\documentclass[a4paper,reqno]{amsart}

\usepackage{graphicx}
\usepackage{tikz}
\usetikzlibrary{patterns,shapes,decorations.pathmorphing,decorations.pathreplacing,calc,arrows}
\usepackage{verbatim}
\usepackage{amssymb}
\usepackage{amstext}
\usepackage{amsmath}
\usepackage{amscd}
\usepackage{amsthm}
\usepackage{amsfonts}
\usepackage{enumerate}
\usepackage{latexsym}
\usepackage{mathrsfs}
\usepackage{hyperref}
\usepackage{mathtools}
\usepackage[all]{xy}
\usepackage{pgfplots}
\usepackage{caption}

\xyoption{all}

\usepackage{pstricks}
\usepackage{lscape}
\usepackage{comment}

\newtheorem{theorem}{Theorem}[section]

\newtheorem{corollary}[theorem]{Corollary}
\newtheorem{lemma}[theorem]{Lemma}
\newtheorem{proposition}[theorem]{Proposition}
\newtheorem{definition-theorem}[theorem]{Definition-Theorem}
\newtheorem{definition-proposition}[theorem]{Definition-Proposition}
\newtheorem{problem}[theorem]{Problem}

\theoremstyle{definition}
\newtheorem{definition}[theorem]{Definition}

\newtheorem{remark}[theorem]{Remark}
\newtheorem{example}[theorem]{Example}

\renewcommand{\AA}{\mathscr{A}}
\newcommand{\CC}{\mathscr{C}}

\newcommand{\FF}{\mathscr{F}}
\newcommand{\KKK}{\mathsf{K}}

\newcommand{\TT}{\mathscr{T}}

\newcommand{\WW}{\mathscr{W}}

\newcommand{\Z}{\mathbb{Z}}
\newcommand{\R}{\mathbb{R}}

\newcommand{\Ann}{\operatorname{Ann}\nolimits}
\newcommand{\bo}{\operatorname{b}\nolimits}

\newcommand{\soc}{\operatorname{soc}\nolimits}
\newcommand{\Ext}{\operatorname{Ext}\nolimits}

\newcommand{\Hom}{\operatorname{Hom}\nolimits}
\newcommand{\rad}{\operatorname{rad}\nolimits}
\newcommand{\End}{\operatorname{End}\nolimits}

\newcommand{\op}{\operatorname{op}\nolimits}

\newcommand{\RHom}{\mathbf{R}\strut\kern-.2em\operatorname{Hom}\nolimits}
\newcommand{\Lotimes}{\mathop{\stackrel{\mathbf{L}}{\otimes}}\nolimits}
\newcommand{\Image}{\operatorname{Im}\nolimits}

\newcommand{\gen}{\operatorname{gen}\nolimits}

\DeclareMathOperator{\moduleCategory}{\mathsf{mod}} \renewcommand{\mod}{\moduleCategory}

\DeclareMathOperator{\proj}{\mathsf{proj}}
\DeclareMathOperator{\inj}{\mathsf{inj}}

\DeclareMathOperator{\Sub}{\mathsf{Sub}}

\DeclareMathOperator{\thick}{\mathsf{thick}}

\DeclareMathOperator{\per}{\mathsf{per}}

\DeclareMathOperator{\add}{\mathsf{add}}

\newcommand{\cut}{\ar@{-}@[|(5)]}

\newcommand{\sttilt}{\mathsf{s\tau\mbox{-}tilt}}
\newcommand{\tors}{\mathsf{tors}}
\newcommand{\torf}{\mathsf{torf}}
\newcommand{\ftors}{\mathsf{f\mbox{-}tors}}
\newcommand{\ftorf}{\mathsf{f\mbox{-}torf}}
\newcommand{\silt}{\mathsf{silt}}
\newcommand{\psilt}{\mathsf{psilt}}

\newcommand{\twosilt}{\mathsf{2\mbox{-}silt}}
\newcommand{\rtwosilt}[1]{\mathsf{2\mbox{$_{#1}$-}silt}}

\newcommand{\wide}{\mathsf{wide}}
\newcommand{\sbrick}{\mathsf{sbrick}}

\newcommand{\flsbrick}{\mathsf{f_L\mbox{-}sbrick}}
\newcommand{\frsbrick}{\mathsf{f_R\mbox{-}sbrick}}
\newcommand{\twopsilt}{\mathsf{2\mbox{-}psilt}}
\newcommand{\rtwopsilt}[1]{\mathsf{2\mbox{$_{#1}$-}psilt}}

\newcommand{\Db}{\mathsf{D}^{\rm b}}
\newcommand{\Kb}{\mathsf{K}^{\rm b}}

\newcommand{\smc}{\mathsf{smc}}
\newcommand{\twosmc}{\mathsf{2\mbox{-}smc}}
\newcommand{\indtwosmc}{\mathsf{ind\mbox{-}2\mbox{-}smc}}
\newcommand{\wPi}{\widetilde{\Pi}}
\newcommand{\I}{\widetilde{I}}

\newcommand{\Lwotimes}{\mathop{{\otimes}^\mathbf{L}_{\widetilde{\Pi}}}\nolimits}

\DeclareMathOperator{\Fac}{\mathsf{Fac}}
\newcommand{\conv}{\operatorname{conv}\nolimits}
\newcommand{\cone}{\operatorname{cone}\nolimits}
\renewcommand{\span}{\operatorname{span}\nolimits}
\numberwithin{equation}{section}

\newcommand{\fT}{{\rm t}}
\newcommand{\fF}{{\rm f}}
\newcommand{\fW}{{\rm w}}
\renewcommand{\P}{{\rm P}}
\renewcommand{\c}{{\rm c}}
\newcommand{\n}{{\rm n}}
\newcommand{\N}{{\rm N}}
\newcommand{\Hasse}{\operatorname{Hasse}\nolimits}

\begin{document}

\title{Fans and polytopes in tilting theory I: Foundations}

\author{Toshitaka Aoki}
\address{Graduate School of Information Science and Technology,
Osaka University, 1-5 Yamadaoka, Suita, Osaka 565-0871, Japan}
\email{aoki-t@ist.osaka-u.ac.jp}
\thanks{}

\author{Akihiro Higashitani}
\address{Department of Pure and Applied Mathematics,
Graduate School of Information Science and Technology,
Osaka University,
1-5 Yamadaoka, Suita, Osaka 565-0871, Japan}
\email{higashitani@ist.osaka-u.ac.jp}
\thanks{}

\author{Osamu Iyama}
\address{Graduate School of Mathematical Sciences,
University of Tokyo,  
3-8-1 Komaba Meguro-ku Tokyo 153-8914, Japan}
\email{iyama@ms.u-tokyo.ac.jp}
\thanks{}

\author{Ryoichi Kase}

\address{Department of Information Science and Engineering, Okayama University of Science, 1-1 Ridaicho, Kita-ku, Okayama 700-0005, Japan}
\email{r-kase@ous.ac.jp}
\thanks{}

\author{Yuya Mizuno}
\address{Faculty of Liberal Arts, Sciences and Global Education / Graduate School of Science, Osaka Metropolitan University, 1-1 Gakuen-cho, Naka-ku, Sakai, Osaka 599-8531, Japan}
\email{yuya.mizuno@omu.ac.jp}

\begin{abstract}
For a finite dimensional algebra $A$ over a field $k$, the 2-term silting complexes of $A$ gives a simplicial complex $\Delta(A)$ called the \emph{$g$-simplicial complex}.
We give tilting theoretic interpretations of the $h$-vectors and Dehn-Sommerville equations of $\Delta(A)$.
Using $g$-vectors of 2-term silting complexes, $\Delta(A)$ gives a nonsingular fan $\Sigma(A)$ in the real Grothendieck group $K_0(\proj A)_\R$ called the \emph{$g$-fan}. 
We give several basic properties of $\Sigma(A)$ including sign-coherence, sign decomposition, idempotent reductions, Jasso reductions, pairwise positivity and a connection with Newton polytopes of $A$-modules. 
Moreover, $\Sigma(A)$ gives a (possibly infinite and non-convex) polytope $\P(A)$ in $K_0(\proj A)_\R$ called the \emph{$g$-polytope} of $A$.
We call $A$ \emph{$g$-convex} if $\P(A)$ is convex. In this case, we show that it is a reflexive polytope, and that the dual polytope is given by the 2-term simple minded collections of $A$.
There are precisely 7 convex $g$-polyogons up to isomorphism. We give a classification of algebras whose $g$-polytopes are smooth Fano. 

We study $g$-fans and $g$-polytopes of two important classes of algebras. We show that the $g$-fan of a classical or generalized preprojective algebra is given by the Coxeter fan. It is $g$-convex if and only if it is of type $A$ or $B$, and in this case, its $g$-polytope is the dual polytope of the short root polytope. 
Moreover we classify Brauer graph algebras which are $g$-convex, and describe their $g$-polytopes as the root polytopes of type $A$ or $C$.
\end{abstract}

\maketitle
\setcounter{tocdepth}{1}
\tableofcontents

\section{Introduction}
The notion of tilting objects is basic to study the structure of a given derived category. The set of partial tilting modules over a finite dimensional algebra has a structure of a simplicial complex, and gives rise to a fan in the Grothendieck group. Their structure has been studied by a number of authors including \cite{RiS,U,Hi1}.
The class of silting objects gives a completion of the class of tilting objects from a point of view of mutation \cite{KV,AI}. Silting objects correspond bijectively with other important objects in the derived category, including (co-)t-structures and simple-minded collections \cite{AIR,IT,KY}.
The subset of 2-term silting complexes enjoys especially nice properties, e.g.\ \cite{An,As1,AMV,BY,DIRRT}. It plays an important role in categorification of cluster algebras of Fomin and Zelevinsky.
The 2-term silting version of the simplicial complex and the fan as well as their applications to cluster algebras have been studied e.g.\ in \cite{Pl2,DF,B,DIJ,BST,As2}.

One of main problems in tilting theory is to classify algebras which are \emph{$g$-finite} (that is, there are only finitely many isomorphism classes of basic 2-term silting complexes, also called as \emph{$\tau$-tilting finite} \cite{DIJ}). Giving an explicit classification up to algebra isomorphism is not a reasonable problem to study (e.g.\ any local algebra is $g$-finite). Thus the Hasse quivers of basic 2-term silting complexes have been studied by a number of researchers. As is pointed out in \cite{DIJ}, the associated fan is a stronger combinatorial invariant than the Hasse quiver. The aim of this paper is to make a systematic use of the simplicial complexes and the fans as essential combinatorial invariants of tilting theory of given algebras.

To explain more details, let $A$ be a finite dimensional algebra $A$ over a field $k$, and let $n:=|A|$, where $|X|$ is the number of non-isomorphic indecomposable direct summands of $X\in\mod A$.
The 2-term presilting complexes of $A$ give a simplicial complex $\Delta(A)$ called the \emph{$g$-simplicial complex}.
Moreover, each 2-term presilting complex $T$  of $A$ gives a simplicial cone in the real Grothendieck group $K_0(\proj A)_\R\simeq\R^n$ generated by the $g$-vectors of the indecomposable direct summands of $T$, and we obtain a nonsingular fan $\Sigma(A)$ in $K_0(\proj A)_\R$ called the \emph{$g$-fan} of $A$.
Each 2-term silting complex $T$ also gives an $n$-simplex $C_{\le1}(T)$ as the convex hull of the origin and the $g$-vectors of the indecomposable direct summands of $T$. Gluing them together, we obtain the \emph{$g$-polytope} $\P(A)$.
We study the $g$-simplicial complex $\Delta(A)$, the $g$-fan $\Sigma(A)$ and the $g$-polytope $\P(A)$ of a finite dimensional algebra $A$ mainly in the case $A$ is $g$-finite. We give some examples of $\P(A)$.

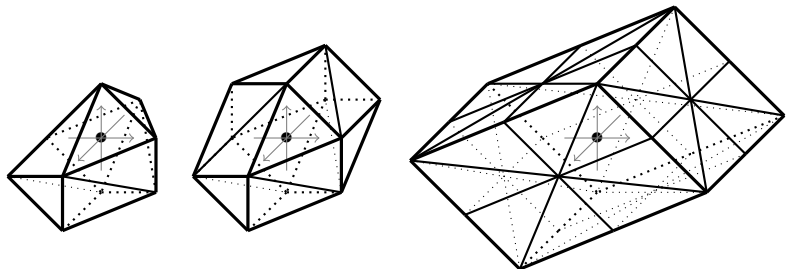
\[\begin{tabular}{cccccccc}
\begin{tikzpicture}[baseline=0mm,scale=0.6]
\node(0) at(0:0) {$\bullet$}; 
\node(x) at(225:1.2) {}; 
\node(-x) at($-1*(x)$) {}; 
\node(y) at(0:1.2) {}; 
\node(-y) at($-1*(y)$) {}; 
\node(z) at(90:1.2) {}; 
\node(-z) at($-1*(z)$) {}; 
\draw[gray, <-] ($0.6*(x)$)--($0.6*(-x)$); \draw[gray, <-] ($0.6*(y)$)--($0.6*(-y)$); \draw[gray, <-] ($0.6*(z)$)--($0.6*(-z)$);

\coordinate(1) at($1*(x) + 0*(y) + 0*(z)$) ;
\coordinate(2) at($0*(x) + 1*(y) + 0*(z)$) ;
\coordinate(3) at($0*(x) + 0*(y) + 1*(z)$) ;

\coordinate(4) at($-1*(x) + 0*(y) + 0*(z)$) ;
\coordinate(5) at($0*(x) + -1*(y) + 0*(z)$) ;
\coordinate(6) at($0*(x) + 0*(y) + -1*(z)$) ;

\coordinate(7) at($0*(x) + 1*(y) + -1*(z)$) ;
\coordinate(8) at($1*(x) + -1*(y) + 0*(z)$) ;
\coordinate(9) at($1*(x) + 0*(y) + -1*(z)$) ;

\draw[very thick] (1)--(2) ;
\draw[very thick] (1)--(3) ;
\draw[very thick] (2)--(3) ;
\draw[thick,dotted] (4)--(5) ;
\draw[thick,dotted] (4)--(6) ;
\draw[thick,dotted] (5)--(6) ;
\draw[thick] (1)--(7) ;
\draw[very thick] (2)--(7) ;
\draw[very thick] (1)--(8) ;
\draw[very thick] (3)--(8) ;
\draw[very thick] (3)--(4) ;
\draw[very thick] (2)--(4) ;
\draw[very thick] (1)--(9) ;
\draw[very thick] (9)--(7) ;
\draw[thick, dotted] (7)--(4) ;
\draw[very thick] (9)--(8) ;
\draw[dotted] (3)--(5) ;
\draw[dotted] (8)--(5) ;
\draw[dotted,thick] (7)--(6) ;
\draw[dotted,thick] (9)--(6) ;
\draw[dotted] (8)--(6) ;
\end{tikzpicture} 
& 
\begin{tikzpicture}[baseline=0mm, scale=0.6]

\node(0) at(0:0) {$\bullet$}; 
\node(x) at(225:1.2) {}; 
\node(-x) at($-1*(x)$) {}; 
\node(y) at(0:1.2) {}; 
\node(-y) at($-1*(y)$) {}; 
\node(z) at(90:1.2) {}; 
\node(-z) at($-1*(z)$) {}; 
\draw[gray, <-] ($0.6*(x)$)--($0.6*(-x)$); \draw[gray, <-] ($0.6*(y)$)--($0.6*(-y)$); \draw[gray, <-] ($0.6*(z)$)--($0.6*(-z)$);

\coordinate(1) at($1*(x) + 0*(y) + 0*(z)$) ; 
\coordinate(2) at($0*(x) + 1*(y) + 0*(z)$) ; 
\coordinate(3) at($0*(x) + 0*(y) + 1*(z)$) ;

\coordinate(4) at($-1*(x) + 0*(y) + 0*(z)$) ;  
\coordinate(5) at($0*(x) + -1*(y) + 0*(z)$) ; 
\coordinate(6) at($0*(x) + 0*(y) + -1*(z)$) ; 

\coordinate(7) at($0*(x) + 1*(y) + -1*(z)$) ; 
\coordinate(8) at($1*(x) + -1*(y) + 0*(z)$) ; 
\coordinate(9) at($-1*(x) + 0*(y) + 1*(z)$) ; 

\coordinate(10) at($1*(x) + 0*(y) + -1*(z)$) ;
\coordinate(11) at($-1*(x) + 1*(y) + 0*(z)$) ;
\coordinate(12) at($0*(x) + -1*(y) + 1*(z)$) ;

\draw[very thick]  (1)--(2);
\draw[very thick]  (1)--(3);
\draw[very thick]  (2)--(3);
\draw[thick, dotted]  (4)--(5);
\draw[thick, dotted]  (4)--(6);
\draw[thick, dotted]  (5)--(6);
\draw[thick]  (1)--(7);
\draw[very thick]  (2)--(7);
\draw[very thick]  (1)--(8);
\draw[thick]  (3)--(8);
\draw[thick]  (9)--(2);
\draw[very thick]  (9)--(3);
\draw[very thick]  (1)--(10);
\draw[very thick]  (10)--(7);
\draw[very thick]  (11)--(2);
\draw[very thick]  (11)--(7);
\draw[very thick]  (10)--(8);
\draw[very thick]  (12)--(3);
\draw[very thick]  (12)--(8);
\draw[very thick]  (11)--(9);
\draw[very thick]  (12)--(9);
\draw[thick, dotted]  (7)--(6);
\draw[thick, dotted]  (10)--(6);
\draw[thick, dotted]  (11)--(4);
\draw[dotted]  (7)--(4);
\draw[dotted]  (8)--(6);
\draw[thick, dotted]  (12)--(5);
\draw[thick, dotted]  (8)--(5);
\draw[thick, dotted]  (9)--(4);
\draw[dotted]  (9)--(5);
\end{tikzpicture}   
&
\begin{tikzpicture}[baseline=0mm, scale=0.6]

\node(0) at(0:0) {$\bullet$}; 
\node(x) at(225:1.2) {}; 
\node(-x) at($-1*(x)$) {}; 
\node(y) at(0:1.2) {}; 
\node(-y) at($-1*(y)$) {}; 
\node(z) at(90:1.2) {}; 
\node(-z) at($-1*(z)$) {}; 
\draw[gray, <-] ($0.6*(x)$)--($0.6*(-x)$); \draw[gray, <-] ($0.6*(y)$)--($0.6*(-y)$); \draw[gray, <-] ($0.6*(z)$)--($0.6*(-z)$);

\coordinate(1) at($1*(x) + 0*(y) + 0*(z)$) ; 
\coordinate(2) at($0*(x) + 1*(y) + 0*(z)$) ; 
\coordinate(3) at($0*(x) + 0*(y) + 1*(z)$) ;

\coordinate(4) at($-1*(x) + 0*(y) + 0*(z)$) ;  
\coordinate(5) at($0*(x) + -1*(y) + 0*(z)$) ; 
\coordinate(6) at($0*(x) + 0*(y) + -1*(z)$) ; 

\coordinate(7) at($0*(x) + 2*(y) + -1*(z)$) ; 
\coordinate(8) at($1*(x) + -1*(y) + 1*(z)$) ; 
\coordinate(9) at($-1*(x) + 1*(y) + 0*(z)$) ; 

\coordinate(10) at($1*(x) + 1*(y) + -1*(z)$) ;
\coordinate(11) at($2*(x) + -2*(y) + 1*(z)$) ;
\coordinate(12) at($0*(x) + -1*(y) + 1*(z)$) ;

\coordinate(13) at($-1*(x) + 0*(y) + 1*(z)$) ;
\coordinate(14) at($0*(x) + 1*(y) + -1*(z)$) ;
\coordinate(15) at($2*(x) + 0*(y) + -1*(z)$) ;

\coordinate(16) at($-1*(x) + 2*(y) + -1*(z)$) ;
\coordinate(17) at($2*(x) + -1*(y) + 0*(z)$) ;
\coordinate(18) at($-2*(x) + 0*(y) + 1*(z)$) ;

\coordinate(19) at($-2*(x) + 2*(y) + -1*(z)$) ;
\coordinate(20) at($1*(x) + -2*(y) + 1*(z)$) ;
\coordinate(21) at($1*(x) + -1*(y) + 0*(z)$) ;

\coordinate(22) at($-2*(x) + 1*(y) + 0*(z)$) ;
\coordinate(23) at($1*(x) + 0*(y) + -1*(z)$) ;
\coordinate(24) at($0*(x) + -2*(y) + 1*(z)$) ;

\coordinate(25) at($-1*(x) + -1*(y) + 1*(z)$) ;
\coordinate(26) at($-1*(x) + 1*(y) + -1*(z)$) ;

\draw[thick]  (1)--(2) ;
\draw[thick]  (1)--(3) ;
\draw[very thick]  (2)--(3) ;
\draw[dotted]  (4)--(5) ;
\draw[dotted]  (4)--(6) ;
\draw[dotted, thick]  (5)--(6) ;
\draw[thick]  (7)--(1) ;
\draw[very thick]  (7)--(2) ;
\draw[thick]  (8)--(1) ;
\draw[very thick]  (8)--(3) ;
\draw[thick]  (9)--(2) ;
\draw[thick]  (9)--(3) ;
\draw[very thick]  (10)--(7) ;
\draw[thick]  (10)--(1) ;
\draw[thick]  (9)--(7) ;
\draw[very thick]  (11)--(8) ;
\draw[thick]  (11)--(1) ;
\draw[thick]  (12)--(8) ;
\draw[thick]  (12)--(3) ;
\draw[thick]  (13)--(9) ;
\draw[very thick]  (13)--(3) ;
\draw[dotted]  (10)--(14) ;
\draw[dotted]  (7)--(14) ;
\draw[very thick]  (15)--(10) ;
\draw[thick]  (15)--(1) ;
\draw[thick]  (16)--(9) ;
\draw[very thick]  (16)--(7) ;
\draw[thick]  (12)--(11) ;
\draw[very thick]  (17)--(11) ;
\draw[thick]  (17)--(1) ;
\draw[thick]  (13)--(12) ;
\draw[very thick]  (18)--(13) ;
\draw[thick]  (18)--(9) ;
\draw[dotted]  (15)--(14) ;
\draw[dotted]  (16)--(14) ;
\draw[very thick]  (17)--(15) ;
\draw[very thick]  (19)--(16) ;
\draw[thick]  (19)--(9) ;
\draw[thick]  (20)--(12) ;
\draw[very thick]  (20)--(11) ;
\draw[dotted]  (17)--(21) ;
\draw[dotted]  (11)--(21) ;
\draw[thick]  (18)--(12) ;
\draw[very thick]  (22)--(18) ;
\draw[thick]  (22)--(9) ;
\draw[dotted, thick]  (15)--(23) ;
\draw[dotted]  (23)--(14) ;
\draw[dotted]  (19)--(14) ;
\draw[dotted]  (15)--(21) ;
\draw[very thick]  (22)--(19) ;
\draw[very thick]  (24)--(20) ;
\draw[thick]  (24)--(12) ;
\draw[dotted]  (20)--(21) ;
\draw[very thick]  (25)--(18) ;
\draw[thick]  (25)--(12) ;
\draw[dotted]  (18)--(4) ;
\draw[dotted]  (22)--(4) ;
\draw[dotted]  (14)--(6) ;
\draw[dotted, thick]  (23)--(6) ;
\draw[dotted]  (23)--(21) ;
\draw[dotted, thick]  (19)--(26) ;
\draw[dotted]  (26)--(14) ;
\draw[dotted]  (19)--(4) ;
\draw[dotted]  (24)--(21) ;
\draw[very thick]  (25)--(24) ;
\draw[dotted]  (25)--(4) ;
\draw[dotted, thick]  (26)--(6) ;
\draw[dotted]  (21)--(6) ;
\draw[dotted]  (26)--(4) ;
\draw[dotted, thick]  (24)--(5) ;
\draw[dotted]  (21)--(5) ;
\draw[dotted]  (24)--(4) ;
\end{tikzpicture} 
\end{tabular}
\]

We denote by $\twosilt A$ the set of isomorphism classes of basic 2-term silting complexes of $A$, which has a natural partial order such that the Hasse quiver $\Hasse(\twosilt A)$ is $n$-regular (see Section \ref{section 2}). The $f$-vector of $\Delta(A)$ gives the number of isomorphism classes of basic 2-term presilting complexes with a fixed number of indecomposable direct summands.
Our first main result gives the following representation theoretic interpretation of the $h$-vector of $\Delta(A)$.

\begin{theorem}[Theorem \ref{h_j}]\label{h_j in intro}
Let $A$ be a finite dimensional algebra over a field $k$ which is $g$-finite, $n:=|A|$ and $(h_0,\ldots,h_n)$ the $h$-vector of $\Delta(A)$. Then, for each $0\le j\le n$, we have
\[h_j=\#\twosilt_jA=\#\sbrick_jA,\]
where $\twosilt_jA$ is the set of isomorphism classes of basic 2-term silting complexes $T$ such that precisely $j$ arrows start at $T$ in $\Hasse(\twosilt A)$, and $\sbrick_jA$ is the set of isomorphism classes of basic semibricks $S$ of $A$ satisfying $|S|=j$.
\end{theorem}

It was shown in \cite{DIRRT} that there is a canonical bijection $\twosilt_1A\simeq\twosilt_{n-1}A$ between join-irreducible elements in $\twosilt A$ and meet-irreducible elements in $\twosilt A$.
We give the following generalization, which categorifies the famous Dehn-Sommerville equations $h_j=h_{n-j}$ for $h$-vectors.

\begin{theorem}[Theorem \ref{symmetry of h_j}]
Let $A$ be a finite dimensional algebra over a field $k$ which is $g$-finite, $n:=|A|$ and $(h_0,\ldots,h_n)$ the $h$-vector of $\Delta(A)$. For $0\le j\le n$, there are canonical bijections
\[\sbrick_jA\simeq\sbrick_{n-j}A\ \mbox{and}\  \twosilt_jA\simeq\twosilt_{n-j}A.\]
In particular, we have $h_j=h_{n-j}$.
\end{theorem}

The $h$-vectors of boundary complexes of simplicial polytopes are known to be unimodal (see \cite[Section 8]{Zi}).
Recently, this was generalized to simplicial spheres \cite{Adi}. As an application, we obtain the unimodality result below. It will be interesting to give a representation theoretic proof.

\begin{corollary}[Corollary \ref{unimodality of h_j}]
Let $A$ be a finite dimensional algebra over a field $k$ which is $g$-finite, and $n:=|A|$. Then we have
\begin{eqnarray*}
&\#\sbrick_1A\le\#\sbrick_2A\le\cdots\le\#\sbrick_{\lfloor\frac{n}{2}\rfloor-1}A\le\#\sbrick_{\lfloor\frac{n}{2}\rfloor}A,&\\
&\#\sbrick_{\lceil\frac{n}{2}\rceil}A\ge\#\sbrick_{\lceil\frac{n}{2}\rceil+1}A\ge\cdots\ge\#\sbrick_{n-1}A\ge\#\sbrick_nA.&
\end{eqnarray*}
\end{corollary}

In Section \ref{section 4}, we study $g$-fans of finite dimensional algebras.
The following are straightforward consequences of known results in tilting theory.

\begin{proposition}[Propositions \ref{characterize g-finite},  \ref{tilting is sign-coherent}, \ref{tilting is ordered}, \ref{g-fan pairwise positive}]
Let $A$ be a finite dimensional algebra over a field $k$ and $n:=|A|$.
\begin{enumerate}[\rm(a)]
\item $\Sigma(A)$ is a  nonsingular fan in $K_0(\proj A)_\R$.
\item Any cone in $\Sigma(A)$ is a face of a cone of dimension $n$.
\item Any cone in $\Sigma(A)$ of dimension $n-1$ is a face of precisely two cones of dimension $n$.
\item $A$ is $g$-finite (or equivalently, $\Sigma(A)$ is finite) if and only if $\Sigma(A)$ is complete.
\item $\Sigma(A)$ is sign-coherent (see Definition \ref{define sign-coherent}), ordered (see Definition \ref{define ordered}), and pairwise positive (see Definition \ref{define pairwise}).
\end{enumerate}
\end{proposition}

For each idempotent $e\in A$ and the corresponding subalgebra $eAe$ of $A$, we regard $K_0(\proj eAe)_\R$ as a subspace of $K_0(\proj A)_\R$.
Thanks to sign-coherence, one can restrict $\Sigma(A)$ to $K_0(\proj eAe)_\R$ to get a subfan. It has the following representation theoretic meaning.

\begin{theorem}[Theorem \ref{eAe}]
There exists an isomorphism of fans
\[\Sigma(eAe)\simeq\{\sigma\in\Sigma(A)\mid \sigma\subset K_0(\proj eAe)_\R\}.\]
\end{theorem}

We also show that the restriction of $\Sigma(A)$ to each orthant can be described by a simpler algebra (Theorem \ref{sign decomposition}) as an analog of the sign decomposition \cite{Ao}.

For each 2-term presilting complex $U$ of $A$, we obtain a new fan $\Sigma(A)/C(U)$ (Definition \ref{define reduction}). On the other hand, there exists a finite dimensional algebra $B$ (called \emph{Jasso reduction} \cite{J}) such that there exist a canonical bijection
\[\{T\in\twosilt A\mid U\in\add T\}\simeq\twosilt B\]
and an isomorphism $K_0(\proj A)_\R/\R C(U)\simeq K_0(\proj B)_\R$. These two constructions are compatible in the following sense.

\begin{theorem}[Theorem \ref{Jasso reduction}]
There exists an isomorphism of fans
\[\Sigma(A)/C(U)\simeq\Sigma(B)\]
induced by a natural isomorphism $K_0(\proj A)_\R/\R C(U)\simeq K_0(\proj B)$.
\end{theorem}

Even in rank 2, there are infinitely many $g$-fans (Example \ref{kase example}). It is interesting to classify all possible $g$-fans in $\R^d$.
In a forthcoming paper \cite{AHIKM1}, we will give a complete answer to the following problem for $d=2$. 

\begin{problem}
Characterize sign-coherent fans in $\R^d$ which can be realized as a $g$-fan of some finite dimensional algebra.
\end{problem}

In Section \ref{section 5}, using the $g$-fans, we introduce the $g$-polytopes $\P(A)$ of finite dimensional algebras $A$.
We study the Ehrhart series ${\rm Ehr}_A(x)$ of $\P(A)$, which is the generating function of the number of isomorphism classes of 2-term presilting complexes of $A$ with at most $\ell$ (possibly isomorphic) indecomposable direct summands. Using the $h$-vector, we will give the following formula.

\begin{theorem}[Theorem \ref{enumerate psilt}]
Let $A$ be a finite dimensional algebra over a field $k$ which is $g$-finite, $n:=|A|$ and $(h_0,\ldots,h_n)$ the $h$-vector of $\Delta(A)$. 
Then the Ehrhart series of $A$ is given by
\[{\rm Ehr}_A(x)=\frac{\sum_{i=0}^nh_ix^i}{(1-x)^{n+1}}.\]
\end{theorem}

Using silting theory, we give explicit connections between the $g$-fan $\Sigma(A)$ and the normal fans $\Sigma(\N(X))$ of the Newton polytopes $\N(X)$ of $A$-modules $X$ (Definitions \ref{define Newton}, \ref{define normal fan}), and between the Hasse quiver $\Hasse(\twosilt A)$ and the 1-skeleton of $\N(X)$. In particular, the following result recovers results in \cite{Fe1}  for $g$-finite case (see also \cite{BCDMTY,PPPP}).

\begin{theorem}[Theorem \ref{normal fan is coarsening}, Corollary \ref{Delzant exist}]
Let $A$ be a finite dimensional algebra over a field $k$ which is $g$-finite. For each $X\in\mod A$, we have
\[\Sigma(\N(X))=\Sigma(A)/\sim_X\ \mbox{ and }\ \overrightarrow{\N}_1(X)\simeq\Hasse(\twosilt A)/\sim_X\]
(see Definition \ref{desine sim_X}). Moreover, there exists $X\in\mod A$ such that $\Sigma(\N(X))=\Sigma(A)$ and $\overrightarrow{\N}_1(X)\simeq\Hasse(\twosilt A)$.
\end{theorem}

The rest of this paper is devoted to study finite dimensional algebras whose $g$-polytopes are convex.
We call $A$ \emph{$g$-convex} if $\P(A)$ is convex. 
We give characterizations of $g$-convexity (Theorem \ref{characterize g-convex}). In particular, $g$-convexity implies 
$g$-finiteness. We introduce the \emph{$c$-polytope} $\P^\c(A)$ by using 2-simple-minded collections (Definition \ref{define c-polytope}).
Using silting-t-structure correspondence, we prove the following result.

\begin{theorem}[Theorem \ref{reflexive polytope}]
Let $A$ be a finite dimensional algebra over a field $k$. Then $A$ is $g$-convex if and only if
\[\P(A)=(\P^\c(A))^*.\]
In this case, both $\P(A)$ and $\P^\c(A)$ are reflexive polytopes.
\end{theorem}

For each $n\ge1$, there exists only finitely many convex $g$-polytopes of dimension $n$ up to isomorphisms of $g$-polytopes (see Definition \ref{define isomorphism of g-fans}, Proposition \ref{finiteness of g-convex fans}). 
It is interesting to know the maximal number of $\#\twopsilt A$ for $g$-convex algebras $A$ with $|A|=n$, see Problem \ref{volume} below.

In Section \ref{section 6}, we give the following classification of convex $g$-polygons by using the well-known list of 16 reflexive polygons \cite{PR}.

\begin{corollary}[Theorem \ref{classify rank2}]
There are precisely 7 convex $g$-polygons up to isomorphism of $g$-polytopes.
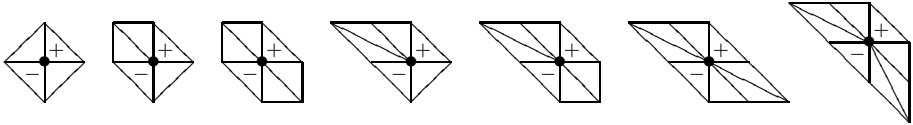
\[{\begin{xy}
0;<3pt,0pt>:<0pt,3pt>::
(0,-5)="0",
(-5,0)="1",
(0,0)*{\bullet},
(0,0)="2",
(5,0)="3",
(0,5)="4",
(1.5,1.5)*{{\scriptstyle +}},
(-1.5,-1.5)*{{\scriptstyle -}},
\ar@{-}"0";"1",
\ar@{-}"1";"4",
\ar@{-}"4";"3",
\ar@{-}"3";"0",
\ar@{-}"2";"0",
\ar@{-}"2";"1",
\ar@{-}"2";"3",
\ar@{-}"2";"4",
\end{xy}}\ \ \ 
{\begin{xy}
0;<3pt,0pt>:<0pt,3pt>::
(0,-5)="0",
(-5,0)="1",
(0,0)*{\bullet},
(0,0)="2",
(5,0)="3",
(-5,5)="4",
(0,5)="5",
(1.5,1.5)*{{\scriptstyle +}},
(-1.5,-1.5)*{{\scriptstyle -}},
\ar@{-}"0";"1",
\ar@{-}"1";"4",
\ar@{-}"4";"5",
\ar@{-}"5";"3",
\ar@{-}"3";"0",
\ar@{-}"2";"0",
\ar@{-}"2";"1",
\ar@{-}"2";"3",
\ar@{-}"2";"4",
\ar@{-}"2";"5",
\end{xy}}\ \ \ 
{\begin{xy}
0;<3pt,0pt>:<0pt,3pt>::
(0,-5)="0",
(5,-5)="1",
(-5,0)="2",
(0,0)*{\bullet},
(0,0)="3",
(5,0)="4",
(-5,5)="5",
(0,5)="6",
(1.5,1.5)*{{\scriptstyle +}},
(-1.5,-1.5)*{{\scriptstyle -}},
\ar@{-}"0";"2",
\ar@{-}"2";"5",
\ar@{-}"5";"6",
\ar@{-}"6";"4",
\ar@{-}"4";"1",
\ar@{-}"1";"0",
\ar@{-}"3";"0",
\ar@{-}"3";"1",
\ar@{-}"3";"2",
\ar@{-}"3";"4",
\ar@{-}"3";"5",
\ar@{-}"3";"6",
\end{xy}}\ \ \ 
{\begin{xy}
0;<3pt,0pt>:<0pt,3pt>::
(0,-5)="0",
(-5,0)="1",
(0,0)*{\bullet},
(0,0)="2",
(5,0)="3",
(-10,5)="4",
(-5,5)="5",
(0,5)="6",
(1.5,1.5)*{{\scriptstyle +}},
(-1.5,-1.5)*{{\scriptstyle -}},
\ar@{-}"0";"3",
\ar@{-}"3";"6",
\ar@{-}"6";"4",
\ar@{-}"4";"0",
\ar@{-}"2";"0",
\ar@{-}"2";"1",
\ar@{-}"2";"3",
\ar@{-}"2";"4",
\ar@{-}"2";"5",
\ar@{-}"2";"6",
\end{xy}}\ \ \ 
{\begin{xy}
0;<3pt,0pt>:<0pt,3pt>::
(0,-5)="0",
(5,-5)="1",
(-5,0)="2",
(0,0)*{\bullet},
(0,0)="3",
(5,0)="4",
(-10,5)="5",
(-5,5)="6",
(0,5)="7",
(1.5,1.5)*{{\scriptstyle +}},
(-1.5,-1.5)*{{\scriptstyle -}},
\ar@{-}"0";"1",
\ar@{-}"1";"4",
\ar@{-}"4";"7",
\ar@{-}"7";"5",
\ar@{-}"5";"0",
\ar@{-}"3";"0",
\ar@{-}"3";"1",
\ar@{-}"3";"2",
\ar@{-}"3";"4",
\ar@{-}"3";"5",
\ar@{-}"3";"6",
\ar@{-}"3";"7",
\end{xy}}\ \ \ 
{\begin{xy}
0;<3pt,0pt>:<0pt,3pt>::
(0,-5)="0",
(5,-5)="1",
(10,-5)="2",
(-5,0)="3",
(0,0)*{\bullet},
(0,0)="4",
(5,0)="5",
(-10,5)="6",
(-5,5)="7",
(0,5)="8",
(1.5,1.5)*{{\scriptstyle +}},
(-1.5,-1.5)*{{\scriptstyle -}},
\ar@{-}"0";"2",
\ar@{-}"2";"8",
\ar@{-}"8";"6",
\ar@{-}"6";"0",
\ar@{-}"4";"0",
\ar@{-}"4";"1",
\ar@{-}"4";"2",
\ar@{-}"4";"3",
\ar@{-}"4";"5",
\ar@{-}"4";"6",
\ar@{-}"4";"7",
\ar@{-}"4";"8",
\end{xy}}
{\begin{xy}
0;<3pt,0pt>:<0pt,3pt>::
(0,-2.5)="0",
(5,-7.5)="1",
(5,-2.5)="2",
(-5,2.5)="3",
(0,2.5)*{\bullet},
(0,2.5)="4",
(5,2.5)="5",
(-10,7.5)="6",
(0,7.5)="7",
(-5,7.5)="8",
(1.5,4)*{{\scriptstyle +}},
(-1.5,1)*{{\scriptstyle -}},
\ar@{-}"6";"1",
\ar@{-}"1";"5",
\ar@{-}"5";"7",
\ar@{-}"7";"6",
\ar@{-}"4";"0",
\ar@{-}"4";"1",
\ar@{-}"4";"2",
\ar@{-}"4";"3",
\ar@{-}"4";"5",
\ar@{-}"4";"6",
\ar@{-}"4";"7",
\ar@{-}"4";"8",
\end{xy}}
\]
The corresponding $c$-polygons are the following.
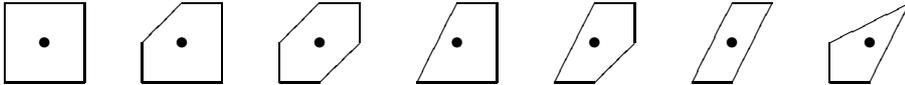
\[
{\begin{xy}
0;<3pt,0pt>:<0pt,3pt>::
(-5,-5)="0",
(5,-5)="1",
(0,0)*{\bullet},
(5,5)="2",
(-5,5)="3",
\ar@{-}"0";"1",
\ar@{-}"1";"2",
\ar@{-}"2";"3",
\ar@{-}"3";"0",
\end{xy}}\ \ \ \ \ \ 
{\begin{xy}
0;<3pt,0pt>:<0pt,3pt>::
(-5,-5)="0",
(5,-5)="1",
(0,0)*{\bullet},
(5,5)="2",
(0,5)="3",
(-5,0)="4",
\ar@{-}"0";"1",
\ar@{-}"1";"2",
\ar@{-}"2";"3",
\ar@{-}"3";"4",
\ar@{-}"4";"0",
\end{xy}}\ \ \ \ \ \ 
{\begin{xy}
0;<3pt,0pt>:<0pt,3pt>::
(-5,-5)="0",
(0,-5)="1",
(5,0)="2",
(0,0)*{\bullet},
(5,5)="3",
(0,5)="4",
(-5,0)="5",
\ar@{-}"0";"1",
\ar@{-}"1";"2",
\ar@{-}"2";"3",
\ar@{-}"3";"4",
\ar@{-}"4";"5",
\ar@{-}"5";"0",
\end{xy}}\ \ \ \ \ \ 
{\begin{xy}
0;<3pt,0pt>:<0pt,3pt>::
(-5,-5)="0",
(5,-5)="1",
(5,5)="2",
(0,0)*{\bullet},
(0,5)="3",
\ar@{-}"0";"1",
\ar@{-}"1";"2",
\ar@{-}"2";"3",
\ar@{-}"3";"0",
\end{xy}}\ \ \ \ \ \ 
{\begin{xy}
0;<3pt,0pt>:<0pt,3pt>::
(-5,-5)="0",
(0,-5)="1",
(5,0)="2",
(5,5)="3",
(0,0)*{\bullet},
(0,5)="4",
\ar@{-}"0";"1",
\ar@{-}"1";"2",
\ar@{-}"2";"3",
\ar@{-}"3";"4",
\ar@{-}"4";"0",
\end{xy}}\ \ \ \ \ \ 
{\begin{xy}
0;<3pt,0pt>:<0pt,3pt>::
(-5,-5)="0",
(0,-5)="1",
(5,5)="2",
(0,0)*{\bullet},
(0,5)="3",
\ar@{-}"0";"1",
\ar@{-}"1";"2",
\ar@{-}"2";"3",
\ar@{-}"3";"0",
\end{xy}}\ \ \ \ \ \ 
{\begin{xy}
0;<3pt,0pt>:<0pt,3pt>::
(-5,-5)="0",
(0,-5)="1",
(5,5)="2",
(0,0)*{\bullet},
(-5,0)="3",
\ar@{-}"0";"1",
\ar@{-}"1";"2",
\ar@{-}"2";"3",
\ar@{-}"3";"0",
\end{xy}}
\]
\end{corollary}

It is well-known that there are 4319 reflexive polytopes in dimension 3 \cite{KS}. It is interesting to know which one can be realized as a $g$-polytope. In a forthcoming paper \cite{AHIKM2}, we will give a complete answer to the following problem for  $d=3$.

\begin{problem}
Classify convex $g$-polytopes in $\R^d$.
\end{problem}

We also give a classification of algebras whose $g$-polytopes are smooth Fano (Definition \ref{define smooth Fano}), a much stronger notion than convexity.

\begin{theorem}[Theorem \ref{smooth Fano tilting polytope}]
Let $A$ be a finite dimensional algebra over a field $k$.
Then $\P(A)$ is a smooth Fano polytope if and only if $A$ is a product of local algebras, algebras of pentagon type and algebras of hexagon type.
\end{theorem}

In Sections \ref{section 7}--\ref{section 9}, we give explicit descriptions of the $g$-fans and/or $g$-polytopes of certain important classes of algebras.
In Section \ref{section 7}, we describe $g$-polytopes for classical and generalized preprojective algebras due to Geiss-Leclerc-Schr\"oer \cite{GLS} by using the root polytope. 

\begin{theorem}[Theorem \ref{thm p.p. polytope}]
Let $\Pi$ be a classical or generalized preprojective algebra of Dynkin type.
\begin{enumerate}[\rm(a)]
\item $\Sigma(\Pi)$ is the Coxeter fan.
\item $\Pi$ is $g$-convex if and only if it is either of type $A_n$ or $B_n$. In this case, $\P(\Pi)$ is the dual polytope of the short root polytope of type $A_n$ or $B_n$ respectively.
\end{enumerate}
\end{theorem}

In particular, $\#\twosilt\Pi$ is the order of the Weyl group. In type $B_n$, it is $2^nn!$ and hence the volume of $\P(\Pi)$ is $2^n$. The following is a list of natural questions to study.

\begin{problem}\label{volume}
\begin{enumerate}[\rm(a)]
\item Is the volume of convex $g$-polytopes of rank $n$ at most $2^n$?
\item Classify symmetric convex $g$-polytopes.
\end{enumerate}
\end{problem}

Note that the part (a) is true for centrally symmetric convex polytopes by Minkowski's convex body Theorem \cite[Chapter III]{C}.

In Section \ref{section 9}, we study $g$-polytopes of Brauer graph algebras. In particular, we give the following characterization of $g$-convex Brauer graph algebras.

\begin{theorem}[Theorem \ref{BTA-BOA-Phi}]
    Let $\Gamma$ be a connected Brauer graph with $n$ edges and $B_{\Gamma}$ the Brauer graph algebra associated to $\Gamma$. 
    \begin{enumerate}[\rm(a)]
    \item $B_{\Gamma}$ is $g$-finite if and only if $B_{\Gamma}$ is $g$-convex if and only if $\Gamma$ is either a Brauer tree or a Brauer odd-cycle.     
    \item If $\Gamma$ is a Brauer tree (respectively, Brauer odd-cycle), then $\P(B_{\Gamma})$ is the root polytope of type $A_n$ (respectively, $C_n$).
    \end{enumerate}
\end{theorem}

Most results in this paper are valid in a more general setting as in the following remark, 
but for simplicity, we will work on finite dimensional algebras over fields throughout this paper.

\begin{remark}\label{general setting}
Let $A$ be a dg algebra over a field $k$. Assume that the following conditions hold.
\begin{enumerate}[{\rm(i)}]
\item  $A$ is \emph{non-positive}, that is, $H^i(A)=0$ holds for each $i\ge1$.
\item $\dim_kH^0(A)<\infty$.
\end{enumerate}
Then, for a finite dimensional algebra $B:=H^0(A)$, we have 
\[\Delta(A)=\Delta(B),\ \Sigma(A)=\Sigma(B)\ \mbox{ and }\ \P(A)=\P(B),\]
see Proposition \ref{dg case} for a more general result stated in terms of algebraic triangulated categories

Note that one can weaken the condition (ii) above by using results in \cite{Kim}.
\end{remark}

\noindent{\bf Conventions }
All modules are right modules. The composition of morphisms $f:X\to Y$ and $g:Y\to Z$ is denoted by $gf:X\to Z$. The composition of arrows $a:i\to j$ and $b:j\to k$ in a quiver is denoted by $ab:i\to k$.
Let $A$ be a finite dimensional algebra over a field $k$.
We denote by $\mod A$ the category of finitely generated right $A$-modules and by $\proj A$ the category of finitely generated projective $A$-modules.
We denote by $\Db(\mod A)$ the bounded derived category of $\mod A$ and by $\Kb(\proj A)$ the bounded homotopy category of $\proj A$. 
For an object $X$ in a Krull-Schmidt category (e.g.\ $\mod A$, $\Kb(\proj A)$, $\Db(\mod A)$), we denote by $|X|$ the number of non-isomorphic indecomposable direct summands of $X$.

\section{Preliminaries}\label{section 2}

\subsection{Preliminaries on fans and polytopes}

We recall some fundamental materials on fans and polytopes. 
We refer the reader to e.g.\ \cite{F,BR,BP} for these materials. 

\medskip
A \textit{convex polyhedral cone} $\sigma$ is a set of the form $\sigma=\{ \sum_{i=1}^s r_i v_i \mid r_i \geq 0\}$, 
where $v_1,\ldots,v_s \in \mathbb{R}^d$. We denote it by $\sigma=\cone\{v_1,\ldots,v_s\}$. 
Note that $\{0\}$ is regarded as a convex polyhedral cone. 
We collect some notions concerning convex polyhedral cones. Let $\sigma$ be a convex polyhedral cone. 
\begin{enumerate}[$\bullet$]
\item The dimension of $\sigma$ is the dimension of the linear space generated by $\sigma$. 
\item We say that $\sigma$ is \textit{strongly convex} if $\sigma \cap (-\sigma)=\{0\}$ holds, 
i.e., $\sigma$ does not contain a linear subspace of positive dimension. 
\item We call $\sigma$ \textit{rational} if each $v_i$ can be taken from $\mathbb{Q}^d$. 
\item We denote by $\langle \cdot, \cdot \rangle$ the usual inner product. A \emph{supporting hyperplane} of $\sigma$ is a hyperplane $\{v \in \sigma \mid \langle u,v \rangle=0\}$ in $\R^d$ given by some $u \in \mathbb{R}^d$ satisfying $\sigma\subset\{v \in \R^d \mid \langle u,v \rangle\ge0\}$.
\item A \textit{face} $\tau$ of $\sigma$ is the intersection of $\sigma$ with a supporting hyperplane of $\sigma$.
\item If a face $\tau$ is maximal with respect to the inclusion, then $\tau$ is called a \textit{facet} of $\sigma$.
\end{enumerate}

\emph{In what follows, a cone means a strongly convex rational polyhedral cone for short.}

\begin{definition}
A \textit{fan} $\Sigma$ in $\mathbb{R}^d$ is a collection of cones in $\mathbb{R}^d$ such that 
\begin{enumerate}[\rm(a)]
\item each face of a cone in $\Sigma$ is also contained in $\Sigma$, and 
\item the intersection of two cones in $\Sigma$ is a face of each of those two cones. 
\end{enumerate}
In this case, for each $0\le i\le d$, we denote by $\Sigma_i$ the subset of cones of dimension $i$.
For example, $\Sigma_0$ consists of the trivial cone $\{0\}$. 
We call each element in $\Sigma_1$ a \emph{ray} of $\Sigma$.
\end{definition}

We collect some notions concerning fans used in this paper. Let $\Sigma$ be a fan in $\mathbb{R}^d$. 
\begin{enumerate}[$\bullet$]
\item We call $\Sigma$ \textit{finite} if it consists of a finite number of cones. 
\item We call $\Sigma$ \textit{complete} if $\bigcup_{\sigma \in \Sigma}\sigma=\mathbb{R}^d$. 
\item We call $\Sigma$ \textit{nonsingular} (or \textit{smooth}) if each maximal cone in $\Sigma$ is generated by a $\Z$-basis for $\Z^d$. 
\end{enumerate}

We prepare some notions which will be used in this paper.

\begin{definition}\label{define isomorphism of fans}
Let $\Sigma$ and $\Sigma'$ be fans in $\R^d$ and $\R^{d'}$ respectively. An \emph{isomorphism $\Sigma\simeq\Sigma'$ of fans} is an isomorphism $\Z^d\simeq\Z^{d'}$ of abelian groups such that the induced linear isomorphism $\R^d\to\R^{d'}$ gives a bijection $\Sigma\simeq\Sigma'$ between cones.
\end{definition}

\begin{definition}\label{define coarsening fan}
Let $\Sigma$ be a finite fan in $\R^d$, and let $\sim$ be an equivalence relation on $\Sigma_d$.
We say that \emph{$\sim$ coarsens $\Sigma$} if, for each $\sigma\in\Sigma_d$, the set $[\sigma]:=\bigcup_{\tau\sim\sigma}\tau$ is convex. In this case, we define a fan $\Sigma/\sim$ called the \emph{coarsening} of $\Sigma$ by
\[\Sigma/\sim\,:=\{[\sigma_1]\cap\cdots\cap[\sigma_s]\mid s\ge1,\ \sigma_1,\ldots,\sigma_s\in\Sigma_d\}.\]
\end{definition}

A \emph{polytope} $P$ is a convex hull of a finite subset $K$ of $\mathbb{R}^d$. It is called \emph{lattice polytope} if $K$ is contained in $\mathbb{Z}^d$.  
A \textit{supporting hyperplane} of $P$ is a hyperplane $\{v \in P\mid\langle u, v \rangle = a\}$ in $\R^d$ given by some $u \in \R^d$ and $a \in \R$ satisfying $P \subsetneq \{v \in \R^d\mid\langle u, v\rangle \geq a\}$. A \textit{face} of $P$ is the intersection of $P$ with a supporting hyperplane of $P$. A maximal face is called a \textit{facet} of $P$.

For lattice polytopes $P$ and $P'$ in $\Z^d$ and $\Z^{d'}$ respectively, an \emph{isomorphism $P\simeq P'$ of lattice polytopes} is an isomorphism $\Z^d\simeq\Z^{d'}$ of abelian groups such that the induced linear isomorphism $\R^d\to\R^{d'}$ gives a bijection $P\simeq P'$.

\begin{definition}\label{define normal fan}
Let $P$ be a polytope in $V=\R^d$, and $V^*$ the dual space of $V$. For each face $F$ of $P$, let $F_1,\ldots,F_s$ be all facets of $P$ containing $F$. For each $1\le i\le s$, let $v_i\in V^*$ be an outer normal vector of $F_i$, and let
\[\sigma_F:=\cone\{v_1,\ldots,v_s\}.\]
The \emph{normal fan} of $P$ is
\[\Sigma(P):=\{\sigma_F\mid\mbox{$F$ is a face of $P$}\}.\]
\end{definition}

If $P$ has dimension $d$, then $\Sigma(P)$ is a finite complete fan in $\R^d$. Otherwise, the cones of $\Sigma(P)$ are not strongly convex.

Each element $f\in V^*$ gives a face of $P$:
\[P_f:=\{v\in P\mid f(v)=\max f(P)\}.\]
For each face $F$ of $P$, the corresponding cone $\sigma_F\in\Sigma(P)$ can be written as
\[\sigma_F=\{f\in V^*\mid P_f\supseteq F\}.\]
Then each $f\in\sigma_F^\circ$ satisfies $P_f=F$.

\subsection{Preliminaries on tilting theory}
We recall basic results on silting theory from \cite{AI,AIR}. We refer to \cite{AI,BY} for mutation in more general setting.
First we recall the definition of 2-term silting objects/complexes.

\begin{definition}\label{define silting}
Let $\TT$ be a Krull-Schmidt triangulated category 
\begin{enumerate}[\rm(a)]
\item An object $T\in\TT$ is called  \emph{presilting} if $\Hom_{\TT}(T,T[\ell])=0$ for all positive integers $\ell$.
\item An object $T\in\TT$ is called \emph{silting} if it is presilting and $\TT=\thick T$. 
\item We denote by $\silt\TT$ (respectively, $\psilt\TT$) the set of isomorphism classes of basic silting (respectively, presilting) objects of $\TT$. 
\item For $T,U\in\silt\TT$, we write $T\ge U$ if $\Hom_{\TT}(T,U[\ell])=0$ holds for all positive integers $\ell$. Then $(\silt\TT,\ge)$ is a partially ordered set \cite{AI}.
\item For $T\in\silt\TT$, let
\begin{eqnarray*}
\rtwosilt{T}\TT&:=&\{U\in\silt\TT\mid T\ge U\ge T[1]\},\\ 
\rtwopsilt{T}\TT&:=&\{V\in\TT\mid \exists U\in\rtwosilt{T}\TT\mbox{ such that }V\in\add U\},
\end{eqnarray*}
\item For a ring $A$, let $\TT:=\Kb(\proj A)$ and
\[\silt A:=\silt\TT,\ \psilt A:=\psilt\TT,\ \twosilt A:=\rtwosilt{A}\TT\ \mbox{ and }\ \twopsilt A:=\rtwopsilt{A}\TT.\]
We apply the same definitions for a non-positive dg ring $A$ and $\TT:=\per A$.
\end{enumerate}
\end{definition}

Note that $\twopsilt A$ consists of \emph{2-term} complexes $T=(T^i,d^i)$, that is, $T^i = 0$ for all $i\not= 0,-1$. Moreover, $T\in\twopsilt A$ is silting if and only if $|T|=|A|$ holds.

Later we use the following basic fact.

\begin{proposition}{\cite[Theorem A.7]{BY}}\label{dg case}
Let $\TT$ be a Krull-Schmidt algebraic triangulated category. 
For $T\in\silt\TT$, let $A:=\End_{\TT}(T)$. Then there exists a triangle functor $F:\TT\to\Kb(\proj A)$ which sends $T$ to $A$ and gives an isomorphism $K_0(\TT)\simeq K_0(\proj A)$ and bijections
\[\rtwosilt{T}\TT\simeq\twosilt A\ \mbox{ and }\ \rtwopsilt{T}\TT\simeq\twopsilt A.\]
Moreover, the bijection $\rtwosilt{T}\TT\simeq\twosilt A$ commutes with mutation, and the functor $F$ sends exchange triangles in $\rtwosilt{T}\TT$ to those in $\twosilt A$.
\end{proposition}

\begin{proof}
This is \cite[Theorem A.7]{BY}. The last assertion follows from \cite[Proposition A.6]{BY} and its proof.\end{proof}

A typical setting of Proposition \ref{dg case} is the following.

\begin{example}
Let $A$ be a non-positive dg ring $A$, and $B:=H^0(A)$. 
Then there exists a triangle functor $F:\per A\to\Kb(\proj B)$ which sends $A$ to $B$ and gives an isomorphism $K_0(\per A)\simeq K_0(\proj B)$ and bijections
\[\twosilt A\simeq\twosilt B\ \mbox{ and }\ \twopsilt A\simeq\twopsilt B.\]
\end{example}

In the rest, let $A$ be a finite dimensional algebra over a field $k$. The subposet $(\twosilt A,\ge)$ of $(\silt A,\ge)$ plays a central role in this paper.

Recall that the \emph{Hasse quiver} $\Hasse P$ of a poset $P$ has the set $P$ of vertices, and an arrow $x\to y$ if $x>y$ and there does not exist $z\in P$ satisfying $x>z>y$. 
It is known that $\Hasse(\twosilt A)$ is $n$-regular for $n:=|A|$. More precisely, let $T=T_1\oplus\cdots\oplus T_n\in\twosilt A$ with indecomposable $T_i$. For each $1\le i\le n$, there exists precisely one $T'\in\twosilt A$ such that $T'=T'_i\oplus(\bigoplus_{j\neq i}T_j)$ for some $T'_i\neq T_i$. In this case, we call $T'$ \emph{mutation} of $T$ at $T_i$ and write
\[T'=\mu_i(T).\]
In this case, either $T>T'$ or $T'<T$ holds. The following result is fundamental in silting theory. 
 
\begin{proposition}\label{mutation=exchange}
Let $A$ be a finite dimensional algebra over a field $k$, and $T,T'\in\twosilt A$. Take a decomposition $T=T_1\oplus\cdots\oplus T_n$ with indecomposable $T_i$. Then the following conditions are equivalent.
\begin{enumerate}[\rm(a)]
\item $T>T'$, and $T$ and $T'$ are mutation of each other.
\item There is an arrow $T\to T'$ in $\Hasse(\twosilt A)$.
\item $T'=T'_i\oplus(\bigoplus_{j\neq i}T_j)$ and there is a triangle
\[T_i\xrightarrow{f} U_i\to T'_i\to T_i[1]\]
such that $f$ is a minimal left $(\add \bigoplus_{j\neq i}T_j)$-approximation.
\item $T'=T'_i\oplus(\bigoplus_{j\neq i}T_j)$ and there is a triangle 
\[T_i\to U_i\xrightarrow{g} T'_i\to T_i[1]\]
such that $g$ is a minimal right $(\add \bigoplus_{j\neq i}T_j)$-approximation.
\end{enumerate}
The triangles in (c) and (d) are isomorphic, and called an \emph{exchange triangle}.
\end{proposition}

The class of support $\tau$-tilting modules complements the class of classical tilting modules.

\begin{definition}
Let $A$ be a finite dimensional algebra over a field $k$, $\tau$ the Auslander-Reiten translation of $A$, and $X\in\mod A$.
\begin{enumerate}[\rm(a)]
\item $X$ is called \emph{$\tau$-rigid} if $\Hom_A(X,\tau X)=0$.
\item $X$ is called \emph{$\tau$-tilting} if it is $\tau$-rigid and satisfies $|X|=|A|$.
\item $X$ is called \emph{support $\tau$-tilting} if there exists an idempotent $e\in A$ such that $X$ is a $\tau$-tilting $A/(e)$-module.
\item We denote by $\sttilt A$ the set of isomorphism classes of basic support $\tau$-tilting $A$-modules.
\end{enumerate}
\end{definition}

We often identify 2-term silting complexes with support $\tau$-tilting modules via the following bijection.

\begin{proposition}\cite[Theorem 3.2]{AIR}
Let $A$ be a finite dimensional algebra over a field $k$. Then there exist bijections
\[\twosilt A\simeq\sttilt A\ \mbox{ given by }T\mapsto H^0(T).\]
\end{proposition}

For a class $\CC$ of objects in $\mod A$, let
\begin{eqnarray*}
&\CC^\perp:=\{X\in\mod A\mid\Hom_A(\CC,X)=0\},&\\
&{}^\perp\CC:=\{X\in\mod A\mid\Hom_A(X,\CC)=0\}.&
\end{eqnarray*}
There is a strong connection between 2-term silting complexes and some important subcategories defined as follows. See section \ref{section 3} for more details.

\begin{definition}
Let $A$ be a finite dimensional algebra over a field $k$. A full subcategory $\CC$ of $\mod A$ is called a \emph{torsion class} (respectively, \emph{torsionfree class}) if it is closed under extensions and factor modules (respectively, submodules).
It is called \emph{functorially finite} if there exists $M\in\CC$ satisfying $\CC=\Fac M$ (respectively, $\CC=\Sub M$).
We denote by $\tors A$ (respectively, $\ftors A$, $\torf A$, $\ftorf A$) the set of torsion classes (respectively, functorially finite torsion classes, torsionfree classes, functorially finite torsionfree classes) in $\mod A$.

We have mutually inverse bijections
\begin{equation}\label{TF bijection}
\tors A\simeq\torf A,\ \CC\mapsto\CC^\perp\ \mbox{ and }\ \torf A\simeq\tors A,\ \CC\mapsto{}^\perp\CC.
\end{equation}
A pair $(\TT,\FF)\in\tors A\times\torf A$ is called a \emph{torsion pair} if $\FF=\TT^\perp$, or equivalently, $\TT={}^\perp\FF$. In this case, we have functors
\[\fT_\TT:\mod A\to\TT\ \mbox{ and }\ \fF_\FF:\mod A\to\FF\]
such that each $X\in\mod A$ admits an exact sequence
\begin{equation}\label{canonical exact}
0\to\fT_\TT X\to X\to \fF_\FF X\to0.
\end{equation}
\end{definition}

The bijections \eqref{TF bijection} restrict to bijections \cite{Sm}
\[\ftors A\simeq\ftorf A\ \mbox{ and }\ \ftorf A\simeq\ftors A.\]
The following bijection is also important.

\begin{definition-proposition}{\cite[Proposition 1.2(b), Lemma 3.4]{AIR}}\label{surjection to f-tors}
Let $A$ be a finite dimensional algebra over a field $k$.
\begin{enumerate}[\rm(a)]
\item We have surjections
\begin{eqnarray*}
&\twopsilt A\to\ftors A,\ U\mapsto\TT_U:=\Fac H^0(U),&\\
&\twopsilt A\to\ftors A,\ U\mapsto\overline{\TT}_U:={}^\perp H^{-1}(\nu U),&\\
&\twopsilt A\to\ftorf A,\ U\mapsto\FF_U:=\Sub H^{-1}(\nu U),&\\
&\twopsilt A\to\ftorf A,\ U\mapsto\overline{\FF}_U:=H^0(U)^\perp&
\end{eqnarray*}
such that $(\TT_U,\overline{\FF}_U)$ and $(\overline{\TT}_U,\FF_U)$ form torsion pairs. Thus each $X\in\mod A$ admits exact sequences
\begin{eqnarray*}
\mbox{$0\to\fT_U X\to X\to \overline{\fF}_U X\to0$ for $\fT_UX:=\fT_{\TT_U}X$ and $\overline{\fF}_UX:=\fF_{\overline{\FF}_U}X$,}\\
\mbox{$0\to\overline{\fT}_U X\to X\to\fF_U X\to0$ for $\overline{\fT}_UX:=\fT_{\overline{\TT}_U}X$ and $\fF_UX:=\fF_{\FF_U}X$.}
\end{eqnarray*}
\item We regard $\ftors A$ and $\ftorf A$ as posets with respect to the inclusion relation. Then
the first two surjections in (a) restrict to the same isomorphism of posets
\[\twosilt A\simeq\ftors A,\ T\mapsto\TT_T=\overline{\TT}_T.\]
The last two surjections in (a) restrict to the same anti-isomorphism of posets
\[\twosilt A\simeq\ftorf A,\ T\mapsto\FF_T=\overline{\FF}_T.\]
\end{enumerate}
\end{definition-proposition}

The following finiteness condition plays a central role in this paper.

\begin{definition}
Let $A$ be a finite dimensional algebra over a field $k$.
We say that $A$ is \emph{$g$-finite} if $\#\twosilt A<\infty$. (This is called \emph{$\tau$-tilting finite} in \cite{DIJ}.)
\end{definition}

We give a characterization of $g$-finiteness.

\begin{proposition}{\cite[Theorem 1.2]{DIJ}}\label{characterize tautilt-finite}
Let $A$ be a finite dimensional algebra over a field $k$. Then $A$ is $g$-finite if and only if $\tors A=\ftors A$ if and only if $\torf A=\ftorf A$.
\end{proposition}

There is a strong connection between 2-term silting complexes and the following class of modules. 

\begin{definition}
Let $A$ be a finite dimensional algebra over a field $k$.
\begin{enumerate}[\rm(a)]
\item An object $X=X_1\oplus\cdots\oplus X_r\in\mod A$ is called a \emph{semibrick} if
\begin{equation}\label{define semibrick}
\Hom_A(X_i,X_j)=
\begin{cases}
\textnormal{division ring}&(i=j)\\
0&(j\neq i).
\end{cases}
\end{equation}
We denote by $\sbrick A$ the set of isomorphism classes of semibricks in $\mod A$.
\item A full subcategory $\CC$ of $\mod A$ is called \emph{wide} if it is closed under extensions, kernels and cokernels. We denote by $\wide A$ the set of wide subcategories of $\mod A$.
\end{enumerate}
\end{definition}

Note that the usual definition of a semibrick is more general: a (possibly infinite) set of modules satisfying \eqref{define semibrick} \cite{As1}. In this paper, we only need to consider semibricks in the sense above. It is basic that there is a bijection
\[\sbrick A\simeq\wide A\]
sending $X$ to the smallest extension closed subcategory containing $\add X$ \cite{Ri}.

The following notion is a derived category version of semibricks.

\begin{definition}
Let $A$ be a finite dimensional algebra over a field $k$. An object $X=X_1\oplus\cdots\oplus X_r\in\Db(\mod A)$ is called a \emph{simple-minded collection} if the following conditions hold.
\begin{enumerate}[$\bullet$]
\item $\Hom_{\Db(\mod A)}(X,X[\ell])=0$ for all negative integers $\ell$. 
\item For $1\leq i,j\leq r$, $\Hom_{\Db(\mod A)}(X_i,X_j)=
\begin{cases}
\textnormal{division ring}&(i=j)\\
0&(j\neq i).
\end{cases}$ 
\item $\Db(\mod A)=\thick X$.
\end{enumerate}
Note that $r=|A|$ holds in this case. A simple-minded collection $X$ is called \emph{2-term} if $H^i(X)=0$ holds for all integers $i\neq -1,0$. We denote by $\smc A$ (respectively, $\twosmc A$) the set of isomorphism classes of simple-minded collections (respectively, 2-term simple-minded collections) on $\Db(\mod A)$.
\end{definition}

We have the following silting-t-structure correspondence. 

\begin{proposition}\cite{KY}\label{KY thm}
Let $A$ be a finite dimensional algebra over a field $k$ and $n=|A|$. Then there exists a bijection between $\silt A$ and $\smc A$ such that $T = T_1\oplus\cdots\oplus T_n\in\silt A$ and the corresponding $S=S_1\oplus\cdots\oplus S_n\in\smc A$ satisfy
\begin{equation}\label{TX duality}
\Hom_{\Db(\mod A)}(T_i,S_j[p])=\begin{cases}
\End_{\Db(\mod A)}(S_i)&(i=j\ \textnormal{and}\ p=0)\\
0&(\textnormal{otherwise}).
\end{cases}
\end{equation}
In particular, we have
\begin{equation}\label{TS dual}
([T_i],[S_j])=\delta_{ij}\cdot\dim_k\End_{\Db(\mod A)}(S_j).
\end{equation}
\end{proposition}


\section{$g$-simplicial complexes}\label{section 3}
In this section, we introduce $g$-simplicial complexes and study their basic properties. In particular, we give a representation theoretic interpretations of their $h$-vectors. Moreover we give a proof of Dehn-Sommerville equations in terms of the representation theory.

Throughout this section, let $A$ be a finite dimensional algebra over a field $k$.

\begin{definition}\label{psilt^j}
For $j\ge0$, let $\twopsilt^jA$ be the set of isomorphism classes of basic 2-term presilting complexes $T$ such that $|T|=j$. 
We define a simplicial complex $\Delta(A)$ called the \emph{$g$-simplicial complex} of $A$ as follows: The set of $j$-simplices is $\twosilt^{j+1}A$.
\end{definition}

We give the following basic properties.

\begin{proposition}\label{properties of Delta}
Let $A$ be a finite dimensional algebra over a field $k$ and $n:=|A|$.
\begin{enumerate}[\rm(a)]
\item $\Delta(A)$ is pure of dimension $n-1$, that is, all facets of $\Delta(A)$ have dimension $n-1$.
\item Each face of dimension $n-2$ in $\Delta(A)$ is contained in precisely two facets.
\item $\Delta(A)$ is flag, that is, each minimal subset which is not a face of $\Delta(A)$ consists of two points.
\end{enumerate}
\end{proposition}

\begin{proof}
For (a) and (b), see \cite{DIJ}. (c) For distinct elements $T_1,\cdots,T_j\in\twopsilt^1A$, their direct sum belongs to $\twopsilt^jA$ if and only if $\Hom_{\Kb(\proj A)}(T_i,T_{i'}[1])=0$ for each $1\le i\neq i'\le j$ if and only if $T_i\oplus T_{i'}\in\twopsilt^2A$ for each $1\le i\neq i'\le j$. 
\end{proof}

Now we assume that $A$ is $g$-finite. 
For $n:=|A|$, we denote by
\[(f_{-1},f_0,\ldots,f_{n-1})\ \mbox{ and }\ (h_0,h_1,\ldots,h_n)\]
the \emph{$f$-vector} and the \emph{$h$-vector} of the $g$-simplicial complex $\Delta(A)$. Thus
\[f_{-1}:=1\ \mbox{ and }\ f_j:=\#\twopsilt^{j+1}A\]
is the number of the $j$-simplices in $\Delta(A)$ for $j\ge 0$, and $h_j$ is defined by
\begin{equation}\label{define h_j}
h_j=\sum_{i=0}^j(-1)^{j-i}{n-i\choose j-i}f_{i-1}\ \mbox{ for }\ 0\le j\le n.
\end{equation}
In other words, the \emph{$f$-polynomial} and \emph{$h$-polynomial}
\[f(x):=\sum_{i=0}^nf_{i-1}x^{n-i}\ \mbox{ and }\ h(x):=\sum_{i=0}^nh_ix^{n-i}\]
are related by $h(x)=f(x-1)$. Thus the equations \eqref{define h_j} are equivalent to
\begin{equation}\label{recover f}
f_{j-1}=\sum_{i=0}^j{n-i\choose j-i}h_i\ \mbox{ for }\ 0\le j\le n,
\end{equation}
which recover the $f$-vector from the $h$-vector. 
We give a representation theoretic meaning of the $f$- and $h$-vectors.

\begin{definition}
Let $A$ be a finite dimensional algebra over a field $k$ and $j\ge0$.
\begin{enumerate}[\rm(a)]
\item Let $\twosilt_jA$ be the set of all $T\in\twosilt A$ such that there exist precisely $j$ arrows starting at $T$ in $\Hasse(\twosilt A)$.
\item Let $\sbrick_jA$ be the set of all $S\in\sbrick A$ such that $|S|=j$.
\end{enumerate}
\end{definition}

The following is the first main result in this section.

\begin{theorem}\label{h_j}
Let $A$ be a finite dimensional algebra over a field $k$ which is $g$-finite. Then for each $0\le j\le |A|$, we have
\[h_j=\#\twosilt_jA=\#\sbrick_jA.\]
\end{theorem}

To prove Theorem \ref{h_j}, we need preparations. 
In the rest, let $A$ be a finite dimensional algebra over a field $k$ (which is not necessarily $g$-finite).

First, we recall the following well-known fact (a), see e.g. \cite[Corollary 2.4(a)]{AS}.

\begin{definition-proposition}
Let $A$ be a finite dimensional algebra over a field $k$.
\begin{enumerate}[\rm(a)]
\item Each $X\in\mod A$ has a direct summand $X_{\gen}$ such that a direct summand $Y$ of $X$ satisfies $\Fac X=\Fac Y$ if and only if $Y$ has a direct summand which is isomorphic to $X_{\gen}$.
\item For a 2-term complex $T\in \Kb(\proj A)$, we denote by $T_{\gen}$ a minimal direct summand of $T$ such that $H^0(T)_{\gen}=H^0(T_{\gen})$.
\end{enumerate}
\end{definition-proposition}

Clearly $X_{\gen}$ and $T_{\gen}$ are unique up to isomorphism.

\begin{example}
Consider a finite dimensional algebra $A$ given by
\[Q=\left[\xymatrix@C2em{1\ar@<2pt>[r]^a&2\ar@<2pt>[r]^b\ar@<2pt>[l]^{c}&3\ar@<2pt>[l]^{d}}\right]\ \mbox{ and }\ A:=kQ/\langle ab,dc,ca-bd\rangle.\]
Then $A=P_1\oplus P_2\oplus P_3 =\begin{smallmatrix}1\\ 2\\ 1\end{smallmatrix}\oplus\begin{smallmatrix}&2\\ 1&&3\\ &2\end{smallmatrix}\oplus\begin{smallmatrix}3\\ 2\\ 3\end{smallmatrix}$.
For example, we can take a 2-term silting complex $$T:=[0\to P_1]\oplus [P_2\to P_1\oplus P_3]\oplus[0\to P_3].$$
Then $H^0(T) =\begin{smallmatrix}1\\ 2\\ 1\end{smallmatrix}\oplus\begin{smallmatrix}&\\ 1&&3\\ &2\end{smallmatrix}\oplus\begin{smallmatrix}3\\ 2\\ 3\end{smallmatrix}$ and 
$$\left(\begin{smallmatrix}1\\ 2\\ 1\end{smallmatrix}\oplus\begin{smallmatrix}&\\ 1&&3\\ &2\end{smallmatrix}\oplus\begin{smallmatrix}3\\ 2\\ 3\end{smallmatrix}\right)_{\gen}=\begin{smallmatrix}1\\ 2\\ 1\end{smallmatrix}\oplus\begin{smallmatrix}3\\ 2\\ 3\end{smallmatrix}.$$
Thus we have $T_{\gen}=[0\to P_1]\oplus[0\to P_3]$. 
\end{example}

For $T\in\twosilt A$, we have the following description of $T_{\gen}$ in terms of $\Hasse(\twosilt A)$.

\begin{lemma}\label{describe T_gen}
\begin{enumerate}[\rm(a)]
\item For $T=T_1\oplus\cdots\oplus T_n\in\twosilt A$ with indecomposable $T_i$, we have
\[T_{\gen}=\bigoplus_{T>\mu_i(T)} T_i.\]
\item $\twosilt_jA=\{T\in\twosilt A\mid|T_{\gen}|=j\}$.
\end{enumerate}
\end{lemma}

\begin{proof}
(a) It is immediate from definition of $T_{\gen}$ that $T_i$ is a direct summand of $T_{\gen}$ if and only if $H^0(T_i)\notin\Fac H^0(T/T_i)$. This is clearly equivalent to $T>\mu_i(T)$ (see the proof of \cite[Theorem 2.7]{IRRT}). Thus the assertion holds.

(b) Immediate from (a).
\end{proof}

Recall from Definition-Proposition \ref{surjection to f-tors} that each $U\in\twopsilt A$ gives torsion classes
\[\TT_U=\Fac H^0(U)\subseteq\overline{\TT}_U={}^\perp H^{-1}(\nu U)\]
such that the equality holds if $U\in\twosilt A$.
We need the following notions.

\begin{definition-proposition}\label{define co-Bongartz}
Let $U\in\twopsilt A$.
\begin{enumerate}[\rm(a)]
\item We call $T\in\twosilt A$ a \emph{completion} of $U$ if it satisfies $U\in\add T$.
\item A completion $T$ of $U$ is called \emph{minimal} if
\[\TT_T=\TT_U.\]
Then $U$ has a unique minimal completion up to isomorphism, which we denote by $U_{\min}$.

\item A completion $T$ of $U$ is called \emph{maximal} (or \emph{Bongartz}) if
\[\TT_T=\overline{\TT}_U.\]
Then $U$ has a unique maximal completion up to isomorphism, which we denote by $U_{\max}$.
\item \emph{(Jasso reduction)} We have
\[\WW_U:=\overline{\TT}_U\cap\overline{\FF}_U\in\wide A.\]
Define a functor
\[\fW_U:\mod A\to\WW_U\ \mbox{ by }\ \fW_UX:=\overline{\fT}_UX/\fT_UX=\fT_{U_{\max}}X/\fT_{U_{\min}}X.\]
Let $[U]$ be an ideal of $\End_{\Kb(\proj A)}(U_{\max})$ consisting of all morphisms factoring through objects in $\add U$, and
\[B:=\End_{\Kb(\proj A)}(U_{\max})/[U].\]
Then $|B|=|A|-|U|$ holds, and there exists an equivalence
\[\mod B\simeq\WW_U\subseteq\mod A\]
which induces an injective homomorphism $K_0(\mod B)\subseteq K_0(\mod A)$.
\end{enumerate}
\end{definition-proposition}

\begin{proof} 
(b) By Definition-Proposition \ref{surjection to f-tors}(a), $\TT_U\in\ftors A$ holds. By Definition-Proposition \ref{surjection to f-tors}(c), there exists unique $T\in\twosilt A$ satisfying $\TT_T=\TT_U$.

(c) This is shown similarly.

(d) This is \cite[Theorem 3.8]{J}, see also \cite[Theorem 4.12]{DIRRT}.
\end{proof}

\begin{lemma}\label{gen min}
Let $U\in\twopsilt A$. Then we have
\[(U_{\min})_{\gen}\simeq U_{\gen}\ \mbox{and}\ (U_{\gen})_{\min}\simeq U_{\min}.\]
\end{lemma}

\begin{proof}
Since $U_{\min}$ has $U$ as a direct summand and $\TT_{U_{\min}}=\TT_U$,  we have $(U_{\min})_{\gen}\simeq U_{\gen}$. Since $\TT_{(U_{\gen})_{\min}}=\TT_{U_{\gen}}=\TT_U=\TT_{U_{\min}}$ holds, we have $(U_{\gen})_{\min}\simeq U_{\min}$ by Definition-Proposition \ref{surjection to f-tors}(c).
\end{proof}

\begin{definition-proposition}\label{define 2-presilt_jA}
For $0\le j\le |A|$, let
\begin{eqnarray}
\twopsilt_jA &:=&\{U\in\twopsilt^jA\mid U_{\min}\in\twosilt_jA\} \label{U_min}\\
&=&\{U\in\twopsilt^jA\mid U_{\gen}=U\}.\label{U_gen}
\end{eqnarray}
\end{definition-proposition}

\begin{proof}
By Lemma \ref{describe T_gen}, $U_{\min}\in\twosilt_jA$ holds if and only if $|(U_{\min})_{\gen}|=j$ holds.
By Lemma \ref{gen min}, this is equivalent to $|U_{\gen}|=j$.
Since $|U|=j$ holds and $U$ has $U_{\gen}$ as a direct summand, this is equivalent to $U_{\gen}\simeq U$.
\end{proof}

Now we are ready to prove two key results. The first one is the following.

\begin{theorem}\label{coBongartz}
Let $A$ be a finite dimensional algebra over a field $k$ and $0\le j\le|A|$. Then we have bijections
\[(-)_{\min}:\twopsilt_jA\simeq\twosilt_jA,\]
whose inverse is given by $(-)_{\gen}$.
\end{theorem}

\begin{proof}
For each $U\in\twopsilt_jA$, we have $U_{\min}\in\twosilt_jA$ by \eqref{U_min}. Thus the map $(-)_{\min}:\twopsilt_jA\to\twosilt_jA$ is well-defined. Moreover, $(U_{\min})_{\gen}\simeq U_{\gen}\simeq U$ holds by Lemma \ref{gen min} and \eqref{U_gen}.

For each $T\in\twosilt_jA$, let $U:=T_{\gen}$. Then $U_{\gen}=U$ holds, and moreover $|U|=j$ holds by Lemma \ref{describe T_gen}(b). Thus $U\in\twopsilt_jA$ holds by \eqref{U_gen}, and the map $(-)_{\gen}:\twosilt_jA\to\twopsilt_jA$ is well-defined. 
By Lemma \ref{gen min}, we have $(T_{\gen})_{\min}\simeq T_{\min}=T$. Thus the assertion follows.
\end{proof}

The second one is the following.

\begin{theorem}\label{decomposition}
Let $A$ be a finite dimensional algebra over a field $k$ and $0\le j\le|A|$.
Then we have a bijection
\begin{equation}\label{decomposition 2}
\twopsilt^jA\simeq\bigsqcup_{i=0}^j\{(V,W)\in\twopsilt_iA\times\twopsilt^{j-i}A\mid\mbox{$W$ is a direct summand of $V_{\min}/V$}\}
\end{equation}
given by $U\mapsto (U_{\gen},U/U_{\gen})$, and the converse is given by $(V,W)\mapsto V\oplus W$.
\end{theorem}

\begin{proof}
For $U\in\twopsilt^jA$, let $(V,W):=(U_{\gen},U/U_{\gen})$ and $i:=|U_{\gen}|$. Then $0\le i\le j$ and $V_{\gen}=V$ hold, and hence $V\in \twopsilt_iA$. Clearly $W\in\twopsilt^{j-i}A$ holds.
Since $V_{\min}=(U_{\gen})_{\min}\simeq U_{\min}$ holds by Lemma \ref{gen min}, $V_{\min}/V\simeq U_{\min}/U_{\gen}$ has $W=U/U_{\gen}$ as a direct summand. Thus the map $U\mapsto (U_{\gen},U/U_{\gen})$ is well-defined.
It is injective since $U\simeq U_{\gen}\oplus(U/U_{\gen})$ holds.

To prove surjectivity, take $(V,W)\in\twopsilt_iA\times\twopsilt^{j-i}A$ such that $W$ is a direct summand of $V_{\min}/V$, and let $U:=V\oplus W\in\twopsilt^jA$.
Since $\TT_V\subseteq\TT_U\subseteq\TT_{V_{\min}}=\TT_V$ holds, we have $\TT_V=\TT_U$. Since $V_{\gen}=V$ holds by \eqref{U_gen}, 
we have $U_{\gen}\simeq V$ and hence $(U_{\gen},U/U_{\gen})\simeq (V,W)$.
\end{proof}

We need the following preparation on semibricks.

\begin{definition}
Let $S$ be a semibrick of $A$.
\begin{enumerate}[\rm(a)]
\item We call $S$ \emph{left-finite} if the smallest torsion class containing $S$ is functorially finite. We denote by $\flsbrick A$ the set of isomorphism classes of left-finite semibricks of $A$.
\item We call $S$ \emph{right-finite} if the smallest torsionfree class containing $S$ is functorially finite. We denote by $\frsbrick A$ the set of isomorphism classes of right-finite semibricks of $A$.
\end{enumerate}
If $A$ is $g$-finite, then all semibricks of $A$ are left-finite and right-finite by Proposition \ref{characterize tautilt-finite}.
\end{definition}

\begin{proposition}{\cite[Theorem 2.3]{As1}}\label{2silt sbrick}
Let $A$ be a finite dimensional algebra over a field $k$ with $n:=|A|$.
Then we have the following bijections, where $\nu=-\otimes_ADA:\Kb(\proj A)\simeq\Kb(\inj A)$.
\begin{eqnarray*}
\twosilt_jA\simeq\flsbrick_jA,&&T\mapsto H^0(T)/\rad_{\End_A(H^0(T))}H^0(T).\\
\twosilt_jA\simeq\frsbrick_{n-j}A,&&T\mapsto\soc_{\End_A(H^{-1}(\nu T))}H^{-1}(\nu T).
\end{eqnarray*}
\end{proposition}

We are ready to prove Theorem \ref{h_j}.

\begin{proof}[Proof of Theorem \ref{h_j}]
By Theorem \ref{decomposition}, $f_{j-1}=\#\twopsilt^jA$ is equal to the cardinality of the right-hand side of \eqref{decomposition 2}. For each $V\in\twopsilt_iA$, there are ${n-i\choose j-i}$ choices of $W$. Thus the equality
\[f_{j-1}=\#\twopsilt^jA=\sum_{i=0}^j{n-i\choose j-i}\#\twopsilt_iA\]
holds. Comparing with \eqref{recover f}, we obtain
\[h_j=\#\twopsilt_jA,\]
which is equal to $\#\twosilt_jA$ by Theorem \ref{coBongartz}.
Finally, $\#\twosilt_jA=\#\flsbrick_jA=\#\sbrick_jA$ holds by Proposition \ref{2silt sbrick}. 
\end{proof}

Since the $g$-simplicial complex $\Delta(A)$ is a simplicial sphere \cite{DIJ}, Dehn-Sommerville equations 
\[h_j=h_{n-j}\]
are satisfied \cite[Theorem 6.8.8]{V} (see also \cite[Theorem 8.21]{Zi}).
Our next result categorifies these equations by giving a symmetry of the set $\twosilt A$.

\begin{theorem}\label{symmetry of h_j}
Let $A$ be a finite dimensional algebra over a field $k$ which is $g$-finite.
For $0\le j\le n:=|A|$, there is a canonical bijection
\[\sbrick_jA\simeq\sbrick_{n-j}A\ \mbox{and}\ \twosilt_jA\simeq\twosilt_{n-j}A.\]
In particular, we have $h_j=h_{n-j}$.
\end{theorem}

For the case $j=1$, the bijection $\twosilt_1A\simeq\twosilt_{n-1}A$ between join-irreducible elements in $\twosilt A$ and meet-irreducible elements in $\twosilt A$ was shown in \cite{DIRRT} (see also \cite{IRRT}). 
To prove Theorem \ref{symmetry of h_j}, we need the following result.

\begin{proposition}\label{2silt sbrick2}\cite{As1}
Let $A$ be a finite dimensional algebra over a field $k$ with $n:=|A|$.
\begin{enumerate}[\rm(a)]
\item There exist bijections
\[H^0:\twosmc A\simeq\flsbrick A\ \mbox{and}\ H^{-1}:\twosmc A\simeq\frsbrick A\]
such that $S=H^0(S)\oplus H^{-1}(S)[1]$ holds for each $S\in\twosmc A$. 
\item The following diagram commutes.
\[\xymatrix@R1.5em{
&&\twosilt A\ar[d]^{{\rm Prop.}\,\ref{KY thm}}_\wr\ar@/^-3mm/[dll]_{{\rm Prop.}\,\ref{2silt sbrick}}^\sim\ar@/^3mm/[drr]^{{\rm Prop.}\,\ref{2silt sbrick}}_\sim\\
\flsbrick A&&\twosmc A\ar[ll]^{H^{0}}_\sim\ar[rr]_{H^{-1}}^\sim&&\frsbrick A.
}\]
\item Assume that $T=T_1\oplus\cdots\oplus T_n\in\twosilt A$ and $S=S_1\oplus\cdots\oplus S_n\in\twosmc A$ correspond to each other by the bijection in Proposition \ref{2silt sbrick}.
For each $1\le i\le n$, $\mu_i(T)<T$ if and only if $S_i\in\mod A$, and $\mu_i(T)>T$ if and only if $S_i\in(\mod A)[1]$.\end{enumerate}
\end{proposition}

\begin{proof}
(a) and (b) are \cite[Theorem 3.3]{As1}. To prove (c), it suffices to show the first equivalence. By the left part of the commutative diagram in (b), we have $H^0(S)=H^0(T)/\rad_{\End_A(H^0(T))}H^0(T)$ and hence
\[H^0(S_i)=H^0(T_i)/\sum_{f\in\rad_{\mod A}(H^0(T),H^0(T_i))}\Image f.\]
Thus $S_i\in\mod A$ if and only if $H^0(S_i)\neq0$ if and only if $H^0(T_i)\notin\TT_{T/T_i}$. By Lemma \ref{describe T_gen}(a), this is equivalent to $\mu_i(T)<T$.
\end{proof}

We are ready to prove Theorem \ref{symmetry of h_j}.

\begin{proof}[Proof of Theorem \ref{symmetry of h_j}]
It suffices to give a bijection $\sbrick_jA\simeq\sbrick_{n-j}A$. 
Since $A$ is $g$-finite, Proposition \ref{2silt sbrick2}(a) gives bijections $H^0:\twosmc A\simeq\sbrick A$ and $H^{-1}:\twosmc A\simeq\sbrick A$ such that $S=H^0(S)\oplus H^{-1}(S)[1]$ for each $S\in\twosmc A$. Since $|H^0(S)|+|H^{-1}(S)|=|S|=n$ holds, they give a bijection $\sbrick_jA\simeq\sbrick_{n-j}A$ for each $j$.
\end{proof}

Using unimodality results of $h$-vectors in combinatorics, we obtain the following result as an application. It will be an interesting question if there is a direct proof using tilting theory.

\begin{corollary}\label{unimodality of h_j}
Let $A$ be a finite dimensional algebra over a field $k$ which is $g$-finite, and $n:=|A|$. Then we have
\begin{eqnarray*}
&\#\sbrick_1A\le\#\sbrick_2A\le\cdots\le\#\sbrick_{\lfloor\frac{n}{2}\rfloor-1}A\le\#\sbrick_{\lfloor\frac{n}{2}\rfloor}A,&\\
&\#\sbrick_{\lceil\frac{n}{2}\rceil}A\ge\#\sbrick_{\lceil\frac{n}{2}\rceil+1}A\ge\cdots\ge\#\sbrick_{n-1}A\ge\#\sbrick_nA.&
\end{eqnarray*}
\end{corollary}

\begin{proof}
The unimodality of $h$-vectors was originally proved for boundary complexes of simplicial polytopes, and this is generalized for simplicial spheres (see \cite{Adi}). 
Since $\Delta(A)$ gives a simplicial sphere (see \cite[Theorem 5.4]{DIJ}), the result follows from Theorem~\ref{h_j}.
\end{proof}

There are several natural problems \cite{Pe} in view of Corollary \ref{unimodality of h_j} and the fact that $\Delta(A)$ is flag (Proposition \ref{properties of Delta}(c)). The \emph{$\gamma$-vector} $(\gamma_0,\ldots,\gamma_{\lfloor\frac{n}{2}\rfloor})$ is defined by the equality
\[h(x)=\sum_{i=0}^{\lfloor\frac{n}{2}\rfloor}\gamma_ix^i(1+x)^{n-2i}.\]

\begin{problem}
Let $A$ be a finite dimensional algebra over a field $k$ which is $g$-finite, and let $h(x)$ be the $h$-polynomial of $\Delta(A)$.
\begin{enumerate}[\rm(a)]
\item Is $h(x)$ real-rooted (that is, all roots are real numbers)?
\item Is $h(x)$ log-concave (that is, $h_i^2\ge h_{i-1}h_{i+1}$ holds for each $i$)?
\item Is $h(x)$ $\gamma$-nonnegative (that is, $\gamma_i\ge0$ for each $i$)?
\end{enumerate}
\end{problem}

Gal's conjecture asks if each flag simplicial complex is $\gamma$-nonnegative. The following implications are known.
\[\xymatrix@R1em{
\mbox{real-rooted}\ar@{=>}[r]\ar@{=>}[d]&\mbox{log-concave}\ar@{=>}[d]\\
\mbox{$\gamma$-nonnegative}\ar@{=>}[r]&\mbox{unimodal}
}\]


\section{$g$-fans}\label{section 4}

\subsection{Definition and basic properties}

We introduce the $g$-fan of a finite dimensional algebra. 
Let $A$ be a finite dimensional algebra over a field $k$. 
Let $K_0(\proj A)$ be the Grothendieck group of $\Kb(\proj A)$ and
\[K_0(\proj A)_\R:=K_0(\proj A)\otimes_{\mathbb{Z}}\R\simeq \R^{|A|}.\] 
The $g$-simplicial complex has a canonical geometric realization as a fan in the real Grothendieck group $K_0(\proj A)_{\mathbb{R}}$ \cite{DIJ}.
Now we introduce its fan version.

\begin{definition}
For $T=T_1\oplus\cdots\oplus T_j\in\twopsilt A$ with indecomposable $T_i$, let
\begin{eqnarray*}
C(T) := \{\sum_{i=1}^ja_i[T_i]\mid a_1,\ldots,a_j\ge 0\}\subset K_0(\proj A)_\R.
\end{eqnarray*}
We call the set
\[\Sigma(A):=\{C(T)\mid T\in\twopsilt A\}\]
of cones the \emph{$g$-fan} of $A$.
\end{definition}

Notice that $\Sigma(A)$ can be an infinite set. For each $0\le i\le |A|$, the subset $\Sigma_i(A)$ of cones of dimension $i$ is given by
\[\Sigma_i(A)=\{C(U)\mid U\in\twopsilt^iA\}.\]
We give the following basic properties of $g$-fans.

\begin{proposition}\label{characterize g-finite}
Let $A$ be a finite dimensional algebra over a field $k$ and $n:=|A|$.
\begin{enumerate}[\rm(a)]
\item $\Sigma(A)$ is a  nonsingular fan in $K_0(\proj A)_\R$.
\item Any cone in $\Sigma(A)$ is a face of a cone of dimension $n$.
\item Any cone in $\Sigma(A)$ of dimension $n-1$ is a face of precisely two cones of dimension $n$.
\item $A$ is $g$-finite (or equivalently, $\Sigma(A)$ is finite) if and only if $\Sigma(A)$ is complete.
\end{enumerate}
\end{proposition}

\begin{proof}
For (a), (b) and (c), see \cite{DIJ} and Proposition \ref{properties of Delta}. (d) is \cite[Theorem 4.7]{As2}.
\end{proof}

\begin{example}\label{convex rank 2}
We give examples of $g$-fans of algebras of rank 2.
\begin{eqnarray*}
&\Sigma(k\times k)={\begin{xy}
0;<3pt,0pt>:<0pt,3pt>::
(0,-5)="0",
(-5,0)="1",
(0,0)*{\bullet},
(0,0)="2",
(5,0)="3",
(0,5)="4",
(1.5,1.5)*{{\scriptstyle +}},
(-1.5,-1.5)*{{\scriptstyle -}},
\ar@{-}"0";"1",
\ar@{-}"1";"4",
\ar@{-}"4";"3",
\ar@{-}"3";"0",
\ar@{-}"2";"0",
\ar@{-}"2";"1",
\ar@{-}"2";"3",
\ar@{-}"2";"4",
\end{xy}}\ ,\ \ 
\Sigma\left(k[\begin{xy} ( 0,0) *+{1}="1",( 8,0) *+{2}="2",\ar "1";"2"  \end{xy}]\right)={\begin{xy}
0;<-3pt,0pt>:<0pt,-3pt>::
(0,-5)="0",
(-5,0)="1",
(0,0)*{\bullet},
(0,0)="2",
(5,0)="3",
(-5,5)="4",
(0,5)="5",
(1.5,1.5)*{{\scriptstyle -}},
(-1.5,-1.5)*{{\scriptstyle +}},
\ar@{-}"0";"1",
\ar@{-}"1";"4",
\ar@{-}"4";"5",
\ar@{-}"5";"3",
\ar@{-}"3";"0",
\ar@{-}"2";"0",
\ar@{-}"2";"1",
\ar@{-}"2";"3",
\ar@{-}"2";"4",
\ar@{-}"2";"5",
\end{xy}}\ ,\ \ 
\Sigma\left(k[\begin{xy}( 0,0) *+{1}="1",  ( 8,0) *+{2}="2",\ar@<1.5pt> "1";"2"^a \ar@<1.5pt> "2";"1"^b  \end{xy}]/\langle ab,ba \rangle\right)={\begin{xy}
0;<3pt,0pt>:<0pt,3pt>::
(0,-5)="0",
(5,-5)="1",
(-5,0)="2",
(0,0)*{\bullet},
(0,0)="3",
(5,0)="4",
(-5,5)="5",
(0,5)="6",
(1.5,1.5)*{{\scriptstyle +}},
(-1.5,-1.5)*{{\scriptstyle -}},
\ar@{-}"0";"2",
\ar@{-}"2";"5",
\ar@{-}"5";"6",
\ar@{-}"6";"4",
\ar@{-}"4";"1",
\ar@{-}"1";"0",
\ar@{-}"3";"0",
\ar@{-}"3";"1",
\ar@{-}"3";"2",
\ar@{-}"3";"4",
\ar@{-}"3";"5",
\ar@{-}"3";"6",
\end{xy}}&\\
&\Sigma\left(k[\begin{xy} ( 0,0) *+{1}="1",( 8,0) *+{2}="2",\ar "1";"2"^{a}  \ar @(ru,rd)"2";"2"^{c} \end{xy}]/\langle c^2\rangle\right)=\begin{xy}
0;<-3pt,0pt>:<0pt,-3pt>::
(0,-5)="0",
(-5,0)="1",
(0,0)*{\bullet},
(0,0)="2",
(5,0)="3",
(-10,5)="4",
(-5,5)="5",
(0,5)="6",
(1.5,1.5)*{{\scriptstyle -}},
(-1.5,-1.5)*{{\scriptstyle +}},
\ar@{-}"0";"3",
\ar@{-}"3";"6",
\ar@{-}"6";"4",
\ar@{-}"4";"0",
\ar@{-}"2";"0",
\ar@{-}"2";"1",
\ar@{-}"2";"3",
\ar@{-}"2";"4",
\ar@{-}"2";"5",
\ar@{-}"2";"6",
\end{xy}\ ,\ \ 
                    \Sigma\left(k[\begin{xy}( 0,0) *+{1}="1", ( 8,0) *+{2}="2",\ar@<1.5pt> "1";"2"^a \ar@<1.5pt> "2";"1"^b  \ar @(ru,rd)"2";"2"^{c} \end{xy}]/\langle ab,ba,cb,c^2 \rangle\right)=\begin{xy}
0;<-3pt,0pt>:<0pt,-3pt>::
(0,-5)="0",
(5,-5)="1",
(-5,0)="2",
(0,0)*{\bullet},
(0,0)="3",
(5,0)="4",
(-10,5)="5",
(-5,5)="6",
(0,5)="7",
(1.5,1.5)*{{\scriptstyle -}},
(-1.5,-1.5)*{{\scriptstyle +}},
\ar@{-}"0";"1",
\ar@{-}"1";"4",
\ar@{-}"4";"7",
\ar@{-}"7";"5",
\ar@{-}"5";"0",
\ar@{-}"3";"0",
\ar@{-}"3";"1",
\ar@{-}"3";"2",
\ar@{-}"3";"4",
\ar@{-}"3";"5",
\ar@{-}"3";"6",
\ar@{-}"3";"7",
\end{xy}&\\
&\Sigma\left(k[\begin{xy}( 0,0) *+{1}="1", ( 8,0) *+{2}="2",\ar@<1.5pt> "1";"2"^a \ar@<1.5pt> "2";"1"^b \ar @(ul,ur)"2";"2"^(.8){c} \ar @(dr,dl)"2";"2"^(.2){d} \end{xy}]/\langle ab,ba,ad,cb, c^2, d^2, cd,dc \rangle\right)=\begin{xy}
0;<3pt,0pt>:<0pt,3pt>::
(0,-5)="0",
(5,-5)="1",
(10,-5)="2",
(-5,0)="3",
(0,0)*{\bullet},
(0,0)="4",
(5,0)="5",
(-10,5)="6",
(-5,5)="7",
(0,5)="8",
(1.5,1.5)*{{\scriptstyle +}},
(-1.5,-1.5)*{{\scriptstyle -}},
\ar@{-}"0";"2",
\ar@{-}"2";"8",
\ar@{-}"8";"6",
\ar@{-}"6";"0",
\ar@{-}"4";"0",
\ar@{-}"4";"1",
\ar@{-}"4";"2",
\ar@{-}"4";"3",
\ar@{-}"4";"5",
\ar@{-}"4";"6",
\ar@{-}"4";"7",
\ar@{-}"4";"8",
\end{xy}&\\
&\Sigma\left(k[\begin{xy}( 0,0) *+{1}="1", ( 8,0) *+{2}="2",\ar@<1.5pt> "1";"2"^a \ar@<1.5pt> "2";"1"^b \ar @(ld,lu)"1";"1"^{d} \ar @(ru,rd)"2";"2"^{c} \end{xy}]\langle ab,ba,da, cb,  c^2, d^2 \rangle\right)=\begin{xy}
0;<-3pt,0pt>:<0pt,-3pt>::
(0,-2.5)="0",
(5,-7.5)="1",
(5,-2.5)="2",
(-5,2.5)="3",
(0,2.5)*{\bullet},
(0,2.5)="4",
(5,2.5)="5",
(-10,7.5)="6",
(0,7.5)="7",
(-5,7.5)="8",
(1.5,4)*{{\scriptstyle -}},
(-1.5,1)*{{\scriptstyle +}},
\ar@{-}"6";"1",
\ar@{-}"1";"5",
\ar@{-}"5";"7",
\ar@{-}"7";"6",
\ar@{-}"4";"0",
\ar@{-}"4";"1",
\ar@{-}"4";"2",
\ar@{-}"4";"3",
\ar@{-}"4";"5",
\ar@{-}"4";"6",
\ar@{-}"4";"7",
\ar@{-}"4";"8",
\end{xy}&\end{eqnarray*}
\end{example}

We give examples of more complicated $g$-fans.

\begin{example}{\cite[Proposition 6.1]{K}}\label{kase example}
For positive integers $\ell\ge1,m\ge1$, we define an algebra $A:=kQ/I$ as follows. The quiver $Q$ is the following
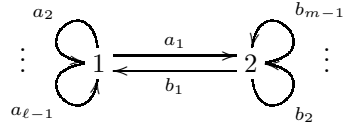
\[\begin{xy}
(0,0) *[o]+{1}="A",
(-10,1.5) *{\vdots},
(20,0)*[o]+{2}="B",
(30,1.5) *{\vdots},
\ar @<1mm> "A";"B"^{a_1}
\ar @(u,l) "A";"A"_{a_2}
\ar @(l,d) "A";"A"_{a_{\ell-1}}
\ar @<1mm> "B";"A"^{b_1}
\ar @(d,r) "B";"B"_{b_2}
\ar @(r,u) "B";"B"_{b_{m-1}}
\end{xy}
\]
The ideal $I$ of $kQ$ is generated by the following elements for all possible $i,j$:
\[a_ia_j\ (i-j\neq1),\ b_ib_j\ (i-j\neq1),\ a_ib_j\ \mbox{ and }\ b_ia_j.\]
Then $\Hasse(\twosilt A_{\ell,m})$ is
\[\xymatrix@R=0em@C=1.5em{
&P^{(0)}\oplus P^{(1)}\ar[r]&P^{(1)}\oplus P^{(2)}\ar[r]&\cdots\ar[r]&P^{(\ell-1)}\oplus P^{(\ell)}\ar[dr]\\
A=P^{(0)}\oplus Q^{(0)}\ar[ur]\ar[dr]&&&&&A[1]=P^{(\ell)}\oplus Q^{(m)}\\
&Q^{(0)}\oplus Q^{(1)}\ar[r]&Q^{(1)}\oplus Q^{(2)}\ar[r]&\cdots\ar[r]&Q^{(m-1)}\oplus Q^{(m)}\ar[ur]}\]
with $P^{(0)}=P_1:=e_1 A$, $P^{(\ell)}=P_1[1]$, $Q^{(0)}=P_2:=e_2 A$ and $Q^{(m)}=P_2[1]$. Since
\[[P^{(i)}]=[P_1]-i[P_2]\ (0\le i\le\ell-1)\ \mbox{ and }\ [Q^{(j)}]=[P_2]-j[P_1]\ (0\le j\le m-1),\]
the $g$-fan $\Sigma(A)$ consists of
\begin{eqnarray*}
&\cone\{[P_1],[P_2]\},\ \cone\{-[P_1],-[P_2]\},\\
&\cone\{[P_1]-(i-1)[P_2],\ [P_1]-i[P_2]\}\ (1\le i\le\ell-1),\ \cone\{[P_1]-(\ell-1)[P_2],\ -[P_2]\},\\
&\cone\{[P_2]-(j-1)[P_1],\ [P_2]-j[P_1]\}\ (1\le j\le m-1),\ \cone\{[P_2]-(m-1)[P_1],\ -[P_1]\}.
\end{eqnarray*}

For example, if $(\ell,m)=(4,5)$, then $\Sigma(A)$ is 
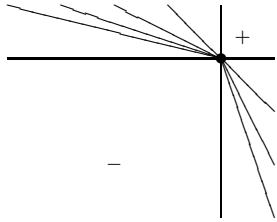
\[\begin{xy}
0;<3pt,0pt>:<0pt,3pt>::
(0,0)*{\bullet},
(0,0)="0",
(5,0)="1",
(5,-5)="2",
(5,-10)="3",
(5,-15)="4",
(0,-5)="5",
(0,5)="6",
(-5,5)="7",
(-10,5)="8",
(-15,5)="9",
(-20,5)="10",
(-5,0)="11",
(1.5,1.5)*{{\scriptstyle +}},
(-1.5,-1.5)*{{\scriptstyle -}},
\ar@{-}"1";"2",
\ar@{-}"2";"3",
\ar@{-}"3";"4",
\ar@{-}"4";"5",
\ar@{-}"1";"6",
\ar@{-}"6";"7",
\ar@{-}"7";"8",
\ar@{-}"8";"9",
\ar@{-}"9";"10",
\ar@{-}"10";"11",
\ar@{-}"5";"11",
\ar@{-}"0";"1",
\ar@{-}"0";"2",
\ar@{-}"0";"3",
\ar@{-}"0";"4",
\ar@{-}"0";"5",
\ar@{-}"0";"6",
\ar@{-}"0";"7",
\ar@{-}"0";"8",
\ar@{-}"0";"9",
\ar@{-}"0";"10",
\ar@{-}"0";"11",
\end{xy}\]
\end{example}

\begin{example}\label{rank 3 example}
Consider a finite dimensional algebra $A$ given by
\[Q=\left[\begin{xy}
(0,4) *+{1}="1",
(0,-4) *+{2}="2",
(8,0) *+{3}="3",
\ar@<1.5pt> "3";"1"_d
\ar@<1.5pt> "1";"2"_a
\ar@<1.5pt> "2";"3"_b
\ar @(rd,ru)"3";"3"_{c}
\end{xy}\right]\mbox{ and }\ A:=kQ/\langle c^2, bd, abcda, cdab - dabc\rangle\] which is a Brauer graph algebra (see Example \ref{sec:example}).
Then $A=\begin{smallmatrix}1\\ 2\\ 3\\ 3\\ 1\end{smallmatrix}\oplus\begin{smallmatrix}2\\ 3\\ 3\\ 1\\ 2\end{smallmatrix}\oplus\begin{smallmatrix}&3\\ 1&&3\\ 2&&1\\ 3&&2\\ &3\end{smallmatrix}$, and $\Sigma(A)$ is given by the following.
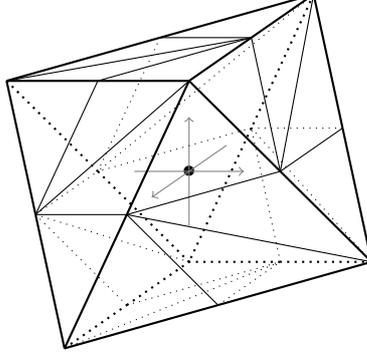
\[\begin{tikzpicture}[baseline=0mm,scale=1]
            \node(0) at(0:0) {$\bullet$}; 
            \node(x) at(215:1) {}; 
            \node(-x) at($-1*(x)$) {}; 
            \node(y) at(0:1.2) {}; 
            \node(-y) at($-1*(y)$) {}; 
            \node(z) at(90:1.2) {}; 
            \node(-z) at($-1*(z)$) {}; 
            \draw[gray, <-] ($0.6*(x)$)--($0.6*(-x)$); \draw[gray, <-] ($0.6*(y)$)--($0.6*(-y)$); \draw[gray, <-] ($0.6*(z)$)--($0.6*(-z)$);

            \coordinate(1) at($1*(x) + 0*(y) + 0*(z)$) ;
            \coordinate(2) at($0*(x) + 1*(y) + 0*(z)$) ;
            \coordinate(3) at($0*(x) + 0*(y) + 1*(z)$) ;

            \coordinate(4) at($-1*(x) + 0*(y) + 0*(z)$) ;
            \coordinate(5) at($0*(x) + -1*(y) + 0*(z)$) ;
            \coordinate(6) at($0*(x) + 0*(y) + -1*(z)$) ;

            \coordinate(7) at($0*(x) + 2*(y) + -1*(z)$) ;
            \coordinate(8) at($1*(x) + -1*(y) + 0*(z)$) ;
            \coordinate(9) at($-1*(x) + 0*(y) + 1*(z)$) ;

            \coordinate(10) at($1*(x) + 1*(y) + -1*(z)$) ;
            \coordinate(11) at($-1*(x) + 1*(y) + 0*(z)$) ;
            \coordinate(12) at($2*(x) + 0*(y) + -1*(z)$) ;

            \coordinate(13) at($0*(x) + -1*(y) + 1*(z)$) ;
            \coordinate(14) at($-2*(x) + 0*(y) + 1*(z)$) ;
            \coordinate(15) at($0*(x) + 1*(y) + -1*(z)$) ;

            \coordinate(16) at($1*(x) + 0*(y) + -1*(z)$) ;
            \coordinate(17) at($0*(x) + -2*(y) + 1*(z)$) ;
            \coordinate(18) at($-1*(x) + -1*(y) + 1*(z)$) ;

            \draw[] (1)--(2) ;
            \draw[thick] (1)--(3) ;
            \draw[thick] (2)--(3) ;
            \draw[dotted] (4)--(5) ;
            \draw[dotted, thick] (4)--(6) ;
            \draw[dotted, thick] (5)--(6) ;
            \draw[] (7)--(1) ;
            \draw[thick] (7)--(2) ;
            \draw[] (1)--(8) ;
            \draw[] (3)--(8) ;
            \draw[] (9)--(2) ;
            \draw[thick] (9)--(3) ;
            \draw[thick] (10)--(7) ;
            \draw[] (10)--(1) ;
            \draw[thick] (11)--(7) ;
            \draw[] (11)--(2) ;
            \draw[thick] (12)--(1) ;
            \draw[thick] (12)--(8) ;
            \draw[thick] (13)--(3) ;
            \draw[] (13)--(8) ;
            \draw[thick] (14)--(9) ;
            \draw[] (14)--(2) ;
            \draw[] (13)--(9) ;
            \draw[dotted] (10)--(15) ;
            \draw[dotted, thick] (7)--(15) ;
            \draw[thick] (12)--(10) ;
            \draw[dotted] (7)--(4) ;
            \draw[dotted] (11)--(4) ;
            \draw[dotted, thick] (12)--(16) ;
            \draw[dotted] (16)--(8) ;
            \draw[thick] (17)--(13) ;
            \draw[thick] (17)--(8) ;
            \draw[thick] (18)--(14) ;
            \draw[] (18)--(9) ;
            \draw[thick] (11)--(14) ;
            \draw[] (17)--(9) ;
            \draw[dotted] (12)--(15) ;
            \draw[dotted] (15)--(4) ;
            \draw[dotted, thick] (16)--(6) ;
            \draw[dotted] (8)--(6) ;
            \draw[dotted, thick] (17)--(5) ;
            \draw[dotted] (8)--(5) ;
            \draw[dotted] (14)--(5) ;
            \draw[dotted] (18)--(5) ;
            \draw[dotted, thick] (14)--(4) ;
            \draw[thick] (18)--(17) ;
            \draw[dotted] (16)--(15) ;
            \draw[dotted, thick] (15)--(6) ;
    \end{tikzpicture}\]
\end{example}

The following general results give canonical isomorphisms between the $g$-fans of two algebras.

\begin{proposition}\label{field extension}
Let $A$ be a finite dimensional algebra over a field $k$.
\begin{enumerate}[\rm(a)]
\item The isomorphism $\Hom_A(-,A):K_0(\proj A)\simeq K_0(\proj A^{\op})$ gives an isomorphism of $g$-fans:
\[\Sigma(A)\simeq\Sigma(A^{\op}).\]
\item Let $K/k$ be a field extension such that the functor $K\otimes_k-:\proj A\to\proj K\otimes_kA$ preserves the indecomposability. Then the isomorphism $K_0(\proj A)\simeq K_0(\proj K\otimes_kA)$ gives an isomorphism of $g$-fans:
\[\Sigma(A)\simeq\Sigma(K\otimes_kA).\]
\end{enumerate}
\end{proposition}

\begin{proof}
(a) We have a duality $\RHom_A(-,A)[1]:\KKK^{\bo}(\proj A)\simeq\KKK^{\bo}(\proj A^{\op})$, which gives the isomorphism $-\Hom_A(-,A):K_0(\proj A)\simeq K_0(\proj A^{\op})$ and bijections $\twosilt A\simeq\twosilt A^{\op}$ and $\twopsilt A\simeq\twopsilt A^{\op}$. Thus the assertion follows.

(b) The assertion follows from \cite{IK}.
\end{proof}

We also study the $g$-fans of arbitrary factor algebras.
Let $A$ be a finite dimensional algebra over a field $k$, and  $B=A/I$ a factor algebra of $A$. 
Using the triangle functor $-\otimes_AB:\Kb(\proj A)\to\Kb(\proj B)$, we define an equivalence relation $\sim_B$ on $\twosilt A$: For $T,U\in\twosilt A$, we write $T\sim_BU$ if $\add(T\otimes_AB)=\add(U\otimes_AB)$.
Then we have the following description of $\Sigma(B)$, see Definition \ref{define coarsening fan} for $\Sigma(A)/\sim_B$.

\begin{proposition}
Let $A$ be a finite dimensional algebra over a field $k$, $I$ an ideal of $A$ contained in $\rad A$,  and $B:=A/I$ a factor algebra of $A$.
\begin{enumerate}[\rm(a)]
\item $I$ annihilates all bricks of $A$ if and only if the isomorphism $-\otimes_AB:K_0(\proj A)\simeq K_0(\proj B)$ gives an isomorphism of $g$-fans:
\[\Sigma(A)\simeq\Sigma(B).\]
For example, if $I$ is generated by a set of central nilpotent elements of $A$, then we have the isomorphism of $g$-fans above.
\item If $A$ is $g$-finite, then the equivalence relation $\sim_B$ coarsens $\Sigma(A)$, and the isomorphism $-\otimes_AB:K_0(\proj A)\simeq K_0(\proj B)$ gives an isomorphism of $g$-fans:
\[\Sigma(B)=\Sigma(A)/\sim_B.\]
\end{enumerate}
\end{proposition}

\begin{proof}
Since $I\subset\rad A$, $-\otimes_AB:K_0(\proj A)\to K_0(\proj B)$ is an isomorphism.

(a) The assertion follows from \cite[Theorem 5.12]{DIRRT} and \cite[Theorem 1]{EJR}.

(b) $-\otimes_AB$ gives an injective map $\twosilt A/\sim_B\ \to\twosilt B$. Thus for each $T\in\twosilt A$, we have $\bigcup_{U\in\twosilt A,\ U\sim_BT}C(U)\subset C(T\otimes_AB)$. Since $\Sigma(A)$ is complete, the equality holds, and the map $\twosilt A/\sim_B\ \to\twosilt B$ is bijective. Thus $\Sigma(B)=\Sigma(A)/\sim_B$ holds.
\end{proof}

Note that one can not drop the $g$-finiteness in (b) above. For example, let $A$ be an $\ell$-Kronecker algebra with $\ell\ge2$, and $B$ the factor algebra of $A$ by the ideal generated by $\ell-1$ arrows. Then $\Sigma(B)$ is complete and hence can not be a coarsening of $\Sigma(A)$ which is not complete.

\medskip
One of the importance of the $g$-fan is that it gives an explicit description of torsion classes given by stability conditions. For each $\theta\in K_0(\proj A)_\R$, let
\begin{eqnarray*}
\TT_\theta&:=&\{X\in\mod A\mid\theta(X')>0\ \text{for all factor modules $X'\neq 0$ of $X$}\},\\
\overline{\TT}_\theta&:=&\{X\in\mod A\mid\theta(X')\ge0\ \text{for all factor modules $X'$ of $X$}\},\\
\FF_\theta&:=&\{X\in\mod A\mid\theta(X')<0\ \text{for all submodules $X'\neq 0$ of $X$}\},\\
\overline{\FF}_\theta&:=&\{X\in\mod A\mid\theta(X')\le0\ \text{for all submodules $X'$ of $X$}\},\\
\WW_\theta&:=&\overline{\TT}_\theta\cap\overline{\FF}_\theta.
\end{eqnarray*}
The following properties are elementary.

\begin{definition-proposition}\cite{BKT,Ki}\label{basic fact for theta}
Let $\theta\in K_0(\proj A)_\R$.
\begin{enumerate}[\rm(a)]
\item $(\TT_\theta,\overline{\FF}_\theta)$ and $(\overline{\TT}_\theta,\FF_\theta)$ are torsion pairs in $\mod A$ satisfying $\TT_\theta\subseteq\overline{\TT}_\theta$ and $\FF_\theta\subseteq\overline{\FF}_\theta$.
In particular, for $\fT_\theta:=\fT_{\TT_\theta}$, $\overline{\fT}_\theta:=\fT_{\overline{\TT}_\theta}$, $\fF_\theta:=\fF_{\FF_\theta}$, $\overline{\fF}_\theta:=\fF_{\overline{\FF}_\theta}$ in \eqref{canonical exact}, each $X\in\mod A$ admits exact sequences
\begin{eqnarray*}
&\mbox{$0\to\fT_\theta X\to X\to\overline{\fF}_\theta X\to0$ with $\fT_\theta X\in\TT_\theta$ and $\overline{\fF}_\theta X\in\overline{\FF}_\theta$,}&\\
&\mbox{$0\to\overline{\fT}_\theta X\to X\to\fF_\theta X\to0$ with $\overline{\fT}_\theta X\in\overline{\TT}_\theta$ and $\overline{\fF}_\theta X\in\FF_\theta$.}&
\end{eqnarray*}
\item $\WW_\theta$ is a wide subcategory of $\mod A$. Moreover, each $X\in\mod A$ admits a filtration \[0\subseteq\fT_\theta X\subseteq\overline{\fT}_\theta X\subseteq X\ \mbox{ such that }\ \fW_\theta X:=\overline{\fT}_\theta X/\fT_\theta X\in\WW_\theta.\]
\end{enumerate}
\end{definition-proposition}

For $T=T_1\oplus\cdots\oplus T_j\in\twopsilt A$ with indecomposable $T_i$, let
\[C^+(T):=\{\sum_{i=1}^ja_i[T_i]\mid a_1,\ldots,a_j>0\}\subset C(T).\]
We have the following descriptions of torsion pairs given by stability conditions in terms of 2-term presilting complexes (see Definition-Propositions \ref{surjection to f-tors} and \ref{define co-Bongartz} for functors $\fF_U$, $\fW_U$ and so on). We refer to \cite{As2,AsI} for more explicit results.

\begin{proposition}\cite[Proposition 3.3]{Y}\cite[Proposition 3.27]{BST}\label{presilting semistable}
Let $U\in\twopsilt A$ and $\theta\in C^+(U)$. Then we have
\[\TT_\theta=\TT_U,\ \overline{\TT}_\theta=\overline{\TT}_U,\ \FF_\theta=\FF_U,\ \overline{\FF}_\theta=\overline{\FF}_U\mbox{ and }\ \WW_\theta=\WW_U.\]
Therefore we have
\[\fT_\theta=\fT_U,\ \overline{\fT}_\theta=\overline{\fT}_U,\ \fF_\theta=\fF_U,\ \overline{\fF}_\theta=\overline{\fF}_U,\mbox{ and }\ \fW_\theta=\fW_U.\]
\end{proposition}

To explain a remarkable property of $g$-fans, we introduce the following notion.

\begin{definition}\label{define reduction}
For a fan $\Sigma$ in $\R^d$ and $\sigma\in\Sigma$, we define the \emph{reduction} of $\Sigma$ at $\sigma$ by
\[\Sigma/\sigma:=\{\pi(\tau)\mid \tau\in\Sigma,\ \sigma\subseteq\tau\},\]
where $\pi:\R^d\to\R^d/\R\sigma$ is a natural projection. Then $\Sigma/\sigma$ is a fan in $\R^d/\R\sigma$.
\end{definition}

This reduction process of fans corresponds to Jasso reduction given in Definition \ref{define co-Bongartz}(d), see also \cite[Theorem 4.5]{As2}.

\begin{theorem}\label{Jasso reduction}
Let $A$ be a finite dimensional algebra over a field $k$, and $U\in\twopsilt A$. For the Bongartz completion $U_{\max}$ of $U$, let $B:=\End_{\Kb(\proj A)}(U_{\max})/[U]$.
\begin{enumerate}[\rm(a)]
\item There exists a triangle functor $\Kb(\proj A)\to\Kb(\proj B)$ which sends $U_{\max}$ to $B$ and gives an isomorphism $K_0(\proj A)/K_0(\add U)\simeq K_0(\proj B)$ and bijections
\[\{T\in\twosilt A\mid U\in\add T\}\simeq\twosilt B\mbox{ and }\ \{T\in\twopsilt A\mid U\in\add T\}\simeq\twopsilt B.\]
\item The isomorphism $K_0(\proj A)_\R/K_0(\add U)_\R\simeq K_0(\proj B)_\R$ gives an isomorphism of fans 
\[\Sigma(A)/C(U)\simeq\Sigma(B).\]
\end{enumerate}
\end{theorem}
	
\begin{proof}
(a) Let $\TT:=\Kb(\proj A)/\thick U$ be the Verdier quotient, $\pi:\Kb(\proj A)\to\TT$ the canonical functor, and $V:=\pi(U_{\max})\in\TT$. Then we have an isomorphism $\pi:K_0(\proj A)/K_0(\add U)\simeq K_0(\TT)$. By \cite[Theorems 3.1, 3.7, Corollary 3.8]{IY}, $\TT$ is a $k$-linear Hom-finite Krull-Schmidt triangulated category such that $\pi$ gives an isomorphism
\[\{T\in\silt A\mid U\in\add T\}\simeq\silt\TT\]
of posets, a bijection
\[\{T\in\psilt A\mid U\in\add T\}\simeq\psilt\TT\]
and an isomorphism
\[B\simeq\End_{\TT}(V)\]
of $k$-algebras. Since $\pi(A)\simeq V$ holds by \cite[Propositions 4.10]{J}, 
$\pi$ gives bijections
\begin{eqnarray*}
\{T\in\twosilt A\mid U\in\add T\}\simeq\rtwosilt{V}\TT\mbox{ and }\{T\in\twopsilt A\mid U\in\add T\}\simeq\rtwopsilt{V}\TT.
\end{eqnarray*}
On the other hand, $\TT$ is algebraic \cite{Ke2,D}. Applying Proposition \ref{dg case} to $\TT$ and $V$, there is a triangle functor $F:\TT\to\Kb(\proj B)$ which sends $V$ to $B$ and gives an isomorphism $K_0(\TT)\simeq K_0(\proj B)$ and bijections
\[\rtwosilt{V}\TT\simeq\twosilt B\mbox{ and }\rtwopsilt{V}\TT\simeq\twopsilt B.\]
Thus the composition $F\circ\pi:\Kb(\proj A)\to\Kb(\proj B)$ gives the desired triangle functor.

(b) For $\ell:=|U|\le j\le |A|$, the triangle functors $\pi$ and $F$ give bijections
\[\{T\in\twopsilt^jA\mid U\in\add T\}\simeq\rtwopsilt{V}^{j-\ell}\TT\simeq\twopsilt^{j-\ell}B\]
such that $F\circ\pi(C(T))=C(F\circ\pi(T))$ as cones in $K_0(\proj B)_\R$. Thus the assertion follows.
\end{proof}

\begin{example}
Let $A$ be the algebra in Example \ref{rank 3 example}.
For $\sigma:=C(e_1A)$, the cones $\tau\in\Sigma(A)$ satisfying $\sigma\subset\tau$ is shown by the left picture below. Moreover, $\Sigma(A)/\sigma$ is the fan shown by the right picture below.
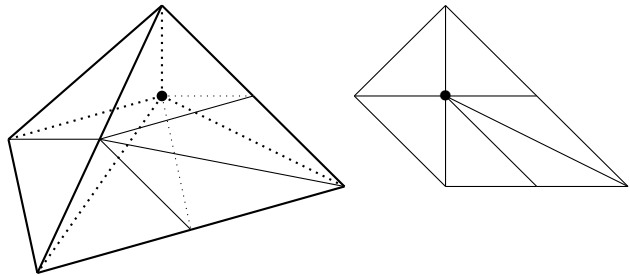
\[\begin{tikzpicture}[baseline=0mm,scale=1]
            \node(x) at(215:1) {}; 
            \node(-x) at($-1*(x)$) {}; 
            \node(y) at(0:1.2) {}; 
            \node(-y) at($-1*(y)$) {}; 
            \node(z) at(90:1.2) {}; 
            \node(-z) at($-1*(z)$) {}; 
            \coordinate(0) at(0:0);
            \fill (0) circle (2pt);            
            \coordinate(1) at($1*(x) + 0*(y) + 0*(z)$) ;
            \coordinate(2) at($0*(x) + 1*(y) + 0*(z)$) ;
            \coordinate(3) at($0*(x) + 0*(y) + 1*(z)$) ;
            \coordinate(7) at($0*(x) + 2*(y) + -1*(z)$) ;
            \coordinate(8) at($1*(x) + -1*(y) + 0*(z)$) ;
            \coordinate(10) at($1*(x) + 1*(y) + -1*(z)$) ;
            \coordinate(12) at($2*(x) + 0*(y) + -1*(z)$) ;
            \draw[dotted] (0)--(2); \draw[dotted,thick] (0)--(3); \draw[dotted,thick] (0)--(7); \draw[dotted,thick] (0)--(8); \draw[dotted] (0)--(10); \draw[dotted,thick] (0)--(12);            
            \draw[] (1)--(2) ;
            \draw[thick] (1)--(3) ;
            \draw[thick] (2)--(3) ;
            \draw[] (7)--(1) ;
            \draw[thick] (7)--(2) ;
            \draw[] (1)--(8) ;
            \draw[thick] (3)--(8) ;
            \draw[thick] (10)--(7) ;
            \draw[] (10)--(1) ;
            \draw[thick] (12)--(1) ;
            \draw[thick] (12)--(8) ;
            \draw[thick] (12)--(10) ;
    \end{tikzpicture}
\begin{tikzpicture}[baseline=0mm,scale=1]
            \node(0) at(0:0) {$\bullet$}; 
            \node(x) at(215:1) {}; 
            \node(-x) at($-1*(x)$) {}; 
            \node(y) at(0:1.2) {}; 
            \node(-y) at($-1*(y)$) {}; 
            \node(z) at(90:1.2) {}; 
            \node(-z) at($-1*(z)$) {}; 
            
            \coordinate(1) at($0*(x) + 0*(y) + 0*(z)$) ;
            \coordinate(2) at($0*(x) + 1*(y) + 0*(z)$) ;
            \coordinate(3) at($0*(x) + 0*(y) + 1*(z)$) ;
            \coordinate(7) at($0*(x) + 2*(y) + -1*(z)$) ;
            \coordinate(8) at($0*(x) + -1*(y) + 0*(z)$) ;
            \coordinate(10) at($0*(x) + 1*(y) + -1*(z)$) ;
            \coordinate(12) at($0*(x) + 0*(y) + -1*(z)$) ;
            \draw[] (1)--(2) ;
            \draw[] (1)--(3) ;
            \draw[] (2)--(3) ;
            \draw[] (7)--(1) ;
            \draw[] (7)--(2) ;
            \draw[] (1)--(8) ;
            \draw[] (3)--(8) ;
            \draw[] (10)--(7) ;
            \draw[] (10)--(1) ;
            \draw[] (12)--(1) ;
            \draw[] (12)--(8) ;
            \draw[] (12)--(10) ;
    \end{tikzpicture}    
    \]
\end{example}

\subsection{Idempotents and $g$-fans}

In this subsection, we observe that $g$-fans of finite dimensional algebras are rather special.

\begin{definition}\label{define sign-coherent}
A \emph{sign-coherent fan} is a pair $(\Sigma,\sigma_+)$ satisfying the following conditions.
\begin{enumerate}[\rm(a)]
\item $\Sigma$ is a nonsingular fan in $\R^d$, and $\sigma_+,-\sigma_+\in\Sigma_d$.
\item Take a $\Z$-basis $e_1,\dots,e_d\in\Z^d$ such that $\sigma_+=\cone\{e_i\mid 1\le i\le d\}$. 
Then for each $\sigma\in\Sigma$, there exists $\epsilon_1,\ldots,\epsilon_d\in\{1,-1\}$ such that $\sigma\subseteq\cone\{\epsilon_1e_1,\dots,\epsilon_de_d\}$.
\item Any cone in $\Sigma$ is a face of a cone of dimension $d$. 
\item Any cone in $\Sigma$ of dimension $d-1$ is a face of precisely two cones of dimension $d$.
\end{enumerate}
\end{definition}

The following property of $g$-fans is basic.

\begin{proposition}\label{tilting is sign-coherent}
For a finite dimensional algebra $A$ over a field $k$, $(\Sigma(A),C(A))$ is a sign-coherent fan. 
In this case, the isomorphism classes of indecomposable projective $A$-modules gives the $\Z$-basis which generate $C(A)$ as a cone.
\end{proposition}

\begin{proof}
The conditions (a),(c) and (d) follow from Proposition \ref{characterize g-finite}. The condition (b) is the sign-coherence of $g$-vectors \cite{AIR}.
\end{proof}

Now we show that the $g$-fans form a very special class of sign-coherent fans.

\begin{definition}\label{define ordered}
An \emph{ordered fan} is a sign-coherent fan $(\Sigma,\sigma_+)$ satisfying the following conditions.
\begin{enumerate}[\rm(a)]
\item For each adjacent cones $\sigma,\sigma'\in\Sigma_d$, a normal vector of $\sigma\cap\sigma'$ belongs to the interior of $\sigma_+$.
\item There exists a partial order $\le$ on $\Sigma_d$ such that $(\Sigma_d,\le)$ has the maximum element $\sigma_+$ and the minimum element $-\sigma_+$. 
\item For any $\sigma,\sigma'\in\Sigma_d$, the following conditions are equivalent.
\begin{itemize}
\item There is an arrow $\sigma\to\sigma'$ in the Hasse quiver.
\item $\sigma$ and $\sigma'$ are adjacent. Moreover, consider the hyperplane $H:=\R(\sigma\cap\sigma')\subset \R^d$. Then $\sigma$ and $\sigma_+$ belong to the same connected component of $\R^d\setminus H$.
\end{itemize}
\end{enumerate}
\end{definition}

For $d=2$, it is easy to check that any sign-coherent fan is ordered.
For $d\ge3$, there are many sign-coherent fans which are not ordered.

Notice that, if $A$ is $g$-finite, then the partial order $\le$ is uniquely determined by the condition (c).
Therefore the following property of $\Sigma(A)$ is important since it claims that the partial order on $\twosilt A$ can be recovered from $\Sigma(A)$ if $A$ is $g$-finite.

\begin{proposition}\label{tilting is ordered}
For a finite dimensional algebra $A$ over a field $k$, $(\Sigma(A),C(A))$ is an ordered fan.
\end{proposition}

\begin{proof}
This is shown in \cite[Theorem 6.11]{DIJ}.
\end{proof}

Now we explain idempotent reductions of $g$-fans. We prepare the following general notion.

\begin{definition}\label{define subfan}
Let $(\Sigma,\sigma_+)$ be a sign-coherent fan, and $e_1,\ldots,e_d$ the basis in Definition \ref{define sign-coherent}(c). For each subset $I\subset\{1,\ldots,d\}$, 
consider the subspace
\[\R^I=\bigoplus_{i\in I}\R e_i\subset \R^d.\]
Then the subset
\[\Sigma^I:=\{\sigma\in\Sigma\mid\sigma\subset\R^I\}\]
is a subfan of $\Sigma$ thanks to the condition Definition \ref{define sign-coherent}(c).
\end{definition}

Now let $A$ be a basic finite dimensional algebra over a field $k$, and $1=e_1+\cdots+e_n$ the orthogonal primitive idempotents. 
As in Definition \ref{define subfan}, for each subset $I\subset\{1,\ldots,n\}$, we obtain a subspace
\[K_0(\proj A)_\R^I:=\bigoplus_{i\in I}\R[e_iA]\subset K_0(\proj A)_\R\]
and a subfan of $\Sigma(A)$ given by 
\[\Sigma^I(A):=\{\sigma\in\Sigma(A)\mid \sigma\subset K_0(\proj A)_\R^I\}.\]
On the other hand, we consider the idempotent
\[e=e^I:=\sum_{i\in I}e_i\in A\]
and the corresponding subalgebra $eAe$ of $A$.
Then we have a fully faithful functor
\[-\otimes_{eAe}eA:\proj eAe\to\proj A\]
which induces an isomorphism
\begin{equation}\label{identify K_0}
-\otimes_{eAe}eA:K_0(\proj eAe)_\R\simeq K_0(\proj A)_\R^I.
\end{equation}
We are ready to state the following result.

\begin{theorem}\label{eAe}
Let $A$ be a finite dimensional algebra over a field $k$, and $1=e_1+\cdots+e_n$ the orthogonal primitive idempotents. 
For each subset $I\subset\{1,\ldots,n\}$ and $e:=e^I$, the isomorphism \eqref{identify K_0} gives an isomorphism of fans
\[\Sigma(eAe)\simeq\Sigma^I(A).\] 
\end{theorem}

\begin{proof} 
For the thick subcategory $\Kb(\add eA)$ of $\Kb(\proj A)$, we have a triangle equivalence
\[-\otimes_{eAe}eA:\Kb(\proj eAe)\simeq\Kb(\add eA)\subset \Kb(\proj A).\]
Therefore we have a bijection 
\[-\otimes_{eAe}eA:\twopsilt(eAe) \simeq \{T\in\twopsilt A\ |\ T^0,T^{-1}\in\add eA\}.\]
Since $K_0(\add eA)_\R=K_0(\proj A)_\R^I$ holds, the assertion follows immediately.
\end{proof}

\begin{example}
Let $A$ be the algebra in Example \ref{rank 3 example}.
Then the subfans $\Sigma^{\{1,2\}}(A)$, $\Sigma^{\{1,3\}}(A)$, $\Sigma^{\{2,3\}}(A)$ are the following.
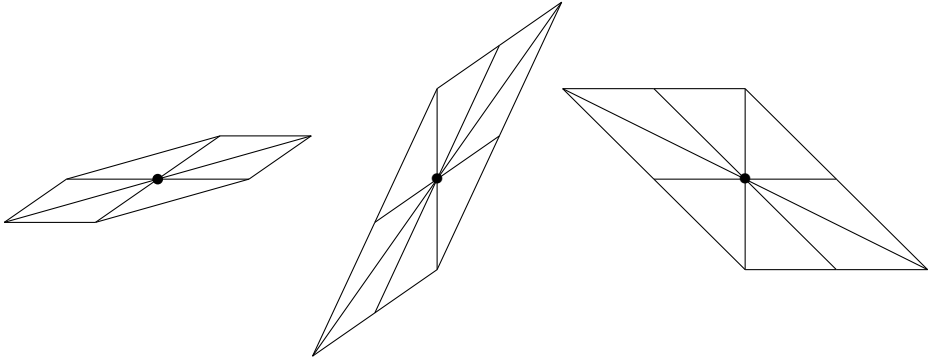
\[\begin{tikzpicture}[baseline=0mm,scale=1]
            \node(x) at(215:1) {}; 
            \node(-x) at($-1*(x)$) {}; 
            \node(y) at(0:1.2) {}; 
            \node(-y) at($-1*(y)$) {}; 
            \node(z) at(90:1.2) {}; 
            \node(-z) at($-1*(z)$) {}; 
            \coordinate(0) at(0:0);
            \fill (0) circle (2pt);
            \coordinate(1) at($1*(x) + 0*(y) + 0*(z)$) ;
            \coordinate(2) at($0*(x) + 1*(y) + 0*(z)$) ;
            \coordinate(4) at($-1*(x) + 0*(y) + 0*(z)$) ;
            \coordinate(5) at($0*(x) + -1*(y) + 0*(z)$) ;
            \coordinate(8) at($1*(x) + -1*(y) + 0*(z)$) ;
            \coordinate(11) at($-1*(x) + 1*(y) + 0*(z)$) ;
            \draw[] (1)--(4); \draw[] (2)--(5); \draw[] (8)--(11);
            \draw[] (1)--(2) ;
            \draw[] (4)--(5) ;
           \draw[] (1)--(8) ;
            \draw[] (11)--(2) ;
            \draw[] (11)--(4) ;
            \draw[] (8)--(5) ;
    \end{tikzpicture}
\begin{tikzpicture}[baseline=0mm,scale=1]
            \node(0) at(0:0) {$\bullet$}; 
            \node(x) at(215:1) {}; 
            \node(-x) at($-1*(x)$) {}; 
            \node(y) at(0:1.2) {}; 
            \node(-y) at($-1*(y)$) {}; 
            \node(z) at(90:1.2) {}; 
            \node(-z) at($-1*(z)$) {}; 
            \coordinate(1) at($1*(x) + 0*(y) + 0*(z)$) ;
            \coordinate(3) at($0*(x) + 0*(y) + 1*(z)$) ;
            \coordinate(4) at($-1*(x) + 0*(y) + 0*(z)$) ;
            \coordinate(6) at($0*(x) + 0*(y) + -1*(z)$) ;
            \coordinate(9) at($-1*(x) + 0*(y) + 1*(z)$) ;
            \coordinate(12) at($2*(x) + 0*(y) + -1*(z)$) ;
            \coordinate(14) at($-2*(x) + 0*(y) + 1*(z)$) ;
            \coordinate(16) at($1*(x) + 0*(y) + -1*(z)$) ;
            \draw[] (1)--(4); \draw[] (3)--(6); \draw[] (9)--(16); \draw[] (12)--(14);
            \draw[] (1)--(3) ;
            \draw[] (4)--(6) ;
            \draw[] (9)--(3) ;
            \draw[] (12)--(1) ;
            \draw[] (14)--(9) ;
            \draw[] (12)--(16) ;
            \draw[] (16)--(6) ;
            \draw[] (14)--(4) ;
    \end{tikzpicture}
\begin{tikzpicture}[baseline=0mm,scale=1]
            \node(0) at(0:0) {$\bullet$}; 
            \node(x) at(215:1) {}; 
            \node(-x) at($-1*(x)$) {}; 
            \node(y) at(0:1.2) {}; 
            \node(-y) at($-1*(y)$) {}; 
            \node(z) at(90:1.2) {}; 
            \node(-z) at($-1*(z)$) {}; 
            \coordinate(2) at($0*(x) + 1*(y) + 0*(z)$) ;
            \coordinate(3) at($0*(x) + 0*(y) + 1*(z)$) ;
            \coordinate(5) at($0*(x) + -1*(y) + 0*(z)$) ;
            \coordinate(6) at($0*(x) + 0*(y) + -1*(z)$) ;
            \coordinate(7) at($0*(x) + 2*(y) + -1*(z)$) ;
            \coordinate(13) at($0*(x) + -1*(y) + 1*(z)$) ;
            \coordinate(15) at($0*(x) + 1*(y) + -1*(z)$) ;
            \coordinate(17) at($0*(x) + -2*(y) + 1*(z)$) ;
            \draw[] (2)--(5); \draw[] (3)--(6); \draw[] (7)--(17); \draw[] (13)--(15);
            \draw[] (2)--(3) ;
            \draw[] (5)--(6) ;
            \draw[] (7)--(2) ;
           \draw[] (13)--(3) ;
            \draw[] (7)--(15) ;
            \draw[] (17)--(13) ;
            \draw[] (17)--(5) ;
            \draw[] (15)--(6) ;
    \end{tikzpicture}
    \]
\end{example}

We show that the product of fans corresponds to the product of algebras. 

\begin{definition}\label{define product fan}\cite{F}
Let $\Sigma$ and $\Sigma'$ be fans in $\mathbb{R}^d$ and $\mathbb{R}^{d'}$ respectively. 
We define a \emph{product fan} $\Sigma\times\Sigma'$ in $\mathbb{R}^{d+d'}$ by
\[\Sigma\times \Sigma':=\{\sigma\times\sigma'\mid \sigma\in\Sigma,\ \sigma'\in\Sigma'\}.\]
We say that $\Sigma$ is \emph{indecomposable} if it can not be written as a product of two fans.
\end{definition}

The following result shows that the decomposition of $g$-fans precisely corresponds to the decomposition of algebras.

\begin{theorem}\label{idempotent decomp}
Let $A$ be a finite dimensional algebra over a field $k$.
\begin{enumerate}[\rm(a)]
\item If $A=A_1\times \cdots\times A_\ell$ for a finite dimensional algebra $A_i$, then we have
$$\Sigma(A)=\Sigma(A_1)\times \cdots\times\Sigma(A_\ell).$$
\item In (a), assume that each $A_i$ is ring-indecomposable. Then, for each decomposition $\Sigma(A)=\Sigma^1\times\cdots\times\Sigma^m$, there exists a decomposition $\{1,\ldots,\ell\}=\bigsqcup_{j=1}^m I_j$ such that $\Sigma^j=\Sigma(\prod_{i\in I_j}A_i)$ for each $1\le j\le m$.
\item If $A$ is ring-indecomposable, then the fan $\Sigma(A)$ is indecomposable.
\end{enumerate}
\end{theorem}

\begin{proof}[Proof of (a)]
There is a bijection 
\[\twopsilt A_1\times\cdots\times\twopsilt A_\ell\simeq\twopsilt(A_1\times \cdots\times A_\ell)\ \mbox{ given by }\ (T_1,\cdots,T_\ell)\mapsto T_1\oplus\cdots\oplus T_\ell.\]
Since $C(T_1\oplus\cdots\oplus T_\ell)=C(T_1)\times\cdots\times C(T_\ell)$, we obtain the desired equation.
\end{proof}

To prove Theorem \ref{idempotent decomp}(b)(c), we need the following observation.

\begin{lemma}\label{presilt and idempotent}
Let $e\in A$ be an idempotent and $f:=1-e$.
\begin{enumerate}[\rm(a)]
\item We have an injective map
\[\twopsilt^1(eAe)\sqcup\twopsilt^1(fAf)\subset\twopsilt^1A\]
sending $P\in\twopsilt^1(eAe)$ to $P\otimes_{eAe}eA$ and $Q\in\twopsilt^1(fAf)$ to $Q\otimes_{fAf}fA$.
\item The equality holds in (a) if and only if $e$ is a central idempotent. 
\end{enumerate}
\end{lemma}

\begin{proof}
(a) Immediate from Theorem \ref{eAe}.

(b) It suffices to prove ``only if'' part. Assume that the equality holds.
If $e$ is not central, then at least one of $\Hom_A(eA,fA)$ and $\Hom_A(fA,eA)$ is non-zero. Without loss of generality, assume $\Hom_A(eA,fA)\neq0$.
Take an indecomposable direct summand $e_iA$ of $eA$ such that $\Hom_A(e_iA,fA)\neq0$, and take a minimal left $(\add fA)$-approximation of $e_iA$
\[e_iA\to Q\to X\to e_iA[1].\]
Then $X$ is an indecomposable direct summand of $\mu^-_{eA}(A)\in\twosilt A$, and thus $X\in\twopsilt^1A$. Moreover $Q$ is non-zero by our choice, and hence $X$ is not contained in the image of the map in (a).
Thus $e$ has to be central.
\end{proof}

We are ready to prove Theorem \ref{idempotent decomp}(b)(c).

\begin{proof}[Proof of Theorem \ref{idempotent decomp}(b)(c)]
It suffices to prove (b).
Without loss of generality, we can assume that $A$ is basic. Let $1=e_1+\cdots+e_n$ be a primitive orthogonal idempotents and $P_i:=e_iA$.
By definition, there exist $\sigma^j\in\Sigma^j$ for each $1\le j\le m$ such that
\[C(A)=\sigma^1\times\cdots\times\sigma^m.\]
Since $C(A)=\cone\{[P_1],\ldots,[P_n]\}$, there exists a decomposition $\{1,\ldots,n\}=\bigsqcup_{j=1}^m J_j$ such that
\[\sigma^j=\cone\{[P_i]\mid i\in J_j\}\]
for each $1\le j\le m$. Let
\[e^j:=\sum_{i\in J_j}e_i\in A.\]
Then $\Sigma^j$ is a fan in $K_0(\proj A)_\R^{J_j}=K_0(\proj e^jAe^j)_\R$, and hence each ray of $\Sigma(A)$ belongs to $K_0(\proj e^jAe^j)$ for some $1\le j\le m$. Thus the map
\[\bigsqcup_{j=1}^m\twopsilt^1(e^jAe^j)\to\twopsilt^1A\]
sending $P\in\twopsilt^1(e^jAe^j)$ to $P\otimes_{e^jAe^j}e^jA$ is bijective.
By Lemma \ref{presilt and idempotent}, $e^j$ is a central idempotent of $A$. Thus there exists a decomposition $\{1,\ldots,\ell\}=\bigsqcup_{j=1}^m I_j$ such that $Ae^j=\prod_{i\in I_j}A_i$, and the assertion follows.
\end{proof}

Next we explain the sign decomposition of $g$-fans \cite{Ao}. We prepare the following general notion.

\begin{definition}\label{define subfan2}
Let $(\Sigma,\sigma_+)$ be a sign-coherent fan in $\R^d$ and $\epsilon\in\{\pm1\}^d$. 
Consider the basis $e_1,\ldots,e_d$ of $\R^d$ and the orthant $\R^d_\epsilon$ given in Definition \ref{define sign-coherent}(b). 
Define a subfan of $\Sigma$ by
\[\Sigma_\epsilon:=\{\sigma\in\Sigma\mid\sigma\subset\R^d_\epsilon\}.\]
Thanks to the condition Definition \ref{define sign-coherent}(b), we have $\Sigma=\bigcup_{\epsilon\in\{\pm1\}^d}\Sigma_\epsilon$. 
\end{definition}

Now let $A$ be a basic finite dimensional algebra over a field $k$ with $|A|=n$, and $1=e_1+\cdots+e_n$ the orthogonal primitive idempotents. For $\epsilon\in\{\pm1\}^n$, 
as in Definition \ref{define subfan2}, we obtain an orthant
\[K_0(\proj A)_{\epsilon,\R}:=\cone(\epsilon_i[e_iA]\mid i\in\{1,\ldots,n\})\]
and a subfan of $\Sigma(A)$ given by
\[\Sigma_\epsilon(A):=\{\sigma\in\Sigma(A)\mid \sigma\subset K_0(\proj A)_{\epsilon,\R}\}.\]

We can describe the fan $\Sigma_\epsilon(A)$ by a simpler algebra defined as follows.

\begin{definition}\label{define A_epsilon}
For $\epsilon\in\{\pm1\}^n$ , let
\[e^+_\epsilon:=\sum_{\epsilon_i=1}e_i\ \mbox{ and }\ e^-_\epsilon:=\sum_{\epsilon_i=-1}e_i.\]
We denote by $A_{\epsilon}$ the subalgebra of $A$ given by
\[A_\epsilon:=\left[\begin{array}{cc}e^+_\epsilon Ae^+_\epsilon&e^+_\epsilon Ae^-_\epsilon\\
0&e^-_\epsilon Ae^-_\epsilon\end{array}\right].\]
\end{definition}

The functor $-\otimes_{A_\epsilon}A:\proj A_\epsilon\to\proj A$ gives an isomorphism
\begin{equation}\label{identify K_0 2}
-\otimes_{A_\epsilon}A:K_0(\proj A_\epsilon)_\R\simeq K_0(\proj A)_\R
\end{equation}
and a functor
\begin{equation}\label{epsilon functor}
-\otimes_{A_\epsilon} A:\KKK^\epsilon(\proj A_\epsilon)\to\KKK^\epsilon(\proj A).
\end{equation}
We denote by $\KKK^\epsilon(\proj A)$ the full subcategory of $\Kb(\proj A)$ consisting of 2-term complexes of the form $P^{-1}\to P^0$ with $P^{-1}\in\add e^-_\epsilon A$ and $P^0\in\add e^+_\epsilon A$. Let
\[\twopsilt_\epsilon A:=\twopsilt A\cap\KKK^\epsilon(\proj A)\ \mbox { and }\ \twosilt_\epsilon A:=\twosilt A\cap\KKK^\epsilon(\proj A).\]
Thus the cones in $\Sigma_\epsilon(A)$ are given by the elements in $\twopsilt_\epsilon A$. 
Now we state main properties of sign decomposition, where the part (b) is a generalization of \cite[Theorem 4.5]{Ao}.

\begin{theorem}\label{sign decomposition}
Let $A$ be a basic finite dimensional algebra algebra over a field $k$, and $1=e_1+\cdots+e_n$ the orthogonal primitive idempotents. Let $\epsilon\in\{\pm1\}^n$.
\begin{enumerate}[\rm(a)]
\item \cite[Proposition 3.2]{Ao} The subset $\twosilt_\epsilon A$ of the poset $\twosilt A$ is an interval with a maximal element $(e_+A)_{\min}$ and a minimal element $(e_-A)_{\max}$.
\item The functor \eqref{epsilon functor} gives bijections $\twopsilt_\epsilon A_\epsilon\simeq\twopsilt_\epsilon A$ and $\twosilt_\epsilon A_\epsilon\simeq\twosilt_\epsilon A$, and the isomorphism \eqref{identify K_0 2} gives an isomorphism of fans
\[\Sigma_\epsilon(A_\epsilon)\simeq\Sigma_\epsilon(A).\]
\end{enumerate}
\end{theorem}

\begin{proof}
(b) The following properties of the functor $F:\KKK^\epsilon(\proj A_\epsilon)\to\KKK^\epsilon(\proj A)$ in \eqref{epsilon functor} can be checked easily.
\begin{enumerate}[\rm(i)]
\item $F$ is full and dense.
\item If $P\in\KKK^\epsilon(\proj A_\epsilon)$ satisfies $F(P)\simeq 0$, then $P\simeq 0$.
\item $\Hom_{\Kb(\proj A_\epsilon)}(P,Q[1])\simeq \Hom_{\Kb(\proj A)}(F(P),F(Q)[1])$.
\end{enumerate}
By (i) and (ii), $F$ gives a bijection between the isomorphism classes of indecomposable objects in $\KKK^\epsilon(\proj A_\epsilon)$ and those in $\KKK^\epsilon(\proj A)$.
Therefore by (iii), $F$ gives a bijection 
\[F:\twopsilt_\epsilon A_\epsilon\simeq\twopsilt_\epsilon A.\]
Thus the assertion follows.
\end{proof}

\begin{example}
Let $A$ be the algebra in Example \ref{rank 3 example}.
Then $\Sigma_\epsilon(A)$ is the following.
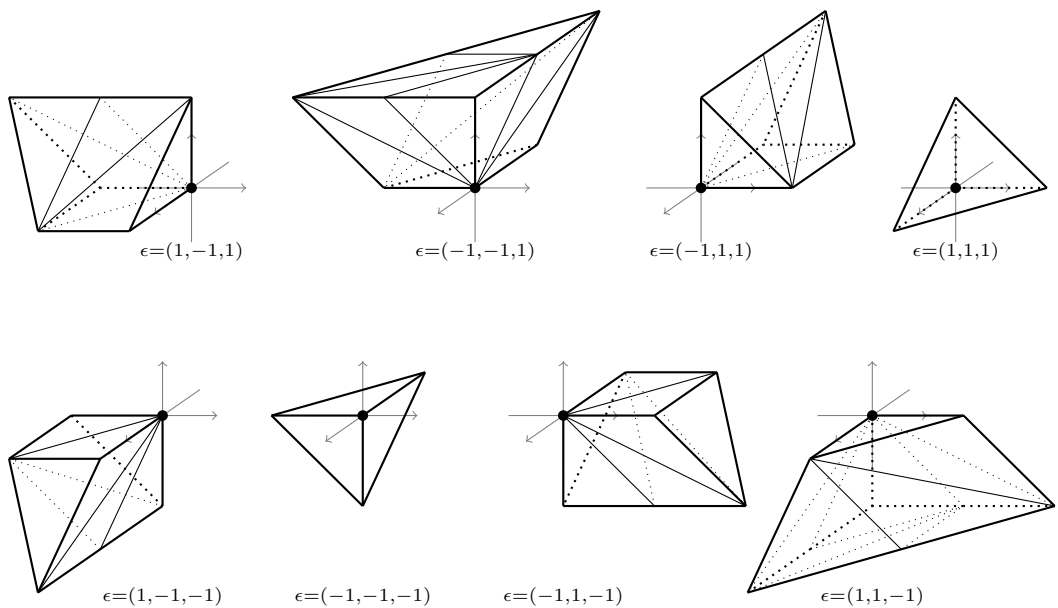
\[\begin{tikzpicture}[baseline=0mm,scale=1]
            \node(x) at(215:1) {}; 
            \node(-x) at($-1*(x)$) {}; 
            \node(y) at(0:1.2) {}; 
            \node(-y) at($-1*(y)$) {}; 
            \node(z) at(90:1.2) {}; 
            \node(-z) at($-1*(z)$) {}; 
            \draw[gray, <-] ($0.6*(x)$)--($0.6*(-x)$); \draw[gray, <-] ($0.6*(y)$)--($0.6*(-y)$); \draw[gray, <-] ($0.6*(z)$)--($0.6*(-z)$);
            \node(-1) at($-.7*(z)$) {${\scriptstyle \epsilon=(1,-1,1)}$};            
            \coordinate(0) at(0:0);
            \fill (0) circle (2pt);
            \coordinate(1) at($1*(x) + 0*(y) + 0*(z)$) ;
            \coordinate(3) at($0*(x) + 0*(y) + 1*(z)$) ;
            \coordinate(5) at($0*(x) + -1*(y) + 0*(z)$) ;
            \coordinate(8) at($1*(x) + -1*(y) + 0*(z)$) ;
            \coordinate(13) at($0*(x) + -1*(y) + 1*(z)$) ;
            \coordinate(17) at($0*(x) + -2*(y) + 1*(z)$) ;
            \draw[thick] (0)--(1); \draw[thick] (0)--(3); \draw[dotted, thick] (0)--(5); \draw[dotted] (0)--(8); \draw[dotted] (0)--(13); \draw[dotted] (0)--(17);             
            \draw[thick] (1)--(3) ;
            \draw[thick] (1)--(8) ;
            \draw[] (3)--(8) ;
            \draw[thick] (13)--(3) ;
            \draw[] (13)--(8) ;
            \draw[thick] (17)--(13) ;
            \draw[thick] (17)--(8) ;
            \draw[dotted, thick] (17)--(5) ;
            \draw[dotted, thick] (8)--(5) ;
    \end{tikzpicture}
\begin{tikzpicture}[baseline=0mm,scale=1]
            \node(x) at(215:1) {}; 
            \node(-x) at($-1*(x)$) {}; 
            \node(y) at(0:1.2) {}; 
            \node(-y) at($-1*(y)$) {}; 
            \node(z) at(90:1.2) {}; 
            \node(-z) at($-1*(z)$) {}; 
            \draw[gray, <-] ($0.6*(x)$)--($0.6*(-x)$); \draw[gray, <-] ($0.6*(y)$)--($0.6*(-y)$); \draw[gray, <-] ($0.6*(z)$)--($0.6*(-z)$);
            \node(-1) at($-.7*(z)$) {${\scriptstyle \epsilon=(-1,-1,1)}$};                        
            \coordinate(0) at(0:0);
            \fill (0) circle (2pt);
            \coordinate(3) at($0*(x) + 0*(y) + 1*(z)$) ;
            \coordinate(4) at($-1*(x) + 0*(y) + 0*(z)$) ;
            \coordinate(5) at($0*(x) + -1*(y) + 0*(z)$) ;
            \coordinate(9) at($-1*(x) + 0*(y) + 1*(z)$) ;
            \coordinate(13) at($0*(x) + -1*(y) + 1*(z)$) ;
            \coordinate(14) at($-2*(x) + 0*(y) + 1*(z)$) ;
            \coordinate(17) at($0*(x) + -2*(y) + 1*(z)$) ;
            \coordinate(18) at($-1*(x) + -1*(y) + 1*(z)$) ;
            \draw[thick] (0)--(3); \draw[thick] (0)--(4); \draw[thick] (0)--(5); \draw[] (0)--(9); \draw[] (0)--(13); \draw[] (0)--(14); \draw[] (0)--(17);
            \draw[dotted, thick] (4)--(5) ;
            \draw[thick] (9)--(3) ;
            \draw[thick] (13)--(3) ;
            \draw[thick] (14)--(9) ;
            \draw[] (13)--(9) ;
            \draw[thick] (17)--(13) ;
            \draw[thick] (18)--(14) ;
            \draw[] (18)--(9) ;
            \draw[] (17)--(9) ;
            \draw[thick] (17)--(5) ;
            \draw[dotted] (14)--(5) ;
            \draw[dotted] (18)--(5) ;
            \draw[thick] (14)--(4) ;
            \draw[thick] (18)--(17) ;
    \end{tikzpicture}
\begin{tikzpicture}[baseline=0mm,scale=1]
            \node(x) at(215:1) {}; 
            \node(-x) at($-1*(x)$) {}; 
            \node(y) at(0:1.2) {}; 
            \node(-y) at($-1*(y)$) {}; 
            \node(z) at(90:1.2) {}; 
            \node(-z) at($-1*(z)$) {}; 
            \draw[gray, <-] ($0.6*(x)$)--($0.6*(-x)$); \draw[gray, <-] ($0.6*(y)$)--($0.6*(-y)$); \draw[gray, <-] ($0.6*(z)$)--($0.6*(-z)$);
            \coordinate(0) at(0:0);
            \fill (0) circle (2pt);
            \node(-1) at($-.7*(z)$) {${\scriptstyle \epsilon=(-1,1,1)}$};                        
            \coordinate(2) at($0*(x) + 1*(y) + 0*(z)$) ;
            \coordinate(3) at($0*(x) + 0*(y) + 1*(z)$) ;
            \coordinate(4) at($-1*(x) + 0*(y) + 0*(z)$) ;
            \coordinate(9) at($-1*(x) + 0*(y) + 1*(z)$) ;
            \coordinate(11) at($-1*(x) + 1*(y) + 0*(z)$) ;
            \coordinate(14) at($-2*(x) + 0*(y) + 1*(z)$) ;
            \draw[thick] (0)--(2); \draw[thick] (0)--(3); \draw[thick,dotted] (0)--(4); \draw[dotted] (0)--(9); \draw[dotted] (0)--(11); \draw[dotted] (0)--(14); 
            \draw[thick] (2)--(3) ;
            \draw[] (9)--(2) ;
            \draw[thick] (14)--(3) ;
            \draw[thick] (11)--(2) ;
            \draw[] (14)--(2) ;
            \draw[dotted,thick] (11)--(4) ;
            \draw[thick] (11)--(14) ;
            \draw[dotted, thick] (14)--(4) ;
    \end{tikzpicture}
\begin{tikzpicture}[baseline=0mm,scale=1]
            \node(x) at(215:1) {}; 
            \node(-x) at($-1*(x)$) {}; 
            \node(y) at(0:1.2) {}; 
            \node(-y) at($-1*(y)$) {}; 
            \node(z) at(90:1.2) {}; 
            \node(-z) at($-1*(z)$) {}; 
            \draw[gray, <-] ($0.6*(x)$)--($0.6*(-x)$); \draw[gray, <-] ($0.6*(y)$)--($0.6*(-y)$); \draw[gray, <-] ($0.6*(z)$)--($0.6*(-z)$);
            \node(-1) at($-.7*(z)$) {${\scriptstyle \epsilon=(1,1,1)}$};            

            \coordinate(0) at(0:0);
            \fill (0) circle (2pt);
            \coordinate(1) at($1*(x) + 0*(y) + 0*(z)$) ;
            \coordinate(2) at($0*(x) + 1*(y) + 0*(z)$) ;
            \coordinate(3) at($0*(x) + 0*(y) + 1*(z)$) ;
            \draw[thick,dotted] (0)--(1); \draw[thick,dotted] (0)--(2); \draw[thick,dotted] (0)--(3);
            \draw[thick] (1)--(2) ;
            \draw[thick] (1)--(3) ;
            \draw[thick] (2)--(3) ;
    \end{tikzpicture}    
    \]
    \[
\begin{tikzpicture}[baseline=0mm,scale=1]
            \node(x) at(215:1) {}; 
            \node(-x) at($-1*(x)$) {}; 
            \node(y) at(0:1.2) {}; 
            \node(-y) at($-1*(y)$) {}; 
            \node(z) at(90:1.2) {}; 
            \node(-z) at($-1*(z)$) {}; 
            \draw[gray, <-] ($0.6*(x)$)--($0.6*(-x)$); \draw[gray, <-] ($0.6*(y)$)--($0.6*(-y)$); \draw[gray, <-] ($0.6*(z)$)--($0.6*(-z)$);

            \node(-1) at($-2*(z)$) {${\scriptstyle \epsilon=(1,-1,-1)}$};            
            \coordinate(0) at(0:0);
            \fill (0) circle (2pt);            
            \coordinate(1) at($1*(x) + 0*(y) + 0*(z)$) ;
            \coordinate(5) at($0*(x) + -1*(y) + 0*(z)$) ;
            \coordinate(6) at($0*(x) + 0*(y) + -1*(z)$) ;
            \coordinate(8) at($1*(x) + -1*(y) + 0*(z)$) ;
            \coordinate(12) at($2*(x) + 0*(y) + -1*(z)$) ;
            \coordinate(16) at($1*(x) + 0*(y) + -1*(z)$) ;
            \draw[thick] (0)--(1); \draw[thick] (0)--(5); \draw[thick] (0)--(6); \draw[dotted, thick] (5)--(6) ; \draw[] (0)--(8) ; \draw[] (0)--(12) ; \draw[] (0)--(16) ;           
            \draw[thick] (1)--(8) ;
            \draw[thick] (12)--(1) ;
            \draw[thick] (12)--(8) ;
            \draw[dotted] (16)--(8) ;
            \draw[thick] (12)--(6) ;
            \draw[dotted] (8)--(6) ;
            \draw[thick] (8)--(5) ;
    \end{tikzpicture}
\begin{tikzpicture}[baseline=0mm,scale=1]
            \node(x) at(215:1) {}; 
            \node(-x) at($-1*(x)$) {}; 
            \node(y) at(0:1.2) {}; 
            \node(-y) at($-1*(y)$) {}; 
            \node(z) at(90:1.2) {}; 
            \node(-z) at($-1*(z)$) {}; 
            \draw[gray, <-] ($0.6*(x)$)--($0.6*(-x)$); \draw[gray, <-] ($0.6*(y)$)--($0.6*(-y)$); \draw[gray, <-] ($0.6*(z)$)--($0.6*(-z)$);
            \node(-1) at($-2*(z)$) {${\scriptstyle \epsilon=(-1,-1,-1)}$};            

            \coordinate(0) at(0:0);
            \fill (0) circle (2pt);
            \coordinate(4) at($-1*(x) + 0*(y) + 0*(z)$) ;
            \coordinate(5) at($0*(x) - 1*(y) + 0*(z)$) ;
            \coordinate(6) at($0*(x) + 0*(y) - 1*(z)$) ;
            \draw[thick] (0)--(4); \draw[thick] (0)--(5); \draw[thick] (0)--(6);
            \draw[thick] (4)--(5) ;
            \draw[thick] (4)--(6) ;
            \draw[thick] (5)--(6) ;
    \end{tikzpicture}    
\begin{tikzpicture}[baseline=0mm,scale=1]
            \node(x) at(215:1) {}; 
            \node(-x) at($-1*(x)$) {}; 
            \node(y) at(0:1.2) {}; 
            \node(-y) at($-1*(y)$) {}; 
            \node(z) at(90:1.2) {}; 
            \node(-z) at($-1*(z)$) {}; 
            \draw[gray, <-] ($0.6*(x)$)--($0.6*(-x)$); \draw[gray, <-] ($0.6*(y)$)--($0.6*(-y)$); \draw[gray, <-] ($0.6*(z)$)--($0.6*(-z)$);
            \node(-1) at($-2*(z)$) {${\scriptstyle \epsilon=(-1,1,-1)}$};            
            \coordinate(0) at(0:0);
            \fill (0) circle (2pt);            
            \coordinate(2) at($0*(x) + 1*(y) + 0*(z)$) ;
            \coordinate(4) at($-1*(x) + 0*(y) + 0*(z)$) ;
            \coordinate(6) at($0*(x) + 0*(y) + -1*(z)$) ;
            \coordinate(7) at($0*(x) + 2*(y) + -1*(z)$) ;
            \coordinate(11) at($-1*(x) + 1*(y) + 0*(z)$) ;
            \coordinate(15) at($0*(x) + 1*(y) + -1*(z)$) ;
            \draw[thick] (0)--(4); \draw[thick] (0)--(2); \draw[thick] (0)--(6); \draw[] (0)--(7); \draw[] (0)--(11); \draw[] (0)--(15);
            \draw[dotted, thick] (4)--(6) ;
            \draw[thick] (7)--(2) ;
            \draw[thick] (11)--(7) ;
            \draw[thick] (11)--(2) ;
            \draw[thick] (6)--(7) ;
            \draw[dotted] (7)--(4) ;
            \draw[thick] (11)--(4) ;
            \draw[dotted] (15)--(4) ;
    \end{tikzpicture}
\begin{tikzpicture}[baseline=0mm,scale=1]
            \node(x) at(215:1) {}; 
            \node(-x) at($-1*(x)$) {}; 
            \node(y) at(0:1.2) {}; 
            \node(-y) at($-1*(y)$) {}; 
            \node(z) at(90:1.2) {}; 
            \node(-z) at($-1*(z)$) {}; 
            \draw[gray, <-] ($0.6*(x)$)--($0.6*(-x)$); \draw[gray, <-] ($0.6*(y)$)--($0.6*(-y)$); \draw[gray, <-] ($0.6*(z)$)--($0.6*(-z)$);
            
            \node(-1) at($-2*(z)$) {${\scriptstyle \epsilon=(1,1,-1)}$};            
            \coordinate(0) at(0:0);
            \fill (0) circle (2pt);
            \coordinate(1) at($1*(x) + 0*(y) + 0*(z)$) ;
            \coordinate(2) at($0*(x) + 1*(y) + 0*(z)$) ;
            \coordinate(6) at($0*(x) + 0*(y) + -1*(z)$) ;
           \coordinate(7) at($0*(x) + 2*(y) + -1*(z)$) ;
            \coordinate(10) at($1*(x) + 1*(y) + -1*(z)$) ;
            \coordinate(12) at($2*(x) + 0*(y) + -1*(z)$) ;
            \coordinate(15) at($0*(x) + 1*(y) + -1*(z)$) ;
            \coordinate(16) at($1*(x) + 0*(y) + -1*(z)$) ;
            \draw[thick] (0)--(1); \draw[thick] (0)--(2); \draw[dotted,thick] (0)--(6); \draw[dotted] (0)--(16); \draw[dotted] (0)--(7); \draw[dotted] (0)--(12); \draw[dotted] (0)--(15); \draw[thick] (1)--(2);            
            \draw[] (7)--(1) ;
            \draw[thick] (7)--(2) ;
            \draw[thick] (10)--(7) ;
            \draw[] (10)--(1) ;
            \draw[thick] (12)--(1) ;
            \draw[dotted] (10)--(15) ;
            \draw[dotted, thick] (7)--(15) ;
            \draw[thick] (12)--(10) ;
            \draw[dotted, thick] (12)--(16) ;
            \draw[dotted] (12)--(15) ;
            \draw[dotted, thick] (16)--(6) ;
            \draw[dotted] (16)--(15) ;
            \draw[dotted, thick] (15)--(6) ;
    \end{tikzpicture}
    \]    
\end{example}

We give another condition for two algebras $A$ and $B$ such that $\Sigma_\epsilon(A)$ and $\Sigma_\epsilon(B)$ are isomorphic. For each ideal $I$ of $A$ contained in $\rad A$, let $B:=A/I$. Then the functor $-\otimes_AB:\proj A\to\proj B$ induces an isomorphism
\begin{equation}\label{identify K_0 3}
-\otimes_AB:K_0(\proj A)_\R\simeq K_0(\proj B)_\R
\end{equation}
and a functor
\begin{equation}\label{A/I functor}
-\otimes_AB:\KKK^\epsilon(\proj A)\to\KKK^\epsilon(\proj B).
\end{equation}
We give a sufficient condition for $I$ such that \eqref{identify K_0 3} induces an isomorphism $\Sigma_\epsilon(A)\simeq\Sigma_\epsilon(B)$.

\begin{proposition}\label{A to A/I}
Let $A$ be a basic finite dimensional algebra algebra over a field $k$, $1=e_1+\cdots+e_n$ the orthogonal primitive idempotents, and $\epsilon\in\{\pm1\}^n$. Let $I$ be an ideal of $A$, and $B:=A/I$. Assume that $I$ satisfies
\[I\subseteq\rad A,\ e^+_\epsilon Ie^-_\epsilon=0,\ (e^+_\epsilon Ie^+_\epsilon)(e^+_\epsilon Ae^-_\epsilon)=0\ \mbox{ and }\ (e^+_\epsilon Ae^-_\epsilon)(e^-_\epsilon Ie^-_\epsilon)=0.\]
Then the functor \eqref{A/I functor} gives bijections $\twopsilt_\epsilon A\simeq\twopsilt_\epsilon B$ and $\twosilt_\epsilon A\simeq\twosilt_\epsilon B$, and the isomorphism \eqref{identify K_0 3} gives an isomorphism of fans
\[\Sigma_\epsilon(A)\simeq\Sigma_\epsilon(B).\]
\end{proposition}

\begin{proof}
By our assumptions on $I$, for each $P^{-1}\in\add e^-_\epsilon A$ and $P^0\in\add e^+_\epsilon A$, the functor $-\otimes_AB$ gives an isomorphism $\Hom_A(P^{-1},P^0)\simeq\Hom_B(P^{-1}\otimes_AB,P^0\otimes_AB)$. Thus the functor $F:=-\otimes_AB:\KKK^\epsilon(\proj A)\to\KKK^\epsilon(\proj B)$ satisfies the following conditions 
\begin{enumerate}[\rm(i)]
\item $F$ is full and dense.
\item If $P\in\KKK^\epsilon(\proj A)$ satisfies $F(P)\simeq 0$, then $P\simeq 0$.
\item $\Hom_{\Kb(\proj A)}(P,Q[1])\simeq \Hom_{\Kb(\proj B)}(F(P),F(Q)[1])$.
\end{enumerate}
Thus $F$ gives a bijection $F:\twopsilt A\cap\KKK^\epsilon(\proj A)\simeq\twopsilt B\cap\KKK^\epsilon(\proj B)$, and the assertion follows.
\end{proof}

We record the following useful observation.

\begin{example}
In the setting of Definition \ref{define A_epsilon}, define an ideal $I_\epsilon$ of $A_\epsilon$ by 
\[I_\epsilon:=\left[\begin{array}{cc}\rad (e^+_\epsilon Ae^+_\epsilon)\cap\Ann_{e^+_\epsilon Ae^+_\epsilon}(e^+_\epsilon Ae^-_\epsilon)&0\\
0&\rad(e^-_\epsilon Ae^-_\epsilon)\cap\Ann(e^+_\epsilon Ae^-_\epsilon)_{e^-_\epsilon Ae^-_\epsilon}\end{array}\right].\]
For each ideal $I$ of $A_\epsilon$ contained in $I_\epsilon$, we have an isomorphism of fans
\[\Sigma_\epsilon(A)\simeq\Sigma_\epsilon(A_\epsilon)\simeq\Sigma_\epsilon(A_\epsilon/I)\]
by Theorem \ref{sign decomposition} and Proposition \ref{A to A/I}.
\end{example}

\begin{definition}\label{define subfan3}
Let $(\Sigma,\sigma_+)$ be a sign-coherent fan in $\R^d$, and $e_1,\ldots,e_d$ the basis of $\R^d$ given in Definition \ref{define sign-coherent}(c). For $1\le i\le d$ and $\delta\in\{\pm1\}$, consider a half space
\[\R^d_{i,\delta}:=\{x_1e_1+\cdots+x_de_d\in\R^d\mid \delta x_i\ge0\}\]
and define a subfan of $\Sigma$ by
\[\Sigma_{i,\delta}:=\{\sigma\in\Sigma\mid\sigma\subset\R^d_{i,\delta}\}.\]
Thanks to the condition Definition \ref{define sign-coherent}(c), we have $\Sigma=\Sigma_{i,+}\cup\Sigma_{i,-}$.
\end{definition}

Let $A$ be a basic finite dimensional algebra over a field $k$ with $|A|=n$, and $1=e_1+\cdots+e_n$ the orthogonal primitive idempotents. For $1\le i\le d$ and $\delta\in\{\pm1\}$, as in Definition \ref{define subfan3}, we obtain a half space
\[K_0(\proj A)_{i,\delta,\R}:=\{x_1[e_1A]+\cdots+x_d[e_dA]\in K_0(\proj A)_{\R}\mid \delta x_i\ge0\}\]
and a subfan of $\Sigma(A)$ given by
\[\Sigma_{i,\delta}(A):=\{\sigma\in\Sigma(A)\mid \sigma\subset K_0(\proj A)_{i,\delta,\R}\}.\]
For elements $T\ge T'$ in $\silt A$, we consider the interval
\[[T',T]:=\{U\in\silt A\mid T\ge U\ge T'\}.\]
The following result gives information how $g$-fans change under mutation.

\begin{theorem}
Let $A$ be a basic finite dimensional algebra over a field $k$ with $|A|=n$, and $1=e_1+\cdots+e_n$ the orthogonal primitive idempotents. For $1\le i\le n$, let $B:=\End_A(\mu_i(A))$. Then there exists a triangle functor $\Kb(\proj A)\to\Kb(\proj B)$ which sends $\mu_i(A)$ to $B$ and gives an isomorphism $K_0(\proj A)\simeq K_0(\proj B)$, a bijection
\[[A[1],\mu_i(A)]\simeq[\mu_i(B[1]),B]\]
and an isomorphism of fans
\[\Sigma_{i,-}(A)\simeq\Sigma_{i,+}(B).\]
\end{theorem}

\begin{proof}
Applying Proposition \ref{dg case} to $\TT:=\Kb(\proj A)$ and $T:=\mu_i(A)$, we obtain a triangle functor $F:\TT\to\Kb(\proj B)$ satisfying $F(T)=B$ and giving an isomorphism $K_0(\proj A)\simeq K_0(\proj B)$ and bijections
\[F:\rtwosilt{T}\TT\simeq\twosilt B\mbox{ and }F:\rtwopsilt{T}\TT\simeq\twopsilt B,\]
see Definition \ref{define silting}(e) for $\rtwosilt{T}\TT$ and $\rtwopsilt{T}\TT$.
Since $F(A[1])$ and $F(T[1])=B[1]$ have the same direct summands except the $i$-th one, we have $F(A[1])=\mu_i(B[1])$. Thus the bijection $F:\rtwosilt{T}\TT\simeq\twosilt B$ restricts to the desired bijection
\[[A[1],T]\simeq[\mu_i(B[1]),B].\]
Since the subfan $\Sigma_{i,-}(A)$ (respectively, $\Sigma_{i,+}(B)$) consists of the cones corresponding to the interval $[A[1],T]$ (respectively, $[\mu_i(B[1]),B]$), we obtain the desired isomorphism $\Sigma_{i,-}(A)\simeq\Sigma_{i,+}(B)$ of fans.
\end{proof}


\section{$g$-polytopes, $c$-polytopes and Newton polytopes}\label{section 5}
In this section, we introduce $g$-polytopes of finite dimensional algebras and characterize when they are convex. We show that convex $g$-polytopes are reflexive polytopes, and describe the dual polytopes as \emph{$c$-polytopes} associated with the set of the 2-term simple minded collections. We also describe Newton polytopes of $A$-modules by using the $g$-fan of $A$.

\subsection{Definition and basic properties}
With each nonsingular fan, we associate a (not necessarily convex) polytope as follows.

\begin{definition}\label{from fan to polytope}
Let $\Sigma$ be a nonsingular fan in $\R^d$. For each $\sigma\in\Sigma_d$, take a basis $v_1,\ldots,v_d$ of $\Z^d$ such that $\sigma=\cone\{v_1,\ldots,v_d\}$, and let
\[\sigma_{\le1}:=\conv\{0,v_1,\ldots,v_d\}\subset\R^d\]
be the convex hull. 
Define a (not necessarily convex) polytope in $\R^d$ by
\[\P(\Sigma):=\bigcup_{\sigma\in\Sigma_d}\sigma_{\le1}.\]
We say that $\Sigma$ is \emph{convex} if $\P(\Sigma)$ is convex.
\end{definition}

Applying this construction to $g$-fans, we obtain the following notion.

\begin{definition}
Let $A$ be a finite dimensional algebra over a field $k$. 
\begin{enumerate}[\rm(a)]
\item We call $\P(A):=\P(\Sigma(A))$ the \emph{$g$-polytope} of $A$. More precisely, for $T=T_1\oplus\cdots\oplus T_\ell\in\twopsilt A$ with indecomposable $T_i$, let
\[C_{\le1}(T) := \{\sum_{i=1}^\ell a_i[T_i]\mid a_i\ge 0, \ \sum_{i=1}^\ell a_i\le1\}\subset K_0(\proj A)_\R.\]
Then $\P(A)$ is defined by
\[\P(A):=\bigcup_{T\in\twosilt A}C_{\le1}(T).\]
\item We say that $A$ is \emph{$g$-convex} if the $g$-polytope $\P(A)$ is convex (i.e.\ $\Sigma(A)$ is convex).
\end{enumerate}
\end{definition}

\begin{example}
For integers $\ell,m\ge1$, let $A=A_{\ell,m}$ be the algebra in Example \ref{kase example}. Then $A$ is $g$-convex if and only if $\ell\le 3$ and $m\le 3$. For example, if $(\ell,m)=(4,5)$, then $\P(A)$ is
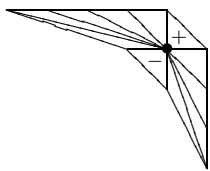
\[\begin{xy}
0;<3pt,0pt>:<0pt,3pt>::
(0,0)*{\bullet},
(0,0)="0",
(5,0)="1",
(5,-5)="2",
(5,-10)="3",
(5,-15)="4",
(0,-5)="5",
(0,5)="6",
(-5,5)="7",
(-10,5)="8",
(-15,5)="9",
(-20,5)="10",
(-5,0)="11",
(1.5,1.5)*{{\scriptstyle +}},
(-1.5,-1.5)*{{\scriptstyle -}},
\ar@{-}"1";"2",
\ar@{-}"2";"3",
\ar@{-}"3";"4",
\ar@{-}"4";"5",
\ar@{-}"1";"6",
\ar@{-}"6";"7",
\ar@{-}"7";"8",
\ar@{-}"8";"9",
\ar@{-}"9";"10",
\ar@{-}"10";"11",
\ar@{-}"5";"11",
\ar@{-}"0";"1",
\ar@{-}"0";"2",
\ar@{-}"0";"3",
\ar@{-}"0";"4",
\ar@{-}"0";"5",
\ar@{-}"0";"6",
\ar@{-}"0";"7",
\ar@{-}"0";"8",
\ar@{-}"0";"9",
\ar@{-}"0";"10",
\ar@{-}"0";"11",
\end{xy}\]
\end{example}

Now we study the Ehrhart series of $A$, which is the generating function of the number of 2-term presilting complexes. 

\begin{definition}
For each integer $\ell\ge0$, we denote by
\[\twopsilt_{\oplus}^{\le\ell}A\]
the set of isomorphism classes of (not necessarily basic) 2-term presilting complexes of $A$ which have at most $\ell$ indecomposable direct summands. 
We define the \emph{Ehrhart series} of $A$ by
\[{\rm Ehr}_A(x):=1+\sum_{\ell\ge1}\#(\twopsilt_{\oplus}^{\le\ell}A)x^\ell.\]
\end{definition}

Since there is a canonical bijection
\[\twopsilt_{\oplus}^{\le\ell}A\simeq \P(A)\cap\frac{1}{\ell}K_0(\proj A)\ \mbox{ given by }U\mapsto\frac{[U]}{\ell},\]
the Ehrhart series ${\rm Ehr}_A(x)$ of $A$ coincides with the Ehrhart series of the $g$-polytope $\P(A)$ (see \cite{BR}) though $\P(A)$ is not necessarily convex.
We give the following explicit description of the Ehrhart series by using the $h$-vector.

\begin{theorem}\label{enumerate psilt}
Let $A$ be a finite dimensional algebra over a field $k$ which is $g$-finite, $n:=|A|$ and $(h_0,\ldots,h_n)$ the $h$-vector of $\Delta(A)$. Then the Ehrhart series of $A$ is given by
\[{\rm Ehr}_A(x)=\frac{\sum_{i=0}^nh_ix^i}{(1-x)^{n+1}}.\]
In other words, for each $\ell\ge0$, we have
\[\#\twopsilt_{\oplus}^{\le\ell}A=\sum_{j=0}^{n}{n+\ell-j\choose n}h_j.\]
\end{theorem}

\begin{proof}
The $g$-fan $\Sigma(A)$ gives a unimodular triangulation (see \cite[page 185]{BR}) of $\P(A)$. Thus the desired assertion is \cite[Theorem 10.3]{BR}. (Note that, even though $\P(A)$ is not necessarily convex, the same proof works).
\end{proof}

To characterize $g$-convexity, we introduce two conditions below. The first one is combinatorial.

\begin{definition}\label{define pairwise}
Let $\Sigma$ be a nonsingular fan in $\R^d$. We call $\Sigma$ \emph{pairwise positive} if the following condition is satisfied.
\begin{enumerate}[$\bullet$]
\item For each two adjacent maximal cones $\sigma,\tau\in\Sigma_d$, take $\Z$-basis $\{v_1,\ldots,v_{d-1},v_d\}$ and $\{v_1,\ldots,v_{d-1},v'_d\}$ of $\Z^d$ such that $\sigma=\cone\{v_1,\ldots,v_{d-1},v_d\}$ and $\tau=\cone\{v_1,\ldots,v_{d-1},v'_d\}$.
Then $v_d+v'_d$ belongs to $\cone\{v_1,\ldots,v_{d-1}\}$. 
\end{enumerate}
We call $\Sigma$ \emph{pairwise convex} if the following condition is satisfied.
\begin{enumerate}[$\bullet$]
\item Define $v_d$ and $v'_d$ as above. Then $v_d+v'_d$ is either $0$, $v_i$ for some $1\le i\le d-1$ or $v_i+v_j$ for some $1\le i,j\le d-1$.
\end{enumerate}
\end{definition}

This notion in fact characterizes convexity of $\P(\Sigma)$ in the following sense.

\begin{proposition}\label{pairwise convex}
Let $\Sigma$ be a nonsingular fan in $\R^d$. Then the following conditions are equivalent.
\begin{enumerate}[\rm(a)]
\item $\Sigma$ is pairwise convex. 
\item For each two adjacent maximal cones $\sigma$ and $\tau$, the union $\sigma_{\le1}\cup\tau_{\le1}$ is convex. 
\end{enumerate}
Moreover, if $\Sigma$ is finite and complete, then the following condition is also equivalent.
\begin{enumerate}[\rm(c)]
\item $\Sigma$ is pairwise positive and $\P(\Sigma)$ is convex.
\end{enumerate}
\end{proposition}

\begin{proof}
(a)$\Rightarrow$(b) This is \cite[Lemma 2.17]{AMN}.

(b)$\Rightarrow$(a) Since both $\{v_1,\ldots,v_{d-1},v_d\}$ and $\{v_1,\ldots,v_{d-1},v'_d\}$ are $\Z$-basis of $\Z^d$, both $v_d$ and $v'_d$ are $\Z$-basis of $\Z^d/\sum_{i=1}^{d-1}\Z v_i\simeq \Z$, and hence $v_d+v_d'$ belongs to $\sum_{i=1}^{d-1}\Z v_i$. Thus $(v_d+v_d')/2$ belongs to $\sigma_{\le1}\cup\tau_{\le1}$ since it is convex. Thus
\begin{equation}\label{vd+vd'}
v_d+v_d'\in(\sigma\cup\tau)\cap(\sum_{i=1}^{d-1}\Z v_i)=\sum_{i=1}^{d-1}\Z_{\ge0} v_i
\end{equation}
holds.
Now take an $\R$-linear form $f:\R^d\to \R$ such that $f(v_i)=1$ holds for each $1\le i\le d$.
Then the hyperplane $f^{-1}(1)$ contains a facet $\conv\{v_i\mid 1\le i\le d\}$ of $\sigma_{\le 1}$.
Since $\sigma_{\le1}\cup\tau_{\le1}$ is convex, it is contained in the half space $\{x\in \R^d\mid f(x)\le 1\}$. In particular, $f(v_d+v'_d)\le 2$ holds. Using \eqref{vd+vd'}, we obtain the assertion.

(a)$\Rightarrow$(c) Clearly $\Sigma$ is pairwise positive. Moreover $\P(\Sigma)$ is convex by the same argument as in \cite[Proposition 2.23]{AMN}.

(c)$\Rightarrow$(a) Since $\Sigma$ is pairwise positive, we have $v_d+v_d'\in\sum_{i=1}^{d-1}\Z_{\ge0} v_i$. Using the latter half of the proof of (b)$\Rightarrow$(a), we obtain that $\Sigma$ is pairwise convex.
\end{proof}

Notice that $g$-fans are always pairwise positive.

\begin{proposition}\label{g-fan pairwise positive}
Let $A$ be a finite dimensional algebra over a field $k$. 
Then $\Sigma(A)$ is pairwise positive.
\end{proposition}

\begin{proof}
Let $T=T_1\oplus\cdots\oplus T_n\in\twosilt A$ and $T_n \to U_n \to T_n'\to T_n[1]$ be an exchange triangle. Then $[T_n]+[T_n']=[U_n]$ belongs to $\sum_{i=1}^{n-1}\Z_{\ge0}[T_i]$. 
\end{proof}

The pairwise convexity of $g$-fans can be stated as follows.

\begin{definition}
Let $A$ be a finite dimensional algebra over a field $k$. 
We say that $A$ is \emph{pairwise $g$-convex} if the following condition is satisfied.
\begin{enumerate}[$\bullet$]
\item For any $T\in\twosilt A$ and an indecomposable direct summand $T_i$ of $T$, let
$$T_i \to U_i \to T_i'\to T_i[1].$$ 
be an exchange triangle. Then $U_i$ has at most two indecomposable direct summands.
\end{enumerate}
\end{definition}

Inspired by \cite{Hi2}, we give characterizations of the convexity of $g$-polytope.

\begin{theorem}\label{characterize g-convex}
Let $A$ be a finite dimensional algebra over a field $k$. 
\begin{enumerate}[\rm(a)] 
\item $A$ is $g$-finite if and only if $\P(A)$ contains the origin in its interior. In this case, the origin is a unique lattice point in the interior of $P(A)$.
\item The following conditions are equivalent.
\begin{enumerate}[\rm(i)] 
\item $A$ is $g$-convex.
\item $\P(A)=\conv\{[U]\mid U\in\twopsilt^1A\}$. 
\item $\Sigma(A)$ is finite and pairwise convex.
\item $A$ is $g$-finite and pairwise $g$-convex.
\end{enumerate}
\end{enumerate}
\end{theorem}

\begin{proof}
(a) The former assertion is immediate from Proposition \ref{characterize g-finite}(d). The latter one is clear since since $\P(A)$ is a union of $C_{\le1}(T)$ for each $T\in\twosilt A$.

(b) (i)$\Leftrightarrow$(ii) and (iii)$\Leftrightarrow$(iv) are clear.

(i)$\Rightarrow$(iii) Since the convex hull of $C_{\le1}(A)\cup C_{\le1}(A[1])$ contains the origin in its interior, so does $\P(A)$. 
Thus $\Sigma(A)$ is complete and finite by (a), and pairwise positive by Proposition \ref{g-fan pairwise positive}. By Proposition \ref{pairwise convex}(c)$\Rightarrow$(a), $\Sigma(A)$ is pairwise convex.

(iii)$\Rightarrow$(i) By Proposition \ref{characterize g-finite}(d), $\Sigma(A)$ is complete. Thus the assertion follows from Proposition \ref{pairwise convex}(a)$\Rightarrow$(c).
\end{proof}

The following is a polytope analog of Definition \ref{define product fan}.

\begin{definition}\label{define free sum}\cite{E,VK}
Let $P$ and $Q$ be polytopes in $\mathbb{R}^d$ and $\mathbb{R}^{d'}$ containing $0$ in its interior respectively. 
We define a polytope $P\oplus Q$ in $\mathbb{R}^{d+d'}$ by
\[P\oplus Q:=\conv(P\cup Q).\]
\end{definition}

We give the following easy observations.

\begin{proposition}\label{idempotent decomp 2}
Let $A$ be a finite dimensional algebra over a field $k$.
\begin{enumerate}[\rm(a)]
\item For each idempotent $e\in A$, we have
\[\P(eAe)=\P(A)\cap K_0(\proj eAe)_\R.\]
In particular, if $A$ is $g$-convex, then so is $eAe$.
\item Assume $A=A_1\times\cdots\times A_\ell$ for a finite dimensional $k$-algebra $A_i$. Then $A$ is $g$-convex if and only if each $A_i$ is $g$-convex. In this case, we have
$$\P(A)=\P(A_1)\oplus \cdots\oplus \P(A_\ell),$$
where we identify $K_0(\proj A)_\R$ and $\bigoplus_{i=1}^\ell K_0(\proj A_i)_\R$ by \eqref{identify K_0}.
\end{enumerate}
\end{proposition}

\begin{proof}
(a) This is immediate from Theorem \ref{eAe}.

(b) We have a bijection
\begin{eqnarray*}
\twosilt A_1\times\cdots\times\twosilt A_\ell\simeq\twosilt A\ \mbox{ given by }\ (T_1,\cdots,T_\ell)\mapsto T_1\oplus\cdots\oplus T_\ell.
\end{eqnarray*}
Moreover the exchange sequences for $A$ are given by the exchange sequences for some $A_i$. 
Thus $A$ is pairwise $g$-convex if and only if each $A_i$ is pairwise $g$-convex.  By Theorem \ref{characterize g-convex}(b)(iii)$\Rightarrow$(i), we obtain the first assertion. Since $\twopsilt^1A=\twopsilt^1A_1\sqcup\cdots\sqcup\twopsilt^1A_\ell$ holds, Theorem \ref{characterize g-convex}(b)(ii) implies
\begin{eqnarray*}
\P(A)&=&\conv\{[U]\mid U\in\twopsilt^1A\}=\conv(\conv\{[U]\mid U_i\in\twopsilt^1A_i\}\mid 1\le i\le\ell)\\
&=&\conv(\P(A_i)\mid 1\le i\le \ell)=\bigoplus_{i=1}^\ell \P(A_i).
\end{eqnarray*} 
Thus the second assertion follows.
\end{proof}

\subsection{Dual polytopes and $c$-polytopes}

Let $P$ be a convex polytope $P$ in $V=\R^n$ containing the origin in its interior. For the dual $\R$-vector space $V^*\simeq \R^n$, we denote by $(-,-):V^*\times V\to\R$ the natural pairing.
Recall that the \emph{dual polytope} is defined by
\[P^*:=\{x\in V^*\mid\forall y\in P,\ (x,y)\le 1\}.\]
Then $P^*$ is again a convex polytope in $V^*$ containing the origin in its interior, and $P^{**}=P$ holds.
We call $P$ \emph{reflexive} if $P$ and $P^*$ are lattice polytopes.

Throughout this subsection, let $A$ be a finite dimensional algebra over a field $k$.
We introduce the notion of $c$-polytope as follows.

\begin{definition}\label{define c-polytope}
For $X\in\Db(\mod A)$, let
\[[X]':=(\dim_k\End_{\Db(\mod A)}(X))^{-1}[X]\in K_0(\mod A)_\R.\]
For the simple $A$-modules $S_1,\ldots,S_n$, we define the \emph{$k$-Grothendieck group} $K_0(\mod A,k)$ of $\mod A$ as the subgroup of $K_0(\mod A)_\R$ generated by $[S_1]',\ldots,[S_n]'$.

For $S=S_1\oplus\cdots\oplus S_n\in\smc A$, let
\[v_S:=\sum_{i=1}^n[S_i]'\in K_0(\mod A)_\R.\]
We define the \emph{$c$-polytope} $\P^\c(A)$ of $A$ as the convex hull
\[\P^\c(A):=\conv\{v_S\mid S\in\twosmc A\}\subset K_0(\mod A)_\R.\]
\end{definition}

Using the Euler form
\[K_0(\proj A)_\R\times K_0(\mod A)_\R\to\R=K_0(\mod k)_\R\ \mbox{ given by }\ (X,Y)\mapsto[\RHom_A(X,Y)],\]
we identify $(K_0(\proj A)_\R)^*$ with $K_0(\mod A)_\R$. 
Using this identification, we can state the following main result in this subsection.

\begin{theorem}\label{reflexive polytope}
Let $A$ be a finite dimensional algebra over a field $k$. Then $A$ is $g$-convex if and only if
\[\P(A)=(\P^\c(A))^*.\]
In this case, both $\P(A)$ and $\P^\c(A)$ are reflexive polytopes.
\end{theorem}

We need the following information on silting-t-structure correspondence (Proposition \ref{KY thm}).

\begin{proposition}
Let $A$ be a finite dimensional algebra over a field $k$.
\begin{enumerate}[\rm(a)]
\item The abelian groups $K_0(\mod A,k)$ and $K_0(\proj A)$ are dual to each other with respect to the Euler form.
\item For any $S=S_1\oplus\cdots\oplus S_n\in\smc A$, the elements $[S_1]',\ldots,[S_n]'$ are the basis of $K_0(\mod A,k)$.
\end{enumerate}
\end{proposition}

\begin{proof}
(a) This is immediate from Proposition \ref{KY thm}.

(b) Take $T=T_1\oplus\cdots\oplus T_n\in\silt A$ corresponding to $S$. Then $[T_1],\ldots,[T_n]$ give a basis of $K_0(\proj A)$ by  \cite{AI}. Thus the claim follows from (a).
\end{proof}

The following simple observation is crucial.

\begin{proposition}\label{normal vector} 
Keep the setting in Proposition \ref{KY thm}. Then the element $v_S\in K_0(\mod A)_\R$ gives a normal vector of the facet of $\P(A)$ corresponding to $T$. 
\end{proposition}

\begin{proof}
For each $1\le i,j\le n$, \eqref{TS dual} implies
\[([T_i]-[T_j],v_S)=([T_i],v_S)-([T_j],v_S)=1-1=0.\]
This shows the assertion.
\end{proof}

Now we are ready to prove Theorem \ref{reflexive polytope}.

\begin{proof}[Proof of Theorem \ref{reflexive polytope}]
It suffices to show the ``only if'' part. Assume that $A$ is $g$-convex. For $x\in K_0(\mod A)_\R$, we consider the half space
\[H_x^{\le 1}:=\{y\in K_0(\proj A)_\R\mid(y,x)\le 1\}\subset K_0(\proj A)_\R.\]
Since $\P(A)$ is a convex polytope by assumption, Proposition \ref{normal vector} shows
\[\P(A)=\bigcap_{S\in\twosmc A}H_{v_s}^{\le 1}=(\P^\c(A))^*,\]
where the first equality follows from \cite[Theorem 2.15(7)]{Zi}.
\end{proof}

We end this subsection by the following remark. 

\begin{remark}
Using $\twosmc A$, we can define the \emph{$c$-simplicial complex} of $A$ by the completely same way as in the definition of $g$-simplicial complex. However, it is not clear if the $c$-simplicial complex enjoys some nice properties. For example, the $c$-polytope $\P^\c(A)$ is not a geometric realization of the $c$-simplicial complex.
\end{remark}

\subsection{Newton polytopes}

Throughout this subsection, let $A$ be a finite dimensional algebra over a field $k$. We study a connection between $g$-fans $\Sigma(A)$ and Newton polytopes (also known as \emph{Harder-Narashimhan polytopes}) of $A$-modules. In particular, our results recover some of results in \cite{Fe1} (see also \cite{BCDMTY,PPPP}).

\begin{definition}\cite{BKT,Fe1}\label{define Newton}
The \emph{Newton polytope} of $X\in\AA$ is the convex hull
\begin{eqnarray*}
\n(X)&:=&\{[Y]\in K_0(\mod A)\mid\mbox{$Y$ is a submodule of $X$}\}\subset K_0(\mod A),\\
\N(X)&:=&\conv \n(X)\subset K_0(\mod A)_\R.
\end{eqnarray*}
\end{definition}

The dimension of $\N(X)$ clearly equals the number of isomorphism classes of simple $A$-modules appearing in $X$ as composition factors. Let $0\neq\theta\in K_0(\proj A)_\R$. Each $r\in\R$ give a hyperplane
\[H_\theta^r:=\{x\in K_0(\mod A)_\R\mid\theta(x)=r\}\subset K_0(\proj A)_\R\]
and a half-space
\[H_\theta^{\le r}:=\{x\in K_0(\mod A)_\R\mid\theta(x)\le r\}\subset K_0(\proj A)_\R.\]
Let $X\in\mod A$. Clearly we have
\begin{eqnarray*}
\N(X)&=&\bigcap_{0\neq\theta\in K_0(\proj A)}H_\theta^{\le\max\theta(\N(X))}.
\end{eqnarray*}
Each $0\neq\theta\in K_0(\proj A)_\R$ gives a face of $\N(X)$:
\[\N(X)_\theta:=\N(X)\cap H_\theta^{\max\theta(\N(X
))}.\]
Conversely each face of $\N(X)$ has a form $\N(X)_\theta$ for some $0\neq\theta\in K_0(\proj A)$.

The following properties are basic to study $\N(X)$, see Definition-Proposition \ref{basic fact for theta} for $\fT_\theta$ and $\overline{\fT}_\theta$.

\begin{lemma}\label{inequalities of theta}
For each $\theta\in K_0(\proj A)_\R$ and $X\in\mod A$, the following assertions hold.
\begin{enumerate}[\rm(a)]
\item $\theta(X)\le\theta(\fT_\theta X)=\theta(\overline{\fT}_\theta X)=\max\theta(\N(X))$.
\item For each submodule $Y$ of $X$, we have $\theta(Y)\le\theta(\fT_\theta X)$.

\item Let $Y$ be a submodule of $X$. Then $\theta(Y)=\theta(\fT_\theta X)$ holds if and only if $\fT_\theta X\subseteq Y\subseteq\overline{\fT}_\theta X$ and $Y/\fT_\theta X\in\WW_\theta$ hold.
\item If a submodule $Y$ of $X$ satisfies $[Y]=[\fT_\theta X]$ (respectively, $[Y]=[\overline{\fT}_\theta X]$), then $Y=\fT_\theta X$ (respectively, $Y=\overline{\fT}_\theta X$).
\end{enumerate}
\end{lemma}

\begin{proof}
(a) The exact sequence $0\to\fT_\theta X\to X\to\overline{\fF}_\theta X\to0$ shows $\theta(X)=\theta(\fT_\theta X)+\theta(\overline{\fF}_\theta X)$. Since $\theta(\overline{\fF}_\theta X)\le0$, we obtain the left inequality. Similarly, the exact sequence $0\to\fT_\theta X\to\overline{\fT}_\theta X\to\fW_\theta X\to0$ shows $\theta(\overline{\fT}_\theta X)=\theta({\fT}_\theta X)+\theta(\fW_\theta X)$, and hence the equality $\theta(\fW_\theta X)=0$ implies the middle equality. To prove the right equality, it suffices to show (b). 

(b) By (a), $\theta(Y)\le\theta(\fT_\theta Y)$ holds. The exact sequence $0\to\fT_\theta Y\to\fT_\theta X\to\fT_\theta X/\fT_\theta Y\to0$ implies $\fT_\theta X/\fT_\theta Y\in\TT_\theta$ and hence $\theta(\fT_\theta X/\fT_\theta Y)\ge0$.
Then the inequality $\theta(\fT_\theta X/\fT_\theta Y)\ge0$ implies
$\theta(\fT_\theta Y)\le\theta(\fT_\theta X)$. Thus the assertion follows.

(c) The ``if'' part is clear from the exact sequence $0\to\fT_\theta X\to Y\to Y/\fT_\theta X\to0$ and $\theta(Y/\fT_\theta X)=0$. 
We prove the ``only if'' part. 
Since $\theta(Y/\fT_\theta Y)=\theta(\overline{\fF}_\theta Y)\le0$ and $\theta(\fT_\theta X/\fT_\theta Y)\ge0$, we have $\theta(Y)\le\theta(\fT_\theta Y)\le\theta(\fT_\theta X)=\theta(Y)$. Thus all the 4 equalities hold, and $\theta(\fT_\theta X/\fT_\theta Y)=0$ implies 
$\fT_\theta X/\fT_\theta Y=0$ and $\fT_\theta X=\fT_\theta Y\subseteq Y$. 
Similarly, since $\theta(Y/\overline{\fT}_\theta Y)=\theta(\fF_\theta Y)\le0$ and $\theta(\overline{\fT}_\theta X/\overline{\fT}_\theta Y)\ge0$ hold, we have $\theta(Y)\le\theta(\overline{\fT}_\theta Y)\le\theta(\overline{\fT}_\theta X)=\theta(\fT_\theta X)=\theta(Y)$. Thus all the 4 equalities hold, 
and $\theta(\fF_\theta Y)=0$ implies $\fF_\theta Y=0$ and  $Y=\overline{\fT}_\theta Y\subseteq\overline{\fT}_\theta X$. 

Finally, the claim $Y/\fT_\theta X\in\WW_\theta$ follows from $Y/\fT_\theta X\subseteq\fW_\theta X$, $\theta(Y/\fT_\theta X)=0$ and a general fact: A submodule $Z'$ of $Z\in\WW_\theta$ belongs to $\WW_\theta$ if and only if $\theta(Z')=0$.

(d) Since the assumption implies $\theta(Y)=\theta(\fT_\theta X)$, we have $\fT_\theta X\subseteq Y\subseteq\overline{\fT}_\theta X$ by (c). Thus the assertion clearly holds.
\end{proof}

Let $\theta\in K_0(\proj A)_\R$. Then the inclusion functor $\WW_\theta\to\mod A$ induces morphisms
\[\iota:K_0(\WW_\theta)\to K_0(\mod A)\ \mbox{ and }\ \iota:K_0(\WW_\theta)_\R\to K_0(\mod A)_\R.\]
As in Definition \ref{define Newton}, for each $X\in\WW_\theta$, let
\begin{eqnarray*}
\n_{\WW_\theta}(X)&:=&\{[Y]\in K_0(\WW_\theta)\mid\mbox{$Y$ is a subobject of $X$}\}\subset K_0(\WW_\theta),\\
\N_{\WW_\theta}(X)&:=&\conv\n_{\WW_\theta}(X)\subset K_0(\WW_\theta)_\R.
\end{eqnarray*}
We have the following descriptions of the faces of $\N(X)$.

\begin{lemma}\label{faces of N(X)}
Let $X\in\mod A$ and $0\neq\theta\in K_0(\proj A)_\R$. 
Then we have
\begin{eqnarray*}
\N(X)_\theta&=&[\fT_\theta X]+\iota(\N_{\WW_\theta}(\fW_\theta X)),\\
\{\mbox{vertices of $\N(X)$}\}&=&\{[\fT_\theta X]\mid0\neq\theta\in K_0(\proj A)\}.
\end{eqnarray*}
Moreover, each edge of $\N(X)$ has a form $\conv\{[\fT_\theta X],[\overline{\fT}_\theta X]\}$ for some $0\neq\theta\in K_0(\proj A)$.
\end{lemma}

\begin{proof}
By Lemma \ref{inequalities of theta}(a)(c), we have  
\[\n(X)\cap H_\theta^{\max\theta(\N(X))}=[\fT_\theta X]+\iota(\n_{\WW_\theta}(\fW_\theta X)).\]
Taking the convex hull, we obtain the first equality.
In particular, $[\fT_\theta X]$ is a vertex of $\N(X)_\theta$ and hence of $\N(X)$. Conversely, for each vertex $v$ of $\N(X)$, there exists $0\neq\theta\in K_0(\proj A)$ such that $\N(X)_\theta=\{v\}$. Then $v=[\fT_\theta X]$ holds by the first equality.

Since each face of $\N(X)$ has a form $\N(X)_\theta$ for some $0\neq\theta\in K_0(\proj A)$ and contains $[\fT_\theta X]$ and $[\overline{\fT}_\theta X]$, the last assertion follows.
\end{proof}

To give a more explicit description of faces of $\N(X)$, we use 2-term presilting complexes. For $0\neq U\in\twopsilt A$, define a face of $\N(X)$ by
\[\N(X)_U:=\N(X)_{[U]}.\] 
We have torsion classes $\TT_U,\overline{\TT}_U$, a wide subcategory $\WW_U$ and functors $\fT_U$, $\overline{\fT}_U$ and $\fW_U$ (see Definition-Propositions \ref{surjection to f-tors} and \ref{define co-Bongartz}). For each $Y\in\WW_U$, we denote by ${\rm s}_{\WW_U}(Y)$ the number of isomorphism classes of simple objects in $\WW_U$ appearing in $Y$ as a composition factor.

\begin{lemma}\label{faces of N(X) 2}
Let $X\in\mod A$ and $0\neq U\in\twopsilt A$.
\begin{enumerate}[\rm(a)]
\item We have
\[\dim\N(X)_U={\rm s}_{\WW_U}(\fW_U X)\le|A|-|U|.\]
\item For each $\theta\in C^+(U)$, we have
\[\N(X)_\theta=\N(X)_U.\]
\end{enumerate}
\end{lemma}

\begin{proof}
(a) By Definition-Proposition \ref{define co-Bongartz}(d), there exists a finite dimensional algebra $B$ such that $\WW_U\simeq \mod B$ and $|B|=|A|-|U|$ and $\iota:K_0(\WW_U)\simeq K_0(\mod B)\to K_0(\mod A)$ is injective. Thus we have
\[\dim\N(X)_\theta\stackrel{\ref{faces of N(X)}}{=}\dim\N_{\WW_U}(\fW_U X)={\rm s}_{\WW_U}(\fW_U X)\le |B|=|A|-|U|.\]

(b) By Proposition \ref{presilting semistable}, we have $\fT_{\theta}=\fT_U=\fT_{[U]}$, $\WW_\theta=\WW_U=\WW_{[U]}$ and $\fW_\theta =\fW_U=\fW_{[U]}$. By Lemma \ref{faces of N(X)}, we obtain
\begin{equation*}
\N(X)_\theta=[\fT_\theta X]+\iota(\N_{\WW_\theta}(\fW_\theta X))=[\fT_{[U]} X]+\iota(\N_{\WW_{[U]}}(\fW_{[U]} X))=\N(X)_{[U]}=\N(X)_U.\qedhere
\end{equation*}
\end{proof}

To state our main result, we introduce some notions. Recall that $\twopsilt^iA$ is the set of isomorphism classes of basic 2-term presilting complex $U\in\Kb(\proj A)$ with $|U|=i$ (Definition \ref{psilt^j}).

\begin{definition}\label{desine sim_X}
Let $A$ be a finite dimensional algebra over a field $k$, $n:=|A|$, and $X\in\mod A$. 
\begin{enumerate}[\rm(a)]
\item An object $X$ in an abelian length category $\WW$ is called \emph{sincere} if each simple object in $\WW$ appears in $X$ as a composition factor.
\item For $0\le i\le n$, let
\begin{eqnarray*}
\twopsilt^i_XA&:=&\{U\in\twopsilt^iA\mid \mbox{$\fW_UX$ is sincere in $\WW_U$}\},\\
\Sigma_{i,X}(A)&:=&\{C(U)\mid U\in\twopsilt^i_XA\}\subseteq\Sigma_i(A).
\end{eqnarray*}
By Lemma \ref{faces of N(X) 2}(a), $U\in\twopsilt^iA$ belongs to $\twopsilt^i_XA$ if and only if $\dim\N(X)_U=n-i$.
\item By Definition-Proposition \ref{surjection to f-tors}(b), we have an order preserving map 
\begin{equation}\label{TTX}
\twosilt A\to\{\mbox{submodules of $X$}\}\ \mbox{ given by }\ U\mapsto\fT_U X.
\end{equation} 
We define an equivalence relation $\sim_X$ on $\twosilt A$ by
\[T\sim_XU\Longleftrightarrow\fT_T X=\fT_U X\stackrel{{\rm\ref{inequalities of theta}(d)}}{\Longleftrightarrow}[\fT_T X]=[\fT_U X].\]
This gives an equivalence relation $\sim_X$ on $\Sigma_n(A)$ since we have a bijection $\twosilt A\simeq\Sigma_n(A)$ given by $T\mapsto C(T)$.
\item We regard the vertices and the edges of $\N(X)$ as a graph. 
Define a quiver $\overrightarrow{\N}_1(X)$ by regarding each edge $\conv\{[\fT_\theta X],[\overline{\fT}_\theta X]\}$ (see Lemma \ref{faces of N(X)}) as an arrow $[\overline{\fT}_\theta X]\to[\fT_\theta X]$.
\item Define a \emph{contraction} $\Hasse(\twosilt A)/\sim_X$ of $\Hasse(\twosilt A)$ by identifying all vertices in each equivalence class of $\sim_X$ and removing all loops.
\end{enumerate}
\end{definition}

Now we give explicit descriptions of the normal fan $\Sigma(\N(X))$ (see Definition \ref{define normal fan}) as a coarsening fan of the $g$-fan $\Sigma(A)$ (see Definition \ref{define coarsening fan}), and of the 1-skeleton $\overrightarrow{\N}_1(X)$ as a contraction of the Hasse quiver $\Hasse(\twosilt A)$. 
The following result is an explicit version of \cite[Propositions 7.4, 8.5]{Fe2} for $g$-finite case.

\begin{theorem}\label{normal fan is coarsening}
Let $A$ be a finite dimensional algebra over a field $k$ which is $g$-finite, and $n:=|A|$.
\begin{enumerate}[\rm(a)]
\item The equivalence relation $\sim_X$ coarsens $\Sigma(A)$, and we have
\begin{equation*}
\Sigma(\N(X))=\Sigma(A)/\sim_X.
\end{equation*}
\item For each $0\le i\le n$, there is a surjection
\begin{equation}\label{n-i face}
\twopsilt^i_XA\simeq\Sigma_{i,X}(A)\to\{\mbox{faces of $\N(X)$ of dimension $n-i$}\}
\end{equation}
given by $U\mapsto C(U)\mapsto\N(X)_U$, which induces bijections
\begin{eqnarray*}
\twosilt A/\sim_X&\simeq&\{\mbox{vertices of $\N(X)$}\},\\
\twopsilt^1_XA&\simeq&\{\mbox{facets of $\N(X)$}\}.
\end{eqnarray*}
\item We have an isomorphism of quivers
\[\overrightarrow{\N}_1(X)\simeq\Hasse(\twosilt A)/\sim_X.\]
\end{enumerate}
\end{theorem}

To prove this, we need the following preparation.

\begin{lemma}\label{g-fan and normal fan}
Let $n:=|A|$ and $X\in\mod A$.
\begin{enumerate}[\rm(a)]
\item Each cone of $\Sigma(A)$ is contained in a cone of $\Sigma(\N(X))$.
\item Let $\sigma,\tau\in\Sigma_n(A)$. Then $\sigma\sim_X\tau$ if and only if $\sigma$ and $\tau$ are contained in the same cone of $\Sigma_n(\N(X))$.
\end{enumerate}
\end{lemma}

\begin{proof}
(a) Immediate from Lemma \ref{faces of N(X) 2}(b).

(b) Take $T,U\in\twosilt A$ such that $\sigma=C(T)$ and $\tau=C(U)$. Then $\N(X)_T=\{[\fT_T X]\}$ and $\N(X)_U=\{[\fT_U X]\}$ hold by Lemmas \ref{faces of N(X)} and \ref{faces of N(X) 2}(a). Thus $\sigma$ and $\tau$ are contained in the same cone of $\Sigma_n(\N(X))$ if and only if $[\fT_T X]=[\fT_U X]$, that is, $\sigma\sim_X\tau$.
\end{proof}

We are ready to prove Theorem \ref{normal fan is coarsening}.

\begin{proof}[Proof of Theorem \ref{normal fan is coarsening}]
(a) The assertion follows from Lemma \ref{g-fan and normal fan}(a)(b).

(b) The map \eqref{n-i face} is well-defined by the definition of $\twopsilt^i_XA$. We prove that it is surjective. Let $F$ be a face of $\N(X)$ of dimension $n-i$. Since $\Sigma(A)$ is complete, Lemma \ref{g-fan and normal fan}(a) implies that the cone $\sigma_F$ of $\Sigma(\N(X))$ of dimension $i$ is a union of cones in $\Sigma(A)$. Thus there exists $U\in\twopsilt^iA$ such that $C(U)\subseteq\sigma_F$. Since $[U]\in \sigma_F^\circ$, we have $\N(X)_U=F$. Thus $U\in\twopsilt^i_XA$ holds, and the map \eqref{n-i face} is surjective.

Since $\twopsilt^n_XA=\twosilt A$, the map \eqref{n-i face} gives the first bijection by Lemma \ref{g-fan and normal fan}(b).
It also gives the second bijection since different elements in $\Sigma_{1,X}(A)$ give different facets of $\N(X)$. 

(c) By (b), both quivers have the set $\twosilt A/\sim_X$ of vertices. We compare their arrows. By Lemma \ref{faces of N(X)}, we have a surjection
\begin{eqnarray*}
\twopsilt^{n-1}_XA\to\{\mbox{arrows in $\overrightarrow{\N}_1(X)$}\}\ \mbox{ given by }\ U\mapsto[\overline{\fT}_UX\to \fT_UX].
\end{eqnarray*}
On the other hand, by Proposition \ref{mutation=exchange}, each arrow in $\Hasse(\twosilt A)$ can be written as $U_{\max}\to U_{\min}$ for some $U\in\twopsilt^{n-1}A$. It is still an arrow in $\Hasse(\twosilt A)/\sim_X$ if and only if $U_{\max}\ {\not\sim_X}\ U_{\min}$ if and only if $\overline{\fT}_UX\neq\fT_UX$ (since $\fT_U=\fT_{U_{\min}}$ and $\overline{\fT}_U=\fT_{U_{\max}}$ hold by Definition-Proposition \ref{define co-Bongartz}) if and only if $U\in\twopsilt^{n-1}_XA$. Thus we have a surjection
\begin{eqnarray*}
\twopsilt^{n-1}_XA\to\{\mbox{arrows in $\Hasse(\twosilt A)/\sim_X$}\}\ \mbox{ given by }\ U\mapsto[U_{\max}\to U_{\min}].
\end{eqnarray*}
Moreover, $U,V\in\twopsilt^{n-1}_XA$ give the same arrow in $\overrightarrow{\N}_1(X)$ if and only if $\overline{\fT}_UX=\overline{\fT}_VX$ and $\fT_UX=\fT_VX$ hold if and only if $U_{\max}\sim_XV_{\max}$ and $U_{\min}\sim_XV_{\min}$ hold if and only if $U$ and $V$ give the same arrow in $\Hasse(\twosilt A)/\sim_X$. Thus the assertion follows.
\end{proof}

As an application of Theorem \ref{normal fan is coarsening}, we obtain the following result.
The part (b) was obtained in \cite[Theorem 4.17, Corollary 6.9]{Fe2}.

\begin{corollary}\label{Delzant exist}
Let $A$ be a finite dimensional algebra over a field $k$ which is $g$-finite.
\begin{enumerate}[\rm(a)]
\item For $X\in\mod A$, the following conditions are equivalent.
\begin{enumerate}[\rm(i)]
\item $\Sigma(\N(X))=\Sigma(A)$ holds.
\item $\overrightarrow{\N}_1(X)\simeq\Hasse(\twosilt A)$ holds.
\item For each $U\in\twopsilt^{n-1}A$, $\fW_UX\neq0$ holds.
\end{enumerate}
\item If one of the following conditions is satisfied, then the conditions in (a) are satisfied.
\begin{enumerate}[\rm(i)]
\item Each brick of $A$ is a direct summand of $X$.
\item For each $V\in\twopsilt^1A$, $H^0(V)$ is a direct summand of $X$.
\end{enumerate}
\end{enumerate}
\end{corollary}

\begin{proof}
(a) By Theorem \ref{normal fan is coarsening}(a)(c), the conditions (i) and (ii) are equivalent to that the map \eqref{TTX} is injective. This is equivalent to that $\fT_TX\neq\fT_{T'}X$ holds for each arrow $T\to T'$ in $\Hasse(\twosilt A)$. Recall that there is a bijection between $\twopsilt^{n-1}A$ and the set of arrows of $\Hasse(\twosilt A)$ given by $U\mapsto[U_{\max}\to U_{\min}]$, and we have
\[\fW_UX=\fT_{U_{\max}}X/\fT_{U_{\min}}X,\]
see Definition-Proposition \ref{define co-Bongartz}(d). Thus the map \eqref{TTX} is injective if and only if (iii) holds.

(b) In each case, it suffices to check the condition (a)(iii).

Consider the case (i). For each $U\in\twopsilt^{n-1}A$, the wide subcategory $\WW_U=\overline{\TT}_U\cap\overline{\FF}_U$ has a unique simple object $S$. Then $\fW_US=\overline{\fT}_US/\fT_US=S\neq0$ holds. Since $S$ is a brick and hence a direct summand of $X$, we have $\fW_UX\neq0$.

Consider the case (ii). For each $U\in\twopsilt^{n-1}A$, let $U_
{\max}=U\oplus V$ for some $V\in\twopsilt^1A$. Then $H^0(V)$ belongs to $\TT_{U_{\max}}=\overline{\TT}_U$ and does not belong to $\TT_{U_{\min}}=\TT_U$. Thus $\overline{\fT}_UH^0(V)=H^0(V)\neq\fT_UH^0(V)$ and hence $\fW_UH^0(V)\neq0$. Since $H^0(V)$ is a direct summand of $X$, we have $\fW_UX\neq0$.
\end{proof}

\begin{example}\label{newton example}
Consider a finite dimensional algebra $A$ given by
\[Q=\left[\xymatrix@C2em{1\ar@<2pt>[r]^a&2\ar@<2pt>[r]^b\ar@<2pt>[l]^{c}&3\ar@<2pt>[l]^{d}}\right]\ \mbox{ and }\ A:=kQ/\langle ab,dc,ca-bd\rangle.\]
Then $A=\begin{smallmatrix}1\\ 2\\ 1\end{smallmatrix}\oplus\begin{smallmatrix}&2\\ 1&&3\\ &2\end{smallmatrix}\oplus\begin{smallmatrix}3\\ 2\\ 3\end{smallmatrix}$.
For $X=P_1^{\oplus p}\oplus P_2^{\oplus q}\oplus P_3^{\oplus r}\oplus S_2^{\oplus s}$ with $p,q,r,s\ge0$, the Newton polytope $\N(X)$ is given by the following.
\[
\begin{tikzpicture}[baseline=0mm, scale=1]
\node(0) at(0:0) {};
\coordinate(x) at(-45:0.7);
\coordinate(y) at(45:0.65);
\coordinate(z) at(90:0.5);

\coordinate(1) at($0*(x) + 0*(y) + 0*(z)$); 
\coordinate(2) at($2.4*(x) + 0*(y) + 0*(z)$); 
\coordinate(3) at($0*(x) + 5.3*(y) + 0*(z)$); 
\coordinate(4) at($0*(x) + 0*(y) + 2.9*(z)$); 
\coordinate(5) at($(2)+($5.3*(x) + 5.3*(y) + 0*(z)$)$); 
\coordinate(6) at($(5)+($0*(x) + 0*(y) + 5.8*(z)$)$); 
\coordinate(7) at($(2)+($0*(x) + 0*(y) + 2.9*(z)$)$); 
\coordinate(8) at($(7)+($2.9*(x) + 2.9*(y) + 2.9*(z)$)$); 
\coordinate(9) at($(5)+($0*(x) + 2.4*(y) + 0*(z)$)$); 
\coordinate(10) at($(3)+($2.4*(x) + 2.4*(y) + 0*(z)$)$); 
\coordinate(11) at($(4)+($0*(x) + 5.8*(y) + 5.8*(z)$)$); 
\coordinate(12) at($(8)+($0*(x) + 2.9*(y) + 2.9*(z)$)$); 
\coordinate(13) at($(12)+($2.4*(x) + 2.4*(y) + 0*(z)$)$); 
\coordinate(14) at($(3)+($0*(x) + 2.9*(y) + 2.9*(z)$)$); 
\coordinate(15) at($(10)+($0*(x) + 2.9*(y) + 2.9*(z)$)$); 
\coordinate(16) at($(9)+($0*(x) + 5.8*(y) + 5.8*(z)$)$); 
\coordinate(17) at($(15)+($2.9*(x) + 2.9*(y) + 2.9*(z)$)$); 
\coordinate(18) at($(11)+($0*(x) + 2.4*(y) + 0*(z)$)$); 
\coordinate(19) at($(18)+($5.3*(x) + 5.3*(y) + 0*(z)$)$); 
\coordinate(20) at($(19)+($2.4*(x) + 0*(y) + 0*(z)$)$);

\draw[very thick] (1)--(2); 
\draw[dotted, very thick] (1)--(3); 
\draw[very thick] (1)--(4); 
\draw[very thick] (2)--(7); 
\draw[dotted, very thick] (9)--(10); 
\draw[dotted, very thick] (14)--(15); 
\draw[very thick] (4)--(7); 
\draw[very thick] (2)--(5); 
\draw[very thick] (5)--(6); 
\draw[very thick] (5)--(9); 
\draw[very thick] (7)--(8); 
\draw[very thick] (4)--(11); 
\draw[very thick] (6)--(8); 
\draw[very thick] (6)--(13); 
\draw[very thick] (8)--(12); 
\draw[dotted, very thick] (3)--(10); 
\draw[dotted, very thick] (3)--(14); 
\draw[very thick] (9)--(16); 
\draw[dotted, very thick] (10)--(15); 
\draw[very thick] (11)--(12); 
\draw[very thick] (11)--(18); 
\draw[very thick] (12)--(13); 
\draw[dotted, very thick] (15)--(17); 
\draw[dotted, very thick] (16)--(17); 
\draw[dotted, very thick] (14)--(18); 
\draw[dotted, very thick] (17)--(19); 
\draw[very thick] (18)--(19); 
\draw[very thick] (16)--(20); 
\draw[very thick] (19)--(20); 
\draw[very thick] (13)--(20); 

\coordinate(a1) at ($(1)!0.45!(2)$); 
\draw[fill=white] (a1)circle (0.25cm); 
\node at ($(a1)+($0*(x)+0*(y)+0*(z)$)$) {$\begin{smallmatrix}1^{p}\end{smallmatrix}$}; 
\coordinate(a2) at ($(1)!0.6!(3)$); 
\draw[fill=white] (a2)circle (0.35cm); 
\node at ($(a2)+($0*(x)+0*(y)+0*(z)$)$) {$\begin{smallmatrix}2^{q+s}\end{smallmatrix}$}; 
\coordinate(a3) at ($(1)!0.5!(4)$); 
\draw[fill=white] (a3)circle (0.25cm); 
\node at ($(a3)+($0*(x)+0*(y)+0*(z)$)$) {$\begin{smallmatrix}3^{r}\end{smallmatrix}$}; 
\coordinate(a4) at ($(2)!0.5!(7)$); 
\draw[fill=white] (a4)circle (0.25cm); 
\node at ($(a4)+($0*(x)+0*(y)+0*(z)$)$) {$\begin{smallmatrix}3^{r}\end{smallmatrix}$}; 
\coordinate(a5) at ($(9)!0.55!(10)$); 
\draw[fill=white] ($(a5)+($-0.1*(x)+-0.1*(y)+-0.1*(z)$)$)circle (0.35cm); 
\node at ($(a5)+($-0.1*(x)+-0.1*(y)+-0.1*(z)$)$) {$\begin{smallmatrix}1^{p+q}\end{smallmatrix}$}; 
\coordinate(a6) at ($(14)!0.6!(15)$); 
\draw[fill=white] (a6)circle (0.3cm); 
\node at ($(a6)+($0*(x)+0*(y)+0*(z)$)$) {$\begin{smallmatrix}\hspace{1mm} 2^{p}\\ \hspace{-0.5mm}1 \end{smallmatrix}$}; 
\coordinate(a7) at ($(4)!0.45!(7)$); 
\draw[fill=white] (a7)circle (0.25cm); 
\node at ($(a7)+($0*(x)+0*(y)+0*(z)$)$) {$\begin{smallmatrix}1^{p}\end{smallmatrix}$}; 
\coordinate(a8) at ($(2)!0.5!(5)$); 
\draw[fill=white] (a8)circle (0.45cm); 
\node at ($(a8)+($0*(x)+0*(y)+0*(z)$)$) {$\begin{smallmatrix}\hspace{1mm} \hspace{0.5mm}1^{p+q}\\\hspace{-3mm}2\end{smallmatrix}$}; 
\coordinate(a9) at ($(5)!0.5!(6)$); 
\draw[fill=white] (a9)circle (0.35cm); 
\node at ($(a9)+($0*(x)+0*(y)+0*(z)$)$) {$\begin{smallmatrix}3^{q+r}\end{smallmatrix}$}; 
\coordinate(a10) at ($(5)!0.5!(9)$); 
\draw[fill=white] (a10)circle (0.25cm); 
\node at ($(a10)+($0*(x)+0*(y)+0*(z)$)$) {$\begin{smallmatrix}2^{s}\end{smallmatrix}$}; 
\coordinate(a11) at ($(7)!0.5!(8)$); 
\draw[fill=white] (a11)circle (0.38cm); 
\node at ($(a11)+($0*(x)+0*(y)+0*(z)$)$) {$\begin{smallmatrix}\hspace{0.7mm} 1 \, 3^{q} \\ 
\hspace{-0.3mm} 2 \end{smallmatrix}$}; 
\coordinate(a12) at ($(4)!0.5!(11)$); 
\draw[fill=white] (a12)circle (0.45cm); 
\node at ($(a12)+($0*(x)+0*(y)+0*(z)$)$) {$\begin{smallmatrix}\hspace{1mm} 3^{q+r} \\ \hspace{-3mm}2\end{smallmatrix}$}; 
\coordinate(a13) at ($(6)!0.6!(8)$); 
\draw[fill=white] (a13)circle (0.3cm); 
\node at ($(a13)+($0*(x)+0*(y)+0*(z)$)$) {$\begin{smallmatrix}\hspace{1mm}1^{p} \\ \hspace{-1mm}2\end{smallmatrix}$}; 
\coordinate(a14) at ($(6)!0.5!(13)$); 
\draw[fill=white] (a14)circle (0.3cm); 
\node at ($(a14)+($0*(x)+0*(y)+0*(z)$)$) {$\begin{smallmatrix}\hspace{0.5mm}3^{r} \\ \hspace{-1mm}2\end{smallmatrix}$}; 
\coordinate(a15) at ($(8)!0.5!(12)$); 
\draw[fill=white] (a15)circle (0.3cm); 
\node at ($(a15)+($0*(x)+0*(y)+0*(z)$)$) {$\begin{smallmatrix}\hspace{0.5mm} 3^{r}\\\hspace{-1mm}2 \end{smallmatrix}$}; 
\coordinate(a16) at ($(3)!0.45!(10)$); 
\draw[fill=white] (a16)circle (0.3cm); 
\node at ($(a16)+($0*(x)+0*(y)+0*(z)$)$) {$\begin{smallmatrix}\hspace{1mm}2^{p} \\ \hspace{-0.5mm}1\end{smallmatrix}$}; 
\coordinate(a17) at ($(3)!0.5!(14)$); 
\draw[fill=white] (a17)circle (0.3cm); 
\node at ($(a17)+($0*(x)+0*(y)+0*(z)$)$) {$\begin{smallmatrix}\hspace{1mm}2^{r}\\\hspace{-0.5mm}3\end{smallmatrix}$}; 
\coordinate(a18) at ($(9)!0.5!(16)$); 
\draw[fill=white] (a18)circle (0.45cm); 
\node at ($(a18)+($0*(x)+0*(y)+0*(z)$)$) {$\begin{smallmatrix}\hspace{1mm}2^{q+r}\\\hspace{-3.5mm}3\end{smallmatrix}$}; 
\coordinate(a19) at ($(10)!0.5!(15)$); 
\draw[fill=white] (a19)circle (0.3cm); 
\node at ($(a19)+($0*(x)+0*(y)+0*(z)$)$) {$\begin{smallmatrix}\hspace{1mm}2^{r}\\\hspace{-0.5mm}3\end{smallmatrix}$}; 
\coordinate(a20) at ($(11)!0.55!(12)$); 
\draw[fill=white] ($(a20)+($0.1*(x)+0.1*(y)+0.1*(z)$)$) circle (0.35cm); 
\node at ($(a20)+($0.1*(x)+0.1*(y)+0.1*(z)$)$) {$\begin{smallmatrix}\hspace{0.5mm}1^{p+q}\end{smallmatrix}$}; 
\coordinate(a21) at ($(11)!0.5!(18)$); 
\draw[fill=white] (a21)circle (0.25cm); 
\node at ($(a21)+($0*(x)+0*(y)+0*(z)$)$) {$\begin{smallmatrix}2^{s}\end{smallmatrix}$}; 
\coordinate(a22) at ($(12)!0.5!(13)$); 
\draw[fill=white] (a22)circle (0.3cm); 
\node at ($(a22)+($0*(x)+0*(y)+0*(z)$)$) {$\begin{smallmatrix}\hspace{1mm}1^{p}\\
\hspace{-0.5mm}2\end{smallmatrix}$}; 
\coordinate(a23) at ($(15)!0.5!(17)$); 
\draw[fill=white] (a23)circle (0.35cm); 
\node at ($(a23)+($0*(x)+0*(y)+0*(z)$)$) {$\begin{smallmatrix}\hspace{1mm}2^{q} \\ \hspace{-0.5mm}1\,3\end{smallmatrix}$}; 
\coordinate(a24) at ($(16)!0.4!(17)$); 
\draw[fill=white] (a24)circle (0.25cm); 
\node at ($(a24)+($0*(x)+0*(y)+0*(z)$)$) {$\begin{smallmatrix}1^{p}\end{smallmatrix}$}; 
\coordinate(a25) at ($(14)!0.55!(18)$); 
\draw[fill=white] (a25)circle (0.35cm); 
\node at ($(a25)+($0*(x)+0*(y)+0*(z)$)$) {$\begin{smallmatrix}3^{q+r}\end{smallmatrix}$}; 
\coordinate(a26) at ($(17)!0.6!(19)$); 
\draw[fill=white] (a26)circle (0.25cm); 
\node at ($(a26)+($0*(x)+0*(y)+0*(z)$)$) {$\begin{smallmatrix}3^{r}\end{smallmatrix}$}; 
\coordinate(a27) at ($(18)!0.5!(19)$); 
\draw[fill=white] (a27)circle (0.45cm); 
\node at ($(a27)+($0*(x)+0*(y)+0*(z)$)$) {$\begin{smallmatrix}\hspace{1mm}2^{p+q}\\\hspace{-3.5mm}1\end{smallmatrix}$}; 
\coordinate(a28) at ($(16)!0.6!(20)$); 
\draw[fill=white] (a28)circle (0.25cm); 
\node at ($(a28)+($0*(x)+0*(y)+0*(z)$)$) {$\begin{smallmatrix}3^{r}\end{smallmatrix}$}; 
\coordinate(a29) at ($(19)!0.6!(20)$); 
\draw[fill=white] (a29)circle (0.25cm); 
\node at ($(a29)+($0*(x)+0*(y)+0*(z)$)$) {$\begin{smallmatrix}1^{p}\end{smallmatrix}$}; 
\coordinate(a30) at ($(13)!0.3!(20)$); 
\draw[fill=white] (a30)circle (0.35cm); 
\node at ($(a30)+($0*(x)+0*(y)+0*(z)$)$) {$\begin{smallmatrix}2^{q+s}\end{smallmatrix}$}; 
\end{tikzpicture}
\]

In particular, $\overrightarrow{\N}_1(X)\simeq\Hasse(\twosilt A)$ holds if all of $p,q,r,s$ are positive.
\end{example}

It is interesting and difficult to realize a fan as a normal fan e.g. \cite{HPS,PPPP}. 
From this viewpoint, the following useful result follows immediately from Corollary \ref{Delzant exist}(a)(i).

\begin{corollary}\label{g is normal}
Let $\Sigma$ be a complete fan in $V=\R^d$. If there exists a finite dimensional algebra $A$ over a field $k$ such that $\Sigma\simeq\Sigma(A)$, then $\Sigma$ is a normal fan of a polytope in $V^*$.
\end{corollary}


\section{Convex $g$-polygons and Smooth Fano $g$-polytopes}\label{section 6}

In this section, we give complete classifications of convex $g$-polygons and smooth Fano $g$-polytopes. To state our result explicitly, we need the following general notion.

\begin{definition}\label{define isomorphism of g-fans}
\begin{enumerate}[\rm(a)]
\item Let $(\Sigma,\sigma_+)$ and $(\Sigma',\sigma'_+)$ be sign-coherent fans in $\R^d$ and $\R^{d'}$ respectively. An \emph{isomorphism of sign-coherent fans} is an isomorphism $f:\Sigma\simeq\Sigma'$ of fans (Definition \ref{define isomorphism of fans}) such that $\{f(\sigma_+),f(-\sigma_+)\}=\{\sigma'_+,-\sigma'_+\}$.
\item Let $A$ and $A'$ be finite dimensional algebras. An \emph{isomorphism of $g$-fans} is an isomorphism $\Sigma(A)\simeq\Sigma(A')$ of sign-coherent fans. The induced isomorphism $\P(A)\simeq \P(A')$ is called an \emph{isomorphism of $g$-polytopes}.
\end{enumerate}
\end{definition}

We have the following finiteness of convex $g$-polytopes/$g$-fans.

\begin{proposition}\label{finiteness of g-convex fans}
For each $n\ge1$, 
there exist only finitely many convex $g$-polytopes (respectively, convex $g$-fans) of dimension $n$ up to isomorphisms of $g$-polytopes (respectively, $g$-fans).
\end{proposition}

\begin{proof}  
By Theorem \ref{characterize g-convex}(a), the origin is a unique lattice point in the interior of $P(A)$. By \cite{LZ}, there exists only finitely many convex $g$-polytopes of dimension $n$ up to isomorphisms.
Each polytope has only finitely many unimodular triangulations, and each nonsingular fan has only finitely many choices of $\sigma_+$. Thus the assertion follows.
\end{proof}

\subsection{Convex $g$-polygons}

In this subsection, we give a classification of convex $g$-polygons.
The following result shows that Example \ref{convex rank 2} gives all convex $g$-polygons up to isomorphism of $g$-polygons.

\begin{theorem}\label{classify rank2}
There are precisely 7 convex $g$-polygons up to isomorphism of $g$-polytopes (Definition \ref{define isomorphism of g-fans}). 
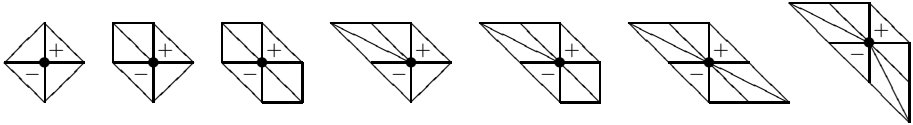
\[{\begin{xy}
0;<3pt,0pt>:<0pt,3pt>::
(0,-5)="0",
(-5,0)="1",
(0,0)*{\bullet},
(0,0)="2",
(5,0)="3",
(0,5)="4",
(1.5,1.5)*{{\scriptstyle +}},
(-1.5,-1.5)*{{\scriptstyle -}},
\ar@{-}"0";"1",
\ar@{-}"1";"4",
\ar@{-}"4";"3",
\ar@{-}"3";"0",
\ar@{-}"2";"0",
\ar@{-}"2";"1",
\ar@{-}"2";"3",
\ar@{-}"2";"4",
\end{xy}}\ \ \ 
{\begin{xy}
0;<-3pt,0pt>:<0pt,-3pt>::
(0,-5)="0",
(-5,0)="1",
(0,0)*{\bullet},
(0,0)="2",
(5,0)="3",
(-5,5)="4",
(0,5)="5",
(1.5,1.5)*{{\scriptstyle -}},
(-1.5,-1.5)*{{\scriptstyle +}},
\ar@{-}"0";"1",
\ar@{-}"1";"4",
\ar@{-}"4";"5",
\ar@{-}"5";"3",
\ar@{-}"3";"0",
\ar@{-}"2";"0",
\ar@{-}"2";"1",
\ar@{-}"2";"3",
\ar@{-}"2";"4",
\ar@{-}"2";"5",
\end{xy}}\ \ \ 
{\begin{xy}
0;<3pt,0pt>:<0pt,3pt>::
(0,-5)="0",
(5,-5)="1",
(-5,0)="2",
(0,0)*{\bullet},
(0,0)="3",
(5,0)="4",
(-5,5)="5",
(0,5)="6",
(1.5,1.5)*{{\scriptstyle +}},
(-1.5,-1.5)*{{\scriptstyle -}},
\ar@{-}"0";"2",
\ar@{-}"2";"5",
\ar@{-}"5";"6",
\ar@{-}"6";"4",
\ar@{-}"4";"1",
\ar@{-}"1";"0",
\ar@{-}"3";"0",
\ar@{-}"3";"1",
\ar@{-}"3";"2",
\ar@{-}"3";"4",
\ar@{-}"3";"5",
\ar@{-}"3";"6",
\end{xy}}\ \ \ 
{\begin{xy}
0;<-3pt,0pt>:<0pt,-3pt>::
(0,-5)="0",
(-5,0)="1",
(0,0)*{\bullet},
(0,0)="2",
(5,0)="3",
(-10,5)="4",
(-5,5)="5",
(0,5)="6",
(1.5,1.5)*{{\scriptstyle -}},
(-1.5,-1.5)*{{\scriptstyle +}},
\ar@{-}"0";"3",
\ar@{-}"3";"6",
\ar@{-}"6";"4",
\ar@{-}"4";"0",
\ar@{-}"2";"0",
\ar@{-}"2";"1",
\ar@{-}"2";"3",
\ar@{-}"2";"4",
\ar@{-}"2";"5",
\ar@{-}"2";"6",
\end{xy}}\ \ \ 
{\begin{xy}
0;<-3pt,0pt>:<0pt,-3pt>::
(0,-5)="0",
(5,-5)="1",
(-5,0)="2",
(0,0)*{\bullet},
(0,0)="3",
(5,0)="4",
(-10,5)="5",
(-5,5)="6",
(0,5)="7",
(1.5,1.5)*{{\scriptstyle -}},
(-1.5,-1.5)*{{\scriptstyle +}},
\ar@{-}"0";"1",
\ar@{-}"1";"4",
\ar@{-}"4";"7",
\ar@{-}"7";"5",
\ar@{-}"5";"0",
\ar@{-}"3";"0",
\ar@{-}"3";"1",
\ar@{-}"3";"2",
\ar@{-}"3";"4",
\ar@{-}"3";"5",
\ar@{-}"3";"6",
\ar@{-}"3";"7",
\end{xy}}\ \ \ 
{\begin{xy}
0;<3pt,0pt>:<0pt,3pt>::
(0,-5)="0",
(5,-5)="1",
(10,-5)="2",
(-5,0)="3",
(0,0)*{\bullet},
(0,0)="4",
(5,0)="5",
(-10,5)="6",
(-5,5)="7",
(0,5)="8",
(1.5,1.5)*{{\scriptstyle +}},
(-1.5,-1.5)*{{\scriptstyle -}},
\ar@{-}"0";"2",
\ar@{-}"2";"8",
\ar@{-}"8";"6",
\ar@{-}"6";"0",
\ar@{-}"4";"0",
\ar@{-}"4";"1",
\ar@{-}"4";"2",
\ar@{-}"4";"3",
\ar@{-}"4";"5",
\ar@{-}"4";"6",
\ar@{-}"4";"7",
\ar@{-}"4";"8",
\end{xy}}\ \ \ 
{\begin{xy}
0;<-3pt,0pt>:<0pt,-3pt>::
(0,-2.5)="0",
(5,-7.5)="1",
(5,-2.5)="2",
(-5,2.5)="3",
(0,2.5)*{\bullet},
(0,2.5)="4",
(5,2.5)="5",
(-10,7.5)="6",
(0,7.5)="7",
(-5,7.5)="8",
(1.5,4)*{{\scriptstyle -}},
(-1.5,1)*{{\scriptstyle +}},
\ar@{-}"6";"1",
\ar@{-}"1";"5",
\ar@{-}"5";"7",
\ar@{-}"7";"6",
\ar@{-}"4";"0",
\ar@{-}"4";"1",
\ar@{-}"4";"2",
\ar@{-}"4";"3",
\ar@{-}"4";"5",
\ar@{-}"4";"6",
\ar@{-}"4";"7",
\ar@{-}"4";"8",
\end{xy}}
\] 
The corresponding $c$-polygons are the following.
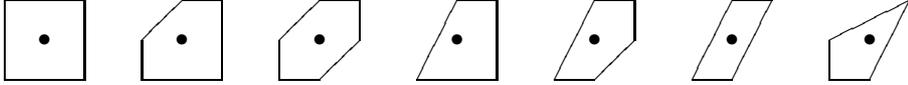
\[
{\begin{xy}
0;<3pt,0pt>:<0pt,3pt>::
(-5,-5)="0",
(5,-5)="1",
(0,0)*{\bullet},
(5,5)="2",
(-5,5)="3",
\ar@{-}"0";"1",
\ar@{-}"1";"2",
\ar@{-}"2";"3",
\ar@{-}"3";"0",
\end{xy}}\ \ \ \ \ \ 
{\begin{xy}
0;<-3pt,0pt>:<0pt,-3pt>::
(-5,-5)="0",
(5,-5)="1",
(0,0)*{\bullet},
(5,5)="2",
(0,5)="3",
(-5,0)="4",
\ar@{-}"0";"1",
\ar@{-}"1";"2",
\ar@{-}"2";"3",
\ar@{-}"3";"4",
\ar@{-}"4";"0",
\end{xy}}\ \ \ \ \ \ 
{\begin{xy}
0;<3pt,0pt>:<0pt,3pt>::
(-5,-5)="0",
(0,-5)="1",
(5,0)="2",
(0,0)*{\bullet},
(5,5)="3",
(0,5)="4",
(-5,0)="5",
\ar@{-}"0";"1",
\ar@{-}"1";"2",
\ar@{-}"2";"3",
\ar@{-}"3";"4",
\ar@{-}"4";"5",
\ar@{-}"5";"0",
\end{xy}}\ \ \ \ \ \ 
{\begin{xy}
0;<-3pt,0pt>:<0pt,-3pt>::
(-5,-5)="0",
(5,-5)="1",
(5,5)="2",
(0,0)*{\bullet},
(0,5)="3",
\ar@{-}"0";"1",
\ar@{-}"1";"2",
\ar@{-}"2";"3",
\ar@{-}"3";"0",
\end{xy}}\ \ \ \ \ \ 
{\begin{xy}
0;<-3pt,0pt>:<0pt,-3pt>::
(-5,-5)="0",
(0,-5)="1",
(5,0)="2",
(5,5)="3",
(0,0)*{\bullet},
(0,5)="4",
\ar@{-}"0";"1",
\ar@{-}"1";"2",
\ar@{-}"2";"3",
\ar@{-}"3";"4",
\ar@{-}"4";"0",
\end{xy}}\ \ \ \ \ \ 
{\begin{xy}
0;<3pt,0pt>:<0pt,3pt>::
(-5,-5)="0",
(0,-5)="1",
(5,5)="2",
(0,0)*{\bullet},
(0,5)="3",
\ar@{-}"0";"1",
\ar@{-}"1";"2",
\ar@{-}"2";"3",
\ar@{-}"3";"0",
\end{xy}}\ \ \ \ \ \ 
{\begin{xy}
0;<3pt,0pt>:<0pt,3pt>::
(-5,-5)="0",
(0,-5)="1",
(5,5)="2",
(0,0)*{\bullet},
(-5,0)="3",
\ar@{-}"0";"1",
\ar@{-}"1";"2",
\ar@{-}"2";"3",
\ar@{-}"3";"0",
\end{xy}}
\]
\end{theorem}

\begin{proof}
By Example \ref{convex rank 2}, these 7 polygons are in fact $g$-polytopes. Thus it suffices to show that they are all.
By Theorem \ref{reflexive polytope}, convex $g$-polytopes are reflexive.
It is well-known that there are precisely 16 reflexive polytopes \cite{PR}.
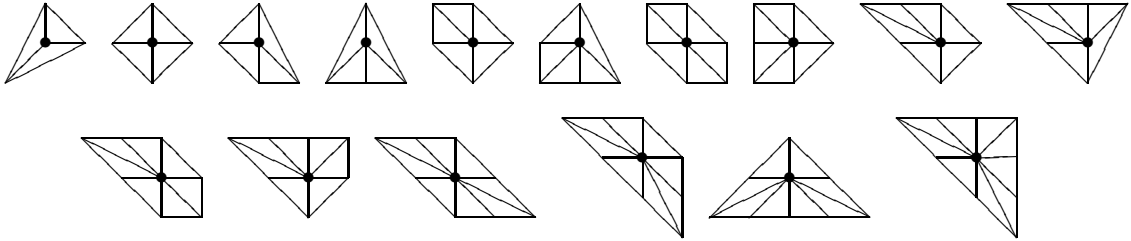
\begin{eqnarray*}
&{\begin{xy}
0;<3pt,0pt>:<0pt,3pt>::
(-5,-5)="1",
(0,0)*{\bullet},
(0,0)="2",
(5,0)="3",
(0,5)="4",
\ar@{-}"2";"1",
\ar@{-}"1";"4",
\ar@{-}"4";"3",
\ar@{-}"2";"3",
\ar@{-}"2";"4",
\ar@{-}"3";"1",
\end{xy}}\ \ \ 
{\begin{xy}
0;<3pt,0pt>:<0pt,3pt>::
(0,-5)="0",
(-5,0)="1",
(0,0)*{\bullet},
(0,0)="2",
(5,0)="3",
(0,5)="4",
\ar@{-}"0";"1",
\ar@{-}"1";"4",
\ar@{-}"4";"3",
\ar@{-}"3";"0",
\ar@{-}"2";"0",
\ar@{-}"2";"1",
\ar@{-}"2";"3",
\ar@{-}"2";"4",
\end{xy}}\ \ \ 
{\begin{xy}
0;<3pt,0pt>:<0pt,3pt>::
(0,-5)="0",
(-5,0)="1",
(0,0)*{\bullet},
(0,0)="2",
(5,-5)="3",
(0,5)="4",
\ar@{-}"0";"1",
\ar@{-}"1";"4",
\ar@{-}"4";"3",
\ar@{-}"3";"0",
\ar@{-}"2";"0",
\ar@{-}"2";"1",
\ar@{-}"2";"3",
\ar@{-}"2";"4",
\end{xy}}\ \ \ 
{\begin{xy}
0;<3pt,0pt>:<0pt,3pt>::
(0,-5)="0",
(-5,-5)="1",
(0,0)*{\bullet},
(0,0)="2",
(5,-5)="3",
(0,5)="4",
\ar@{-}"0";"1",
\ar@{-}"1";"4",
\ar@{-}"4";"3",
\ar@{-}"3";"0",
\ar@{-}"2";"0",
\ar@{-}"2";"1",
\ar@{-}"2";"3",
\ar@{-}"2";"4",
\end{xy}}\ \ \ 
{\begin{xy}
0;<-3pt,0pt>:<0pt,-3pt>::
(0,-5)="0",
(-5,0)="1",
(0,0)*{\bullet},
(0,0)="2",
(5,0)="3",
(-5,5)="4",
(0,5)="5",
\ar@{-}"0";"1",
\ar@{-}"1";"4",
\ar@{-}"4";"5",
\ar@{-}"5";"3",
\ar@{-}"3";"0",
\ar@{-}"2";"0",
\ar@{-}"2";"1",
\ar@{-}"2";"3",
\ar@{-}"2";"4",
\ar@{-}"2";"5",
\end{xy}}\ \ \ 
{\begin{xy}
0;<3pt,0pt>:<0pt,3pt>::
(0,-5)="0",
(-5,-5)="1",
(0,0)*{\bullet},
(0,0)="2",
(5,-5)="3",
(-5,0)="4",
(0,5)="5",
\ar@{-}"0";"1",
\ar@{-}"1";"4",
\ar@{-}"4";"5",
\ar@{-}"5";"3",
\ar@{-}"3";"0",
\ar@{-}"2";"0",
\ar@{-}"2";"1",
\ar@{-}"2";"3",
\ar@{-}"2";"4",
\ar@{-}"2";"5",
\end{xy}}\ \ \ 
{\begin{xy}
0;<3pt,0pt>:<0pt,3pt>::
(0,-5)="0",
(5,-5)="1",
(-5,0)="2",
(0,0)*{\bullet},
(0,0)="3",
(5,0)="4",
(-5,5)="5",
(0,5)="6",
\ar@{-}"0";"2",
\ar@{-}"2";"5",
\ar@{-}"5";"6",
\ar@{-}"6";"4",
\ar@{-}"4";"1",
\ar@{-}"1";"0",
\ar@{-}"3";"0",
\ar@{-}"3";"1",
\ar@{-}"3";"2",
\ar@{-}"3";"4",
\ar@{-}"3";"5",
\ar@{-}"3";"6",
\end{xy}}\ \ \ 
{\begin{xy}
0;<3pt,0pt>:<0pt,3pt>::
(-5,-5)="0",
(0,-5)="1",
(-5,0)="2",
(0,0)*{\bullet},
(0,0)="3",
(5,0)="4",
(-5,5)="5",
(0,5)="6",
\ar@{-}"0";"2",
\ar@{-}"2";"5",
\ar@{-}"5";"6",
\ar@{-}"6";"4",
\ar@{-}"4";"1",
\ar@{-}"1";"0",
\ar@{-}"3";"0",
\ar@{-}"3";"1",
\ar@{-}"3";"2",
\ar@{-}"3";"4",
\ar@{-}"3";"5",
\ar@{-}"3";"6",
\end{xy}}\ \ \ 
{\begin{xy}
0;<-3pt,0pt>:<0pt,-3pt>::
(0,-5)="0",
(-5,0)="1",
(0,0)*{\bullet},
(0,0)="2",
(5,0)="3",
(-10,5)="4",
(-5,5)="5",
(0,5)="6",
\ar@{-}"0";"3",
\ar@{-}"3";"6",
\ar@{-}"6";"4",
\ar@{-}"4";"0",
\ar@{-}"2";"0",
\ar@{-}"2";"1",
\ar@{-}"2";"3",
\ar@{-}"2";"4",
\ar@{-}"2";"5",
\ar@{-}"2";"6",
\end{xy}}&
\end{eqnarray*}
\begin{eqnarray*}
&{\begin{xy}
0;<-3pt,0pt>:<0pt,-3pt>::
(0,-5)="0",
(-5,0)="1",
(0,0)*{\bullet},
(0,0)="2",
(5,5)="3",
(-10,5)="4",
(-5,5)="5",
(0,5)="6",
\ar@{-}"0";"3",
\ar@{-}"3";"6",
\ar@{-}"6";"4",
\ar@{-}"4";"0",
\ar@{-}"2";"0",
\ar@{-}"2";"1",
\ar@{-}"2";"3",
\ar@{-}"2";"4",
\ar@{-}"2";"5",
\ar@{-}"2";"6",
\end{xy}}\ \ \ {\begin{xy}
0;<-3pt,0pt>:<0pt,-3pt>::
(0,-5)="0",
(5,-5)="1",
(-5,0)="2",
(0,0)*{\bullet},
(0,0)="3",
(5,0)="4",
(-10,5)="5",
(-5,5)="6",
(0,5)="7",
\ar@{-}"0";"1",
\ar@{-}"1";"4",
\ar@{-}"4";"7",
\ar@{-}"7";"5",
\ar@{-}"5";"0",
\ar@{-}"3";"0",
\ar@{-}"3";"1",
\ar@{-}"3";"2",
\ar@{-}"3";"4",
\ar@{-}"3";"5",
\ar@{-}"3";"6",
\ar@{-}"3";"7",
\end{xy}}\ \ \ 
{\begin{xy}
0;<-3pt,0pt>:<0pt,-3pt>::
(-5,0)="0",
(0,-5)="1",
(-10,5)="2",
(0,0)*{\bullet},
(0,0)="3",
(5,0)="4",
(-5,5)="5",
(0,5)="6",
(5,5)="7",
\ar@{-}"1";"4",
\ar@{-}"4";"7",
\ar@{-}"7";"2",
\ar@{-}"2";"1",
\ar@{-}"3";"0",
\ar@{-}"3";"1",
\ar@{-}"3";"2",
\ar@{-}"3";"4",
\ar@{-}"3";"5",
\ar@{-}"3";"6",
\ar@{-}"3";"7",
\end{xy}}\ \ \ 
{\begin{xy}
0;<3pt,0pt>:<0pt,3pt>::
(0,-5)="0",
(5,-5)="1",
(10,-5)="2",
(-5,0)="3",
(0,0)*{\bullet},
(0,0)="4",
(5,0)="5",
(-10,5)="6",
(-5,5)="7",
(0,5)="8",
\ar@{-}"0";"2",
\ar@{-}"2";"8",
\ar@{-}"8";"6",
\ar@{-}"6";"0",
\ar@{-}"4";"0",
\ar@{-}"4";"1",
\ar@{-}"4";"2",
\ar@{-}"4";"3",
\ar@{-}"4";"5",
\ar@{-}"4";"6",
\ar@{-}"4";"7",
\ar@{-}"4";"8",
\end{xy}}\ \ \ 
{\begin{xy}
0;<-3pt,0pt>:<0pt,-3pt>::
(0,-2.5)="0",
(5,-7.5)="1",
(5,-2.5)="2",
(-5,2.5)="3",
(0,2.5)*{\bullet},
(0,2.5)="4",
(5,2.5)="5",
(-10,7.5)="6",
(0,7.5)="7",
(-5,7.5)="8",
\ar@{-}"6";"1",
\ar@{-}"1";"5",
\ar@{-}"5";"7",
\ar@{-}"7";"6",
\ar@{-}"4";"0",
\ar@{-}"4";"1",
\ar@{-}"4";"2",
\ar@{-}"4";"3",
\ar@{-}"4";"5",
\ar@{-}"4";"6",
\ar@{-}"4";"7",
\ar@{-}"4";"8",
\end{xy}}\ \ \ 
{\begin{xy}
0;<3pt,0pt>:<0pt,3pt>::
(0,-5)="0",
(5,-5)="1",
(10,-5)="2",
(-5,-5)="3",
(0,0)*{\bullet},
(0,0)="4",
(5,0)="5",
(-10,-5)="6",
(-5,0)="7",
(0,5)="8",
\ar@{-}"0";"2",
\ar@{-}"2";"8",
\ar@{-}"8";"6",
\ar@{-}"6";"0",
\ar@{-}"4";"0",
\ar@{-}"4";"1",
\ar@{-}"4";"2",
\ar@{-}"4";"3",
\ar@{-}"4";"5",
\ar@{-}"4";"6",
\ar@{-}"4";"7",
\ar@{-}"4";"8",
\end{xy}}\ \ \ 
{\begin{xy}
0;<-3pt,0pt>:<0pt,-3pt>::
(0,-2.5)="0",
(5,-7.5)="1",
(5,-2.5)="2",
(-5,2.5)="3",
(0,2.5)*{\bullet},
(0,2.5)="4",
(5,2.7)="5",
(-10,7.5)="6",
(0,7.5)="7",
(-5,7.5)="8",
(5,7.5)="9",
\ar@{-}"6";"1",
\ar@{-}"1";"9",
\ar@{-}"9";"6",
\ar@{-}"4";"0",
\ar@{-}"4";"1",
\ar@{-}"4";"2",
\ar@{-}"4";"3",
\ar@{-}"4";"5",
\ar@{-}"4";"6",
\ar@{-}"4";"7",
\ar@{-}"4";"8",
\ar@{-}"4";"9",
\end{xy}}& 
\end{eqnarray*}
On the other hand, $g$-fans are sign-coherent by Proposition \ref{tilting is sign-coherent}. 
One can easily check by case-by-case analysis that, among 16 fans given by these 16 reflexive polygons, there are exactly 7 sign-coherent fans listed above up to isomorphism of sign-coherent fans.
\end{proof}


\subsection{Smooth Fano $g$-polytopes}
The aim of this subsection is to characterize finite dimensional algebras whose $g$-polytopes satisfy the following property.

\begin{definition}\label{define smooth Fano}\cite{O}
Let $d$ be a positive integer. 
A lattice polytope $P$ in $\mathbb{R}^d$ containing the origin in its interior is called a 
\emph{smooth Fano $d$-polytope} if the vertices of every facet $F$ of $P$ is $\mathbb{Z}^d$-basis of the lattice $\mathbb{Z}^d$.
\end{definition}

Among 16 reflexive polygons, there are 5 smooth Fano polygons.
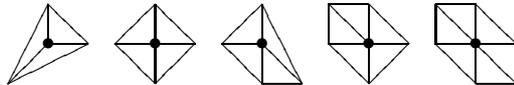
\[{\begin{xy}
0;<3pt,0pt>:<0pt,3pt>::
(-5,-5)="1",
(0,0)*{\bullet},
(0,0)="2",
(5,0)="3",
(0,5)="4",
\ar@{-}"2";"1",
\ar@{-}"1";"4",
\ar@{-}"4";"3",
\ar@{-}"2";"3",
\ar@{-}"2";"4",
\ar@{-}"3";"1",
\end{xy}}\ \ \ 
{\begin{xy}
0;<3pt,0pt>:<0pt,3pt>::
(0,-5)="0",
(-5,0)="1",
(0,0)*{\bullet},
(0,0)="2",
(5,0)="3",
(0,5)="4",
\ar@{-}"0";"1",
\ar@{-}"1";"4",
\ar@{-}"4";"3",
\ar@{-}"3";"0",
\ar@{-}"2";"0",
\ar@{-}"2";"1",
\ar@{-}"2";"3",
\ar@{-}"2";"4",
\end{xy}}\ \ \ 
{\begin{xy}
0;<3pt,0pt>:<0pt,3pt>::
(0,-5)="0",
(-5,0)="1",
(0,0)*{\bullet},
(0,0)="2",
(5,-5)="3",
(0,5)="4",
\ar@{-}"0";"1",
\ar@{-}"1";"4",
\ar@{-}"4";"3",
\ar@{-}"3";"0",
\ar@{-}"2";"0",
\ar@{-}"2";"1",
\ar@{-}"2";"3",
\ar@{-}"2";"4",
\end{xy}}\ \ \ 
{\begin{xy}
0;<-3pt,0pt>:<0pt,-3pt>::
(0,-5)="0",
(-5,0)="1",
(0,0)*{\bullet},
(0,0)="2",
(5,0)="3",
(-5,5)="4",
(0,5)="5",
\ar@{-}"0";"1",
\ar@{-}"1";"4",
\ar@{-}"4";"5",
\ar@{-}"5";"3",
\ar@{-}"3";"0",
\ar@{-}"2";"0",
\ar@{-}"2";"1",
\ar@{-}"2";"3",
\ar@{-}"2";"4",
\ar@{-}"2";"5",
\end{xy}}\ \ \ 
{\begin{xy}
0;<3pt,0pt>:<0pt,3pt>::
(0,-5)="0",
(5,-5)="1",
(-5,0)="2",
(0,0)*{\bullet},
(0,0)="3",
(5,0)="4",
(-5,5)="5",
(0,5)="6",
\ar@{-}"0";"2",
\ar@{-}"2";"5",
\ar@{-}"5";"6",
\ar@{-}"6";"4",
\ar@{-}"4";"1",
\ar@{-}"1";"0",
\ar@{-}"3";"0",
\ar@{-}"3";"1",
\ar@{-}"3";"2",
\ar@{-}"3";"4",
\ar@{-}"3";"5",
\ar@{-}"3";"6",
\end{xy}}
\]

The following special class of algebras plays an important role.

\begin{definition}
A finite dimensional algebra $A$ over a field $k$ is \emph{pentagon type} (respectively, \emph{hexagon type}) if $\P(A)$ is the left (respectively, right) polygon below.
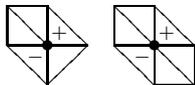
\[{\begin{xy}
0;<-3pt,0pt>:<0pt,-3pt>::
(0,-5)="0",
(-5,0)="1",
(0,0)*{\bullet},
(0,0)="2",
(5,0)="3",
(-5,5)="4",
(0,5)="5",
(1.5,1.5)*{{\scriptstyle -}},
(-1.5,-1.5)*{{\scriptstyle +}},
\ar@{-}"0";"1",
\ar@{-}"1";"4",
\ar@{-}"4";"5",
\ar@{-}"5";"3",
\ar@{-}"3";"0",
\ar@{-}"2";"0",
\ar@{-}"2";"1",
\ar@{-}"2";"3",
\ar@{-}"2";"4",
\ar@{-}"2";"5",
\end{xy}}\ \ \ 
{\begin{xy}
0;<3pt,0pt>:<0pt,3pt>::
(0,-5)="0",
(5,-5)="1",
(-5,0)="2",
(0,0)*{\bullet},
(0,0)="3",
(5,0)="4",
(-5,5)="5",
(0,5)="6",
(1.5,1.5)*{{\scriptstyle +}},
(-1.5,-1.5)*{{\scriptstyle -}},
\ar@{-}"0";"2",
\ar@{-}"2";"5",
\ar@{-}"5";"6",
\ar@{-}"6";"4",
\ar@{-}"4";"1",
\ar@{-}"1";"0",
\ar@{-}"3";"0",
\ar@{-}"3";"1",
\ar@{-}"3";"2",
\ar@{-}"3";"4",
\ar@{-}"3";"5",
\ar@{-}"3";"6",
\end{xy}}\]
\end{definition}

We refer to \cite{AHIKM1} for characterizations of these algebras.

The following is a main result of this section.

\begin{theorem}\label{smooth Fano tilting polytope}
Let $A$ be a finite dimensional algebra over a field $k$.
Then $\P(A)$ is a smooth Fano polytope if and only if $A$ is a product of local algebras, algebras of pentagon type and algebras of hexagon type.
\end{theorem}

The rest of this section is devoted to a proof of Theorem \ref{smooth Fano tilting polytope}.
We start with the following, where we refer to Definition \ref{define free sum} for the definition of sums of polytopes.

\begin{lemma}\label{Fano decomp}
The following assertions hold.
\begin{enumerate}[\rm(a)]
\item Let $P_i$ be a lattice polytope in $\R^{d_i}$ for $1\le i\le\ell$. Then $P_1\oplus\cdots\oplus P_\ell$ is smooth Fano if and only if $P_i$ is smooth Fano for each $1\le i\le\ell$.
\item Let $A_i$ be a finite dimensional algebra over a field $k$ for $1\le i\le\ell$ and $A=A_1\times\cdots\times A_\ell$. 
Then $\P(A)$ is smooth Fano if and only if $\P(A_i)$ is smooth Fano for each $1\le i\le\ell$.
\end{enumerate}
\end{lemma}

\begin{proof}
(a) This is clear. 

(b) This is immediate from (a) and Proposition \ref{idempotent decomp 2}(b).
\end{proof}

The following technical notion plays a central role.

\begin{definition}
Let $k\ge1$. A \emph{del Pezzo $2k$-polytope} (repectively, \emph{pseudo del Pezzo $2k$-polytope}) is
\[V_{2k}:=\conv\{\pm e_1,\ldots,\pm e_{2k},\pm\sum_{i=1}^{2k}e_i\}\ \ \ \mbox{(respectively, }\widetilde{V}_{2k}:=\conv\{\pm e_1,\ldots,\pm e_{2k},\sum_{i=1}^{2k}e_i\}).\]
\end{definition}

The following is clear.

\begin{lemma}\label{equivalent}
Let $A$ be a finite dimensional algebra.
\begin{enumerate}[\rm(a)]
\item $A$ is local if and only if  $\P(A)$ is isomorphic to a line segment $[e_1,-e_1]$.
\item $A$ is hexagon type if and only if $\P(A)\simeq V_2$. 
\item $A$ is pentagon type if and only if $\P(A)\simeq\widetilde{V}_2$. 
\end{enumerate}
\end{lemma}

Recall that a polytope $P$ in $\mathbb{R}^d$ is called a \emph{symmetric} if $v\in P$ implies $-v\in P$. More generally, a polytope $P$ in $\mathbb{R}^d$ is called a \emph{pseudo-symmetric} \cite{O} if there exists a facet $F$ of $P$ such that $-F$ is also a facet.

The following result provides a classification of pseudo-symmetric smooth Fano $d$-polytopes.

\begin{proposition}\cite{E, VK}\label{classify symmetric}
Any pseudo-symmetric smooth Fano $d$-polytope $P$ 
splits into copies of line segments, del Pezzo polytopes and pseudo del Pezzo polytopes. That is, $P$ is isomorphic to a convex polytope 
$$ \overset{i\textnormal{-times}}{\overbrace{L\oplus\ldots\oplus L}}\oplus V_{2a_1}\oplus\ldots\oplus V_{2a_j}\oplus 
 \widetilde{V}_{2b_1}\oplus\ldots\oplus \widetilde{V}_{2b_k},$$ 
where each $L$ is a lattice convex polytope isomorphic to the line segment $[e_1,-e_1]$ and $i+2(a_1+\ldots+a_j+b_1+\ldots+b_k) =d$. 
\end{proposition}

We also need the following description of the facets of (pseudo) del Pezzo polytopes.

\begin{lemma}\label{facets}
Let $s:=\sum_{i=1}^{2k}e_i$. \\
{\rm (a)} The facets of a del Pezzo $2k$-polytope $V_{2k}$ consist of 
\begin{eqnarray}\label{facet1}
\conv\{e_{i_1},\ldots,e_{i_k},-e_{j_1},\ldots,-e_{j_k}\}, 
\end{eqnarray}
where $1 \leq i_1 < \cdots < i_k \leq 2k$, $1 \leq j_1 < \cdots < j_k \leq 2k$ and $\{i_1,\ldots,i_k\} \cap \{j_1,\ldots,j_k\} = \emptyset$, and 
\begin{eqnarray}
&&\conv\{e_{i_1},\ldots,e_{i_k},-e_{j_1},\ldots,-e_{j_{k-1}},s\} \text{ and }\label{facet2} \\
&&\conv\{-e_{i_1},\ldots,-e_{i_k},e_{j_1},\ldots,e_{j_{k-1}},-s\}, \label{facet3}
\end{eqnarray}
where $1 \leq i_1 < \cdots < i_k \leq 2k$, $1 \leq j_1 < \cdots < j_{k-1} \leq 2k$ and $\{i_1,\ldots,i_k\} \cap \{j_1,\ldots,j_{k-1}\} = \emptyset$. \\
{\rm (b)} The facets of a pseudo del Pezzo $2k$-polytope $\widetilde{V}_{2k}$  consist of \eqref{facet2} and 
\begin{eqnarray}\label{facet4}
\conv\{e_{i_1},\ldots,e_{i_\ell},-e_{j_1},\ldots,-e_{j_{2k-\ell}}\}, 
\end{eqnarray}
where $1 \leq i_1 < \cdots < i_\ell \leq 2k$, $1 \leq j_1 < \cdots < j_{2k-\ell} \leq 2k$, $\{i_1,\ldots,i_\ell\} \cap \{j_1,\ldots,j_{2k-\ell}\} = \emptyset$ and $\ell \leq k$. 
\end{lemma}

\begin{proof}
Although the proof is routine, we include it for convenience of the reader.

(a) Let $F$ be a facet of $V_{2k}$ and let $H$ be the supporting hyperplane of $F$. Then $\pm e_i$ is not contained in $F$ at the same time, and so is $\pm s$. 
\begin{itemize}
\item When $F$ does not contain $\pm s$, the vertices of $F$ look like, say, $e_1,\ldots,e_\ell,-e_{\ell+1},\ldots,-e_{2k}$. 
Then $H$ is defined by the equality $f(x)=1$ for
\[f(x):=x_1+\cdots+x_\ell-x_{\ell+1}-\cdots-x_{2k}\]
and $V_{2k}$ is contained in the half space $\{x\in \R^d\mid f(x)\le1\}$. If $\ell \neq k$, then one of $f(s)$ and $f(-s)$ is greater than $1$, a contradiction. 
Hence, $\ell=k$, i.e., $F$ is of the form \eqref{facet1}. 
\item When $F$ contains $\pm s$, the vertices of $F$ look like $\pm(e_1,\ldots,e_\ell,-e_{\ell+1},\ldots,-e_{2k-1},s)$. 
Then $H$ is defined by the equality $f(x)=1$ for
\[f(x):=x_1+\cdots+x_\ell-x_{\ell+1}-\cdots-x_{2k-1}+(2k-2\ell)x_{2k}\]
and $V_{2k}$ is contained in the half space $\{x\in \R^d\mid f(x)\le1\}$. If $\ell \neq k$, then one of $f(e_{2k})$ and $f(-e_{2k})$ is greater than $1$, a contradiction. 
Hence, $\ell=k$, i.e., $F$ is of the forms \eqref{facet2} and \eqref{facet3}. 
\end{itemize}

(b) Let $F$ be a facet of $\widetilde{V}_{2k}$ and let $H$ be the supporting hyperplane of $F$. 
When $F$ does not contain $s$, the vertices of $F$ look like $e_1,\ldots,e_\ell,-e_{\ell+1},\ldots,-e_{2k}$. 
Then $H$ is defined by the equality $f(x)=1$ for
\[f(x):=x_1+\cdots+x_\ell-x_{\ell+1}-\cdots-x_{2k}.\]
and $\widetilde{V}_{2k}$ is contained in the half space $\{x\in \R^d\mid f(x)\le1\}$. If $\ell > k$, then $f(s)>1$, a contradiction. 
Hence, $\ell \leq k$, i.e., $F$ is of the form \eqref{facet4}. 
When $F$ contains $s$, the same proof as above can be applied, and we conclude that $F$ is of the form \eqref{facet2}. 
\end{proof}

Let $P$ be a smooth Fano polytope in $\R^d$. Then there exists a unique fan $\Sigma$ in $\R^d$ such that $P=\P(\Sigma)$, see Definition \ref{from fan to polytope}. 
We call $P$ \emph{ordered} if there exists $\sigma_+\in\Sigma_d$ such that $(\Sigma,\sigma_+)$ is an ordered fan in the sense of Definition \ref{define ordered}.

Now we prove the following key observation.

\begin{proposition}\label{del Pezzo}
\begin{enumerate}[\rm(a)]
\item A del Pezzo $2k$-polytope $V_{2k}$ is ordered if and only if $k=1$.
\item A pseudo del Pezzo $2k$-polytope $\widetilde{V}_{2k}$ is ordered if and only if $k=1$.
\end{enumerate}
\end{proposition}

\begin{proof}
Recall that an ordered polytope $P$ has a facet $F_+$ such that $(\cone F_+)^\circ\cap\span F'=\emptyset$ for each non-maximal face $F'$ of $P$.

(a) `if' part is clear. To prove `only if' part, we assume that $V_{2k}$ is ordered for $k\ge2$. 
By Lemma~\ref{facets} (a), we may assume that 
\[F_+=\conv\{e_1,\ldots,e_k,-e_{k+1},\ldots,-e_{2k}\} \text{ or }F_+=\conv\{e_1,\ldots,e_k,-e_{k+1},\ldots,-e_{2k-1},s\}\] 
without loss of generality. 
In the first case, since $k \ge 2$, we have 
\[(\cone F_+)^\circ\ni e_1+\cdots+e_k-e_{k+1}-\cdots-e_{2k}=s+2(-e_{k+1}-\cdots-e_{2k})\in\span F'\]
for the non-maximal face $F'=\conv\{e_{k+1},\ldots,e_{2k},s\}$, a contradiction. In the second case, we have 
\[(\cone F_+)^\circ\ni e_1+\cdots+e_k-e_{k+1}-\cdots-e_{2k-1}+s=2(e_1+\cdots+e_k)+e_{2k} \in\span F'\]
for the non-maximal face $F'=\conv\{e_1,\ldots,e_k,-e_{2k}\}$, a contradiction. 

(b) `if' part is clear. To prove `only if' part, we assume that $\widetilde{V}_{2k}$ is ordered for $k \ge 2$.
By Lemma~\ref{facets} (b), we may assume that
\[F_+=\conv\{e_1,\ldots,e_\ell,-e_{\ell+1},\ldots,-e_{2k}\} \text{ or }F_+=\conv\{e_1,\ldots,e_k,-e_{k+1},\ldots,-e_{2k-1},s\}\] 
where $\ell\le k$.
In the first case, we have 
\[(\cone F_+)^\circ\ni e_1+\cdots+e_\ell-e_{\ell+1}-\cdots-e_{2k}=2(e_1+\cdots+e_{\ell})-s\in\span F'\]
for the non-maximal face $F'=\conv\{e_1,\ldots,e_\ell,s\}$, a contradiction. 
In the second case, we get a contradiction since $-F_+$ is not a facet of $\widetilde{V}_{2k}$.
\end{proof}

Now we are ready to give a proof of Theorem \ref{smooth Fano tilting polytope}.

\begin{proof}[Proof of Theorem \ref{smooth Fano tilting polytope}]
``If'' part is clear. In fact, each of local algebras, algebras of pentagon type and algebras of hexagon type is smooth Fano, and therefore their product is also smooth Fano by Lemma \ref{Fano decomp}(b).

To prove ``only if'' part, assume that $\P(A)$ is smooth Fano.
By Lemma \ref{Fano decomp}(b), we only have to consider the case $A$ is ring-indecomposable.
Then $\Sigma(A)$ is indecomposable by Theorem \ref{idempotent decomp}(c).
By Propositions \ref{classify symmetric} and \ref{del Pezzo}, $\P(A)$ is one of the line segments, del Pezzo 2-polytopes and pseudo del Pezzo 2-polytopes. Hence $A$ is the one of the local algebras, algebras of pentagon type and algebras of hexagon type by Lemma \ref{equivalent}.
\end{proof}


\section{Preprojective algebras and Coxeter fans}\label{section 7}

The aim of this section is to show that the $g$-fans of preprojective algebras $\Pi$ of Dynkin type are all the Coxeter fans, and the $c$-polytopes $\P^c(\Pi)$ are the short root polytopes.
In particular, $\Pi$ is $g$-convex if and only if it is of type $A_n$ or $B_n$. In this case, $\P(\Pi)$ is the dual polytope of the short root polytope.

\subsection{Classical preprojective algebras}\label{CPPA}
To recall the definition of preprojective algebras associated to symmetrizable generalized Cartan matrices, we start with the definition of hereditary algebras.

\begin{definition}
\begin{enumerate}[\rm(a)]
\item
We call a pair $(D_i,{}_iM_j)_{1\leq i,j\leq n}$ \emph{$k$-species} if 
\begin{enumerate} 
\item[(i)] $D_i$ is a finite dimensional division $k$-algebra. 
\item[(ii)] ${}_iM_j$ is a finitely generated $D_j\otimes_kD_i^{\op}$-module. 
In other words, ${}_iM_j$ is a $(D_i,D_j)$-bimodule and $k$ acts centrally on ${}_iM_j$.
\end{enumerate}
\item 
We call a $k$-species $(D_i,{}_iM_j)$ \emph{acyclic} if 
there does not exist a sequence $i_1$, $i_2$, \ldots, $i_\ell$, $i_{\ell+1}=i_1$ such that  ${}_{i_j}M_{i_{j+1}}\neq 0$ for $1\leq j\leq n$.
\end{enumerate}
\end{definition}

Let $(D_i,{}_iM_j)$ be a $k$-species, $\displaystyle D:=\prod_{i=1}^nD_i$ and $\displaystyle M:=\bigoplus_{1\leq,i,j\leq n}{}_iM_j$. 
We define the tensor algebra by $$T_D(M):=\bigoplus_{n=0}^\infty T^n(M),$$
where $T^0(M)=D$ and $T^n(M):=M^{\otimes n}:=M\otimes_D\cdots\otimes_D M$ ($n$ times).
This is a $k$-algebra, which is finite dimensional if and only if $(D_i,{}_iM_j)$  is acyclic. 
Moreover, in this case, $T_D(M)$ is hereditary \cite[Theorem 2.35]{L}.

\begin{definition}
\begin{enumerate}[\rm(a)]
\item A matrix $C = (c_{ij}) \in M_n(\mathbb{Z})$ is a \emph{symmetrizable generalized 
Cartan matrix} if the following conditions hold.
\begin{enumerate} 
\item[(C1)] $c_{ii} = 2$ for all $i$;

\item[(C2)] $c_{ij} \leq 0$ for all $i \neq j$;

\item[(C3)] $c_{ij} \neq 0$ if and only if $c_{ji} \neq 0$.

\item[(C4)] There is a diagonal integer matrix $D = \textnormal{diag}(c_1,\cdots , c_n)$ with $c_i \geq 1$ for all $i$ such
that $CD$ is symmetric. It is called a \emph{symmetrizer} of $C$.
\end{enumerate}

\item Let 
$(D_i,{}_iM_j)$ be an acyclic $k$-species. 
We define the matrix $C=(c_{i,j})_{1\leq i,j\leq n}$ associated to it as follows: We define $c_{ii}=2$ for any $i$. If $i\neq j$, then we define $c_{ij}$ and $c_{ji}$ as follows. 
\begin{enumerate} 
\item[(i)] If ${}_iM_j=0={}_jM_i$, then $c_{ij}=0=c_{ji}.$
\item[(ii)] If ${}_iM_j\neq 0$, then $c_{ij}:=-\dim({}_iM_j)_{D_j}$ and 
$c_{ji}:=-\dim{}_{D_i}({}_iM_j)$.
\item[(iii)] If ${}_jM_i\neq 0$, then $c_{ij}:=-\dim{}_{D_j}({}_jM_i)$  
and $c_{ji}:=-\dim({}_jM_i)_{D_i}$.
\end{enumerate}
This is well-defined since the acyclicity implies that at least one of the ${}_iM_j$ or ${}_jM_i$ is zero.

Then the matrix $C$ is a symmetrizable generalized Cartan matrix. Indeed, let $c_i:=\dim_kD_i$ and $D:=\textnormal{diag}(c_1,\cdots , c_n)$. Then if ${}_iM_j\neq 0$, then we have $c_{ij}c_j=-\dim_k({}_iM_j)=c_{ji}c_i$ and hence $CD$ is symmetric. The other cases are similar. 

\item Let $C=(c_{i,j})$ be the symmetrizable generalized Cartan matrix corresponding an acyclic $k$-species $(D_i,{}_iM_j)$. 
We call $H:=T_D(M)$ is Dynkin type if the matrix $C$ is Dynkin, that is, one of the type $A_n,B_n,\ldots,G_2$.
\end{enumerate}
\end{definition}

For example, the Cartan matrix of type $A_n$ and $B_n$ are, respectively,  
$$\left(\begin{smallmatrix}2 & -1 & 0 & \dots & \dots & 0\\
-1 & 2 & -1 &\ddots  &  & \vdots\\
0 & -1 & \ddots & \ddots & 0 & 0\\
\vdots &\ddots  & \ddots & 2 & -1 & 0\\
\vdots &  & 0 & -1 & 2 & -1\\
0 & \dots & 0 & 0 & -1 & 2\end{smallmatrix}\right)
\ \textnormal{and} \ \left(\begin{smallmatrix}2 & -1 & 0 & \dots & \dots & 0\\
-1 & 2 & -1 &\ddots  &  & \vdots\\
0 & -1 & \ddots & \ddots & 0 & 0\\
\vdots &\ddots  & \ddots & 2 & -1 & 0\\
\vdots &  & 0 & -1 & 2 & -1\\
0 & \dots & 0 & 0 & -2 & 2\end{smallmatrix}\right)
$$

\begin{definition}
Let $H^e:=H\otimes H^{\op}$ and $E:=\Ext_{H^e}^1(H,H^e)\in\mod H^e$. 
We define the tensor algebra 
$\Pi:=T_{H^e}(E)$ and call it the \emph{preprojective algebra} of $H$. 
We remark that it is a $\mathbb{Z}$-graded algebra with $\Pi_i=E^{\otimes i}$, and 
$(\Pi_i)_H\simeq\tau^{-i}H$ in $\mod H$ and 
${}_H(\Pi_i)\simeq\tau^{-i}H$ in $\mod H^{\op}$. 
\end{definition}

Let $C$ be a symmetrizable Cartan matrix of Dynkin type of rank $n$, that is, one of the type $A_n,B_n,\ldots,G_2$.  
Let $\Phi=\Phi(C)$ be the root system of $C$.    
Let $\{\alpha_1,\ldots,\alpha_n\}\subset\Phi$ be a set of simple roots and 
$L$ the root lattice.  We let   $V:=L\otimes_{\mathbb{Z}}\R$ and denote by 
$V^\ast$ the dual of $V$ with the basis 
$\alpha_1^*,\alpha_2^*,\ldots,\alpha_n^*$.
We denote by the natural pairing $(v^*,v)$ for $v\in V$ and $v^*\in V^*$.

Define a reflection $s_i:V\to V$ by
\[s_i(\alpha_j):=\alpha_j-c_{ij}\alpha_i.\]
The \emph{Weyl group} is defined as a subgroup
\[W=W(C)=\langle s_1,\ldots,s_n\rangle\]
of ${\rm GL}(V)$. Then $W$ acts also on $V^*$ by
\[(wf)(v)=(f, w^{-1}v)\ \mbox{ for }\ f\in V^*,\ v\in V.\]

Define the \emph{dominant chamber} as follows:
$$D:=\bigcap_{i\in Q_0}
\{v^*\in V^*\ |\ ( v^*,\alpha_i) \geq 0\}
=\{\sum_{i=1}^n a_i\alpha_i^*\ |\ a_i\geq 0\}.
$$
Then the set $\{w D\ |w\in W \}$
of all cones $wD$ and their faces consists of a fan in $V^*$ and we call the \emph{Coxeter fan}, see e.g.\ \cite{ReS2}. 

Recall from Definition \ref{define c-polytope} that, 
for $X\in\Db(\mod\Pi)$, let
\[[X]':=(\dim_k\End_{\Db(\mod \Pi)}(X))^{-1}[X]\in K_0(\mod \Pi)_\R.\]
For the simple $\Pi$-modules $D_1,\ldots,D_n$, we define the isomorphism
\[\iota:K_0(\mod \Pi)_\R\simeq V\ \mbox{ by }\ \iota([D_i]')=\alpha_i,\]
and we identify $K_0(\mod \Pi)_\R$ with $V$ via $\iota$. Moreover, for $X=X_1\oplus\cdots\oplus 
X_n\in\smc\Pi$, let
\[v_X:=\iota(\sum_{i=1}^n[X_i]')\in V.\] 
We have an action of $W$ on $K_0(\proj\Pi)_\R$ via the isomorphism $\iota^*:V^*\simeq K_0(\proj\Pi)_\R$.

The following is the main result of this subsection. 

\begin{theorem}\label{thm p.p. polytope}
Let $H$ be a finite dimensional hereditary $k$-algebra of Dynkin type and $\Pi$ the preprojective algebra of $H$.
\begin{enumerate}[\rm(a)] 
\item The $g$-fan $\Sigma(\Pi)$ is the Coxeter fan and we have $$\P(\Pi)=\bigcup_{w\in W}w C_{\le1}(\Pi).$$ 

\item The $g$-polytope $\P(\Pi)$ is convex if and only if the Cartan matrix $C$ is type $A_n$ or $B_n$.
\item
\begin{enumerate} 
\item[\rm(i)] Let $\indtwosmc\Pi$ be the set of indecomposable direct summands of 2-simple minded collections of $\Pi$. 
We have 
$$\{[S]'\ |\ S\in\indtwosmc\Pi\}=\Phi.$$
\item[\rm(ii)] Let $\Phi_{\textnormal{short}}$ be the set of short roots of $\Phi$. 
We have 
$$\{v_X\ |\ X\in\twosmc\Pi\}=\Phi_{\textnormal{short}}\quad\textnormal{and}\quad \P^\c(\Pi)=\conv(\Phi_{\textnormal{short}}).$$
\item[\rm(iii)] If the Cartan matrix $C$ is type $A_n$ or 
$B_n$, then $\P(\Pi)$ is the dual polytope of $\P^\c(\Pi)=\conv(\Phi_{\textnormal{short}})$.
\end{enumerate}
\end{enumerate}
\end{theorem}

From now on, we will give a proof of Theorem \ref{thm p.p. polytope}. 
For this purpose, following \cite{ARS}, we introduce some notations. 
Let $X,Y\in\mod H$ be indecomposable modules. 
In the Auslander-Reiten quiver (AR quiver, for short), 
we write $$\xymatrix@C30pt{X\ar[r]^{(a,b)}&Y}$$ if 
there is a minimal right almost split $X^a\oplus M\to Y$ such that $M$ contains no summands isomorphic to $X$, and there is a left right almost split $X\to Y^b\oplus N$ such that $N$ contains no summands isomorphic to $Y$. 
Let $D_X:=\End_H(X)/\rad_H(X,X)$ and $D_Y:=\End_H(Y)/\rad_H(Y,Y)$. 
Then we have the following basic result.

\begin{proposition}\label{irr mor} \cite[section VII]{ARS}\begin{enumerate}[\rm(a)] 
\item If there exists $\xymatrix@C30pt{X\ar[r]^{(a,b)}&Y}$ in the AR quiver of $\mod H$, 
then we have 
$$\dim(\rad_H(X,Y)/\rad^2_H(X,Y))_{D_X}=a\quad\textnormal{and}\quad
\dim{}_{D_Y}(\rad_H(X,Y)/\rad^2_H(X,Y))=b.$$
Moreover we have $$a\dim_k(D_X)=b\dim_k(D_Y)=\dim{}_{k}(\rad_H(X,Y)/\rad^2_H(X,Y)).$$
\item If there exists an arrow $\xymatrix@C30pt{Y\ar[r]^{(b,a)}&Z}$ in the AR quiver with non-projective $Z$, 
then there exists the arrow $\xymatrix@C30pt{\tau Z\ar[r]^{(a,b)}&Y}$.
\end{enumerate}
\end{proposition}

Then we show the following proposition. 

\begin{proposition}\label{ar seq of H}
Let $H$ be a finite dimensional hereditary $k$-algebra and $e_i$ a primitive idempotent of $H$ and $Q_i:=e_iH$ the indecomposable projective $H$-module. Then there exists the following almost split sequence 
\[\xymatrix@C30pt{  
0\ar[r]&Q_i\ar[r]&
{\displaystyle\bigoplus_{{}_jM_i\neq0}}Q_j^{-c_{ij}}\oplus
{\displaystyle\bigoplus_{{}_iM_k\neq0}}\tau^-Q_k^{-c_{ik}}
\ar[r]^{ }  &  \tau^- Q_i\ar[r]& 0.  }\]
\end{proposition}

\begin{proof}
We denote the multiplicity of $X$ in $M$ by 
$m_X(M)$.
Let $Q_i\to E$ be a minimal left almost split. 
By Proposition \ref{irr mor}, if we have 
$\xymatrix@C30pt{Q_i\ar[r]^{(b,a)\ \ }&\tau^-Q_k}$, then 
we have 
$\xymatrix@C30pt{Q_k\ar[r]^{(a,b)}&Q_i}$ in the AR quiver. 
Since $\rad Q_i\to Q_i$ is a minimal right almost split, we have 
$m_{\tau^-Q_k}(E)=m_{Q_k}(\rad Q_i)$.
Thus we have 
$$m_{Q_k}(\rad Q_i)=\dim((\rad Q_i/\rad^2Q_i)e_k)_{D_k}=\dim({}_iM_k)_{D_k}=-c_{ik}.$$
On the other hand, Proposition \ref{irr mor} implies that, for $\xymatrix@C30pt{Q_i\ar[r]^{(a,b)}&Q_j}$, we have 
$a=m_{Q_i}(\rad Q_j)$, $b=m_{Q_j}(E)$ and 
$$\dim{}_{k}(\rad_H(Q_i,Q_j)/\rad^2_H(Q_i,Q_j))
=a\dim_k(D_{Q_i})=b\dim_k(D_{Q_j}).$$
Since $H$ is acyclic, we have $\dim_k(D_{Q_i})=\dim_k(D_i)=c_i$ and $\dim_k(D_{Q_j})=\dim_k(D_j)=c_j$. Thus,  
Proposition \ref{irr mor} also shows that we have 
\[
\begin{array}{lll}
m_{Q_j}(E)= m_{Q_i}(\rad Q_j)c_i/c_j=\dim({}_jM_i)_{D_i}c_i/c_j= -c_{ji}c_i/c_j=-c_{ij}
\end{array}
\]
and the conclusion follows.
\end{proof}

Let $\Pi$ be the preprojective algebra of $H$ and we denote by $\{e_1,\ldots,e_n\}$ be a complete set of primitive orthogonal idempotents of $\Pi$. 
For $1\le i\le n$, we denote by $i^\pm:=\{1\le j\le n\ |\ {}_jM_i\neq0\ \textnormal{or}\ {}_iM_j\neq0\}$.

\begin{proposition}\label{proj resol of Pi}
Let $H$ be a finite dimensional hereditary $k$-algebra of non-Dynkin type and $\Pi$ the preprojective algebra of $H$. 
For any $1\leq i\leq n$, we have the following minimal projective resolution 
\[\xymatrix@C30pt{  
0\ar[r]&e_i\Pi\ar[r]&
{\displaystyle \bigoplus_{j\in i^\pm}}\ e_j\Pi^{-c_{ij}}
\ar[r]^{ }  &  e_i\Pi\ar[r]&D_i\ar[r]& 0.  }\]
\end{proposition}

Although this is quite standard fact, we give a sketch for the convenience of the reader.

\begin{proof}
Since the preprojective algebra $\Pi$ does not depend on the orientation of $(D_i,{}_iM_j)$, we choose an orientation such that ${}_iM_k=0$ for any $k$ and hence $Q_i$ is simple. 
Then, a minimal projective resolution of $Q_i=D_i$ is obtained as  $\mathbb{Z}$-graded modules and its $n$-degree is given by 
 applying $\tau^{-n}$ to the exact sequence of Proposition \ref{ar seq of H}, see e.g.\ \cite[Proposition 3.7]{MY}, \cite[Theorem 4.12]{GI}, \cite[Proposition 6.8]{S{o}}. 
 Since $(e_i\Pi)_n\cong\tau^{-n}(Q_i)$, we have an exact sequence 
\[\xymatrix@C30pt{  
0\ar[r]&e_i\Pi\ar[r]&
{\displaystyle \bigoplus_{j\in i^\pm}}\ e_j\Pi^{-c_{ij}}
\ar[r]^{ }  &  e_i\Pi(1)\ar[r]&D_i(1)\ar[r]& 0  }\]
as $\mathbb{Z}$-graded modules and this is a minimal projective resolution of $D_i$.
\end{proof}

\begin{lemma}\label{non-Dynkin I_i}
Let $H$ be a finite dimensional hereditary $k$-algebra of non-Dynkin type, $\Pi$ the preprojective algebra of $H$ and $I_i:=\Pi(1-e_i)\Pi$. 
\begin{enumerate}[\rm(a)] 
\item $I_i$ is a classical tilting $\Pi$-module and we have an equivalence $$-\Lotimes_\Pi I_i:\Db(\mod \Pi)\to \Db(\mod \Pi).$$
\item The action $R_i:=-\otimes_\Pi I_i:K_0(\proj\Pi)\to K_0(\proj\Pi),$  $[P]\mapsto[P\otimes_\Pi I_i]$
satisfies 
the following
\[
R_i(e_\ell\Pi) :=
\begin{cases}
[e_\ell\Pi] & \ell\neq i,\\
-[e_i\Pi]+\sum_{j=1}^nc_{ij}[e_j\Pi] & \ell=i.
\end{cases}
\]

\end{enumerate}
\end{lemma}

\begin{proof} 
By Proposition \ref{proj resol of Pi}, we have an exact sequence \[\xymatrix@C30pt{  
0\ar[r]&e_i\Pi\ar[r]&
{\displaystyle \bigoplus_{j\in i^\pm}}\ e_j\Pi^{-c_{ij}}
\ar[r]^{ }& e_i I_i\ar[r] &  0.  }\]
From this sequence, 
the results follow from by the same argument of \cite[section II.1]{BIRS} (some general situation is also discussed in \cite[section 6]{IR}).
\end{proof}

From now on, 
let $H:=T_D(M)$ be a finite dimensional hereditary $k$-algebra of Dynkin type and we denote by the corresponding matrix by $C$.  
Let $\Pi$ be the preprojective algebra of $H$ and $I_i:=\Pi(1-e_i)\Pi$, where $e_i$ the primitive idempotent of $\Pi$. 
We denote by $\langle I_1,\ldots,I_n\rangle$ the set of ideals of $\Pi$ which can be written as 
$$I_{i_1}I_{i_2}\cdots I_{i_\ell}$$ for some $\ell\geq0$ and $i_1,\ldots,i_\ell\in Q_0$. 
Then we have the following results.

\begin{theorem}\label{tau-weyl group}
There exists a bijection $W\to\langle I_1,\ldots,I_n\rangle$, which is given by $w\mapsto I_w =
I_{i_\ell}I_{i_{\ell-1}}\cdots I_{i_1}$ for any reduced 
expression $w=s_{i_1}\cdots s_{i_\ell}$.
Moreover the map gives an anti-isomorphism of posets 
$$W \longrightarrow \sttilt\Pi,$$
where we regard $W$ as a poset with weak order, and also $\sttilt\Pi$ as a poset via the bijection $\twosilt\Pi\simeq\sttilt\Pi$.
\end{theorem}

\begin{proof}
In simply-laced case, the result is \cite[Theorem 2.21]{M}. 
In non-simply-laced case, results of \cite{IR,BIRS} implies that the map is bijection. Using Proposition \ref{proj resol of Pi}, it is also easy to check that the same argument of \cite{M} shows that they are support $\tau$-tilting modules. 
\end{proof}

Moreover, we prepare the following set up. Let 
$\widetilde{H}$ be a finite dimensional hereditary $k$-algebra of affine type $\widetilde{C}$ whose restriction to 1 to $n$ columns and 1 to $n$ rows is $C$. 
Let $\wPi$ be the preprojective algebra of $\widetilde{H}$ such that $\wPi/\langle e_{n+1}\rangle\simeq\Pi$ and 
$\widetilde{W}$ the Coxeter group of $\widetilde{C}$. 
As same as Dynkin type, we define $\I_i:=\wPi(1-e_i)\wPi$  and $\I_w$ for $w\in\widetilde{W}.$ 
We naturally regard $W=\langle s_1,\ldots,s_n\rangle$ as a subgroup of $\widetilde{W}=\langle s_1,\ldots,s_n,s_{n+1}\rangle.$
Then we recall the following lemma.

\begin{proposition}\label{2-silt classification}
\begin{enumerate}[\rm(a)] 
\item We have 
$$\twosilt\Pi=\{\I_w\Lwotimes\Pi\ |\ w\in W\}.$$

\item For any  $w\in W$, 
we have $$[e_i\I_w\Lwotimes\Pi]=w\alpha^*_i$$
\end{enumerate}
\end{proposition}

\begin{proof}
(a) Theorem \ref{tau-weyl group} shows $\sttilt\Pi=\{I_w |\ w\in W\}.$ Then  \cite[Proposition 5.2]{M} (which works for arbitrary Dynkin type) 
implies that the correspondence $I_w\mapsto \I_w\Lwotimes\Pi$ gives a bijective map  $\sttilt\Pi \rightarrow \twosilt\Pi$ of \cite[Theorem 3.2]{AIR}.

(b) 
Since $\I_w$ is a classical tilting module, we have the following projective resolution 
\begin{eqnarray*}
\xymatrix@C20pt@R20pt{0\ar[r]& \widetilde{P}^1\ar[r]^(0.3){}  &\widetilde{P}^0  \ar[r]^{ } & e_i\I_w   \ar[r]^{ } &0,      }
\end{eqnarray*}
where $\widetilde{P}^0,\widetilde{P}^1\in\proj\wPi$. 
Therefore, we have 
$[e_i\I_w\Lwotimes\Pi]=[\widetilde{P}^0\otimes_{\wPi}\Pi]-[\widetilde{P}^1\otimes_{\wPi}\Pi]$. 
Since $\Pi\simeq\wPi/\langle e_{n+1}\rangle$, $[\widetilde{P}^j\otimes_{\wPi}\Pi]$ is given by the restriction of $[\widetilde{P}^j]$ to $K_0(\proj\Pi)_{\mathbb{R}}$ for $j\in\{0,1\}$.
Hence, Lemma \ref{non-Dynkin I_i} implies that 
we have $[e_i\I_w\Lwotimes\Pi]=w\alpha_i^*$ in $K_0(\proj\Pi)_{\mathbb{R}}$.
\end{proof}

Then we are ready to give a proof Theorem \ref{thm p.p. polytope} (a) and (b).

\begin{proof}[Proof of Theorem \ref{thm p.p. polytope} (a) and (b)]
(a) Recall that 
$C(T) := \{\sum_{i=1}^ja_i[T_i]\mid a_i\ge 0\}$
for $T=T_1\oplus\cdots\oplus T_n\in\twosilt \Pi$ with indecomposable $T_i$. Hence 
we have $C(\Pi)=\{\sum_{i=1}^ja_i[P_i]\mid a_i\ge 0\}=\{\sum_{i=1}^na_i\alpha_i^*\mid a_i\ge 0\}=D$. 
Then, by Proposition \ref{2-silt classification} (b), 
we have 
\[
\begin{array}{lll}
C(\I_w\Lwotimes\Pi)&=&\{\sum_{i=1}^na_i[e_i\I_w\Lwotimes\Pi]\mid a_i\ge 0\}\\
&=&\{\sum_{i=1}^na_iw\alpha_i^*\mid a_i\ge 0\}\\
&=&w\{\sum_{i=1}^na_i\alpha_i^*\mid a_i\ge 0\}\\
&=&wD.
\end{array}
\]
Therefore, Proposition \ref{2-silt classification} (a) implies $$\bigcup_{T\in\twosilt \Pi}C(T)=\bigcup_{w\in W}wD.$$
As a consequence, the $g$-fan $\Sigma(\Pi):=\{C(T)\mid T\in\twopsilt \Pi\}$ 
is the set of all cones of $wD$ and their faces, and we are done. 
By the same argument, we have $C_{\le1}(\I_w\Lwotimes\Pi)=wC_{\le1}(\Pi)$ and hence 
\[\P(\Pi) =\bigcup_{T\in\twosilt \Pi}C_{\le1}(T)= \bigcup_{w\in W}w C_{\le1}(\Pi).\] 

(b) 
Since $$\pi:=-\otimes_{\wPi}\Pi :K_0(\proj\wPi)\to K_0(\proj\Pi)$$ 
is compatible with the action of $W$,  Theorem \ref{characterize g-convex} and Proposition \ref{2-silt classification} implies that 
$\P(\Pi)$ is convex if and only if $\sum_{j\in i^\pm}|c_{ij}|\leq 2$. 
This is equivalent to saying that the Cartan matrix $C$ is type $A_n$ or $B_n$. 
\end{proof}

For a proof of Theorem \ref{thm p.p. polytope} (c), we use the following well-known facts.

\begin{lemma}\label{Worbit}
\begin{enumerate}[\rm(a)] 
\item All roots of a given length are conjugate under $W$.
\item  The sum of all simple roots is a short root.
\end{enumerate}
\end{lemma}

\begin{proof}
(a) is \cite[III.10.4 Lemma C]{Hu}. (b) is also well-known and it can be checked by case-by-case analysis.
\end{proof}

Then we give a proof of Theorem \ref{thm p.p. polytope} (c). 

\begin{proof}[Proof of Theorem \ref{thm p.p. polytope} (c)]
Recall that $\twosilt\Pi=\{\I_w\Lwotimes\Pi\ |\ w\in W\}$ from Proposition \ref{2-silt classification}. 
For simplicity, we let $P(w):=\I_w\Lwotimes\Pi$ and $P(w)_i:=e_i\I_w\Lwotimes\Pi$. 
By Proposition \ref{KY thm}, there exists $S(w)\in\twosmc\Pi$
and 
$\twosmc\Pi=\{S(w)\ |\ w\in W\}$ such that $S(w)=S(w)_1\oplus\cdots\oplus S(w)_n$ 
satisfying 
$$([P(w)_i],[S(w)_j]')=\delta_{ij}.$$

On the other hand, Proposition \ref{2-silt classification} implies that we have $[P(w)_i]=w\alpha_i^*.$
Therefore, we have 
$$( w\alpha_i^*,[S(w)_j]')=\delta_{ij}.$$
Then, since the bilinear form $(-,-)$ is non-degenerate, we have $[S(w)_j]'=w\alpha_j$ and hence
$$\{[S(w)_j]'\ |\ 1\leq j\leq n, w\in W\}=\{w\alpha_j\ |\ 1\leq j\leq n, w\in W\}.$$
Thus (i) follows. 
Moreover, by this argument, we have 
$$v_{S(w)}=\sum_{i=1}^n[S(w)_i]'=w\sum_{i=1}^n\alpha_i.$$
Then, Lemma \ref{Worbit} shows that $\{w(\sum_{i=1}^n\alpha_i)\ |\ w\in W\}$ coincide with all short  roots of $\Phi$ and we get (ii). 
Finally, (iii) immediately follows from (b) and Theorem \ref{reflexive polytope}.
\end{proof}

\begin{example}
\begin{enumerate}[\rm(a)] 
\item
Let $H=\left(\begin{smallmatrix}\mathbb{R}&0\\\mathbb{C}&\mathbb{C}\end{smallmatrix}\right)$. Then the corresponding Cartan matrix is $\left(\begin{smallmatrix}2&-1\\-2&2\end{smallmatrix}\right)$. Let $\Pi$ be the preprojective algebra of $H$.
We obtain the Hasse quiver of $\sttilt\Pi$ as follows.

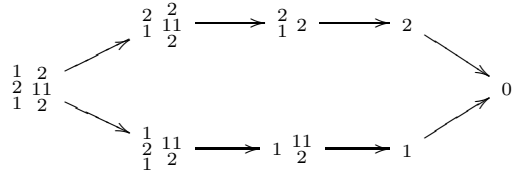
\[\xymatrix@R0em{&{\begin{smallmatrix}2\\ 1\end{smallmatrix}}{\begin{smallmatrix}\\ 2\\11\\2\end{smallmatrix}}\ar[r]^{}&{\begin{smallmatrix}2\\1\end{smallmatrix}}{\begin{smallmatrix}\\ 2\end{smallmatrix}}\ar[r]&{\begin{smallmatrix}\\ 2\end{smallmatrix}}\ar[rd]^{}&\\
{\begin{smallmatrix}1\\ 2\\1\end{smallmatrix}}{\begin{smallmatrix}\\ 2\\11\\2\end{smallmatrix}}\ar[ur]^{}\ar[dr]_{}&&&&{\begin{smallmatrix}\\ 0\end{smallmatrix}}\\
&{\begin{smallmatrix} 1\\2\\1\end{smallmatrix}}{\begin{smallmatrix}\\ 11\\2\end{smallmatrix}}\ar[r]_{}&{\begin{smallmatrix} 1\end{smallmatrix}}{\begin{smallmatrix}11\\ 2\end{smallmatrix}}\ar[r]&{\begin{smallmatrix}\\ 1\end{smallmatrix}}\ar[ru]_{}&}\]

On the other hand, the corresponding 2-simple minded collections are shown as follows.

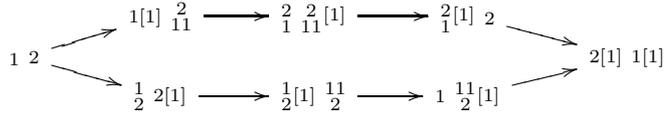
\[\xymatrix@R0em{&{\begin{smallmatrix}\\ 1[1]\end{smallmatrix}}{\begin{smallmatrix}2\\ 11\end{smallmatrix}}\ar[r]^{}&{\begin{smallmatrix}\\ 2\\1\end{smallmatrix}}{\begin{smallmatrix}\\2\\11\end{smallmatrix}\hspace{-0.1cm}\begin{smallmatrix}[1]\end{smallmatrix}}\ar[r]&{\begin{smallmatrix}\\2\\1\end{smallmatrix}\hspace{-0.1cm}\begin{smallmatrix}[1]\end{smallmatrix}}{\begin{smallmatrix}\\ 2\end{smallmatrix}}\ar[rd]^{}&\\
{\begin{smallmatrix}\\ 1\end{smallmatrix}}{\begin{smallmatrix}2\\ \end{smallmatrix}}\ar[ur]^{}\ar[dr]_{}&&&&
{\begin{smallmatrix}2[1]\end{smallmatrix}}{\begin{smallmatrix}1[1]\end{smallmatrix}}\\
&{\begin{smallmatrix} 1\\2\end{smallmatrix}}{\begin{smallmatrix} 2[1]\end{smallmatrix}}\ar[r]_{}&{\begin{smallmatrix} 1\\2\end{smallmatrix}}\hspace{-0.1cm}\begin{smallmatrix}[1]\end{smallmatrix}{\begin{smallmatrix} 11\\2\end{smallmatrix}}\ar[r]&{\begin{smallmatrix} 1\end{smallmatrix}}{\begin{smallmatrix} 11\\2\end{smallmatrix}}\hspace{-0.1cm}\begin{smallmatrix}[1]\end{smallmatrix}\ar[ru]_{}&}\]

Then, by taking $[-]'$, we have the set of roots as follows

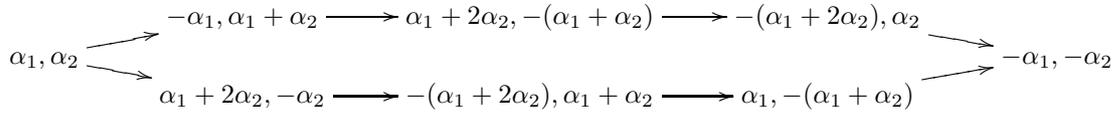
\[\xymatrix@R0em{&-\alpha_1,\alpha_1+\alpha_2\ar[r]^{}&\alpha_1+2\alpha_2,-(\alpha_1+\alpha_2)\ar[r]&-(\alpha_1+2\alpha_2),\alpha_2\ar[rd]^{}&\\
\alpha_1,\alpha_2\ar[ur]^{}\ar[dr]_{}&&&&-\alpha_1,-\alpha_2\\
&\alpha_1+2\alpha_2,-\alpha_2\ar[r]_{}&-(\alpha_1+2\alpha_2),\alpha_1+\alpha_2\ar[r]&\alpha_1,-(\alpha_1+\alpha_2)\ar[ru]_{}&}\]
and hence we have 
$$\{v_X\ |\ X\in\twosmc\Pi\}=\{\pm(\alpha_2,\alpha_1+\alpha_2)\}.$$

Thus $c$-polytope $\P^\c(\Pi)$ is illustrated as a dotted line in the left picture below and the $g$-polytope  $\P(\Pi)$ is illustrated in the right picture below.  
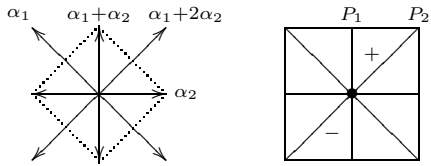
\[
\begin{xy}
0;<5pt,0pt>:<0pt,5pt>::
(  0,0) ="(0,0)",
( 0,5) ="(0,15)",
( 0,6) *+{{\scriptstyle \alpha_1+\alpha_2}},
( 5,0) ="(15,0)",
( 6.5,0) *+{{\scriptstyle \alpha_2}},
(5,5) ="(15,15)",
(6.5,6) *+{{\scriptstyle\alpha_1+2\alpha_2}},
( 0,-5) ="(0,-15)",
(  -5,0) ="(-15,0)",
(-5,-5) ="(-15,-15)",
(5,-5) ="(15,-15)",
(-5,5) ="(-15,15)",
(-6,6) *{{\scriptstyle\alpha_1}},
\ar@{<->} "(15,0)";"(-15,0)"
\ar@{<->} "(0,15)";"(0,-15)"
\ar@{<->} "(15,15)";"(-15,-15)"
\ar@{<->} "(-15,15)";"(15,-15)"
\ar@{.} "(15,0)";"(0,15)"
\ar@{.} "(0,15)";"(-15,0)"
\ar@{.} "(-15,0)";"(0,-15)"
\ar@{.} "(0,-15)";"(15,0)"
\end{xy}
\ \ \ \ \ \ \begin{xy}
0;<0pt,5pt>:<-5pt,5pt>::
(  0,0) *{\bullet}="(0,0)",
(4.5,-1.5) *{{\scriptstyle +}},
(-4.5,1.5) *{{\scriptstyle -}},
(11,-5) *{{\scriptstyle P_2}},
( 6,0) *{{\scriptstyle P_1}},
( 0,5) ="(0,15)",
( 5,0) ="(15,0)",
(-10,5) ="(-30,15)",
( 0,-5) ="(0,-15)",
(  -5,0) ="(-15,0)",
(10,-5) ="(30,-15)",
(5,-5) ="(15,-15)",
(-5,5) ="(-15,15)",
\ar@{-} "(15,0)";"(-15,0)"
\ar@{-} "(0,15)";"(0,-15)"
\ar@{-} "(-30,15)";"(30,-15)"
\ar@{-} "(-15,15)";"(15,-15)"
\ar@{-} "(15,0)";"(0,15)"
\ar@{-} "(0,15)";"(-30,15)"
\ar@{-} "(-30,15)";"(-15,0)"
\ar@{-} "(-15,0)";"(0,-15)"
\ar@{-} "(0,-15)";"(30,-15)"
\ar@{-} "(30,-15)";"(15,0)"
\end{xy}\]

\item
Let $H=\left(\begin{smallmatrix}\mathbb{R}&0&0\\
\mathbb{R}&\mathbb{R}&0\\
\mathbb{C}&\mathbb{C}&\mathbb{C}\end{smallmatrix}\right)$. Then the corresponding Cartan matrix is $\left(\begin{smallmatrix}2&-1&0\\
-1&2&-1\\
0&-2&2\end{smallmatrix}\right)$. Let $\Pi$ be the preprojective algebra of $H$.

The root system of type $B_3$ is illustrated in the left picture below and the $g$-polytope $\P(\Pi(B_3))=(\P^\c(\Pi(B_3)))^*$ is illustrated in the right picture below.
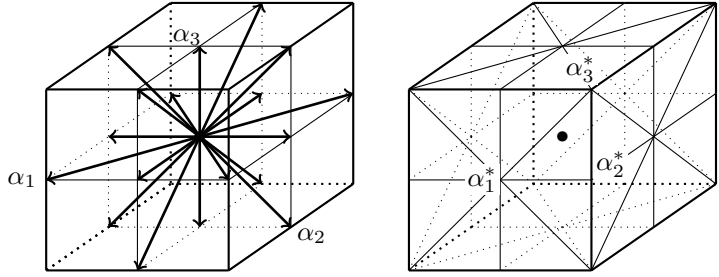
\[  \begin{tikzpicture}[baseline=0mm, scale=1]
        \coordinate(0) at(0:0); 
        \node(x) at(215:1) {}; 
        \node(-x) at($-1*(x)$) {}; 
        \node(y) at(0:1.2) {}; 
        \node(-y) at($-1*(y)$) {}; 
        \node(z) at(90:1.2) {}; 
        \node(-z) at($-1*(z)$) {}; 

        \coordinate(1) at($1*(x) + 0*(y) + 0*(z)$) ; 
        \coordinate(2) at($0*(x) + 1*(y) + 0*(z)$) ; 
        \coordinate(3) at($0*(x) + 0*(y) + 0*(z)$) ; 
        \coordinate(4) at($1*(x) + 1*(y) + 0*(z)$) ; 
        \coordinate(5) at($-1*(x) + 1*(y) + 0*(z)$) ; 
        \coordinate(6) at($-1*(x) + 0*(y) + 0*(z)$) ; 
        \coordinate(7) at($0*(x) + -1*(y) + 0*(z)$) ; 
        \coordinate(8) at($1*(x) + -1*(y) + 0*(z)$) ; 
        \node[left] at($(8)+($0*(x) + 0*(y) + 0*(z)$)$) {$\alpha_1$}; 
        \coordinate(9) at($-1*(x) + -1*(y) + 0*(z)$) ; 

        \coordinate(11) at($1*(x) + 0*(y) + 1*(z)$) ; 
        \coordinate(12) at($0*(x) + 1*(y) + 1*(z)$) ; 
        \coordinate(13) at($0*(x) + 0*(y) + 1*(z)$) ; 
        \node[above] at($(13)+($0.2*(x) + 0*(y) + 0*(z)$)$) {$\alpha_3$}; 
        \coordinate(14) at($1*(x) + 1*(y) + 1*(z)$) ; 
        \coordinate(15) at($-1*(x) + 1*(y) + 1*(z)$) ; 
        \coordinate(16) at($-1*(x) + 0*(y) + 1*(z)$) ; 
        \coordinate(17) at($0*(x) + -1*(y) + 1*(z)$) ; 
        \coordinate(18) at($1*(x) + -1*(y) + 1*(z)$) ; 
        \coordinate(19) at($-1*(x) + -1*(y) + 1*(z)$) ; 

        \coordinate(-11) at($1*(x) + 0*(y) + -1*(z)$) ; 
        \coordinate(-12) at($0*(x) + 1*(y) + -1*(z)$) ; 
        \node[right] at($(-12)+($0.2*(x) + 0.1*(y) + 0*(z)$)$) {$\alpha_2$}; 
        \coordinate(-13) at($0*(x) + 0*(y) + -1*(z)$) ; 
        \coordinate(-14) at($1*(x) + 1*(y) + -1*(z)$) ; 
        \coordinate(-15) at($-1*(x) + 1*(y) + -1*(z)$) ; 
        \coordinate(-16) at($-1*(x) + 0*(y) + -1*(z)$) ; 
        \coordinate(-17) at($0*(x) + -1*(y) + -1*(z)$) ; 
        \coordinate(-18) at($1*(x) + -1*(y) + -1*(z)$) ; 
        \coordinate(-19) at($-1*(x) + -1*(y) + -1*(z)$) ;

        \draw[] (4)--(5);
        \draw[dotted] (5)--(9);
        \draw[dotted] (9)--(8); 
        \draw[] (8)--(4);
        \draw[thick] (14)--(15); 
        \draw[thick] (15)--(19);
        \draw[thick] (19)--(18);
        \draw[thick] (18)--(14);
        \draw[] (-14)--(-15); 
        \draw[thick,dotted] (-15)--(-19); 
        \draw[thick,dotted] (-19)--(-18); 
        \draw[thick] (-18)--(-14); 

        \draw[] (11)--(16);
        \draw[dotted] (16)--(-16); 
        \draw[dotted] (-16)--(-11); 
        \draw[] (-11)--(11);
        \draw[thick] (15)--(-15);
        \draw[thick] (-15)--(-14);
        \draw[thick] (-14)--(14);
        \draw[thick,dotted] (19)--(-19); 
        \draw[thick,dotted] (-19)--(-18); 
        \draw[thick] (-18)--(18); 

        \draw[] (-12)--(12);
        \draw[] (12)--(17); 
        \draw[dotted] (17)--(-17); 
        \draw[dotted] (-17)--(-12);

        \begin{scope}[line width=1pt, ->]
        \draw[] (0)--(1);
        \draw[] (0)--(2);
        \draw[] (0)--(4);
        \draw[] (0)--(5);
        \draw[] (0)--(6);
        \draw[] (0)--(7);
        \draw[] (0)--(8);
        \draw[] (0)--(9);

        \draw[] (0)--(11);
        \draw[] (0)--(12);
        \draw[] (0)--(13);
        \draw[] (0)--(16);
        \draw[] (0)--(17);

        \draw[] (0)--(-11);
        \draw[] (0)--(-12);
        \draw[] (0)--(-13);
        \draw[] (0)--(-16);
        \draw[] (0)--(-17);
        \end{scope}
        \fill (0) circle (2pt);
    \end{tikzpicture} \ \ \ \ \ \ 
    \begin{tikzpicture}[baseline=0mm, scale=1]
        \node(0) at(0:0) {$\bullet$}; 
        \node(x') at(215:1) {}; 
        \node(-x') at($-1*(x)$) {}; 
        \node(y') at(0:1.2) {}; 
        \node(-y') at($-1*(y)$) {}; 
        \node(z') at(90:1.2) {}; 
        \node(-z') at($-1*(z)$) {}; 

        \node(x) at($1*(x') + 0*(y') + 0*(z')$) {}; 
        \node(-x) at($-1*(x)$) {}; 
        \node(y) at($1*(x') + 1*(y') + 0*(z')$) {}; 
        \node(-y) at($-1*(y)$) {}; 
        \node(z) at($1*(x') + 1*(y') + 1*(z')$) {}; 
        \node(-z) at($-1*(z)$) {};

        \coordinate(1) at($1*(x) + 0*(y) + 0*(z)$) ; 
        \coordinate(2) at($0*(x) + 1*(y) + 0*(z)$) ; 
        \coordinate(3) at($0*(x) + 0*(y) + 1*(z)$) ;
        \coordinate(4) at($-1*(x) + 0*(y) + 0*(z)$) ; 
        \coordinate(5) at($0*(x) + -1*(y) + 0*(z)$) ; 
        \coordinate(6) at($0*(x) + 0*(y) + -1*(z)$) ; 
    
        \coordinate(7) at($0*(x) + 2*(y) + -1*(z)$) ; 
        \coordinate(8) at($1*(x) + -1*(y) + 1*(z)$) ; 
        \coordinate(9) at($-1*(x) + 1*(y) + 0*(z)$) ; 
    
        \coordinate(10) at($1*(x) + 1*(y) + -1*(z)$) ;
        \coordinate(11) at($2*(x) + -2*(y) + 1*(z)$) ;
        \coordinate(12) at($0*(x) + -1*(y) + 1*(z)$) ;
    
        \coordinate(13) at($-1*(x) + 0*(y) + 1*(z)$) ;
        \coordinate(14) at($0*(x) + 1*(y) + -1*(z)$) ;
        \coordinate(15) at($2*(x) + 0*(y) + -1*(z)$) ;
    
        \coordinate(16) at($-1*(x) + 2*(y) + -1*(z)$) ;
        \coordinate(17) at($2*(x) + -1*(y) + 0*(z)$) ;
        \coordinate(18) at($-2*(x) + 0*(y) + 1*(z)$) ;
    
        \coordinate(19) at($-2*(x) + 2*(y) + -1*(z)$) ;
        \coordinate(20) at($1*(x) + -2*(y) + 1*(z)$) ;
        \coordinate(21) at($1*(x) + -1*(y) + 0*(z)$) ;
    
        \coordinate(22) at($-2*(x) + 1*(y) + 0*(z)$) ;
        \coordinate(23) at($1*(x) + 0*(y) + -1*(z)$) ;
        \coordinate(24) at($0*(x) + -2*(y) + 1*(z)$) ;
    
        \coordinate(25) at($-1*(x) + -1*(y) + 1*(z)$) ;
        \coordinate(26) at($-1*(x) + 1*(y) + -1*(z)$) ;
        
        \draw[]  (1)--(2) ;
        \draw[]  (1)--(3) ;
        \draw[thick]  (2)--(3) ;
        \draw[dotted]  (4)--(5) ;
        \draw[dotted]  (4)--(6) ;
        \draw[dotted, thick]  (5)--(6) ;
        \draw[]  (7)--(1) ;
        \draw[thick]  (7)--(2) ;
        \draw[]  (8)--(1) ;
        \draw[thick]  (8)--(3) ;
        \draw[]  (9)--(2) ;
        \draw[]  (9)--(3) ;
        \draw[thick]  (10)--(7) ;
        \draw[]  (10)--(1) ;
        \draw[]  (9)--(7) ;
        \draw[thick]  (11)--(8) ;
        \draw[]  (11)--(1) ;
        \draw[]  (12)--(8) ;
        \draw[]  (12)--(3) ;
        \draw[]  (13)--(9) ;
        \draw[thick]  (13)--(3) ;
        \draw[dotted]  (10)--(14) ;
        \draw[dotted]  (7)--(14) ;
        \draw[thick]  (15)--(10) ;
        \draw[]  (15)--(1) ;
        \draw[]  (16)--(9) ;
        \draw[thick]  (16)--(7) ;
        \draw[]  (12)--(11) ;
        \draw[thick]  (17)--(11) ;
        \draw[]  (17)--(1) ;
        \draw[]  (13)--(12) ;
        \draw[thick]  (18)--(13) ;
        \draw[]  (18)--(9) ;
        \draw[dotted]  (15)--(14) ;
        \draw[dotted]  (16)--(14) ;
        \draw[thick]  (17)--(15) ;
        \draw[thick]  (19)--(16) ;
        \draw[]  (19)--(9) ;
        \draw[]  (20)--(12) ;
        \draw[thick]  (20)--(11) ;
        \draw[dotted]  (17)--(21) ;
        \draw[dotted]  (11)--(21) ;
        \draw[]  (18)--(12) ;
        \draw[thick]  (22)--(18) ;
        \draw[]  (22)--(9) ;
        \draw[dotted, thick]  (15)--(23) ;
        \draw[dotted]  (23)--(14) ;
        \draw[dotted]  (19)--(14) ;
        \draw[dotted]  (15)--(21) ;
        \draw[thick]  (22)--(19) ;
        \draw[thick]  (24)--(20) ;
        \draw[]  (24)--(12) ;
        \draw[dotted]  (20)--(21) ;
        \draw[thick]  (25)--(18) ;
        \draw[]  (25)--(12) ;
        \draw[dotted]  (18)--(4) ;
        \draw[dotted]  (22)--(4) ;
        \draw[dotted]  (14)--(6) ;
        \draw[dotted, thick]  (23)--(6) ;
        \draw[dotted]  (23)--(21) ;
        \draw[dotted, thick]  (19)--(26) ;
        \draw[dotted]  (26)--(14) ;
        \draw[dotted]  (19)--(4) ;
        \draw[dotted]  (24)--(21) ;
        \draw[thick]  (25)--(24) ;
        \draw[dotted]  (25)--(4) ;
        \draw[dotted, thick]  (26)--(6) ;
        \draw[dotted]  (21)--(6) ;
        \draw[dotted]  (26)--(4) ;
        \draw[dotted, thick]  (24)--(5) ;
        \draw[dotted]  (21)--(5) ;
        \draw[dotted]  (24)--(4) ;

        \node[left,fill=white, inner sep=1pt] at($(1)+($0*(x) + 0*(y) + 0*(z)$)$) {$\alpha_1^{\ast}$};
        \node[right,fill=white, inner sep=1pt] at($(2)+($0*(x) + -0.15*(y) + 0.2*(z)$)$) {$\alpha_2^{\ast}$};
        \node[above,fill=white, inner sep=1pt] at($(3)+($0.15*(x) + -0.15*(y) + 0.1*(z)$)$) {$\alpha_3^{\ast}$};
    
    \end{tikzpicture} 
        \]

\end{enumerate}
\end{example}

\subsection{Generalized preprojective algebras}

In this subsection, we study $g$-polytopes of generalized 
preprojective algebras introduced by \cite{GLS}.
We show that the $g$-polytope is convex if and only if 
the Cartan matrices is type $A_n$ or $B_n$.

First we introduce  the  notion  of  generalized  preprojective algebras associated with symmetrizable generalized Cartan matrices \cite{GLS}.

Let $C = (c_{ij}) \in M_n(\mathbb{Z})$ be a symmetrizable generalized 
Cartan matrix. 
We denote by
$g_{ij} := | \operatorname{gcd}(c_{ij} , c_{ji})|$ and $f_{ij} := |c_{ij} |/g_{ij}.$

An {\it orientation} of $C$ is a subset $\Omega$ of
$\{1, 2,\cdots , n\} \times \{1, 2, \cdots , n\} $ such that the followings hold:

\begin{enumerate} 
\item[(i)] $\{(i, j), (j, i)\} \cap \Omega
 \neq \emptyset $ if and only if $c_{ij} < 0$;

\item[(ii)] For each sequence $((i_1, i_2), (i_2, i_3), \cdots , (i_t, i_{t+1}))$ with $t \geq 1$ and $(i_s, i_{s+1}) \in \Omega$
 for
all $1 \leq s \leq t$, we have $i_1 \neq i_{t+1}$.
\end{enumerate}

The \emph{opposite orientation} of an orientation $\Omega$ is defined as
$\Omega^* := \{ (j,i) \mid (i,j) \in \Omega \}$.
Let
$\overline{\Omega} := \Omega \cup \Omega^*$ 
and we define
$$\overline{\Omega}(-,i) := \{ j \in Q_0 \mid (j,i) \in \overline{\Omega} \}$$

For an orientation $\Omega$ of $C$, define the quiver
$Q := Q(C,\Omega) := (Q_0,Q_1)$ with the
set of vertices $Q_0 := \{ 1,\ldots, n\}$ and 
with the set of arrows 
\[
Q_1 := \{ \alpha_{ij}^{(g)}: j \to i\ |\ (i,j) \in \Omega, 1 \leq g \leq g_{ij} \}
\cup \{ \epsilon_i: i \to i \ |\ 1 \leq i \leq n \}.
\]
We call $Q$ \emph{a quiver of type $C$}. 
Let $Q^\circ:=Q^\circ(C,\Omega)$ be the quiver obtained from $Q$ by deleting all loops $\epsilon_i$.

Then we define the generalized preprojetive algebra  associated to $C$ as follows. 

\begin{definition}
Let $C$ be a symmetrizable Cartan matrix with a symmetrizer $D$. 

For $(i,j) \in \overline{\Omega}$, define
\[
\textnormal{sgn}(i,j) :=
\begin{cases}
1 & \textnormal{if $(i,j) \in \Omega$},\\
-1 & \textnormal{if $(i,j) \in \Omega^*$}.
\end{cases}
\]

For $Q = Q(C,\Omega)$ and a symmetrizer $D = \textnormal{diag}(c_1,\ldots,c_n)$ of $C$, 
we define an algebra
\[
\Pi= \Pi(C,D,\Omega) := K\overline{Q}/\overline{I}\ \ 
\] 
as follows.
The \emph{double quiver} $\overline{Q} = \overline{Q}(C)$ is obtained from 
$Q$ by adding a new arrow 
$\alpha_{ji}^{(g)}: j \to i$
for each arrow $\alpha_{ij}^{(g)}: i \to j$ of $Q^\circ$. 

The ideal $\overline{I}$ of the path algebra
$K\overline{Q}$ is defined by the following
relations:

\begin{itemize}
\item[(P1)]
For each $i\in Q_0$, we have the \emph{nilpotency relation}
\[
\epsilon_i^{c_i} = 0.
\]

\item[(P2)]
For each $(i,j) \in \overline{\Omega}$ and each $1 \leq g \leq g_{ij}$, we have
the \emph{commutativity relation}
\[
\alpha_{ij}^{(g)}\epsilon_i^{f_{ij}} = \epsilon_j^{f_{ji}}\alpha_{ij}^{(g)}.
\]

\item[(P3)]
For each $i\in Q_0$, we have the \emph{mesh relation}
\[
\sum_{j\in \overline{\Omega}(-,i)}
\sum_{g=1}^{g_{ji}} \sum_{f=0}^{f_{ij}-1} 
\textnormal{sgn}(i,j)
\epsilon_i^{f_{ij}-1-f}
\alpha_{ji}^{(g)}\alpha_{ij}^{(g)}\epsilon_i^f  
 = 0.
\]
\end{itemize}

\begin{remark}
Our definition of the preprojective algebra is slightly different from the original one given by \cite{GLS}  (which we denote by $\Pi^{GLS}$), but two definitions are essentially the same objects. 
More precisely, we have $\Pi(C,D) =\Pi^{GLS}({}^tC,D)$, where ${}^tC$ is the transposed Cartan matrix.  
\end{remark}

We remark that $\Pi$ does not depend on the orientation $\Omega$ of $C$, so that we can write $\Pi=\Pi(C,D)$.
\end{definition}

For the one dimensional simple $\Pi$-modules $S_1,\ldots,S_n$, we define the 
$\iota:K_0(\mod \Pi)_\R\simeq V$ by $\iota([S_i])=\alpha_i$  
and, for $X=X_1\oplus\cdots\oplus 
X_n\in\smc\Pi$, let
\[v_X:=\iota(\sum_{i=1}^n[X_i])\in V.\]

Then we have the following analogous result of Theorem \ref{thm p.p. polytope}, where the only difference is that $[S]'$ in Theorem \ref{thm p.p. polytope}(c)(i) is replaced by $[S]$.

\begin{theorem}\label{thm p.p. polytope2}
Let $C$ be a symmetrizable Cartan matrix of Dynkin type with a symmetrizer $D$ and $\Pi=\Pi(C,D)$. 
\begin{enumerate}[\rm(a)] 
\item $\Sigma(\Pi)$ is the Coxeter fan and we have $$\P(\Pi)=\bigcup_{w\in W}w C_{\le1}(\Pi).$$ 

\item $\P(\Pi)$ is convex if and only if the Cartan matrix $C$ is type $A_n$ or 
$B_n$.
\item
\begin{enumerate} 
\item[\rm(i)] 
We have 
$$\{[S]\ |\ S\in\indtwosmc\Pi\}=\Phi.$$
\item[\rm(ii)] Let $\Phi_{\textnormal{short}}$ be the set of short roots of $\Phi$. 
We have 
$$\{v_X\ |\ X\in\twosmc\Pi\}=\Phi_{\textnormal{short}}\quad\textnormal{and}\quad \P^\c(\Pi)=\conv(\Phi_{\textnormal{short}}).$$
\item[\rm(iii)] If the Cartan matrix $C$ is type $A_n$ or 
$B_n$, then $\P(\Pi)$ is the dual polytope of $\P^\c(\Pi)=\conv(\Phi_{\textnormal{short}})$.
\end{enumerate}
\end{enumerate}
\end{theorem}

The proof of Theorem \ref{thm p.p. polytope2} is completely parallel to that of Theorem \ref{thm p.p. polytope}. Below we only highlight the differences.  
Instead of Proposition \ref{proj resol of Pi}, we use the following result. 

\begin{proposition}\label{proj resol of PiFG}\cite{FG}
Let $C$ be a symmetrizable Cartan matrix of non-Dynkin type with a symmetrizer $D$ and $\Pi=\Pi(C,D)$. 
For any $1\leq i\leq n$, we have the following minimal projective resolution\footnote{Because of our definition, we use the index $c_{ij}$ instead of $c_{ji}$ used in \cite{FG}} 
\[\xymatrix@C30pt{  
0\ar[r]&e_i\Pi\ar[r]&
{\displaystyle \bigoplus_{j\in i^\pm}}\ e_j\Pi^{-c_{ij}}
\ar[r]^{ }  &  e_i\Pi\ar[r]&\hat{S}_i\ar[r]& 0,  }\]where 
$\hat{S}_i$ denotes by the generalized simple module. In particular, 
we have an exact sequence \[\xymatrix@C30pt{  
0\ar[r]&e_i\Pi\ar[r]&
{\displaystyle \bigoplus_{j\in i^\pm}}\ e_j\Pi^{-c_{ij}}
\ar[r]^{ }& I_i\ar[r] &  0.  }\]
\end{proposition}

Then, by combining with results of \cite{M,FG},
Theorem \ref{thm p.p. polytope2} (a) and (b) follows from the same argument of subsection \ref{CPPA}. 
To show Theorem \ref{thm p.p. polytope2} (c), it is enough to show the following result.

\begin{lemma}\label{end=k}
For any $S\in\indtwosmc\Pi$, we have 
$$\End_{\Db(\mod \Pi)}(S)\cong k.$$
In particular, $$[S]=[S]'.$$
\end{lemma}

For a proof, we recall results of \cite{KM} and 
prepare the following setting. Let $\widetilde{C}$ be a affine Cartan matrix 
whose restriction to 1 to $n$ columns and 1 to $n$ rows is $C$. 
Let $\wPi$ be the complete preprojective algebra of $\widetilde{C}$ with a symmetrizer. Then we have the following result. 

\begin{proposition}\label{end of non-dynkin silting}
For any $T\in\twosilt\wPi$, 
we have $$\End_{\Db(\mod \wPi)}(T)\cong\wPi.$$
Moreover, for any brick $S$, 
we have $\End_{\wPi}(S)\cong k.$
\end{proposition}

\begin{proof}
Recall from \cite{KM} that any two-term silting complex is tilting complex and it belongs to one of the two  connected components of mutation  (which also works for non-simply laced cases by \cite{FG}).
One component contains $\wPi$ and the other contains $\wPi[1]$.
Then by \cite{FG,BIRS}, the endmorphism algebra of them is isomorphic to $\wPi$.
Thus, for $T_i\in\twopsilt^1\wPi$,
we have $\End_{\Db(\mod \wPi)}(T_i)\cong e_i\wPi e_i$
for an idempotent $e_i$ of $\wPi$.
Hence $\End_{\Db(\mod \wPi)}(T_i)/\rad(\End_{\Db(\mod \wPi)}(T_i))\simeq k$.
On the other hand, for any $S\in\indtwosmc\wPi$, there exists  $T_i\in\twopsilt^1\wPi$ such that $$\End_{\Db(\mod \wPi)}(T_i)/\rad(\End_{\Db(\mod \wPi)}(T_i))\cong\End_{\Db(\mod \wPi)}(S).$$

Therefore, we get $\End_{\Db(\mod \wPi)}(S)\cong k.$
\end{proof}

We give a proof of Lemma \ref{end=k}. 

\begin{proof}[Proof of Lemma \ref{end=k}]
Because $\Pi=\wPi/\langle e_{n+1}\rangle$, 
$\mod\Pi$ is a full subcategory of $\mod\wPi$. 
Therefore, Proposition \ref{end of non-dynkin silting} implies that we conclude
$\End_{\Pi}(S)\cong k$ for any brick $S$ of $\mod\Pi$.
\end{proof}

Using these preparations, one can prove Theorem \ref{thm p.p. polytope2} by an argument completely parallel to the proof of Theorem \ref{thm p.p. polytope}.

\subsection{$h$-vectors of the simplicial complexes of preprojective algebras}
In this subsection, we discuss $h$-vectors of the simplicial complexes of preprojective algebras. 
Let $C\in M_n(\mathbb{Z})$ be a symmetrizable Cartan matrix, and $W=W(C)$ the corresponding Weyl group.

\begin{definition}
For  $w\in W$, we define 
 $$\textnormal{Des}(w)=\{s_i\ |\ \ell(w)>\ell(s_iw)\}\ \ \textnormal{and}\ \ \textnormal{des}(w) = |\textnormal{Des}(w)|.$$  
Moreover, we define 
$$E(W,j):=|\{w\in W\ |\ \textnormal{des}(w)=j\}|$$
and it is called \emph{$W$-Eulerian numbers}.
\end{definition}

Then we get the following conclusion. 

\begin{theorem}\label{h-vectors of ppalgs}
Let $C\in M_n(\mathbb{Z})$ be a symmetrizable Cartan matrix, and $\Pi:=\Pi(C)$ the (classical or generalized) preprojective algebra. 
Then the $h$-vectors of $\Delta(\Pi)$ is given by $W$-Eulerian numbers, that is, for each $0\le j\le n$, we have 
$$h_j=E(W,j).$$
\end{theorem}

\begin{proof} 
By Theorem \ref{h_j}, $h_j$ is the cardinality of the set of all $T\in\twosilt\Pi$ such that there exist precisely $j$ arrows starting at $T$ in $\Hasse(\twosilt\Pi)$. 
By using $h_j=h_{n-j}$ and the anti-isomorphism of posets in Theorem \ref{tau-weyl group} and its variation for generalized preprojective algebras \cite{FG}, this equals
$|\{w\in W\ |\ \textnormal{des}(w)=j\}|=E(W,j)$.
\end{proof}

\begin{table}[htbp]
\caption{The Eulerian numbers of type $A_n$.}\begin{center}
      \begin{tabular}{c|c ccccccc} \hline
        $n$ $\backslash$ $k$ &0&1 &2&3&4&5&6 
        \\ \hline
        1 & 1  & 1 \\
        2 & 1  & 4&1 \\
        3 & 1 &11&11 & 1 \\
        4 & 1&26&66&26  & 1\\ 
        5 & 1&57&302&302&57  & 1 \\ 
        6 & 1&120&1191&2416&1191&120&  
        1 \\ 
      \end{tabular}
  \end{center}
\end{table}%

\begin{table}[htbp]
  \begin{center}
\caption{The Eulerian numbers of type $D_n$.}
      \begin{tabular}{c|c cccccccc} \hline
        $n$ $\backslash$ $k$ &0&1 &2&3&4&5&6& 7&8
        \\ \hline
        4 &  1&44&102&44  & 1\\ 
        5 & 1&157&802&802&157 & 1 \\ 
        6 & 1&530&5551&10876&5551&530&  1 \\
        7 & 1&1731&35121&124427&124427&35121& 1731&1 \\
        8 & 1&5528&208732&1265704&2201030&1265704& 208732&5528&1 \\
      \end{tabular}
  \end{center}
\end{table}%


\section{Brauer graph algebras and Root polytopes}\label{section 9}

In this section, we study $g$-polytopes of Brauer graph algebras. We show that every $g$-finite Brauer graph algebra is $g$-convex and the $g$-polytope is isomorphic to the root polytope of type $A_n$ or $C_n$ (Theorem \ref{BTA-BOA-Phi}), where $n$ is the number of edges of the associated Brauer graph.

\subsection{Definition and main result}
    A Brauer graph algebra is defined from a combinatorial object called Brauer graph. We refer to a survey paper \cite{Sc} for the background and basic properties of Brauer graph algebras. 

\begin{definition}
A \textit{ribbon graph} is a triple $\Gamma=(H,\sigma,\overline{(\ )})$, where $H$ is a non-empty finite set, $\sigma$ is a permutation on $H$, and $\overline{(\ )} \colon H\to H$ is a fixed-point free involution.

\begin{itemize}
    \item Each element of $H$ is called \textit{half-edge} of $\Gamma$. 
    \item Let $H\to H/\overline{(\ )}$ be a canonical surjection. Each element of $H/\overline{(\ )}$ is called \emph{edge} of $\Gamma$. We denote by $[h]$ the edge containing $h\in H$. 
    Then, we have $[h]=[\overline{h}]$.
    \item Let $s \colon H\to H/\langle\sigma\rangle$ be a canonical surjection. 
    Each element of $H/\langle \sigma \rangle$ is called \emph{vertex} of $\Gamma$. 
     
    \item For each vertex $v$ of $\Gamma$, the $\sigma$-orbit $(h,\sigma(h),\ldots, \sigma^{\ell-1}(h))$ incident to $v$ is called the \emph{cyclic ordering} around $v$, where $\ell$ is the cardinality of this orbit.    
    \end{itemize}
A \textit{Brauer graph} is a ribbon graph $\Gamma$ equipped with a \emph{multiplicity function}, which assigns a positive integer $\mathfrak{m}(v)>0$, called \emph{multiplicity}, for every vertex $v$ of $\Gamma$. 
\end{definition}

In order to describe a given Brauer/ribbon graph $\Gamma=(H,\sigma, \overline{(\ )})$, we usually use its geometric realization, that is, a graph obtained from $\Gamma$ by gluing half-edges $h$ and $\overline{h}$ together to form a line whose endpoints are $s(h)$ and $s(\overline{h})$. When we describe the cyclic ordering $(h,\sigma(h), \ldots, \sigma^{\ell-1}(h))$ of half-edges around a vertex $v$, 
we draw them in the plane locally so that the half-edges appear around $v$ in this order counterclockwisely. See the following figure. 
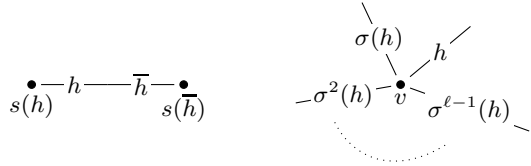
\[
\begin{tikzpicture}[baseline=0mm]
    \node(1) at(0:-1) {};
    \node(2) at(0:1) {};
    \coordinate(0) at(0:0);
    \draw[fill=black] (1)circle(0.5mm) node[below]{\small $s(h)$};
    \draw[fill=black] (2)circle(0.5mm) node[below]{\small $s(\overline{h})$};
    \draw[-] (1)--node[fill=white,inner sep=1]{\small $h$}(0);
    \draw[-] (0)--node[fill=white,inner sep=1]{\small $\overline{h}$}(2);
\end{tikzpicture}
\quad \quad \quad 
\begin{tikzpicture}[baseline=0mm]
    \node(0) at (0:0) {}; 
    \draw[fill=black] (0:0)circle(0.5mm) node[below]{\small $v$};
    \draw (0)--node[fill=white,inner sep=1]{\small $\sigma(h)$}(115:1.2); 
    \draw (0)--node[fill=white,inner sep=1]{\small $\sigma^2(h)$}(190:1.4); 
    \draw (0)--node[fill=white,inner sep=1]{\small $\sigma^{\ell-1}(h)$}(-20:1.8); 
    \draw (0)--node[fill=white,inner sep=1]{\small $h$}(40:1.2); 
    \draw[dotted] (210:1)..controls(240:1.2)and(270:1.2)..(310:1);
    \end{tikzpicture}
\]
We also notice that a geometric realization naturally provides an undirected graph, which we call the \emph{underlying graph} of $\Gamma$. 
From now on, we assume that every ribbon graph is connected, that is, the underlying graph is connected.

\begin{definition}
    Let $\Gamma=(H,\sigma,\overline{(\ )})$ be a Brauer graph with multiplicity function $\mathfrak{m}$. Let $Q_{\Gamma}$ be a finite quiver given as follows: 
\begin{itemize}
    \item The set of vertices is the set $E$ of edges of $\Gamma$. 
    \item We draw an arrow $a_{h} \colon [h] \to [\sigma(h)]$ for every $h\in H$.
\end{itemize}
We define the algebra $B_{\Gamma}:=kQ_{\Gamma}/I_{\Gamma}$, where $I_{\Gamma}$ is a two-sided ideal in the path algebra $kQ_{\Gamma}$ generated by all relations of the following forms: 
\begin{eqnarray*}
    a_{\sigma^{-1}(h)}a_{\overline{h}} \colon \ [\sigma^{-1}(h)] \xrightarrow{a_{\sigma^{-1}(h)}} [h]=[\overline{h}] \xrightarrow{a_{\overline{h}}} [\sigma(\overline{h})] \quad \text{and} 
\end{eqnarray*}
\begin{eqnarray*}
    C_h^{\mathfrak{m}(s(h))} - C_{\overline{h}}^{\mathfrak{m}(s(\overline{h}))}  
\end{eqnarray*}
for all $h\in H$. Here, $C_h$ denotes a cycle 
\[
    C_h \colon \ [h] \xrightarrow{a_{h}} [\sigma(h)] \xrightarrow{a_{\sigma(h)}} [\sigma^2(h)] \longrightarrow \cdots \longrightarrow [\sigma^{\ell-1}(h)] \xrightarrow{a_{\sigma^{\ell-1}(h)}} [\sigma^{\ell}(h)]= [h]
\]
in $Q_{\Gamma}$ and $\ell$ denotes the cardinality of the $\sigma$-orbit incident to $s(h)$.
We call $B_{\Gamma}$ the \textit{Brauer graph algebra} associated to $\Gamma$. 
\end{definition}

It is well-known (see \cite{Sc} for example) that Brauer graph algebras are finite dimensional symmetric algebras, which are special biserial. 
We give an example of some Brauer graphs and their associated Brauer graph algebras in Example \ref{sec:example}.

The following special classes of Brauer graph algebras play a central role in this section. Now, we say that a \emph{cycle graph} of length $\ell$ is an undirected graph which consists of a single cycle having $\ell$ edges. An \emph{odd cycle} (respectively, \emph{even cycle}) is a cycle graph of odd (respectively, even) length.

\begin{definition}
    A Brauer graph $\Gamma$ is called \emph{Brauer tree} if its underlying graph is a tree, and called \emph{Brauer odd-cycle} if its underlying graph contains precisely one odd cycle and no even cycles. In this case, $B_{\Gamma}$ is called \emph{Brauer tree algebra} and \emph{Brauer odd-cycle algebra} respectively.  
\end{definition}

Brauer trees and Brauer odd-cycles have canonical embeddings into the plane so that the cyclic ordering of the half-edges incident to each vertex is described in counterclockwise direction. See Figure \ref{Fig:BTA-BOA}.

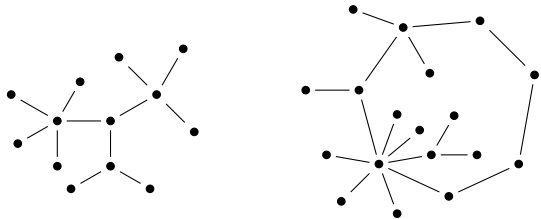
\begin{figure}[t]
    \begin{tabular}{ccccccccc} 
    \begin{tikzpicture}[baseline=0mm]
    \node(1) at(0:0) {};
    \node(2) at(90:0.6) {};
    \node(3) at(210:0.6) {};
    \node(4) at(-30:0.6) {};
    \node(5) at($(180:0.7)+(2)$) {}; 
    \node(6) at($(30:0.7)+(2)$) {};
    \node(7) at($(60:0.7)+(6)$) {};
    \node(8) at($(135:0.7)+(6)$) {};
    \node(9) at($(-45:0.7)+(6)$) {};
    \node(10) at($(210:0.6)+(5)$) {};
    \node(11) at($(150:0.7)+(5)$) {};
    \node(12) at($(60:0.6)+(5)$) {};
    \node(13) at($(-90:0.6)+(5)$) {};

    \draw[fill=black] (1)circle(0.5mm);
    \draw[fill=black] (2)circle(0.5mm);
    \draw[fill=black] (3)circle(0.5mm);
    \draw[fill=black] (4)circle(0.5mm);
    \draw[fill=black] (5)circle(0.5mm);
    \draw[fill=black] (6)circle(0.5mm);
    \draw[fill=black] (7)circle(0.5mm);
    \draw[fill=black] (8)circle(0.5mm);
    \draw[fill=black] (9)circle(0.5mm);
    \draw[fill=black] (10)circle(0.5mm);
    \draw[fill=black] (11)circle(0.5mm);
    \draw[fill=black] (12)circle(0.5mm);
    \draw[fill=black] (13)circle(0.5mm);

    \draw (1)--(2) (1)--(3) (1)--(4) (2)--(5) (2)--(6) (6)--(7) (6)--(8) (6)--(9) (5)--(10) (5)--(11) (5)--(12) (5)--(13); 
    \end{tikzpicture}
    \hspace{10mm}
    \begin{tikzpicture}[baseline=-8mm]

        \node(a0) at(20:1.2) {};  
        \draw[fill=black] (a0)circle(0.5mm);
        \node(a1) at(70:1.2) {};  
        \draw[fill=black] (a1)circle(0.5mm);
        \node(a2) at(120:1.2) {};  
        \draw[fill=black] (a2)circle(0.5mm);
        \node(a3) at(170:1.2) {};  
        \draw[fill=black] (a3)circle(0.5mm);
        \node(a4) at(220:1.2) {};  
        \draw[fill=black] (a4)circle(0.5mm);
        \node(a5) at(270:1.2) {};  
        \draw[fill=black] (a5)circle(0.5mm);
        \node(a6) at(320:1.2) {};  
        \draw[fill=black] (a6)circle(0.5mm);

        \node(b1) at($(10:0.7)+(a4)$) {};  
        \draw[fill=black] (b1)circle(0.5mm);
        \node(b2) at($(40:0.7)+(a4)$) {};  
        \draw[fill=black] (b2)circle(0.5mm);
        \node(b3) at($(70:0.7)+(a4)$) {};  
        \draw[fill=black] (b3)circle(0.5mm);
        \node(b4) at($(-70:0.7)+(a4)$) {};  
        \draw[fill=black] (b4)circle(0.5mm);
        \node(b5) at($(-135:0.7)+(a4)$) {};  
        \draw[fill=black] (b5)circle(0.5mm);
        \node(b6) at($(170:0.7)+(a4)$) {};  
        \draw[fill=black] (b6)circle(0.5mm);
        \node(b7) at($(180:0.7)+(a3)$) {};  
        \draw[fill=black] (b7)circle(0.5mm);

        \node(b8) at($(-60:0.7)+(a2)$) {};  
        \draw[fill=black] (b8)circle(0.5mm);
        \node(b9) at($(160:0.7)+(a2)$) {};  
        \draw[fill=black] (b9)circle(0.5mm);
        
       

        \node(b14) at($(0:0.6)+(b1)$) {};  
        \node(b15) at($(60:0.6)+(b1)$) {};  
        \draw[fill=black] (b14)circle(0.5mm);
        \draw[fill=black] (b15)circle(0.5mm);
        
        \draw (a0)--(a1)--(a2)--(a3)--(a4)--(a5)--(a6);
        \draw (a4)--(b1) (a4)--(b2) (a4)--(b3) (a4)--(b4) (a4)--(b5) (a4)--(b6) (a3)--(b7);
        \draw[] (a0)--(a6);
        \draw (a2)--(b8) (a2)--(b9) (b1)--(b14) (b1)--(b15);
        
    \end{tikzpicture}
\end{tabular} 
\caption{Examples of a Brauer tree (in the left) and a Brauer odd-cycle (in the right).} 
\label{Fig:BTA-BOA}
\end{figure}

The following is a main result in this section, where the equivalence of (i) and (iii) in (a) below was shown in \cite{AAC} (see also \cite{STV}), and 
(b) for Brauer tree algebras was shown in \cite{AMN}.

\begin{theorem} \label{BTA-BOA-Phi}
    Let $\Gamma$ be a Brauer graph with $n$ edges and $B_{\Gamma}$ the Brauer graph algebra of $\Gamma$. 
    \begin{enumerate}[\rm(a)]
    \item $B_{\Gamma}$ is pairwise $g$-convex. Moreover, the following conditions are equivalent. 
    \begin{enumerate}[\rm(i)]
        \item $B_{\Gamma}$ is $g$-finite, 
        \item $B_{\Gamma}$ is $g$-convex,
        \item $\Gamma$ is either a Brauer tree or a Brauer odd-cycle.     
    \end{enumerate}
    \item Let $\Phi$ be a root system of type $A_n$ or $C_n$ and $L$ the root lattice of $\Phi$. 
    If one of the equivalent conditions of (a) hold, then we have an isomorphism 
    $K_0(\proj B_{\Gamma}) \simeq L$ which restricts to a bijection $[\twopsilt^1 B_{\Gamma}]  \simeq \Phi$:
    \begin{equation}\label{eq:twosilt-aw}
        \xymatrix{
            K_0(\proj B_{\Gamma}) \ar[r]^-{\sim} \ar@{}[d]|{\bigcup} & L \ar@{}[d]|{\bigcup} \\ 
            [\twopsilt^1 B_{\Gamma}]  \ar[r]^-{\sim} &  \Phi,
        }
        \end{equation}
    where $\Phi$ is of type $A_n$ (respectively, $C_n$) if $\Gamma$ is a Brauer tree (respectively, Brauer odd-cycle).
    Therefore, it induces an isomorphism 
    \[
        \P(B_{\Gamma}) \cong P_{\Phi} 
    \]
    of lattice polytopes, and the $f$-vector and the $h$-vector are given by 
    \begin{eqnarray*}
        f_{A_n}(x) := \sum_{m=0}^n \frac{(n+m)!}{m!m!(n-m)!}x^{n-m}, \quad h_{A_n}(x) := \sum_{k=0}^n\binom{n}{k}^2 x^k \\ 
        \left(\text{respectively, \ }   f_{C_n}(x) := \sum_{m=0}^n \frac{n2^{2m}}{n+m} \binom{n+m}{2m} x^{n-m}, \quad h_{C_n}(x) := \sum_{k=0}^n\binom{2n}{2k} x^k \right).
    \end{eqnarray*}
    \end{enumerate}
\end{theorem}

We note that two $g$-convex Brauer graph algebras may have the same $g$-polytope even though they are not derived equivalent, see Example \ref{sec:example}. We also note that one can describe the $g$-fan $\Sigma(B_{\Gamma})$ in terms of shear coordinate of laminations corresponding to $\Gamma$ \cite{AY} (see also \cite{FT}).

To prove Theorem \ref{BTA-BOA-Phi}, we need some preparations. It is mentioned in \cite{EJR} (see also \cite{AAC}) that the set $\twosilt B_{\Gamma}$ only depends on the ribbon graph of $\Gamma$ and is independent of the multiplicity function $\mathfrak{m}$. For simplicity, we may replace $\mathfrak{m}$ to a multiplicity function identically equal to $1$, and identify $\Gamma$ with its ribbon graph naturally.

\subsection{Signed half-walks}
In this subsection, we recall a classification of 2-term silting complexes over Brauer graph algebras due to \cite{AAC}. Let $\Gamma=(H,\sigma,\overline{(\ )})$ be an arbitrary Brauer graph with $n$ edges. 

\begin{definition} \label{def:signed walk}
\begin{enumerate}[\rm (1)]
    \item A \textit{half-walk} of $\Gamma$ is a sequence of half-edges $w:=(h_1,\ldots, h_m)$ such that $s(\overline{h_i})=s(h_{i+1})$ and $\overline{h_i}\neq h_{i+1}$ for all $i\in \{1,\ldots,m-1\}$. 
    Defining $\overline{w}=(\overline{h_m},\ldots,\overline{h_1})$ makes $\overline{(\ )}$ an involution on the set of signed half-walks. 
    \item A \emph{walk} of $\Gamma$ is the unordered pair $W=\{w,\overline{w}\}$ of half-walks. For a walk $W=\{w=(h_1,\ldots, h_m),\overline{w}\}$, we define $s(w):=s(h_1)$ and $s(\overline{w})=s(\overline{h_m})$ as the \emph{endpoints} of $W$. 
    \item A \textit{signature} on a walk $W=\{w:=(h_1,\ldots, h_m),\overline{w}\}$ is an assignment of signs $\epsilon(h)= \epsilon(\overline{h})\in \{\pm1\}$ for the half-edges appearing in $W$ such that $\epsilon(h_i)\neq \epsilon(h_{i+1})$ for all $i\in \{1,\ldots, m-1\}$. 
    A half-walk $w$ equipped with a signature $\epsilon$ on a walk $w\in W$ is called \emph{signed half-walk} and written by $w^{\epsilon}$ or $(h_1,\ldots,h_m;\epsilon)$. 
    A \emph{signed walk} is the unordered pair $W^{\epsilon}=\{w^{\epsilon},\overline{w}^{\epsilon}\}$ of signed half-walks. 
\end{enumerate}
\end{definition}

Remark that one can not always find a signature for a given walk $W=\{w:=(h_1,\ldots,h_m),\overline{w}\}$. 
However, if one finds a signature on $W$, then there are precisely two signatures $\pm \epsilon$, giving signed half-walks $w^{\epsilon}=(h_1,\ldots,h_m;\epsilon)$ and $w^{-\epsilon}=(h_1,\ldots,h_m;-\epsilon)$.

Similar to a walk on undirected graphs, we can describe a given (signed) half-walk graphically as follows:
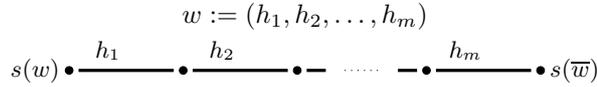
\[
    \begin{tabular}{ccc}
    $w:=(h_1, h_2,\ldots, h_m)$ 
    \\ 
\begin{tikzpicture}[baseline=0mm]
    \node(1) at(0:0) {};
    \node(2) at(0:1.5) {};
    \node(3) at(0:3) {};
    \node(4) at(0:3.5) {};
    \node(4a) at(0:3.5) {};
    \node(5a) at(0:4.2) {};
    \node(5) at(0:4.2) {};
    \node(6) at(0:4.7) {};
    \node(7) at(0:6.2) {};

    \coordinate(0) at(0:0);
    \draw[fill=black] (1)circle(0.5mm) node[left]{\small $s(w)$};
    \draw[fill=black] (2)circle(0.5mm) node[below]{};
    \draw[fill=black] (3)circle(0.5mm) node[below]{};
    \draw[fill=black] (6)circle(0.5mm) node[below]{};
    \draw[fill=black] (7)circle(0.5mm) node[right]{\small $s(\overline{w})$};

    \draw[-,very thick] (1)--node[above]{\small $h_1$}($(1)!0.6!(2)$)--(2); 
    \draw[-,very thick] (2)--node[above]{\small $h_2$}($(2)!0.6!(3)$)--(3); 
    \draw[-,very thick] (3)--(4);
    \draw[dotted] (4a)--(5a); 
    \draw[-,very thick] (5)--(6);
    \draw[-,very thick] (6)--node[above]{\small $h_m$}($(6)!0.6!(7)$)--(7);
    
\end{tikzpicture} 
\end{tabular}
\]
Here, we need to take in account of loops carefully. For example, $w_1\neq w_2$ as half-walks in the next example: 
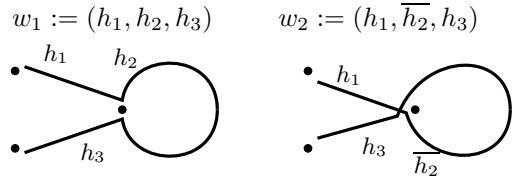
\[
    \begin{tabular}{ccc}
    $w_1:=(h_1,h_2,h_3)$ &  $w_2:=(h_1,\overline{h_2},h_3)$ \\
    \begin{tikzpicture}[baseline=0mm]
        \node(1) at(160:1.5) {};  
        \node(2) at(0:0) {}; 
        \node(3) at(-160:1.5) {};

        \draw[fill=black] (1)circle(0.5mm);
        \draw[fill=black] (2)circle(0.5mm);
        \draw[fill=black] (3)circle(0.5mm);
        \node(11) at($(1)+(90:0.1)$) {}; 
        \node(21) at($(2)+(90:0.1)$) {}; 
        \node(31) at($(3)+(-90:0.1)$) {};
        \node(111) at($(1)+(-90:0.1)$) {}; 
        \node(211) at($(2)+(-90:0.1)$) {};
        \node(311) at($(3)+(-90:0.1)$) {};

        
        \coordinate(2a) at(90:0.125);
        \coordinate(2b) at(-90:0.125);
        \coordinate(2c) at(0:0);
        
        \draw[very thick] (11)--(2a)..controls(0.1,0.8)and(1.25,0.8)..(0:1.25) ;
        \draw[very thick] (31)--(2b)..controls(0.1,-0.8)and(1.25,-0.8)..(0:1.25) ;
        
    \node[above] at($(160:0.9)+(90:0.18)+(180:0.025)$) {\small $h_1$};
    \node[above] at($(0:0)+(60:0.5)+(180:0.2)$) {\small $h_2$};
    \node[below] at($(0:0)+(220:0.5)$) {\small $h_3$};
        
    \end{tikzpicture}
    & \quad \quad 
    \begin{tikzpicture}[baseline=0mm]
        \node(1) at(160:1.5) {};  
        \node(2) at(0:0) {}; 
        \node(3) at(-160:1.5) {};

        \draw[fill=black] (1)circle(0.5mm);
        \draw[fill=black] (2)circle(0.5mm);
        \draw[fill=black] (3)circle(0.5mm);
        \node(11) at($(1)+(-90:0.1)$) {}; 
        \node(21) at($(2)+(90:0.1)$) {}; 
        \node(31) at($(3)+(90:0.1)$) {};
        \node(111) at($(1)+(-90:0.1)$) {}; 
        \node(211) at($(2)+(-90:0.1)$) {};
        \node(311) at($(3)+(-90:0.1)$) {}; 
        
        \coordinate(2a) at(200:0.25);
        \coordinate(2b) at(200:0.125);
        \coordinate(2c) at(0:0);
        
        \draw[very thick] (11)--(2b)..controls(0.1,-0.8)and(1.25,-0.8)..(0:1.25) ;
        \draw[very thick] (31)--(2a)..controls(0.1,0.8)and(1.25,0.8)..(0:1.25) ;
        
    \node[above] at($(160:0.9)+(90:-0.1)+(180:0.025)$) {\small $h_1$};
    \node[below] at($(0:0)+(-60:0.5)+(-180:0.1)$) {\small $\overline{h_2}$};
    \node[below] at($(0:0)+(200:0.7)+(0:0.125)$) {\small $h_3$};
    \end{tikzpicture}
\end{tabular}
\]

We need to attach some extra data to each endpoint of a signed walk, which is uniquely determined by its signature. 

\begin{definition} \label{def:virtual}
A \emph{\text{virtual (half-)edge}} is an element of the set of the symbols $\{\mathrm{vr}_+(h), \mathrm{vr}_-(h)\mid h\in H\}$. We extend the map $s$ to virtual edges by $s(\mathrm{vr}_{\pm}(h))=s(h)$ for all $h\in H$. For a vertex $v$, we extend the cyclic ordering $(h,\sigma(h),\ldots, \sigma^{\ell-1}(h))$ around $v$ to the \emph{cyclic ordering around $v$ accounting the virtual edges} by
\begin{equation} \label{eq:cyc}
    (\mathrm{vr}_-(h), h,\mathrm{vr}_+(h), \mathrm{vr}_-(\sigma(h)), \sigma(h), \mathrm{vr}_+(\sigma(h)), \ldots, \mathrm{vr}_-(\sigma^{\ell-1}(h)), \sigma^{\ell-1}(h),  \mathrm{vr}_+(\sigma^{\ell-1}(h))). 
\end{equation}
A subsequence of (\ref{eq:cyc}) is called a \emph{cyclic subordering around $v$ accounting the virtual edges}. 
To each signed walk $W^{\epsilon}=\{w^{\epsilon}=(h_1,\ldots, h_m;\epsilon),\overline{w}^{\epsilon}\}$, we attach virtual edges given by 
\begin{equation*}
    h_0 = \overline{h_0} := \mathrm{vr}_{-\epsilon(h_1)}(h_1) \quad \text{and} \quad h_{r+1} = \overline{h_{r+1}} := \mathrm{vr}_{-\epsilon(h_m)}(\overline{h_m})  
\end{equation*}
with signs $\epsilon(h_0)= \epsilon(\overline{h_0})=-\epsilon(h_1)$ and $\epsilon(h_{m+1})= \epsilon(\overline{h_{m+1}}) = -\epsilon(h_m)$. 
\end{definition}

Now we define the admissibility of signed walks by the non-crossing conditions (NC0)-(NC3) below. 
Fix a pair of (not necessarily distinct) signed walks $W^{\epsilon}=\{w^{\epsilon}=(h_1,\ldots,h_{m};\epsilon),\overline{w}^{\epsilon}\}$ and 
$W'^{\epsilon'}=\{w'^{\epsilon'}=(h_1',\ldots,h_{\ell}';\epsilon'),\overline{w'}^{\epsilon'}\}$. 
Notice that there are virtual edges $h_0,h_{\ell+1},h_0',h'_{m+1}$, etc.\,with signs attached to their endpoints.   

\begin{definition}\cite[Definition 2.7]{AAC} \label{NC0}
    We say that $W^{\epsilon}$ and $W'^{\epsilon'}$ satisfies \textnormal{(NC0)} (it was called the sign condition in \cite{AAC}) if the following condition is satisfied: 
    \begin{enumerate}
        \item[\textnormal{(NC0)}] Whenever $W^{\epsilon}$ and $W'^{\epsilon'}$ have a common endpoint $
        v$, the signatures on the half-edges of $W^{\epsilon}$ and of $W'^{\epsilon'}$ incident to $v$ are the same. 
    \end{enumerate}
\end{definition}

Next, we say that a \emph{maximal common subwalk} of $W$ and $W'$ is a walk $Z=\{z,\overline{z}\}$ such that $z$ is a common continuous subsequence of some $x\in W$ and $y\in W'$, and there are no common continuous subsequences of $x$ and $y$ properly containing $z$ as a continuous subsequence. It is said to be \emph{proper} if no endpoints of $Z$ are the common endpoint(s) of $W$ and $W'$.
Notice that there are several maximal common subwalks of $W$ and $W'$ in general. 

\begin{definition}\cite[Definition 2.8]{AAC} \label{NC2}
    Let $Z$ be a maximal common subwalk of $W$ and $W'$ given by a half-walk $z=(t_1,\ldots, t_{r})$ so that $t_k=h_{i+k-1}=h'_{j+k-1}$ for all $k\in \{1,\ldots,r\}$. 
    Let $u:=s(t_1)$ and $v:=s(\overline{t_r})$ be endpoints of $Z$.
    We call the pair of cyclic subordering on $\{\overline{h_{i-1}},\overline{h'_{j-1}},t_1\}$ around $u$ and $\{h_{i+r},h'_{j+r}, \overline{t_{\ell}}\}$ around $v$ accounting virtual edges the \emph{neighbourhood cyclic ordering} of $Z$. 
    Then, we say that $W^{\epsilon}$ and $W'^{\epsilon'}$ satisfies \textnormal{(NC1)} and \textnormal{(NC2)} at $Z$ if the following condition is satisfied respectively. 
    \begin{enumerate}
        \item[(\textnormal{NC1})] $\epsilon(t_k)=\epsilon'(t_k)$ for all $k\in \{1,\ldots, r\}$. 
        \item[(\textnormal{NC2})] If $Z$ is proper, then the neighbourhood cyclic ordering of $Z$ is either 
        \[
            \scalebox{1}
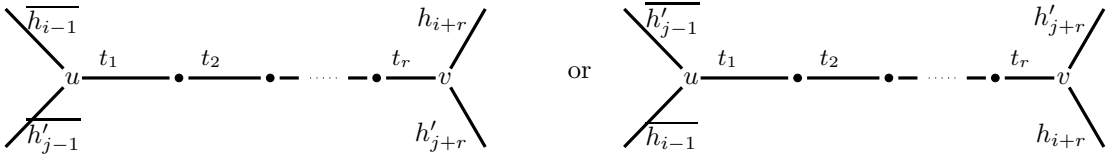
{
        \begin{tikzpicture}[baseline=0mm, scale=1]
            \node(1) at(0:0.4) {};
            \node(2) at(0:1.8) {};
            \node(3) at(0:3) {};
            \node(4) at(0:3.5) {};
            \node(4a) at(0:3.4) {};
            \node(5a) at(0:4) {};
            \node(5) at(0:3.9) {};
            \node(6) at(0:4.4) {};
            \node(7) at(0:5.3) {};

            \coordinate(0) at(0:0);
            \node at(1) {$u$};
            \draw[fill=black] (2)circle(0.5mm) ;
            \draw[fill=black] (3)circle(0.5mm) ;
            \draw[fill=black] (6)circle(0.5mm) ;
            \node at(7) {$v$};
        
            \draw[-,very thick] (1)--node[above,]{\small $t_{1}$}($(1)!0.6!(2)$)--(2); 
            \draw[-,very thick] (2)--node[above]{\small $t_{2}$}($(2)!0.6!(3)$)--(3); 
            \draw[-,very thick] (3)--(4);
            \draw[dotted] (4a)--(5a); 
            \draw[-,very thick] (5)--(6);
            \draw[-,very thick] (6)--node[above]{\small $t_{r}$}($(6)!0.6!(7)$)--(7);

            \node(8c) at($(0:0)+(120:1.2)$) {}; 
            \draw[very thick] (8c)--node[right]{$\overline{h_{i-1}}$}($(8c)!0.4!(1)$)--(1);
        
            \node(8d) at($(0:0)+(-120:1.2)$) {};
            \draw[very thick] (8d)--node[right]{$\overline{h'_{j-1}}$}($(8d)!0.4!(1)$)--(1);
        
            \node(z1) at($(7)+(60:1.2)$) {}; 
            \node(z2) at($(z1)+(0:0.7)$) {}; 
            
            \draw[very thick] (z1)--node[left]{$h_{i+r}$}($(z1)!0.4!(7)$)--(7);

            \node(zl) at($(7)+(-60:1.2)$) {}; 
            \node(zl1) at($(zl)+(0:0.7)$) {}; 
            \draw[very thick] (zl)--node[left]{$h'_{j+r}$}($(zl)!0.4!(7)$)--(7);
       
        \end{tikzpicture}
        }
        \ \text{or} \ 
        \scalebox{1}{
            \begin{tikzpicture}[baseline=0mm, scale=1]
                \node(1) at(0:0.4) {};
                \node(2) at(0:1.8) {};
                \node(3) at(0:3) {};
                \node(4) at(0:3.5) {};
                \node(4a) at(0:3.4) {};
                \node(5a) at(0:4) {};
                \node(5) at(0:3.9) {};
                \node(6) at(0:4.4) {};
                \node(7) at(0:5.3) {};

                \coordinate(0) at(0:0);
                \node at(1) {$u$};
                \draw[fill=black] (2)circle(0.5mm) ;
                \draw[fill=black] (3)circle(0.5mm) ;
                \draw[fill=black] (6)circle(0.5mm) ;
                \node at(7) {$v$};
            
                \draw[-,very thick] (1)--node[above,]{\small $t_{1}$}($(1)!0.6!(2)$)--(2); 
                \draw[-,very thick] (2)--node[above]{\small $t_{2}$}($(2)!0.6!(3)$)--(3); 
                \draw[-,very thick] (3)--(4);
                \draw[dotted] (4a)--(5a); 
                \draw[-,very thick] (5)--(6);
                \draw[-,very thick] (6)--node[above]{\small $t_{r}$}($(6)!0.6!(7)$)--(7);

                \node(8c) at($(0:0)+(120:1.2)$) {}; 
                \draw[very thick] (8c)--node[right]{$\overline{h'_{j-1}}$}($(8c)!0.4!(1)$)--(1);
            
                \node(8d) at($(0:0)+(-120:1.2)$) {};
                \draw[very thick] (8d)--node[right]{$\overline{h_{i-1}}$}($(8d)!0.4!(1)$)--(1);
            
                \node(z1) at($(7)+(60:1.2)$) {}; 
                \node(z2) at($(z1)+(0:0.7)$) {}; 
                
                \draw[very thick] (z1)--node[left]{$h'_{j+r}$}($(z1)!0.4!(7)$)--(7);

                \node(zl) at($(7)+(-60:1.2)$) {}; 
                \node(zl1) at($(zl)+(0:0.7)$) {}; 
                \draw[very thick] (zl)--node[left]{$h_{i+r}$}($(zl)!0.4!(7)$)--(7);
            
            \end{tikzpicture}
        }
        \]
    \end{enumerate}
\end{definition}

Notice that the condition \textnormal{(NC1)} is automatically satisfied for any signed walk with itself by the definition of the signature (Definition \ref{def:signed walk}(3)). 

Lastly, for a vertex $v$ and an integer $i\in \{1,\ldots,r\}$ such that $v=s(h_i)$, we refer to the set $\{\overline{h_{i-1}},h_i\}$ as a \emph{neighbourhood} of $v$ in $W$. Notice that half-edges appearing in the neighbourhood at a vertex can be virtual. 

\begin{definition}\cite[Definition 2.9]{AAC} \label{NC3}
    For a vertex $v$, suppose that $\{a,b\}$ and $\{c,d\}$ are neighbourhoods of $v$ in $W$ and in $W'$ respectively. Then, we say that $v$ is the \emph{intersecting vertex} of $W$ and $W'$ if $a,b,c,d$ are pairwise distinct. 
    We say that $W$ and $W'$ satisfies \textnormal{(NC3)} at the intersecting vertex $v$ if the following condition is satisfied. 
    \begin{enumerate}
        \item[\textnormal{(NC3)}]  If $v$ is an intersecting vertex with respect to the neighbourhoods $\{a,b\}$ in $W$ and $\{c,d\}$ in $W'$, and at most one of $a,b,c,d$ is virtual, then the cyclic subordering around $v$ accounting virtual edges and signatures are either   
    \end{enumerate}
    
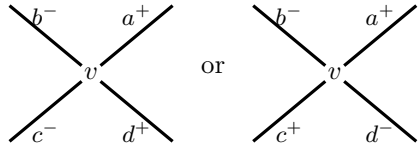
\[
\begin{tikzpicture}[baseline=0mm]
    \node(0) at(0:0) {}; 
    \node at(0) {$v$};
    \draw[very thick] (0)--node[above]{\small $a^+$}(40:1.4); 
    \draw[very thick] (0)--node[above]{\small $b^-$}(140:1.4); 
    \draw[very thick] (0)--node[below]{\small $c^-$}(-140:1.4); 
    \draw[very thick] (0)--node[below]{\small $d^+$}(-40:1.4); 
\end{tikzpicture}
\quad \text{or} \quad  
\begin{tikzpicture}[baseline=0mm]
    \node(0) at(0:0) {}; 
    \node at(0) {$v$};
    \draw[very thick] (0)--node[above]{\small $a^+$}(40:1.4); 
    \draw[very thick] (0)--node[above]{\small $b^-$}(140:1.4); 
    \draw[very thick] (0)--node[below]{\small $c^+$}(-140:1.4); 
    \draw[very thick] (0)--node[below]{\small $d^-$}(-40:1.4); 
\end{tikzpicture}
\]
\end{definition}

\begin{definition}
    We say that two signed walks $W^{\epsilon}$ and $W'^{\epsilon'}$ are \emph{admissible} if they satisfy \textnormal{(NC0)}, \textnormal{(NC1)} and \textnormal{(NC2)} at all maximal common subwalks, and \textnormal{(NC3)} at all intersecting vertices. 
    An \emph{admissible signed walk} is a signed walk that is admissible with itself. 
\end{definition}

Now, we denote by $\mathsf{AW}(\Gamma)$ the set of admissible signed walks of $\Gamma$, by $\mathsf{CW}(\Gamma)$ the set of maximal collections consisting of admissible signed walks of $\Gamma$ which are pairwise admissible.

\begin{theorem}\cite[Theorem 4.6]{AAC} \label{AAC}
Let $\Gamma$ be a Brauer graph and $B_{\Gamma}$ the Brauer graph algebra of $\Gamma$. Then, there are bijections 
    \begin{equation} \label{eq:AAC}
        \twopsilt^1 B_{\Gamma} \overset{\sim}{\longrightarrow} \mathsf{AW}(\Gamma)
        \quad \text{and} \quad 
        \twosilt B_{\Gamma} \overset{\sim}{\longrightarrow} \mathsf{CW}(\Gamma).
    \end{equation}
\end{theorem}

From the definition of Brauer graph algebras, the Grothendieck group $K_0(\proj B_{\Gamma})$ is canonically isomorphic to the free $\mathbb{Z}$-module $\mathbb{Z}E$ over the set $E$ of edges of $\Gamma$. 
For a signed walk $W^{\epsilon}=\{w^{\epsilon}=(h_1,\ldots,h_m;\epsilon),\overline{w}^{\epsilon}\}$, one can define
\begin{equation} \label{eq:classOfwalk}
    [W^{\epsilon}] := \sum_{i=1}^r \epsilon(h_i) [h_i] \in \mathbb{Z}E.
\end{equation}
In fact, it is independent of a choice of $w\in W$. 
In addition, let 
\[
    [\mathsf{AW}(\Gamma)] := \{[W^\epsilon] \mid W^\epsilon \in \mathsf{AW}(\Gamma) \} \subset \mathbb{Z}E.
\]    

Then we have the following.

\begin{proposition} \label{AAC-Thm}
    We have the following commutative diagram  
    \begin{equation} \label{eq:psaw}
        \xymatrix{
            K_0(\proj B_{\Gamma}) \ar[r]^-{\sim} \ar@{}[d]|{\bigcup}  & \mathbb{Z}E \ar@{}[d]|{\bigcup} \\ 
            [\twopsilt^1 B_{\Gamma}] \ar[r]^-{\sim} &  [\mathsf{AW}(\Gamma)]. 
        }
    \end{equation}
    
\end{proposition}

\begin{proof}
    Let $T\in \twopsilt^1 B_{\Gamma}$ and $W^{\epsilon}\in \mathsf{AW}(\Gamma)$ the signed walk corresponding to $T$ under the bijection $\twopsilt^1 B_{\Gamma} \simeq \mathsf{AW}(\Gamma)$ in (\ref{eq:AAC}). 
    According to \cite[Secition 4]{AAC}, one can see that the class $[T]$ is sent to $[W^{\epsilon}]$ under the canonical isomorphism $K_0(\proj B_{\Gamma})\simeq \mathbb{Z}E$. 
\end{proof}

\begin{example} \label{ex:edges}
    For an edge $X=\{h,\overline{h}\}\in E$, the element $X\in \mathbb{Z}E$ belongs to $[\mathsf{AW}(\Gamma)]$.
    In fact, a signed half-walk $W_X^+= \{(h;+), (\overline{h};+)\}$ is clearly admissible and satisfies $[W_X^+]=X$ in $\mathbb{Z}E$. 
    On the other hand, let $P_X$ be the indecomposable projective $B_{\Gamma}$-module corresponding to $X$. Then $P_X\in \twopsilt^1 B_{\Gamma}$ is sent to $W_X^+$ by the bijection 
    in Theorem \ref{AAC}.
\end{example}

\begin{example} \label{sec:example}
We depict the $g$-polytope and the corresponding root polytope for a class of Brauer graph algebras with small number of edges. 

\begin{enumerate}[\rm(a)] 
    \item For $n=2$, there are only one Brauer tree and one Brauer odd-cycle described as follows. 
    \[

    \]
    Then the associated Brauer graph algebras are presented by the following quivers with relations: 
    \[
        k\big(\begin{xy}(0,0)*+{1}="1", (10,0)*+{2}="2", {"1" \ar@<-1mm>_{a} "2"}, {"2" \ar@<-1mm>_{b} "1"} \end{xy}\big)/\langle aba,bab\rangle
        \quad \text{and} \quad 
        k\big(\begin{xy}(0,0)*+{1}="1", (10,0)*+{2}="2", {"1" \ar@<-1mm>_{a} "2"}, {"2" \ar@<-1mm>_{c} "1"}, {"2" \ar@(rd,ru)_{b} "2"}\end{xy}\big)/\langle ac,b^2, abca,bca-cab\rangle
    \]
    The $g$-polytopes of their associated Brauer graph algebras are 
    \[
        %
    \]
    which clearly correspond to root polytopes of type $A_2$ and $C_2$ respectively. 
    \item Let $\Gamma_1$ be the following Brauer tree having $3$ edges.
        \[
            %
 
        \] 
        Then $B_{\Gamma_1}$ is the algebra given in Example \ref{newton example}.
        There are $12$ elements in $\twopsilt^1 B_{\Gamma_1}$. We describe the corresponding $g$-vectors and admissible signed walks in the following table.
    
    \begin{figure}[h]
        \scalebox{0.9}{
    \begin{tabular}{|c|c|c|c|c|c|c|c|}
        \hline
        $[W^{\epsilon}]$ & $(1,0,0)$ & $(0,1,0)$ & $(0,0,1)$ & $(1,-1,0)$ & $(0,1,-1)$ & $(1,-1,1)$ \\
\hline
    \end{tabular}
        }
    \end{figure}

    \noindent
    By Proposition \ref{characterize g-convex}(b), the $g$-polytope of $B_{\Gamma_1}$ is given by the convex hull in $\mathbb{R}^3$ of all integer vectors appearing in the above table and described as follows. 
    \begin{equation}\nonumber
        %
 
    \end{equation}
    It is isomorphic to the root polytope of type $A_{3}$ up to the linear isomorphism by 
    $\begin{psmallmatrix} 1 & 0 & 0 \\ 0 & -1 & 0 \\ 0 & 0 & 1 \end{psmallmatrix}$. 

    \item Let $\Gamma_2$ be the following Brauer odd-cycle having $3$ edges.  
    \[
        %
$, and there are $18$ elements in $\twopsilt^1 B_{\Gamma_2}$ whose $g$-vectors are given by the following table. 
    \begin{figure}[h]
        \scalebox{0.85}{
        \begin{tabular}{|c|c|c|c|c|c|c|c|}
            \hline
            $[W^{\epsilon}]$ & $(1,0,0)$ & $(0,1,0)$ & $(0,0,1)$ & $(1,-1,0)$ & $(1,0,-1)$ & $(0,1,-1)$ \\
\hline
        \end{tabular}
        }
        \end{figure}
    
    \noindent
    Then $\P(B_{\Gamma_2})$ is described in the left diagram of Figure \ref{fig:ExBGA}, which is isomorphic to the root polytope of type $C_3$ up to linear isomorphism by $\begin{psmallmatrix} 1 & 0 & 0 \\ 0 & -1 & -2 \\ 0 & 0 & 1 \end{psmallmatrix}$.
     
    \item The following Brauer odd-cycle $\Gamma_3$ has the same underlying graph as $\Gamma_2$ in (c) but they are different as Brauer graphs. 
    \[
        %
    \] 
    Then, $B_{\Gamma_3}$ is an algebra $A$ in Example \ref{rank 3 example}. 
    There are $18$ elements in $\twopsilt^1 B_{\Gamma_3}$ whose $g$-vectors are given by the following table. 
    \begin{figure}[h]
        \scalebox{0.85}{
        \begin{tabular}{|c|c|c|c|c|c|c|c|}
            \hline
            $[W^{\epsilon}]$ & $(1,0,0)$ & $(0,1,0)$ & $(0,0,1)$ & $(1,-1,0)$ & $(1,0,-1)$ & $(0,1,-1)$ \\
\hline
        \end{tabular}
        }
        \end{figure}

        From the tables in (c) and (d), we have 
        \[
            [\mathsf{AW}(\Gamma_2)] = [\mathsf{AW}(\Gamma_3)] \subseteq \mathbb{Z}^3. 
        \]
        It implies that $B_{\Gamma_2}$ and $B_{\Gamma_3}$ determine the same $g$-polytope, while they do not determine the same $g$-fan, see Figure \ref{fig:ExBGA}. Notice that they are not derived equivalent. 
    \end{enumerate}

    \begin{figure}[h]
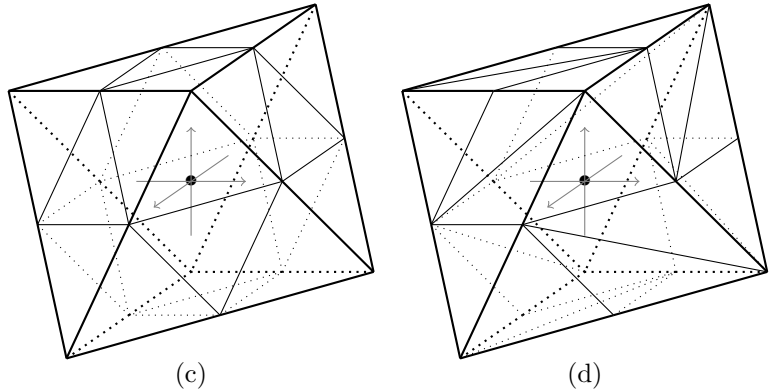

        %
    \caption{The $g$-polytopes for (c) and (d)} 
    \label{fig:ExBGA}
\end{figure}

\end{example}

\subsection{Signed half-walks and root lattices} 
Let $\Gamma=(H,\sigma, \overline{(\ )})$ be a Brauer graph having $n$ edges, which is either a Brauer tree or a Brauer odd-cycle. In this subsection, we associate the root system $\Phi(\Gamma)$ to $\Gamma$ and show that $\Phi(\Gamma)$ is in bijection with the set $\mathsf{AW}(\Gamma)$ of admissible signed walks of $\Gamma$ (Proposition \ref{thm:bijection_roots}). Using this result, we prove Theorem \ref{BTA-BOA-Phi}.

\begin{definition} \label{rootpoly2} 
    Let $V$ be the set of vertices of $\Gamma$ and $\mathbb{R}V$ the vector space with basis $V$. 
    Then 
    \[
        \dim \mathbb{R}V =
        \begin{cases}
           n+1  &\text{if $\Gamma$ is a Brauer tree,} \\ 
           n  &\text{if $\Gamma$ is a Brauer odd-cycle.} 
        \end{cases}
    \]
    We define the inner product on $\mathbb{R}V$ which makes $V$ an orthonormal basis. 
    We define the \emph{root system} associated to $\Gamma$ by 
\[
    \Phi(\Gamma) := 
    \begin{cases}
        \Phi_{A_{n}} = \{ u-v \mid (u,v)\in V\times V, u\neq v\} &\text{if $\Gamma$ is a Brauer tree,} \\
        \Phi_{C_n} = \{ \pm u\pm v \mid (u,v)\in V\times V, u\neq v\} \cup \{\pm 2u \mid u\in V\} &\text{if $\Gamma$ is a Brauer odd-cycle.} 
    \end{cases}
\] 
We denote the root lattice of $\Phi(\Gamma)$ by $L(\Gamma)$. The \emph{root polytope} $P_{\Phi(\Gamma)}$ is defined as the convex hull of $\Phi(\Gamma)$ in $\mathbb{R}\otimes_{\mathbb{Z}}L(\Gamma)$.  
\end{definition}

\begin{proposition} \label{thm:bijection_roots} 
Let $\Gamma$ be a Brauer graph which is either a Brauer tree or a Brauer odd-cycle. 
Let $E$ be the set of edges of $\Gamma$. 
Then, we have an isomorphism $\partial \colon \mathbb{Z}E \simeq L(\Gamma)$ which restricts to a bijection $[\mathsf{AW}(\Gamma)] \simeq \Phi(\Gamma)$. 
\begin{equation}\label{eq:EL}
    \partial \colon
    \xymatrix{
        \mathbb{Z}E \ar[r]^-{\sim} \ar@{}[d]|{\bigcup}  & L(\Gamma) \ar@{}[d]|{\bigcup} \\ 
        [\mathsf{AW}(\Gamma)] \ar[r]^-{\sim} &  \Phi(\Gamma). 
    }
\end{equation}
\end{proposition}

From now on, we prove Proposition \ref{thm:bijection_roots}. 
We need to fix a spanning tree $\Gamma_{\rm sp}$ of $\Gamma$, that is, $\Gamma_{\rm sp}$ is one of subtrees which include all vertices of $\Gamma$. 
Let $H_{\rm sp}$ be the set of half-edges of $\Gamma_{\rm sp}$ and $H_{\rm sp}^c$ the complement of $H_{\rm sp}$. If $\Gamma$ is a Brauer tree, then $H_{\rm sp}^c=\emptyset$; Otherwise, $H_{\rm sp}^c$ consists of two half-edges which form an edge lying in the unique odd cycle of $\Gamma$.

\begin{definition} \label{def:orientation}
    An \textit{orientation} of $\Gamma_{\rm sp}$ is a complete set $\mathfrak{o}$ of representatives of $H_{\rm sp}/\overline{(\ )}$ in $H_{\rm sp}$. 
    For such $\mathfrak{o}$, we say that a vertex $v$ is a \textit{source} 
    (respectively, a \textit{sink})
    if there are no half-edges $h\in \mathfrak{o}$ such that $v=s(\overline{h})$ (resp., $v=s(h)$).  
    A \emph{bipartite orientation} is an orientation such that every vertex of $\Gamma_{\rm sp}$ is either a source or a sink. 
\end{definition}

Since $\Gamma_{\rm sp}$ is a tree, it has precisely two bipartite orientation which are related to each other by involution of half-edges. Take one of such an orientation $\mathfrak{o}$ of $\Gamma_{\rm sp}$. 
For a vertex $v\in V$, we set 
\begin{equation*} \label{eq:bipartite} 
    \mathfrak{o}(v) := 
    \begin{cases}
        1 & \text{if $v$ is a source in $\Gamma_{\rm sp}$},  \\
        -1 & \text{if $v$ is a sink in $\Gamma_{\rm sp}$}.
    \end{cases}
\end{equation*}
By definition, we have $\mathfrak{o}(s(h)) =  - \mathfrak{o}(s(\overline{h}))$ for any $h\in H_{\rm sp}$, while we have $\mathfrak{o}(s(h)) =  \mathfrak{o}(s(\overline{h}))$ for the half-edges $h\in H_{\rm sp}^c$ since $\Gamma$ is a Brauer odd-cycle (if $h$ exists) and $h$ belongs to the odd cycle, see Figure \ref{fig:bipartite}. 
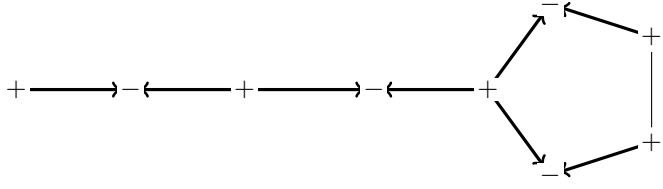
\begin{figure}
\begin{tabular}{cccccccc}
    \begin{tikzpicture}[baseline=0mm, scale=1]
    \node(1) at(0:0) {};
    \node(2) at(0:1.5) {};
    \node(3) at(0:3) {};
    \node(4) at(0:3.5) {};
    \node(4a) at(0:3.5) {};
    \node(5a) at(0:4.2) {};
    \node(5) at(0:4.2) {};
    \node(6) at(0:4.7) {};
    \node(7) at(0:6.2) {};

    \coordinate(0) at(0:0);

    \draw[-,very thick,->] (1)--(2); 
    \draw[-,very thick,<-] (2)--(3); 
    \draw[-,very thick,->] (3)--(6);
    \draw[-,very thick,<-] (6)--(7);

    \node(z1) at($(7)+(54:1.4)$) {}; 
    \node(z2) at($(z1)+(-18:1.4)$) {}; 
    
    \draw[very thick,->] (7)--(z1);
    \draw[very thick,<-] (z1)--(z2);

    \node(zl) at($(7)+(-54:1.4)$) {}; 
    \node(zl1) at($(zl)+(18:1.4)$) {}; 
    \draw[very thick,<-] (zl)--(7);
    \draw[very thick,<-] (zl)--(zl1);


    \node[fill=white,inner sep=1] at(1) {$+$};
    \node[fill=white,inner sep=1] at(2) {$-$};
    \node[fill=white,inner sep=1] at(3) {$+$};
    \node[fill=white,inner sep=1] at(6) {$-$};
    \node[fill=white,inner sep=1] at(7) {$+$};
    \node[fill=white,inner sep=1] at(z1) {$-$};
    \node[fill=white,inner sep=1] at(zl) {$-$};
    \node[fill=white,inner sep=1] at(z2) {$+$};
    \node[fill=white,inner sep=1] at(zl1) {$+$};

    \draw[] ($(z2)+(-90:0.2)$)--($(zl1)+(90:0.2)$);
    
\end{tikzpicture}
\end{tabular}
\caption{A bipartite orientation on the spanning tree.}
\label{fig:bipartite}
\end{figure}
Consider the surjective homomorphism 
\[
    \pi \colon \mathbb{Z} H \to \mathbb{Z}E \quad \text{given by} \quad h \mapsto \mathfrak{o}(s(h))[h] \quad \text{for all $h\in H$,} 
\]
and a homomorphism 
\begin{equation*}
    \delta \colon \mathbb{Z}H \to \mathbb{Z}V \quad \text{given by} \quad 
    h \mapsto 
	\begin{cases} 
		s(h) - s(\overline{h}) & \text{for all $h\in H_{\rm sp}$, } \\ 
		s(h) + s(\overline{h}) & \text{for all $h\in H_{\rm sp}^c$.} 
	\end{cases}
\end{equation*}

\begin{proposition} \label{prop:delpar}
    We have the following commutative diagram. 
    \[
    \xymatrix{
    0 \ar[r] &\ker \pi \ar[r] & \mathbb{Z}H \ar[d]^{\delta} \ar[r]^{\pi} & \mathbb{Z}E \ar[r] \ar@{.>}[ld]^{\partial}& 0 & \textnormal{(exact)}\\ 
    & & \mathbb{Z}V  & &  
    } 
    \]
Moreover, for any $h\in H$, we have 
\begin{equation} \label{eq:map_delta}
    \partial([h]) = 
    \begin{cases} 
        \mathfrak{o}(s(h))(s(h)-s(\overline{h})) & \text{if $h\in H_{\rm sp},$}\\ 
        \mathfrak{o}(s(h))(s(h)+s(\overline{h})) & \text{if $h\in H_{\rm sp}^c$}.
    \end{cases}
\end{equation}
\end{proposition}

\begin{proof}
    For the former assertion, it suffices to show $\delta(\ker \pi)=0$. By definition, the kernel $\ker \pi$ is generated by elements of the forms $h+\overline{h}$ with $h\in H_{\rm sp}$ and $h-\overline{h}$ with $h\in H_{\rm sp}^c$.
    Then, we have  
\begin{eqnarray*}
    \delta(h + \overline{h}) &=&  (s(h)-s(\overline{h})) + (s(\overline{h})- s(h)) = 0 \quad \text{for all $h\in H_{\rm sp}$ and} \\
    \delta(h - \overline{h}) &=&  (s(h) + s(\overline{h})) - (s(\overline{h}) + s(h)) = 0 \quad \text{for all $h\in H_{\rm sp}^c$}.  
\end{eqnarray*}
Thus, we get the former assertion. The latter one is clear. 
\end{proof}

\begin{proposition} \label{prop:parisom}
    The map $\partial$ restricts to an isomorphism $\partial \colon \mathbb{Z}E\simeq L(\Gamma)$.  
\end{proposition}
    
\begin{proof}
    As we mentioned in Example \ref{ex:edges}, a basis $E$ of $\mathbb{Z}E$ is included in $[\mathsf{AW}(\Gamma)]$. 
    We first assume that $\Gamma$ is a Brauer tree and hence $\Phi(\Gamma)=\Phi_{A_n}$. By (\ref{eq:map_delta}), $E$ is bijectively sent to 
    \begin{equation*} \label{eq:basisBT}
        \partial(E) = \{s(h)-s(\overline{h}) \mid h\in \mathfrak{o}\} \subseteq \Phi_{A_n}, 
    \end{equation*}
    which is clearly a basis of $L(\Gamma)$.
    Next, we assume that $\Gamma$ is a Brauer odd-cycle and hence $\Phi(\Gamma)=\Phi_{C_n}$. We denote by $X=\{h_0,\overline{h_0}\}$ the unique edge which does not lie in $\Gamma_{\rm sp}$. 
    Then, $E$ is bijectively sent to 
    \begin{equation*} \label{eq:basisBO}
        \partial(E) \stackrel{\eqref{eq:map_delta}}{=} \{s(h)-s(\overline{h}) \mid h\in \mathfrak{o}\}\cup \{\mathfrak{o}(s(h_0))(s(h_0)+s(\overline{h_0})\} \subseteq \Phi_{C_n}. 
    \end{equation*}
    which is easily shown to be a basis of $L(\Gamma)$. 
\end{proof}

Now, we study signed walks of $\Gamma$. The following is basic.

\begin{lemma} \label{lem:nonsp} \label{lem:SC}
    Let $\Gamma$ be a Brauer tree or a Brauer odd-cycle. 
\begin{enumerate}[\rm (a)]
    \item Let $w^{\epsilon}=(h_1,\ldots,h_m;\epsilon)$ be a signed half-walk. For any edge $X$ lying in the odd cycle in $\Gamma$, there is at most one $1\leq i \leq m$ such that $h_i\in X$.  
    For a signed half-walk $w^{\epsilon}$ of $\Gamma$, every half-edge lying in the odd cycle in $\Gamma$ appears at most once in $w^{\epsilon}$. 
    \item Every signed walk $W^{\epsilon}$ of $\Gamma$ satisfies \textnormal{(NC0)} and \textnormal{(NC1)}. 
\end{enumerate}
    
\end{lemma}

\begin{proof} 
    (a) We only have to consider a Brauer odd-cycle $\Gamma$. 
    We denote by $Z$ the unique odd cycle in $\Gamma$, by $\ell$ the length of $Z$. 
    Take an integers $1\leq j < k \leq m$ such that $h_j,h_k\in X$ and there are no $j<i<k$ satisfying $h_i\in X$. Since $\Gamma$ has no cycle except for $Z$, we have $h_j=h_k$ and $k-j=\ell$. Since $\ell$ is odd, we have 
    $\epsilon(h_j)=\epsilon(h_k) = (-1)^{k-j} \epsilon(h_j) = (-1)^{\ell} \epsilon(h_j)=-\epsilon(h_j)$, a contradiction. 

    (b) The condition \textnormal{(NC1)} is automatic for any signed walk. 

    Let $W^{\epsilon}=\{w^{\epsilon}=(h_1,\ldots,h_m;\epsilon),\overline{w}^{\epsilon}\}$ be a signed walk with endpoints $u:=s(h_1)$ and $v:=s(\overline{h_m})$.
    If $u\neq v$, the condition \textnormal{(NC0)} is automatic. 
    Assume that $u=v$. In this case, $\Gamma$ is a Brauer odd-cycle. We denote by $Z$ the unique odd cycle in $\Gamma$, by $\ell$ the length of $Z$. 
    Since $\Gamma$ has no cycles except for $Z$, using (a), we can write $w=pz\overline{p}$, where      
    $p=(h_1,\ldots,h_r)$ with $r\geq 0$ and $z=(h_{r+1},\ldots,h_{r+\ell})$ forms the cycle $Z$. 
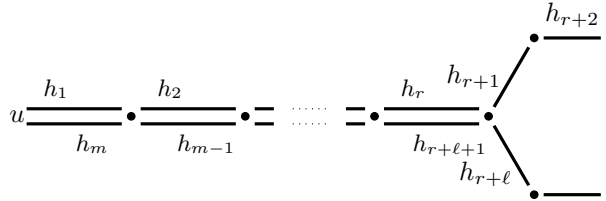
        \[
            \begin{tikzpicture}[baseline=0mm, scale=1]
            \node(1) at(0:0) {};
            \node(2) at(0:1.5) {};
            \node(3) at(0:3) {};
            \node(4) at(0:3.5) {};
            \node(4a) at(0:3.5) {};
            \node(5a) at(0:4.2) {};
            \node(5) at(0:4.2) {};
            \node(6) at(0:4.7) {};
            \node(7) at(0:6.2) {};
        
            \node(1c) at($(0:0)+(90:0.1)$) {};
            \node(2c) at($(0:1.5)+(90:0.1)$) {};
            \node(3c) at($(0:3)+(90:0.1)$) {};
            \node(4c) at($(0:3.5)+(90:0.1)$) {};
            \node(4ac) at($(0:3.5)+(90:0.1)$) {};
            \node(5ac) at($(0:4.2)+(90:0.1)$) {};
            \node(5c) at($(0:4.2)+(90:0.1)$) {};
            \node(6c) at($(0:4.7)+(90:0.1)$) {};
            \node(7c) at($(0:6.2)+(90:0.1)$) {};

            \node(1d) at($(0:0)+(90:-0.1)$) {};
            \node(2d) at($(0:1.5)+(90:-0.1)$) {};
            \node(3d) at($(0:3)+(90:-0.1)$) {};
            \node(4d) at($(0:3.5)+(90:-0.1)$) {};
            \node(4ad) at($(0:3.5)+(90:-0.1)$) {};
            \node(5ad) at($(0:4.2)+(90:-0.1)$) {};
            \node(5d) at($(0:4.2)+(90:-0.1)$) {};
            \node(6d) at($(0:4.7)+(90:-0.1)$) {};
            \node(7d) at($(0:6.2)+(90:-0.1)$) {};

            \coordinate(0) at(0:0);
            \node at(1) {$u$};
            \draw[fill=black] (2)circle(0.5mm) ;
            \draw[fill=black] (3)circle(0.5mm) ;
            \draw[fill=black] (6)circle(0.5mm) ;
            \draw[fill=black] (7)circle(0.5mm) ;
        
            \draw[-,very thick] (1c)--node[above,]{\small $h_{1}$}($(1c)!0.6!(2c)$)--(2c); 
            \draw[-,very thick] (2c)--node[above]{\small $h_{2}$}($(2c)!0.6!(3c)$)--(3c); 
            \draw[-,very thick] (3c)--(4c);
            \draw[dotted] (4ac)--(5ac); 
            \draw[-,very thick] (5c)--(6c);
            \draw[-,very thick] (6c)--node[above]{\small $h_{r}$}($(6c)!0.6!(7c)$)--(7c);

            \draw[-,very thick] (2d)--node[below]{\small $h_{m}$}($(2d)!0.6!(1d)$)--(1d); 
            \draw[-,very thick] (3d)--node[below]{\small $h_{m-1}$}($(3d)!0.6!(2d)$)--(2d); 
            \draw[-,very thick] (3d)--(4d);
            \draw[dotted] (4ad)--(5ad); 
            \draw[-,very thick] (5d)--(6d);
            \draw[-,very thick] (7d)--node[below]{\small $h_{r+\ell+1}$}($(7d)!0.6!(6d)$)--(6d);

            \node(z1) at($(7)+(60:1.2)$) {}; 
            \node(z2) at($(z1)+(0:1)$) {}; 
            
            \draw[very thick] (7)--node[left]{$h_{r+1}$}(z1);
            \draw[fill=black] (z1)circle(0.5mm);
            \draw[very thick] (z1)--node[above]{$h_{r+2}$}(z2);

            \node(zl) at($(7)+(-60:1.2)$) {}; 
            \node(zl1) at($(zl)+(0:1)$) {}; 
            \draw[very thick] (zl)--node[left]{$h_{r+\ell}$}($(zl)!0.4!(7)$)--(7);
            \draw[fill=black] (zl)circle(0.5mm);
            \draw[very thick] (zl)--(zl1);

            \draw[dotted] (z2)--(zl1);
            
        \end{tikzpicture}
        \]
    Since $\ell$ is odd, so is $m=2r+\ell$. Thus, $\epsilon(h_m) = (-1)^{m-1}\epsilon(h_1)=\epsilon(h_1)$ as desired. 
\end{proof}

By Lemma \ref{lem:SC}(b), for a signed walk $W=\{w^{\epsilon}=(h_1,\ldots,h_m;\epsilon), \overline{w}^{\epsilon}\}$, one can extend the signature $\epsilon$ to the endpoints $s(w)=s(h_1)$ and $s(\overline{w})=s(\overline{h_m})$ of $W$ 
by $\epsilon(s(w)):=\epsilon(h_1)$ and $\epsilon(s(\overline{w})):=\epsilon(h_m)$. 
Moreover, one can assign the sign $\mathfrak{o}\epsilon(s(w)):=\mathfrak{o}(s(w))\epsilon(s(w))\in \{\pm1\}$ on $s(w)$, and similarly $\mathfrak{o}\epsilon(s(\overline{w}))$ on $s(\overline{w})$.

The following result means that the element $\partial([W^{\epsilon}])$ is completely determined by the endpoints of $W$ and the signature on them.

\begin{lemma} \label{lem:partialWE}
    For a signed walk $W^{\epsilon}=\{w^{\epsilon}, \overline{w}^{\epsilon}\}$ of $\Gamma$, the following holds.
    \begin{equation} \label{del_sw}
        \partial([W^{\epsilon}]) = 
        \begin{cases}
            \mathfrak{o}\epsilon(s(w))  (s(w) - s(\overline{w})) & \text{if $W$ is a walk of $\Gamma_{\rm sp}$,} \\ 
            \mathfrak{o}\epsilon(s(w)) (s(w) + s(\overline{w})) & \text{else.}
        \end{cases}
    \end{equation}    
\end{lemma}

\begin{proof} 
    Let $W^{\epsilon}=\{w^{\epsilon}=(h_1,\ldots,h_m;\epsilon),\overline{w}^{\epsilon}\}$ be a signed walk. 
    By (\ref{eq:classOfwalk}) and (\ref{eq:map_delta}), we have 
    \begin{eqnarray*}
        [W^{\epsilon}] &=& \sum_{i=1}^m \epsilon(h_i) [h_i] \in \mathbb{Z}E \quad \text{and}\\
        \partial([h_i]) &=& \begin{cases} 
            \mathfrak{o}(s(h_i))(s(h_i) - s(h_{i+1})) & \text{if $h_i\in H_{\rm sp}$,} \\ 
            \mathfrak{o}(s(h_i))(s(h_i) + s(h_{i+1})) & \text{if $h_i \in H_{\rm sp}^c$,} 
        \end{cases} 
    \end{eqnarray*}
    where $1\leq i \leq m$ and $s(h_{m+1})=s(\overline{h_m})$.
    If $W$ is a walk of $\Gamma_{\rm sp}$ (that is, $h_i\in H_{\rm sp}$ for all $1\leq i\leq m$), we have 
    \begin{eqnarray*}
        \partial([W^{\epsilon}]) &=& \sum_{i=1}^m \epsilon(h_i)\partial([h_i]) \nonumber 
        = \mathfrak{o}\epsilon(s(h_1)) 
        \sum_{i=1}^m (s(h_{i})-s(h_{i+1})) \nonumber\\
        &=& \mathfrak{o}\epsilon(s(h_1))  (s(h_1) - s(\overline{h_m})) \nonumber 
        = \mathfrak{o}\epsilon(s(w))  (s(w) - s(\overline{w})).
    \end{eqnarray*} 
    Otherwise, by Lemma \ref{lem:nonsp}(a), there exists a unique integer $1\leq j \leq m$ such that $h_j\in H_{\rm sp}^c$. Then we have 
    \begin{eqnarray*}
        \partial([W^{\epsilon}]) &=& \sum_{i=1}^m \epsilon(h_i)\partial([h_i]) \nonumber \\
        &=& \mathfrak{o}\epsilon(s(h_1)) \left(\sum_{i=1}^{j-1} (s(h_i)-s(h_{i+1}))) + (s(h_j) + s(h_{j+1})) -  \sum_{i=j+1}^m (s(h_i)-s(h_{i+1}))\right) \nonumber\\ 
        &=& \mathfrak{o}\epsilon(s(h_1)) (s(h_1) + s(\overline{h_m})) \nonumber
        = \mathfrak{o}\epsilon(s(w)) (s(w) + s(\overline{w})). 
    \end{eqnarray*}
    Thus, we obtain the desired equalities (\ref{del_sw}).  
\end{proof}

\begin{lemma} \label{lem:onsubtree}
Let $T$ be a subtree of $\Gamma$. Then, the following hold. 
\begin{enumerate}[\rm (a)]
    \item Every signed walk of $T$ is admissible on $\Gamma$. 
    \item For two vertices $u,v$ of $T$ with $u\neq v$, 
    there exists a unique walk $W$ having $u,v$ as endpoints and it gives rise to admissible signed walks $W^{\pm \epsilon}$. Conversely, every admissible signed walks of $T$ can be obtained in this way.
\end{enumerate}
\end{lemma}

\begin{proof}
    (a) Let $W^{\epsilon}$ be a signed walk of $T$. By Lemma \ref{lem:SC}(b), $W^{\epsilon}$ satisfies the conditions \textnormal{(NC0)} and \textnormal{(NC1)}. Since it has no proper maximal common subwalks and no intersecting vertices with itself, \textnormal{(NC2)} and \text{(NC3)} are automatic. 
    We conclude that $W^{\epsilon}$ is admissible on $T$, and also on $\Gamma$. 
    
    (b) Let $u,v$ be distinct vertices of $T$. 
    Since $T$ is a tree, there exists a unique half-walk $w=(h_1,\ldots,h_m)$ such that $s(h_1)=u$ and $s(\overline{h_m})=v$ and it gives rise to a walk $W=\{w,\overline{w}\}$ having $u,v$ as endpoints. 
    In this case, $h_1,\ldots,h_m$ are pairwise distinct because $T$ is a tree. 
    Then one can find a signature $\pm\epsilon$ on $W$, and both of $W^{\pm\epsilon}$ are admissible by (a). 
    Conversely, it is easy to see that every admissible signed walk of $T$ can be obtained in this way. 
\end{proof}

Now, we are ready prove Proposition \ref{thm:bijection_roots}.

\begin{proof}[Proof of Proposition \ref{thm:bijection_roots}] 
    In the above notations, we show that the isomorphism
    \[
        \partial \colon \mathbb{Z}E \overset{\sim}{\longrightarrow} L(\Gamma)
    \]
    in Proposition \ref{prop:parisom} restricts to a bijection 
    \begin{equation} \label{eq:bijGamma}
        \partial|_{[\mathsf{AW}(\Gamma)]} \colon [\mathsf{AW}(\Gamma)] \overset{\sim}{\longrightarrow} \Phi(\Gamma).  
    \end{equation}

    We first consider the case when $\Gamma$ is a Brauer tree and hence $\Phi(\Gamma)=\Phi_A$. 
    From (\ref{del_sw}), we have $\partial([\mathsf{AW}(\Gamma)]) \subseteq \Phi_A$ since every signed walk of $\Gamma$ is admissible and has distinct endpoints by Lemma \ref{lem:onsubtree}(b). On the other hand, by Lemma \ref{lem:onsubtree}(b), for two distinct vertices $u,v$ of $\Gamma$, we have two admissible signed walks $W^{\pm \epsilon}$ having $u,v$ as the endpoints and 
    \[
        \partial(\{W^{\epsilon}, W^{-\epsilon}\}) = \{\pm(u-v)\}. 
    \]
    Therefore, we get the assertion in this case.

Next, we assume that $\Gamma$ is a Brauer odd-cycle and hence $\Phi(\Gamma)=\Phi_C$. 
We denote by $Z$ the unique odd cycle in $\Gamma$, by $\ell$ the length of $Z$. For each vertex $v\in V$, we define a subtree $T_v$ of $\Gamma$ as follows: Consider a graph obtained from $\Gamma$ by deleting all (half-)edges lying in $Z$ and then $T_v$ is defined to be its connected component containing the vertex $v$. We write $\check{v}$ for the the unique common vertex of $T_v$ and $Z$. 

By (\ref{del_sw}), we have $\partial([W^{\epsilon}]) \in \Phi_C$ for any signed walk $W^{\epsilon}$ of $\Gamma$. In fact, for the endpoints $u,v$ of $W^{\epsilon}$, we have $\partial([W^{\epsilon}])=\pm u\pm v$ if $u\neq v$, $\partial([W^{\epsilon}])=\pm 2u$ by \textnormal{(NC0)} otherwise. 
Thus, $\partial([\mathsf{AW}(\Gamma)])\subseteq \Phi_C$. 
In order to show the bijectivity, it is enough to show $\Phi_C \subseteq \partial([\mathsf{AW}(\Gamma)])$.

(a) Firstly, we consider the elements of the form $\pm(u-v)\in \Phi_C$ with $u\neq v$. 
Let $u,v$ be distinct vertices of $\Gamma$. Applying Lemma \ref{lem:onsubtree}(b) to the spanning tree $\Gamma_{\rm sp}$, there are two admissible signed walks $W^{\pm\epsilon}$ having $u,v$ as endpoints. 
By (\ref{del_sw}), we have 
\begin{equation*}
    \partial(\{W^{\epsilon}, W^{-\epsilon}\}) = \{\pm (u-v)\}.
\end{equation*}

(b) Secondly, we show that $\pm(u+v)$ belong to $\partial([\mathsf{AW}(\Gamma)])$ for vertices $u\neq v$ with $T_v\neq T_u$.
Recall from Lemma \ref{lem:onsubtree} that we have a unique half-walk $w$ of $\Gamma_{\rm sp}$ such that $u=s(w)$ and $v=s(\overline{w})$. Since $\Gamma$ contains no cycles except for $Z$, $w$ can be written as $w=pzq$, where 
    \begin{itemize}
        \item $p=(p_1,\ldots,p_{m})$ ($m\geq 0$) is a unique half-walk on $T_u$ such that $u=s(p_1)$ and $\check{u}=s(\overline{p_{m}})$, 
        \item $z:=(z_1,\ldots,z_j)$ ($1\leq j < \ell$) is a half-walk consisting of half-edges appearing in $Z$ such that $\check{u}=s(z_1)$ and $\check{v}= s(\overline{z_{j}})$ with $z_i\in H_{\rm sp}$ for all $1\leq i\leq j$, and
        \item $q=(q_1,\ldots,q_{r})$ ($r\geq 0$) is a unique half-walk on $T_v$ such that $\check{v}=s(q_1)$ and $v=s(\overline{q_{r}})$. 
    \end{itemize}
    For a half-walk $z$, there is a half-walk $z'=(z_{j+1},\ldots, z_{\ell})$ such that $zz'=(z_1,\ldots,,z_j,z_{j+1},\ldots,z_{\ell})$ forms the cycle $Z$. 
    Since $s(\overline{z_{\ell}})=\check{u}$ and$s(z_{j+1})=\check{v}$, we obtain a half-walk $w':=p\overline{z'}q$ and a walk $W':=\{w',\overline{w'}\}$ of $\Gamma$. 
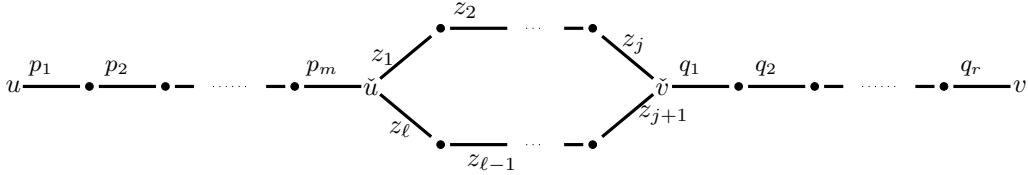
    \[
        \begin{tikzpicture}[baseline=0mm, scale=1]
            \node(1) at(0:0) {};
            \node(2) at(0:1) {};
            \node(3) at(0:2) {};
            \node(4) at(0:2.5) {};
            \node(4a) at(0:2.5) {};
            \node(5a) at(0:3.2) {};
            \node(5) at(0:3.2) {};
            \node(6) at(0:3.7) {};
            \node(7) at(0:4.7) {};
        
            \coordinate(0) at(0:0);
            \node at(1) {$u$};
            \draw[fill=black] (2)circle(0.5mm) ;
            \draw[fill=black] (3)circle(0.5mm) ;
            \draw[fill=black] (6)circle(0.5mm) ;
            \node at(7) {$\check{u}$};
        
            \draw[-,very thick] (1)--node[above,]{\small $p_{1}$}($(1)!0.6!(2)$)--(2); 
            \draw[-,very thick] (2)--node[above]{\small $p_{2}$}($(2)!0.6!(3)$)--(3); 
            \draw[-,very thick] (3)--(4);
            \draw[dotted] (4a)--(5a); 
            \draw[-,very thick] (5)--(6);
            \draw[-,very thick] (6)--node[above]{\small $p_m$}($(6)!0.6!(7)$)--(7);

            \node(z1) at($(7)+(40:1.2)$) {}; 
            \node(z2) at($(z1)+(0:1)$) {}; 
            \node(z3) at($(z2)+(0:0.5)$) {}; 
            \node(z4) at($(z2)+(0:1)$) {}; 
            \node(11) at($(z4)+(-40:1.2)$) {}; 
            \draw[fill=black] (z4)circle(0.5mm) ;
            \node at(11) {$\check{v}$};

            \draw[very thick] (7)--node[left]{$z_1$}(z1);
            \draw[fill=black] (z1)circle(0.5mm);
            \draw[very thick] (z1)--node[above]{$z_2\quad $}(z2);

            \node(zl) at($(7)+(-40:1.2)$) {}; 
            \node(zl1) at($(zl)+(0:1)$) {}; 
            \draw[very thick] (zl)--node[left]{$z_{\ell}$}($(zl)!0.4!(7)$)--(7);
            \draw[fill=black] (zl)circle(0.5mm);
            \draw[very thick] (zl)--node[below]{$\quad z_{\ell-1}$}(zl1);

            \node(-z3) at($(zl1)+(0:0.5)$) {}; 
            \node(-z4) at($(zl1)+(0:1)$) {}; 
            \draw[fill=black] (-z4)circle(0.5mm) ;

            \node(zl) at($(0:0)+(-130:1.6)$) {};
            
            \draw[dotted] (z2)--(z3); 
            \draw[very thick] (z3)--(z4)--node[right]{$z_j$}($(z4)!0.4!(11)$)--(11);

            \draw[dotted] (zl1)--(-z3); 
            \draw[very thick] (11)--node[right]{$z_{j+1}$}(-z4)--(-z3);
        
            \node(12) at($(0:1)+(11)$) {};
            \node(13) at($(0:2)+(11)$) {};
            \node(14) at($(0:2.5)+(11)$) {};
            \node(14a) at($(0:2.5)+(11)$) {};
            \node(15a) at($(0:3.2)+(11)$) {};
            \node(15) at($(0:3.2)+(11)$) {};
            \node(16) at($(0:3.7)+(11)$) {};
            \node(17) at($(0:4.7)+(11)$) {};

            \draw[fill=black] (12)circle(0.5mm) ;
            \draw[fill=black] (13)circle(0.5mm) ;
            \draw[fill=black] (16)circle(0.5mm) ;
            \node at(17) {$v$};
        
            \draw[-,very thick] (11)--node[above,]{\small $q_{1}$}($(11)!0.6!(12)$)--(12); 
            \draw[-,very thick] (12)--node[above]{\small $q_{2}$}($(12)!0.6!(13)$)--(13); 
            \draw[-,very thick] (13)--(14);
            \draw[dotted] (14a)--(15a); 
            \draw[-,very thick] (15)--(16);
            \draw[-,very thick] (16)--node[above]{\small $q_r$}($(16)!0.6!(17)$)--(17);

        \end{tikzpicture}
        \]  
    From our construction, $W'$ is not a walk of $\Gamma_{\rm sp}$ but a walk of a certain subtree of $\Gamma$. By Lemma \ref{lem:onsubtree}(b), it gives rise to two admissible signed walks $W'^{\pm\epsilon'}$. 
    By (\ref{del_sw}), we have 
    \[
        \partial(\{W'^{\epsilon'},W'^{-\epsilon'}\}) = \{\pm(u+v)\}. 
    \]

(c) Thirdly, we consider vertices $u\neq v$ with $T_u=T_v$, i.e., $\check{u}=\check{v}$. 
We will construct $X,Y\in \mathsf{AW}(\Gamma)$ such that 
\begin{equation} \label{eq:casec}
    \partial(\{X,Y\}) = \{\pm (u+v)\}. 
\end{equation}
We may assume $u\neq \check{u}$ by replacing $u$ and $v$ if necessary. 
Then, one can find a half-walk $w=(h_1,\ldots,h_m)$ of the form $w:=pzq$, where 
\begin{itemize}
\item $p=(p_1,\ldots,p_{m})$ ($m\geq 1$) is a unique half-walk on $T_u$ such that $u=s(p_1)$ and $\check{u}=s(\overline{p_m})$, 
\item $z=(z_{1},\ldots,z_{\ell})$ is a half-walk which forms the cycle $Z$ so that $\check{u}=s(z_{1})=s(\overline{z_{\ell}})$, and 
\item $q=(q_{1},\ldots,q_{r})$ ($r\geq 0$) is a unique half-walk on $T_u$ such that $\check{u}=s(q_{1})$ and $v=s(\overline{q_r})$. 
\end{itemize}
\noindent
On the other hand, let $w'$ be a half-walk of $\Gamma$ defined by $w':=p\overline{z}q$. 
Now, we fix a signature $\epsilon$ on a walk $w\in W$. Then, it gives a signature on $w'\in W'$.
By Lemma \ref{lem:SC}(b), all $W^{\pm\epsilon},W'^{\pm\epsilon}$ satisfy \textnormal{(NC0)} and \textnormal{(NC1)}. 
For the admissibility, we need to check the conditions \textnormal{(NC2)} and \textnormal{(NC3)}. 
Depending on the structure of $\Gamma$, we have three cases (c-i)-(c-iii) as follows. 

(c-i) We first consider the case when $r\geq 1$ and $\overline{p_{m}}=q_1$.  In this case, \textnormal{(NC3)} is automatic.  
For a signed walk $W^{\epsilon}$, take an integer $k>0$ such that $t_{i}:=p_{m-k+i}=\overline{q_{i}}$ for all $i\in \{1,\ldots,k\}$ but $\overline{p_{m-k}} \neq q_{k+1}$. 
Since $u\neq v$, at most one of $\overline{p_{m-k}}$ and $q_{k+1}$ is the virtual edge attached to $W^{\epsilon}$. 
Then, $(t_1,\ldots,t_k)$ gives a unique proper maximal common subwalk of $W$ with itself. By definition, $W^{\epsilon}$ satisfies \textnormal{(NC2)} if and only if the cyclic subordering around $\check{u}$ and around $s:=s(t_1)$ accounting the virtual edges is either  
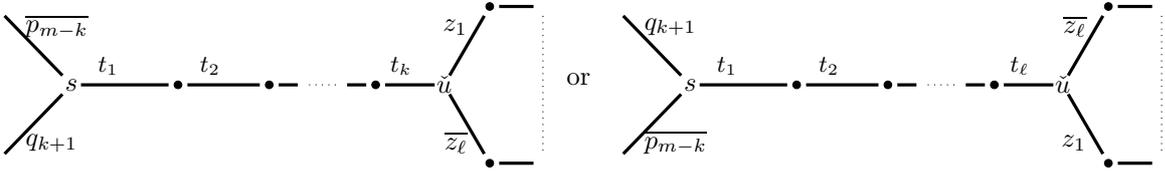
\[
            \hspace{-5mm}\scalebox{1}{
        \begin{tikzpicture}[baseline=0mm, scale=1]
            \node(1) at(0:0.4) {};
            \node(2) at(0:1.8) {};
            \node(3) at(0:3) {};
            \node(4) at(0:3.5) {};
            \node(4a) at(0:3.4) {};
            \node(5a) at(0:4) {};
            \node(5) at(0:3.9) {};
            \node(6) at(0:4.4) {};
            \node(7) at(0:5.3) {};

            \coordinate(0) at(0:0);
            \node at(1) {$s$};
            \draw[fill=black] (2)circle(0.5mm) ;
            \draw[fill=black] (3)circle(0.5mm) ;
            \draw[fill=black] (6)circle(0.5mm) ;
            \node at(7) {$\check{u}$};
        
            \draw[-,very thick] (1)--node[above,]{\small $t_{1}$}($(1)!0.6!(2)$)--(2); 
            \draw[-,very thick] (2)--node[above]{\small $t_{2}$}($(2)!0.6!(3)$)--(3); 
            \draw[-,very thick] (3)--(4);
            \draw[dotted] (4a)--(5a); 
            \draw[-,very thick] (5)--(6);
            \draw[-,very thick] (6)--node[above]{\small $t_{k}$}($(6)!0.6!(7)$)--(7);

            \node(8c) at($(0:0)+(120:1.2)$) {}; 
            \draw[very thick] (8c)--node[right]{$\overline{p_{m-k}}$}($(8c)!0.4!(1)$)--(1);
        
            \node(8d) at($(0:0)+(-120:1.2)$) {};
            \draw[very thick] (8d)--node[right]{$q_{k+1}$}($(8d)!0.4!(1)$)--(1);
        
            \node(z1) at($(7)+(60:1.2)$) {}; 
            \node(z2) at($(z1)+(0:0.7)$) {}; 
            
            \draw[very thick] (z1)--node[left]{$z_{1}$}($(z1)!0.4!(7)$)--(7);
            \draw[fill=black] (z1)circle(0.5mm);
            \draw[very thick] (z1)--node[above]{}(z2);

            \node(zl) at($(7)+(-60:1.2)$) {}; 
            \node(zl1) at($(zl)+(0:0.7)$) {}; 
            \draw[very thick] (zl)--node[left]{$\overline{z_{\ell}}$}($(zl)!0.4!(7)$)--(7);
            \draw[fill=black] (zl)circle(0.5mm);
            \draw[very thick] (zl)--(zl1);

            \draw[dotted] (z2)--(zl1);
            
        \end{tikzpicture}
        }
        \ \text{or} \ 
        \scalebox{1}{
            \begin{tikzpicture}[baseline=0mm, scale=1]
                \node(1) at(0:0.4) {};
                \node(2) at(0:1.8) {};
                \node(3) at(0:3) {};
                \node(4) at(0:3.5) {};
                \node(4a) at(0:3.4) {};
                \node(5a) at(0:4) {};
                \node(5) at(0:3.9) {};
                \node(6) at(0:4.4) {};
                \node(7) at(0:5.3) {};

                \coordinate(0) at(0:0);
                \node at(1) {$s$};
                \draw[fill=black] (2)circle(0.5mm) ;
                \draw[fill=black] (3)circle(0.5mm) ;
                \draw[fill=black] (6)circle(0.5mm) ;
                \node at(7) {$\check{u}$};
            
                \draw[-,very thick] (1)--node[above,]{\small $t_{1}$}($(1)!0.6!(2)$)--(2); 
                \draw[-,very thick] (2)--node[above]{\small $t_{2}$}($(2)!0.6!(3)$)--(3); 
                \draw[-,very thick] (3)--(4);
                \draw[dotted] (4a)--(5a); 
                \draw[-,very thick] (5)--(6);
                \draw[-,very thick] (6)--node[above]{\small $t_{\ell}$}($(6)!0.6!(7)$)--(7);

                \node(8c) at($(0:0)+(120:1.2)$) {}; 
                \draw[very thick] (8c)--node[right]{$q_{k+1}$}($(8c)!0.4!(1)$)--(1);
            
                \node(8d) at($(0:0)+(-120:1.2)$) {};
                \draw[very thick] (8d)--node[right]{$\overline{p_{m-k}}$}($(8d)!0.4!(1)$)--(1);
            
                \node(z1) at($(7)+(60:1.2)$) {}; 
                \node(z2) at($(z1)+(0:0.7)$) {}; 
                
                \draw[very thick] (z1)--node[left]{$\overline{z_{\ell}}$}($(z1)!0.4!(7)$)--(7);
                \draw[fill=black] (z1)circle(0.5mm);
                \draw[very thick] (z1)--node[above]{}(z2);

                \node(zl) at($(7)+(-60:1.2)$) {}; 
                \node(zl1) at($(zl)+(0:0.7)$) {}; 
                \draw[very thick] (zl)--node[left]{$z_{1}$}($(zl)!0.4!(7)$)--(7);
                \draw[fill=black] (zl)circle(0.5mm);
                \draw[very thick] (zl)--(zl1);

                \draw[dotted] (z2)--(zl1);
                
            \end{tikzpicture}
        }
        \] 
By the cyclic ordering around $\check{u}$, precisely one of $W^{\epsilon}$ and $W'^{\epsilon}$ satisfies \textnormal{(NC2)} by taking in account of virtual edges attached to them. 
Similarly, precisely one of $W^{-\epsilon}$ and $W'^{-\epsilon}$ satisfies \textnormal{(NC2)}.  
As a consequence, we obtain the desired admissible signed walks $X$ and $Y$ satisfying (\ref{eq:casec}).

(c-ii) Next, we assume that $r\geq 1$ and $\overline{p_m} \neq q_1$. 
In this case, the condition \textnormal{(NC2)} is automatic. 
For a signed walk $W^{\epsilon}$, $\check{u}$ is the intersecting vertex of $W$ with itself, whose neighbourhoods are $\{a,b\}=\{\overline{p_m},z_1\}$ and $\{c,d\}=\{\overline{z_{\ell}}, q_1\}$. Recall that $W^{\epsilon}$ satisfies \textnormal{(NC3)} if and only only if the cyclic subordering around $\check{u}$ accounting virtual edges and the signatures are either  
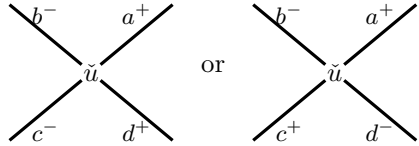
\[
\begin{tikzpicture}[baseline=0mm]
    \node(0) at(0:0) {}; 
    \node at(0) {$\check{u}$};
    \draw[very thick] (0)--node[above]{\small $a^+$}(40:1.4); 
    \draw[very thick] (0)--node[above]{\small $b^-$}(140:1.4); 
    \draw[very thick] (0)--node[below]{\small $c^-$}(-140:1.4); 
    \draw[very thick] (0)--node[below]{\small $d^+$}(-40:1.4); 
\end{tikzpicture}
\quad \text{or} \quad  
\begin{tikzpicture}[baseline=0mm]
    \node(0) at(0:0) {}; 
    \node at(0) {$\check{u}$};
    \draw[very thick] (0)--node[above]{\small $a^+$}(40:1.4); 
    \draw[very thick] (0)--node[above]{\small $b^-$}(140:1.4); 
    \draw[very thick] (0)--node[below]{\small $c^+$}(-140:1.4); 
    \draw[very thick] (0)--node[below]{\small $d^-$}(-40:1.4); 
\end{tikzpicture}
\]
From the cyclic ordering around $\check{u}$, 
precisely one of $W^{\epsilon}$ and $W'^{\epsilon}$ satisfies \textnormal{(NC3)}.  
Similarly, precisely one of $W^{-\epsilon}$ and $W'^{-\epsilon}$ satisfies \textnormal{(NC3)}. Then, we obtain the desired $X$ and $Y$ satisfying (\ref{eq:casec}). 

(c-iii) Lastly, we assume that $r=0$, i.e., $v=\check{u}$. In this case, by letting $q_1:=z_{\ell+1}$ be the virtual edge attached to $W^{\epsilon}$, a similar argument as in (c-ii) gives the desired $X$ and $Y$ satisfying (\ref{eq:casec}). 

(d) Finally, we show that $\pm2u\in \partial([\mathsf{AW}(\Gamma)])$ for any vertex $u\in V$. 
By the same way as a proof of Lemma \ref{lem:SC}(b), one can find a half-walk $w=(h_1,\ldots,h_m)$ of the form $w=pz\overline{p}$, where 
\begin{itemize} 
    \item $p=(h_1,\ldots,h_{r})$ ($r\geq 0$) is a unique half-walk on $T_u$ such that $u=s(h_1)$ and $\check{u}=s(\overline{h_r})$ and 
    \item $z=(h_{r+1},\ldots,h_{r+\ell})$ is a half-walk which 
    forms the cycle $Z$ so that $\check{u}=s(h_{r+1})=s(\overline{h_{r+\ell}})$. 
\end{itemize}
Then, it gives rise to a walk $W$ equipped with signatures $\pm\epsilon$. By Lemma \ref{lem:SC}(b), it satisfies the conditions \textnormal{(NC0)} and \textnormal{(NC1)}. 
Since the endpoints of $W$ are the same, the condition \textnormal{(NC2)} is automatic. 
Now, we consider (NC3). 
This is automatic if $u\neq \check{u}$. 
Assume that $u=\check{u}$ and let $\{a,b\}=\{\overline{h_r},h_{r+1}\}$ and $\{c,d\}=\{\overline{h_{r+\ell}},h_{r+\ell+1}\}$ be neighbourhoods of $\check{u}$ in $\Gamma$. Since two of $a,b,c,d$ are virtual, the condition \textnormal{(NC3)} is satisfied. 
Therefore, both of $W^{\pm\epsilon}$ are admissible. By (\ref{del_sw}), we have
\[
    \partial(\{W^{\epsilon},W^{-\epsilon}\}) = \{\pm 2u\}. 
\]

By (a)-(d), we conclude that the map (\ref{eq:bijGamma}) is bijective.  
\end{proof}

We end this subsection with a proof of Theorem \ref{BTA-BOA-Phi}. 

\begin{proof}[Proof of Theorem \ref{BTA-BOA-Phi}] 
    Let $\Gamma$ be a Brauer graph having $n$ edges and $B:=B_{\Gamma}$ the Brauer graph algebra of $\Gamma$. 

    (a) Thanks to Proposition \ref{field extension}(b), we can assume that a base field $k$ is algebraically closed.  
    In order to show that the algebra $B$ is pairwise $g$-convex, we first consider a left mutation of $B$ with respect to the indecomposable projective $B$-module $P$ so that the exchange triangle is 
    \begin{equation} \label{triangleP}
        P \to U \to P' \to P[1].  
    \end{equation}
    The triangle (\ref{triangleP}) is explicitly described in \cite[Section 6]{Ai}, in particular, the number of indecomposable direct summands of $U$ is at most two. 
    
    Next, let $T\in \twosilt B$. Then $T$ is a tilting complex since $B$ is a symmetric algebra \cite[Example 2.8]{AI}. We have a triangle equivalence 
    \[
        F\colon \Db(\mod B) \overset{\sim}{\longrightarrow} \Db(\mod \End_{\Db(\mod B)}(T))
    \]
    mapping $T$ to $\End_{\Db(\mod B)}(T)$. 
    By \cite[Corollary 1.3]{AZ}, we have $\End_{\Db(\mod B)}(T)\cong B_{\Gamma'}$ for some Brauer graph $\Gamma'$ having $n$ edges. 
    Since $F$ sends an exchange triangle $X\to U'\to Y \to X[1]$ of $T$ to an exchange triangle of $B_{\Gamma'}$, by applying the argument above to $B_{\Gamma'}$, the number of indecomposable direct summands of $U'$ is at most two. 

    Consequently, $B$ is pairwise $g$-convex.

    We show the latter assertion.  
    (i)$\Leftrightarrow$(iii) is \cite[Theorem 6.7]{AAC}.  
    (i)$\Leftrightarrow$(ii) follows from Theorem \ref{characterize g-convex}(b). 

(b) Assume that $\Gamma$ is a Brauer tree or a Brauer odd-cycle. 
Combining (\ref{eq:psaw}) and (\ref{eq:EL}), we obtain the following commutative diagram. 
\begin{equation*} \label{eq:desired isom}
\xymatrix{ 
    K_0(\proj B_{\Gamma}) \ar[r]^-{\sim} \ar@{}[d]|{\bigcup} & \mathbb{Z}E \ar@{}[d]|{\bigcup}\ar[r]^-{\sim} & L(\Gamma) \ar@{}[d]|{\bigcup} \\
    [\twopsilt^1 B_{\Gamma}] \ar[r]^-{\sim} & [\mathsf{AW}(\Gamma)] \ar[r]^-{\sim}& \Phi(\Gamma). 
    }
\end{equation*}
Thus, we obtain (\ref{eq:twosilt-aw}), which clearly gives rise to the desired isomorphism $\P(B_{\Gamma})\cong P_{\Phi(\Gamma)}$ of lattice polytopes. 
The last assertion follows from \cite[Theorem 1]{ABHPS}. 
\end{proof}


\section*{Acknowledgments} 
T.A is supported by JSPS Grants-in-Aid for Scientific Research JP19J11408. 
A.H is supported by JSPS Grant-in-Aid for Scientists Research (C) 20K03513.
O.I is supported by JSPS Grant-in-Aid for Scientific Research (B) 16H03923, (C) 18K3209 and (S) 15H05738. 
R.K is supported by JSPS Grant-in-Aid for Young Scientists (B) 17K14169. 
Y.M is supported by 
Grant-in-Aid for Scientific Research (C) 20K03539. 
Y.M would like to thank Haruhisa Enomoto and Hiroyuki Minamoto for helpful discussions

\end{document}